\documentclass[10pt,reqno,a4paper]{amsart}
\usepackage[dvipsnames]{xcolor}

\usepackage{amsmath, amssymb, amsfonts, latexsym, mdwlist, amsthm}
\usepackage{subfig}
\usepackage{graphicx}
\usepackage{tikz-cd}
\usepackage{wrapfig}
\usepackage{mathrsfs}

\tikzstyle{startstop} = [rectangle, rounded corners, minimum width=3cm, minimum height=1cm,text centered, draw=black, fill=red!30]
\tikzstyle{io} = [trapezium, trapezium left angle=70, trapezium right angle=110, minimum width=3cm, minimum height=1cm, text centered, draw=black, fill=blue!30]
\tikzstyle{process} = [rectangle, minimum width=3cm, minimum height=1cm, text centered, draw=black, fill=orange!30]
\tikzstyle{decision} = [diamond, minimum width=3cm, minimum height=1cm, text centered, draw=black, fill=green!30]
\tikzstyle{arrow} = [thick,->,>=stealth]

\usepackage{mathtools}
\usepackage{tikz}
\usepackage{comment}

\usetikzlibrary{patterns}

\usepackage{enumerate}

\usepackage{amsmath, amssymb, amsfonts, latexsym, mdwlist, amsthm}
\usepackage{subfig}
\usepackage{wrapfig}

\usepackage{mathtools}
\usepackage{bm}

\numberwithin{equation}{section}

\definecolor{lightred}{HTML}{ff4d4d}
\definecolor{lightblue}{HTML}{1F88CD}
\definecolor{lightgrey}{HTML}{727272}
\definecolor{lightblue2}{HTML}{009EC1}
\definecolor{mypink}{HTML}{FD00B0}

\usepackage[bookmarks, colorlinks, breaklinks, pdftitle={},
pdfauthor={Zhiyu Liu}]{hyperref}
\hypersetup{linkcolor=OliveGreen,citecolor=MidnightBlue,filecolor=black,urlcolor=blue}

\usepackage{tikz}
\usetikzlibrary{calc,trees,positioning,arrows,chains,shapes.geometric,%
    decorations.pathreplacing,decorations.pathmorphing,shapes,%
    matrix,shapes.symbols}

\tikzset{
>=stealth',
  punktchain/.style={
    rectangle,
    rounded corners,
    draw=black, thick,
    minimum height=3em,
    text centered,
    on chain},
  line/.style={draw, thick, <-},
  element/.style={
    tape,
    top color=white,
    bottom color=blue!50!black!60!,
    minimum width=8em,
    draw=blue!40!black!90, very thick,
    text width=10em,
    minimum height=3.5em,
    text centered,
    on chain},
  every join/.style={->, thick,shorten >=1pt},
  decoration={brace},
  tuborg/.style={decorate},
  tubnode/.style={midway, right=2pt},
}

\usepackage{paralist}
\setdefaultenum{(a)}{(i)}{}{}
\usepackage[shortlabels]{enumitem} 
\usepackage{derivative}

\makeatletter
\newtheorem*{rep@theorem}{\rep@title}
\newcommand{\newreptheorem}[2]{%
\newenvironment{rep#1}[1]{%
 \def\rep@title{#2 \ref{##1}}%
 \begin{rep@theorem}}%
 {\end{rep@theorem}}}
\makeatother

\usepackage[titletoc,title]{appendix}
\newtheorem{theorem}{Theorem}[section]
\newreptheorem{theorem}{Theorem}
\newtheorem{proposition}[theorem]{Proposition}

\newtheorem{lemma}[theorem]{Lemma}

\newtheorem{corollary}[theorem]{Corollary}
\newreptheorem{corollary}{Corollary}
\newtheorem{conjecture}[theorem]{Conjecture}
\newreptheorem{conjecture}{Conjecture}

\newtheorem{Assum}[theorem]{Assumption}
\newtheorem{Setup}[theorem]{Setup}

\newtheorem{thm-int}{Theorem}

\theoremstyle{definition}
\newtheorem{Def-s}[theorem]{Definition}
\newtheorem{definition}[theorem]{Definition}

\newtheorem{example}[theorem]{Example}
\newtheorem{remark}[theorem]{Remark}

\newcommand{\ignore}[1]{}

\usepackage[numbers]{natbib}
\newcommand{\ra}{\rightarrow}
\newcommand{\xra}{\xrightarrow}

\newcommand{\wt}{\widetilde}

\newcommand{\D}{\mathrm{D}}
\newcommand{\HN}{\mathrm{HN}}
\newcommand{\tf}{\mathrm{tf}}
\newcommand{\LL}{\mathrm{L}}
\newcommand{\Db}{\mathrm{D}^{\mathrm{b}}}
\newcommand{\Dperf}{\mathrm{D}_{\mathrm{perf}}}
\newcommand{\Dqc}{\mathrm{D}_{\mathrm{qc}}}
\newcommand{\Supp}{\mathrm{Supp}}

\newcommand{\sT}{\mathsf{T}}
\newcommand{\Td}{\mathrm{Td}}
\newcommand{\tor}{\mathrm{tor}}
\newcommand{\Ass}{\mathrm{Ass}}
\newcommand{\depth}{\mathrm{depth}}

\newcommand{\ZZ}{\mathbb{Z}}

\newcommand{\QQ}{\mathbb{Q}}

\newcommand{\CC}{\mathbb{C}}
\newcommand{\RR}{\mathbb{R}}
\newcommand{\PP}{\mathbb{P}}

\newcommand{\DD}{\mathbb{D}}
\newcommand{\kk}{\mathsf{k}}
\newcommand{\bv}{\mathbf{v}}

\newcommand{\bp}{\mathbf{p}}
\newcommand{\bch}{\mathbf{ch}}
\newcommand{\st}{\mathrm{st}}
\newcommand{\A}{\mathrm{A}}

\newcommand{\CH}{\mathrm{CH}}

\newcommand{\Dpug}{\mathrm{D}_{\mathrm{pug}}}

\newcommand{\ch}{\mathrm{ch}}

\newcommand{\codim}{\mathrm{codim}}

\newcommand{\td}{\mathrm{td}}
\newcommand{\bn}{\mathrm{BN}}

\newcommand{\KK}{\mathrm{K}}
\newcommand{\Knum}{\mathrm{K}_{\mathrm{num}}}
\newcommand{\num}{\mathrm{num}}

\newcommand{\perf}{\mathrm{perf}}

\newcommand{\pr}{\mathrm{pr}}

\newcommand{\bDelta}{\overline{\Delta}}

\renewcommand{\Re}{\operatorname{Re}}
\renewcommand{\Im}{\operatorname{Im}}

\DeclareMathOperator{\im}{im}

\DeclareMathOperator{\rk}{rk}
\DeclareMathOperator{\cok}{cok}

\DeclareMathOperator{\Coh}{\mathrm{Coh}}

\DeclareMathOperator{\Ext}{Ext}

\DeclareMathOperator{\Hom}{Hom}
\DeclareMathOperator{\RHom}{RHom}

\DeclareMathOperator{\Spec}{Spec}
\DeclareMathOperator{\Pic}{Pic}

\DeclareMathOperator{\cone}{cone}
\DeclareMathOperator{\Stab}{Stab}
\DeclareMathOperator{\Gr}{Gr}

\newcommand{\cX}{\mathcal{X}}

\newcommand{\cZ}{\mathcal{Z}}
\newcommand{\cC}{\mathcal{C}}
\newcommand{\cA}{\mathcal{A}}
\newcommand{\cE}{\mathcal{E}}
\newcommand{\cU}{\mathcal{U}}
\newcommand{\cH}{\mathcal{H}}
\newcommand{\cR}{\mathcal{R}}
\newcommand{\cB}{\mathcal{B}}

\newcommand{\cI}{\mathcal{I}}

\newcommand{\cT}{\mathcal{T}}
\newcommand{\cQ}{\mathcal{Q}}

\newcommand{\cP}{\mathcal{P}}
\newcommand{\cD}{\mathcal{D}}
\newcommand{\cL}{\mathcal{L}}
\newcommand{\cN}{\mathcal{N}}
\newcommand{\cM}{\mathcal{M}}
\newcommand{\rH}{\mathrm{H}}
\newcommand{\cO}{\mathcal{O}}

\newcommand{\bA}{\mathbb{A}}

\DeclareMathOperator{\cF}{\mathcal{F}}
\DeclareMathOperator{\cG}{\mathcal{G}}

\DeclareMathOperator{\oh}{\mathcal{O}}

\usepackage{hyperref}
\hypersetup{
	colorlinks=true,
    linkcolor={blue},
    citecolor={blue},
	urlcolor={black}
}

\begin{document}

\title[Tilt-stability on singular schemes and BG-type inequalities]{Tilt-stability on singular schemes and Bogomolov--Gieseker-type inequalities}
\subjclass[2020]{14F08 (Primary); 14F06, 14D06, 14D20 (Secondary)}
\keywords{Derived categories, Tilt-stability, Generalized Bogomolov--Gieseker inequalities, Stability conditions}

\author{Zhiyu Liu}
\address{School of Mathematical Sciences, Zhejiang University, Hangzhou, Zhejiang Province 310058, P. R. China}
\email{jasonlzy0617@gmail.com}

\author{Tianle Mao}
\address{Kavli Institute for the Physics and Mathematics of the Universe,
The University of Tokyo\\ 
5-1-5 Kashiwanoha, Kashiwa, Chiba, 277-8583, Japan}
\email{tianle.mao@ipmu.jp}

\begin{abstract}
We generalize the framework of tilt-stability to singular schemes and formulate the generalized Bogomolov--Gieseker inequality conjecture of Bayer--Macr\`i--Toda for singular threefolds. We also develop relative versions of these constructions, generalizing corresponding results in \cite{BLMNPS21}. Along the way, we establish Bogomolov--Gieseker-type inequalities for semistable sheaves on any projective scheme.


By extending previous techniques, we verify the conjecture for all Fano threefolds with canonical Gorenstein $\mathbb{Q}$-factorial singularities and a series of singular Calabi--Yau threefolds. Furthermore, we construct stability conditions on the relative Kuznetsov components associated with families of singular Fano threefolds, thereby proving a singular analogue of a conjecture of Kuznetsov--Shinder.

\end{abstract}

\vspace{-1em}
\maketitle

\setcounter{tocdepth}{1}
\tableofcontents

\section{Introduction}

Motivated by Douglas’s work on $\Pi$-stability, Bridgeland introduces the notion of stability conditions on triangulated categories \cite{bridgeland:stability}. More recently, Li \cite{li:remark} proves the existence of stability conditions on the derived categories of projective schemes, while the relative version and the properness of the corresponding moduli spaces are established in \cite{LLLMPSZ}.

For many geometric applications, it is more effective to work with \emph{tilt-stability} rather than full stability conditions. The construction of tilt-stability is introduced by \cite{bridgeland:stab-on-K3,aaron:bridgeland-moduli-k-trivial}, which provides examples of stability conditions on surfaces. Later, it is extended to any smooth projective variety in \cite{bayer2011bridgeland,bayer2016space}. Although tilt-stability is only a weak stability condition on higher-dimensional varieties (cf.~Definition \ref{def:stability-condition}), it provides a powerful and computable tool to study sheaves on varieties and has many applications, including birational geometry of moduli spaces of sheaves \cite{bayer:projectivity,aaron:bridgeland-moduli-k-trivial}, hyper-K\"ahler geometry \cite{bayer:mmp}, Brill--Noether theory of curves \cite{bayer:brill-noether,feyz-li:clifford-indices,bayer:brill-noether-abelian-surface}, effective restriction problems \cite{feyzbakhsh:effective-restriction,feyz:mukai-program}, etc.

Moreover, Bayer--Macr\`i--Toda \cite{bayer2011bridgeland,bayer2016space} introduced a conjectural Bogomolov--Gieseker-type inequality involving $\ch_3$ of tilt-stable objects. Although originally proposed as a route to construct stability conditions on smooth threefolds, such a $\ch_3$-inequality also plays a crucial role in enumerative geometry of Calabi--Yau threefolds \cite{toda:bogomolov-counting,feyz:curve-counting,feyz:rank-r-dt-theory-from-0,feyz:rank-r-dt-theory-from-1,feyz:physics-abelian-dt,liu-ruan:cast-bound,liu:cast-bound-3fold} and in birational geometry \cite{bbmt14}.




On the other hand, singularities arise naturally even when one is ultimately interested only in smooth varieties. For instance, it is often useful to study degenerations of smooth varieties to singular schemes with richer geometric structures. Motivated by degeneration techniques in enumerative geometry (cf.~\cite{PP12}), it is natural to expect that Bayer--Macr\`i--Toda's (BMT) Conjecture might be approached through such a method.

Singular varieties are also unavoidable in birational geometry. A conjectural relation between the minimal model program for smooth threefolds and stability conditions is proposed in \cite{toda:ex-contract,toda:mmp-surface}. Since the threefolds appearing in the program are terminal and $\QQ$-factorial, it is therefore natural to investigate tilt-stability and the BMT Conjecture for these singular threefolds. We refer to Section \ref{subsec:mot} for a more detailed discussion of these motivations and related questions.

To generalize the theory above to singular varieties and pursue applications, the first step is to develop a theory of tilt-stability in the singular setting. The main results of this paper can be summarized as follows.

\begin{itemize}
    \item We extend tilt-stability to singular schemes, formulate the BMT Conjecture in this setting, and develop relative versions of these constructions, establishing their basic properties.

    \item We verify the BMT Conjecture for singular Fano threefolds and a series of singular Calabi--Yau threefolds.

    \item We prove a semistable reduction theorem for tilt-semistable objects, showing that the BMT Conjecture can be verified via degeneration.

    \item We construct stability conditions on relative Kuznetsov components associated with families of singular Fano threefolds.
\end{itemize}


\subsection{Tilt-stability}

To define tilt-stability on singular schemes, we need a notion of the second Chern character for coherent sheaves. Motivated by the Mumford Chern character used in \cite{langer:normal-surface}, we introduce a homomorphism
\[\bch_i(-)\colon \KK(\Coh(X))\to \CH_{\dim X-i}(X)_{\QQ}\]
for $i\leq d$ and any quasi-projective scheme $X$ over a field $\kk$ that is a \emph{local complete intersection (lci)} in codimension $d$, where $\CH_{k}(X)_{\QQ}$ is the rational Chow group of $k$-dimensional cycles in $X$. As explained in Section \ref{subsec:generalization-ch}, these maps $\bch_i$ satisfy most of the functorial properties familiar from the smooth case.


Now, fix a projective scheme $X$ of dimension $n\geq 2$ over a field $\kk$ that is a local complete intersection in codimension $2$ equipped with an ample divisor $H$. The usual slope function of coherent sheaves can be written as
\[\mu_H(E)\coloneqq \frac{\bch_1(E).H^{n-1}}{\bch_0(E).H^n}\]
for $E\in \Coh(X)$ with $\bch_0(E)\neq 0$, and $+\infty$ otherwise. This allows the construction of tilt-stability on smooth varieties to be extended to this setting: for any $(b,w)\in \RR^2$, we define a full subcategory
\begin{align*}
\Coh^b_H(X)\coloneqq &\Big\{E\in \Db(X)\colon \cH^{i}(E)=0  \text{ for }  i\notin \{-1,0\}, 
\text{ any subsheaf } 0\neq F\subset \cH^{-1}(E) \\
 &\text{ satisfies } \mu_H(F)\leq b, \text{ any quotient sheaf } \cH^{0}(E)\twoheadrightarrow G\neq 0 \text{ satisfies } \mu_H(G)> b\Big\}
\end{align*}
and a homomorphism $Z^{b,w}\colon  \KK(\Coh(X))\to \CC$
\[Z^{b,w}(-)\coloneqq -\bch_2(-).H^{n-2}+w\bch_0(-).H^n+\mathfrak{i}(\bch_1(-).H^{n-1}-b\bch_0(-).H^n).\]
Our first main result is:

\begin{theorem}\label{thm:intro-1}
With the above notation, we define 
\[\Phi_{X, H}(x)\coloneqq \limsup_{\mu\to x} \left\{\frac{\bch_2(E).H^{n-2}}{\bch_0(E).H^{n}} \colon E\text{ is }\mu_H\text{-semistable and }\mu_H(E)=\mu\right\}\in \RR\cup\{\pm \infty\}\footnote{In our paper, $\limsup$ is defined using unpunctured neighborhoods. We define $\limsup$ of $\varnothing$ to be $-\infty$.}.\]
Then 

\begin{enumerate}
    \item \emph{(Theorem \ref{thm:exist-bg-function})} there exists a constant $\mathsf{D}_{X,H}\geq 0$ so that $\Phi_{X, H}(x)\leq \frac{1}{2}x^2+\mathsf{D}_{X,H}$,

    \item \emph{(Theorem \ref{thm-bms})} for any $w>\Phi_{X, H}(b)$, the pair $(\Coh^b_H(X), Z^{b,w})$ is a weak stability condition on $\Db(X)$, and 

    \item \emph{(Theorem \ref{thm:wall-chamber-tilt})} there exists a wall-chamber structure on $\{(b,w)\in \RR^2\colon w>\Phi_{X, H}(b)\}\subset \RR^2$.
\end{enumerate}
\end{theorem}

The function $\Phi_{X, H}(x)$ is called \emph{the Le Potier function} in \cite{li:stab-manifold-finite-albanese}. For smooth varieties over a field of characteristic $0$, the classical Bogomolov--Gieseker inequality implies $\Phi_{X, H}(x)\leq \frac{1}{2}x^2$.

For smooth projective varieties and normal projective surfaces over any field, Theorem \ref{thm:intro-1}(a) is proved in \cite{koseki:bg-positive-char,langer:normal-surface}. Our inequality may be regarded as a generalization, although the proof uses a different approach. When $X$ has rational singularities and $\mathrm{char}(\kk)=0$, we can take $\mathsf{D}_{X, H}=0$ (cf.~Theorem \ref{thm-bg-normal}).

\begin{remark}
Parts (b) and (c) of Theorem \ref{thm:intro-1} follow from more general theorems for tilting of abstract weak stability conditions on triangulated categories, see Theorem \ref{thm:tilt-stability} and \ref{thm:wall-chamber-abstract}. In the geometric setting, tilt-stability satisfies many expected properties as listed in Section \ref{subsec-titl-from-slope} and \ref{subsec:lemma-about-tilt}.  

Note that Theorem \ref{thm:intro-1}(b) and (c) also hold for normal projective surfaces. In this case, $\bch(-)$, the corresponding  Bogomolov--Gieseker inequality, and tilt-stability have already been established in \cite{langer:normal-surface}.
\end{remark}

\begin{remark}
In Appendix \ref{appendix:approach}, we present an alternative treatment of tilt-stability on arbitrary projective schemes. Rather than using the Chern characters $\bch_i(-)$, this approach defines tilt-stability using the (normalized) coefficients of the Hilbert polynomial. Furthermore, we establish an analogue of Theorem \ref{thm:intro-1} within this framework. Specifically, we prove a BG-type inequality involving the first three coefficients of the Hilbert polynomial for semistable sheaves on any projective scheme (cf.~Theorem \ref{thm-bg-general}), which may be of independent interest. 

However, the primary focus of this paper is the BMT Conjecture, which requires a sharper Bogomolov-Gieseker-type inequality and a well-behaved notion of $\bch_3$, so we do not explore Appendix \ref{appendix:approach} further in the main text.


\end{remark}


\subsection{Bayer--Macr\`i--Toda Conjecture}

We next formulate the singular analogue of the BMT Conjecture. Fix a $3$-dimensional projective scheme $X$ over a field $\kk$ which is lci in codimension $2$ and an ample divisor $H$. Then $\bch_i$ is well-defined for $0\leq i\leq 2$. We denote by $$\nu_{b,w}(-)\coloneqq -\frac{\Re Z^{b,w}(-)}{\Im Z^{b,w}(-)}$$ the slope function associated with $Z^{b,w}$, so we can define $\nu_{b,w}$-semistable objects in $\Coh^b_H(X)$ in the usual way. 

Furthermore, if $X$ is either lci, or $\QQ$-factorial and normal, then we can define $\bch_3(E)\in \QQ$ for any $E\in \Db(X)$ as in Definition \ref{def:ch-lci} and \ref{def:ch3}. In these two cases, the following conjecture generalizes \cite[Conjecture 4.1]{bayer2016space} and \cite[Question 2.4]{macri:fano-threefold}.

\begin{conjecture}[{Conjecture \ref{conj-2}, simplified version}]\label{intro:conj}
In the above setting, take a constant $\mathsf{D}\geq 0$ so that $$\Phi_{X, H}(x)\leq \frac{1}{2}x^2+\mathsf{D}.$$ Then there exists $\Gamma\in \mathsf{E}(X)_{\QQ}$ with $\Gamma.H\geq 0$, such that for any $w>\frac{1}{2}b^2+\mathsf{D}$ and any $\nu_{b,w}$-semistable object $E\in \Coh^b_H(X)$, we have a quadratic inequality as in Remark \ref{rmk:ch3}, involving $\bch_i(E).H^{3-i}$ and $\Gamma.\bch_1(E)$, whose coefficients depend on $\Gamma. H$, the parameters $(b,w)$, and the constant $\mathsf{D}$.
\end{conjecture}

\begin{remark}
As discussed above, the constant $\mathsf{D}$ in the formulation always exists, and should be viewed as a correction term to the classical Bogomolov–Gieseker inequality for $\mu_H$-semistable sheaves, reflecting singularities of $X$ and $\mathrm{char}(\kk)$.
\end{remark}

Here, $\mathsf{E}(X)_{\QQ}$ denotes the bivariant Chow group $\A^2(X)_{\QQ}$ of degree $2$ when $X$ is lci, or $\CH_1(X)_{\QQ}$ when $X$ is $\QQ$-factorial but not lci. In particular, we always have an intersection pairing $\mathsf{E}(X)_{\QQ}\times \CH_2(X)_{\QQ}\to \QQ$. 


As in the smooth case, the above conjecture has two equivalent formulations; see Conjecture \ref{conj-1} and \ref{conj-3}. Moreover, by Theorem \ref{thm-stab}, Conjecture \ref{intro:conj} allows us to construct an explicit family of stability conditions on $X$, which generalizes the corresponding results in \cite{bayer2016space,macri:fano-threefold}.

Conjecture \ref{intro:conj} has been proved for many smooth threefolds, including smooth Fano threefolds \cite{li:fano-3fold,macri:fano-threefold}, smooth Calabi--Yau complete intersections in weighted projective spaces \cite{Li19,FKLR,koseki:stability-cy-solid,liu:bg-ineqaulity-quadratic}, and a series of smooth Calabi--Yau threefolds \cite{FKLR}. Moreover, by \cite{FKLR}, Conjecture \ref{intro:conj} for smooth Calabi--Yau threefolds can be reduced to a conjectural inequality for $\bch_2$ for $\mu_H$-stable sheaves with small slope. We extend this reduction method to a more general setting in Section \ref{subsec:reduce-to-ch2} and apply it to a broad class of singular threefolds.


For example, if $\kk$ is an algebraically closed field of characteristic $0$, we have the following result. Recall that a \emph{Fano threefold} $X$ is a normal projective $3$-dimensional variety with rational Gorenstein singularities and $-K_X$ ample.

\begin{theorem}[{Corollary \ref{cor:fano3}}]
Let $X$ be a Fano threefold over $\kk$ that is either lci or $\QQ$-factorial. Then Conjecture \ref{intro:conj} holds for $(X, -K_X)$, $\mathsf{D}=0$, some $\Gamma\in \mathsf{E}(X)_{\QQ}$ with $\Gamma.(-K_X)\geq 0$, and $(b,w)$ in the range 
\begin{equation}\label{eq:intro-range}
w>\frac{1}{2}b^2+\frac{1}{2}(b-\lfloor b \rfloor)(\lfloor b \rfloor+1-b).
\end{equation}
\end{theorem}


For a more effective choice of $\Gamma$, see Theorem \ref{thm-fano3}. Note that there are many more deformation types of $\QQ$-factorial Fano threefolds than of smooth Fano threefolds.

For Calabi--Yau threefolds, we have the following singular version of \cite{Li19}.

\begin{theorem}[{Theorem \ref{thm:sing-quintic}}]
Let $X\subset \PP^4_{\kk}$ be a quintic normal threefold with rational singularities. Then Conjecture \ref{intro:conj} holds for $(X,H)$, $\mathsf{D}=\Gamma=0$, and $(b,w)$ in the range \eqref{eq:intro-range}.
\end{theorem}

Similar results hold for examples in \cite{liu:bg-ineqaulity-quadratic,koseki:stability-cy-solid} with rational singularities.

We also generalize criteria in \cite[Theorem 3.1]{FKLR}:

\begin{theorem}[{Theorem \ref{thm:main-criterion}}]
Let $(X, H)$ be a polarised normal projective threefold over $\kk$ with rational Gorenstein singularities, such that $K_X$ is numerically trivial, $\mathrm{H}^1(\cO_X)=0$, and either $X$ is lci or $\QQ$-factorial. Fix divisors $S\in |H|$ and $C\in |H|_S|$. Assume either

\begin{enumerate}
    \item $S$ and $C$ are both smooth with $\bn_C< \chi(\cO_X(H))$ or

    \item $S$ has rational singularities and $C$ is integral with $\bn_C< \chi(\cO_X(H))-1.$
\end{enumerate}

Then Conjecture \ref{intro:conj} holds for $(X, H)$, $\mathsf{D}=0$, some $\Gamma\in \mathsf{E}(X)_{\QQ}$ with $\Gamma.H\geq 0$, and $(b,w)$ in the range \eqref{eq:intro-range}. Here, $\bn_C$ is the invariant defined in Definition \ref{def:bnc}.
\end{theorem}

\begin{example}
In addition to the singular analogs of smooth examples considered in previous papers, we also apply Theorem \ref{thm:main-criterion} to certain singular Calabi--Yau threefolds that do not admit smoothings. An example is a general degree $8$ hypersurface in the weighted projective space $\PP(1,1,1,2,3)$ (cf.~Corollary \ref{cor:x8}).
\end{example}


\subsection{Tilt-stability in the relative setting}

In the spirit of \cite{BLMNPS21,piyaratne2019moduli,toda:moduli-K3}, we also show that tilt-stability behaves well in a flat family (cf.~Theorem \ref{thm:tilt-HN-structure}), which generalizes \cite[Theorem 25.3]{BLMNPS21}. In particular, for a suitable flat projective family of $2$-dimensional schemes in characteristic $0$, the moduli stacks of tilt-semistable objects admit good moduli spaces that are proper over the base, generalizing results in \cite{toda:moduli-K3}. 

\begin{theorem}[{Corollary \ref{cor:moduli-surface}, absolute version}]
Let $X$ be a pure $2$-dimensional projective scheme over a field $\kk$ of characteristic $0$ that is either geometrically normal or lci. Then for any $w> \Phi_{X, H}(b)$, the moduli stack of $\nu_{b,w}$-semistable objects with a fixed class $(\bch_i(-).H^{2-i})_{0\leq i\leq 2}$ is an Artin stack of finite type over $\kk$, and admits a good moduli space which is proper over $\kk$.
\end{theorem}


Another application is a semistable reduction result for tilt-semistable objects, which allows us to verify Conjecture \ref{intro:conj} via degeneration.

\begin{theorem}[{Theorem \ref{thm-degeneration}, simplified version}]
Let $X\to C$ be a projective flat lci morphism to a $1$-dimensional integral regular Noetherian scheme $C$ with the fraction field $K$ such that each fiber is equidimensional of dimension $3$. Fix a $C$-ample divisor $H$ on $X$ and $\Gamma\in \A^2(X)_{\QQ}$. Assume that there exists a constant $\mathsf{D}>0$ so that 
\[\Phi_{X_c, H_c}(x)\leq \frac{1}{2}x^2+\mathsf{D}\]
for any point $c\in C$. If for a closed point $p\in C$, Conjecture \ref{intro:conj} holds for $(X_p, H_p)$, $\Gamma_p$, $\mathsf{D}$, and fixed $(b,w)$ with $w>\frac{1}{2}b^2+\mathsf{D}$, then it also holds for $(X_K, H_K)$, $\Gamma_K$, and the same $\mathsf{D}$ and $(b,w)$.
\end{theorem}

Note that we allow $C$ to be of mixed characteristic. When $X\to C$ is smooth and $C$ has characteristic zero, the result is proved in \cite[Proposition 27.1]{BLMNPS21}. We expect this to be helpful for proving Conjecture \ref{intro:conj}, for instance via degeneration to toric schemes or reduction mod $p$ techniques. We will return to this point in future work.

\subsection{Stability conditions on singular Kuznetsov components}

For many Fano manifolds, the derived category contains distinguished semi-orthogonal components, called \emph{Kuznetsov components}. These subcategories have been studied extensively; see, for example, \cite{ku:hyperplane-section,ku:derived-cat-cubic4,kuz:fractional-CY,kuz09,kuznetsov2004derived,kuznetsov2018derived}.

Using a rotation of tilt-stability, stability conditions on a series of Fano manifolds are constructed in \cite{bayer2017stability}, which is later generalized to the relative setting by \cite[Section 26]{BLMNPS21}. Such a construction is applied to the moduli theory of sheaves and the geometry of Fano threefolds in \cite{bernardara2012categorical,PY20,LZ2021moduli,JLLZ2021gushelmukai,jLz2021brillnoether,feyzbakhsh2023new,jllz:inf-cat-torelli,FeyzbakhshPertusi2021stab}, and the study of hyper-K\"ahler manifolds in \cite{BLMNPS21,perry2019stability,shen-yin:k3-cat-bv,li2018twisted,li2020elliptic,GLZ2021conics,FGLZ,FGLZcube,guo-liu:atomic}.

On the other hand, Kuznetsov components can also be defined for many singular Fano varieties. In Theorem \ref{thm:stab-ox}, \ref{thm:stab-index-2}, and \ref{thm-ku-index-1}, we generalize results in \cite[Section 6]{bayer2017stability} and \cite[Corollary 26.2]{BLMNPS21} to Kuznetsov components of singular Fano threefolds.

As an application, we settle a singular variant of a conjecture of Kuznetsov--Shinder \cite[Conjecture 1.8]{kuznetsov:fano-threefold-degneration}. More precisely, for each $1\leq d\leq 5$, \cite[Theorem 3.6]{kuznetsov:fano-threefold-degneration} constructs a family $\cX\to B$ of Fano threefolds such that $B$ is a smooth complex curve, $\cX_{o}$ is a $1$-nodal index $1$ Fano threefold of genus $2d+2$ for a closed point $o\in B$, and $\cX_b$ is a smooth index $1$ Fano threefold of genus $2d+2$ for each $b\in B\setminus \{o\}$. Moreover, there exists a smooth proper category $\bar{\cA}_{\cX}\subset \Db(\cX)$ over $B$ whose fiber over $b\in B\setminus \{o\}$ is the Kuznetsov component of $\cX_b$, while the fiber over $o$ is equivalent to the Kuznetsov component of a smooth index $2$ Fano threefold of degree $d$.

In \cite[Conjecture 1.8]{kuznetsov:fano-threefold-degneration}, it is expected that $\bar{\cA}_{\cX}$ carries a stability condition over $B$. When $B$ is the spectrum of a complete DVR, this is proved by \cite{LMPSZ:deformation}. Using tilt-stability, we prove an analog of this conjecture for a slightly larger semi-orthogonal component $\cA_{\cX/B}\subset \Db(\cX)$, which differs from $\bar{\cA}_{\cX}$ only over $o\in B$.

\begin{corollary}[{Corollary \ref{cor-ks-conj}}]\label{cor:ku-intro}
In the above setting, if $d\geq 2$, then there exists a stability condition on $\cA_{\cX/B}$ over $B$.
\end{corollary}

Since the orthogonal complement of $\bar{\cA}_{\cX}$ in $\cA_{\cX/B}$ is generated by an explicit object (cf.~\cite[Theorem 3.6]{kuznetsov:fano-threefold-degneration}), we expect that Corollary \ref{cor-ks-conj} can be used to induce stability conditions on $\bar{\cA}_{\cX}$, hence to prove \cite[Conjecture 1.8]{kuznetsov:fano-threefold-degneration} in full generality.

\subsection*{Motivations and further questions}\label{subsec:mot}

Here, we discuss some motivations and problems that are closely related to this paper.

\subsubsection*{BMT Conjecture for more threefolds}

A fundamental technique in algebraic geometry is degeneration. Concretely, one considers a family $\cX \to B$ together with two points $0,1 \in B$ such that the fiber $\cX_0$ is the variety of interest, while $\cX_1$ has richer geometry. This method is widely used in Gromov--Witten theory, for example, in the proof of the MNOP Conjecture \cite{PP12}. A key step is to establish an appropriate degeneration formula for the problem under consideration.

Motivated by this perspective, a main motivation for developing tilt-stability on singular schemes is to study the behavior of Conjecture \ref{intro:conj} under degeneration. By applying Theorem \ref{thm-degeneration} to a suitable degeneration, we expect that Conjecture \ref{intro:conj} can be verified for many smooth threefolds. Indeed, many such threefolds admit degenerations to unions of smooth threefolds for which Conjecture \ref{intro:conj} is already known, such as Fano threefolds. For instance, let $X$ be a smooth Calabi--Yau threefold given by the complete intersection of a quadric hypersurface and a codimension-$4$ linear section of $\Gr(2,6)$. Then $X$ admits a degeneration to a union $X'$ of two codimension-$5$ linear sections of $\Gr(2,6)$, each of which is a smooth Fano threefold of Picard number one. A Bogomolov--Gieseker-type inequality for semistable sheaves on $X'$ is proved in Theorem \ref{thm:exist-bg-function}. Consequently, the framework of this paper applies to $X'$. It is therefore natural to expect that Conjecture \ref{intro:conj} can be proved for such a union by combining the corresponding results for its Fano components, and that the conjecture for $X$ would then follow from Theorem \ref{thm-degeneration}. In the spirit of \cite{LMPSZ:deformation}, one may also consider the logarithmic derived category on the singular fiber; see \cite{hu:coherent-in-log,hu:log-derived-cat}.

We also note that our result for singular Fano threefolds in Corollary \ref{cor:fano3} is not expected to be optimal. For smooth Fano threefolds of Picard number one, Conjecture \ref{conj-2} with $\Gamma = 0$ is proved in \cite{li:fano-3fold}. For the higher Picard rank case, an effective choice of $\Gamma$ is obtained in \cite{macri:fano-threefold}. We expect that a more explicit version of Corollary \ref{cor:fano3} can be established by adapting the methods of \cite{li:fano-3fold,macri:fano-threefold}. Similar ideas may also apply to weak Fano threefolds.

\subsubsection*{Birational geometry}



Singularities arise naturally in the minimal model program. It is conjectured in \cite[Question 1.1]{toda:ex-contract} that each step of the minimal model program for a smooth threefold can be realized as a wall-crossing of moduli spaces of stable objects. This has been established in dimension $2$ in \cite{toda:mmp-surface}, and partial progress in dimension $3$ is made in \cite{toda:ex-contract}. In particular, using the perverse $t$-structure on $\Db(X)$ associated with a birational contraction $X \to Y$, \cite{toda:ex-contract} introduces a perverse version of tilt-stability and a corresponding form of Conjecture \ref{intro:conj} for $X$. Note that every threefold arising in the minimal model program of a smooth threefold has \emph{terminal $\QQ$-factorial singularities}, so it fits perfectly into our framework. It is therefore natural to generalize the constructions of \cite{toda:ex-contract} to singular threefolds and to compare the resulting perverse theory with the theory developed in this paper. We expect that such a comparison could lead to a solution of \cite[Question 1.1]{toda:ex-contract} in dimension $3$.

On the other hand, effective basepoint-freeness results play a central role in birational geometry. For a smooth threefold $X$, it is proved in \cite{bbmt14} that Conjecture \ref{intro:conj} implies an effective generation theorem for adjoint linear series, which in turn yields a version of Fujita’s conjecture. It is therefore an interesting problem to extend these results to singular threefolds.


\subsubsection*{Curve-counting theory}



Another natural source of singular threefolds is provided by quotients by finite groups, and more generally by coarse moduli spaces of Deligne--Mumford stacks. The crepant resolution conjecture relates the curve-counting invariants of a smooth $3$-dimensional Deligne--Mumford stack $\cX$ to those of a crepant resolution $Y$ of its coarse moduli space $X$. In the context of Donaldson--Thomas theory, the invariants of a smooth Calabi--Yau $3$-dimensional stack $\cX$ are related to those of $Y$ in \cite{BCR}. As explained in \cite[Section 2.2]{BCR}, the coarse moduli space $X$ has Gorenstein quotient singularities, and therefore has a well-behaved theory of tilt-stability. On the other hand, wall-crossing for tilt-stability and Conjecture \ref{intro:conj} have been used to study Donaldson--Thomas invariants in \cite{toda:bogomolov-counting,feyz:rank-r-dt-theory-from-0,feyz:rank-r-dt-theory-from-1,feyz:curve-counting,liu-ruan:cast-bound}. It is therefore natural to expect that the relationship between the Donaldson--Thomas invariants of $\cX$ and $Y$ can be understood through the wall-crossing behavior of tilt-stability on the coarse moduli space $X$.

A different approach to using Conjecture \ref{intro:conj} to study curves is developed in \cite{macri:space-curve,liu-ruan:cast-bound,feyz:physics-abelian-dt,liu:cast-bound-3fold}. In particular, Conjecture \ref{intro:conj} yields various bounds on the genus of curves in threefolds, which in turn imply vanishing theorems for curve-counting invariants of smooth Calabi--Yau threefolds. Using the framework of this paper, we expect that these results can also be extended to the singular setting.

\subsubsection*{Study of moduli spaces}



Moduli spaces of semistable objects are among the main geometric outputs of stability conditions. In the smooth setting, stability conditions on Kuznetsov components have led to the construction and study of many remarkable moduli spaces, including hyper-K\"ahler manifolds and their birational models, and moduli of sheaves on smooth Fano threefolds. Since we construct stability conditions on relative Kuznetsov components associated with singular families in Section \ref{sec:ku}, it is natural to ask how the corresponding moduli spaces behave in degenerating families, especially using Corollary \ref{cor:ku-intro}. In particular, one expects moduli spaces on smooth Kuznetsov components to admit natural degenerations to moduli spaces associated with the singular fibers. Such a picture would provide a categorical approach to studying degenerations of hyper-K\"ahler manifolds and moduli spaces of sheaves on Fano varieties.

On the other hand, wall-crossing for stability gives a powerful description of the birational geometry of moduli spaces on smooth K3 surfaces and related hyper-K\"ahler manifolds \cite{bayer:mmp,bayer:projectivity}. It is therefore natural to seek an analogous picture for moduli spaces on singular symplectic surfaces. One may then hope to use the resulting wall-crossing to study the minimal model program for the corresponding moduli spaces and, more generally, to extend the results of \cite{bayer:mmp,bayer:projectivity} from hyper-K\"ahler manifolds to singular hyper-K\"ahler varieties.

\subsection*{Organization}
In Section \ref{sec:pre}, we review preliminaries on derived categories and singular schemes. Then Section \ref{sec:intersect-ch} develops the intersection-theoretic and Chern character formalism needed in the singular setting. Sections \ref{sec:stab} and \ref{sec:stab-family} recall stability conditions and their relative versions in families. In Section \ref{sec:general-tilt}, we prove an abstract tilting theorem for weak stability conditions (cf.~Theorem \ref{thm:tilt-stability}) and establish its wall-chamber structure in Theorem \ref{thm:wall-chamber-abstract}. 

In Section \ref{sec:slope-bg}, we first introduce the necessary notation, then we study slope-stability and its relative version in Sections \ref{subsec:slope} and \ref{subsec:slope-family}. 

In Section \ref{sec:bg}, we first define the Le Potier function in the absolute and relative settings. In Theorem \ref{thm:exist-bg-function}, we show that it is always bounded above by a quadratic function, and in Theorem \ref{thm-bg-normal} we sharpen this bound for normal varieties.

Section \ref{sec:tilt-3} focuses on tilt-stability constructed from slope-stability. We begin with a rotation of slope-stability and its relative version in Proposition \ref{prop-rotate-slope-real-b}. We then establish the basic properties and useful lemmas for tilt-stability in Sections \ref{subsec-titl-from-slope} and \ref{subsec:lemma-about-tilt}. The main result on tilt-stability in families is Theorem \ref{thm:tilt-HN-structure}.



In Section \ref{sec:bmt}, we first formulate the BMT Conjecture on projective threefolds and prove that it is equivalent to two other weaker formulations (cf.~Theorem \ref{thm:equivalent-conj}). Using these conjectures, Theorem \ref{thm-stab} shows that we can construct an explicit family of stability conditions. We end this section by establishing a semistable reduction result, Theorem \ref{thm-lift-tilt-ss}, for tilt-semistable objects, which allows us to check the BMT Conjecture via degeneration (cf.~Theorem \ref{thm-degeneration}).

In Section \ref{sec:bmt-cy-fano}, we first generalize the reduction method in \cite{FKLR} to a more general setting in Theorem \ref{thm:ch2}. Using this, we verify the BMT Conjecture for Fano threefolds and a series of Calabi–Yau threefolds. Moreover, a singular Calabi--Yau threefold without smoothing is considered in Corollary \ref{cor:x8}.

In Section \ref{sec:ku}, we generalize \cite[Theorem 23.1]{BLMNPS21} to families of singular schemes. Using this, we construct stability conditions on Kuznetsov components associated with singular Fano threefolds. A singular variant of Kuznetsov--Shinder's conjecture \cite[Conjecture 1.8]{kuznetsov:fano-threefold-degneration} is proved in Corollary \ref{cor-ks-conj}.

Finally, in Appendix \ref{appendix:approach}, we discuss how to replace the Chern characters in the construction of tilt-stability with coefficients of Hilbert polynomials, and prove an analog of Theorem \ref{thm:intro-1}. In particular, a Bogomolov--Gieseker-type inequality for semistable sheaves on an arbitrary projective scheme is established in Theorem \ref{thm-bg-general}.

\subsection*{Acknowledgments}
We would like to thank Soheyla Feyzbakhsh, Chunyi Li, Emanuele Macr\`i, Alexander Perry, and Yukinobu Toda for discussions that greatly inspired this paper. Z.L. would like to thank his supervisor Yongbin Ruan for encouragement and many valuable suggestions. We would also like to thank Arend Bayer, Tomohiro Karube, Naoki Koseki, Peize Liu, Ziqi Liu, Shiji Lyu, Kenneth Ma, Nick Rekuski, Dongjian Wu, Chenyang Xu, Nantao Zhang, and Xiaolei Zhao for helpful discussions. Z.L. was supported by NSFC Grant 123B2002. T.M. was supported by World Premier International Research Center Initiative (WPI), MEXT, Japan.

\section{Preliminaries}\label{sec:pre}

In this section, we collect preliminaries on derived categories, singular schemes, and relative objects used later.

\subsection{Notations and conventions}

We begin by summarizing some basic notions.

\begin{itemize}
   \item Given a scheme $X$, a \emph{point} $t\to X$ means a morphism from the spectrum of a field, and it is called a \emph{geometric point} if the corresponding field is algebraically closed. We write $t\in X$ when $t\to X$ identifies the source field with the residue field of its image.

    \item We say a locally Noetherian scheme $X$ is \emph{equidimensional} if each irreducible component of $X$ has the same finite Krull dimension and there are no embedded components. 


    \item A \emph{Dedekind scheme} is an integral, Noetherian, one-dimensional regular scheme. For a Dedekind scheme $C$, we write $p\in C$ for a closed point, $c\in C$ for an arbitrary point, $\eta\in C$ for the generic point, and $K$ for its fraction field.

    \item We say a morphism $f\colon X \to Y$ between schemes is \emph{essentially of finite type} if $f$ is either of finite type, or $X$ is affine and $f$ factors as $\Spec(S^{-1}A)\to \Spec(A)\to Y$, where $\Spec(A)\to Y$ is a morphism of finite type and $S^{-1}A$ is a localization of $A$.

    \item A \emph{Nagata scheme} is a scheme $X$ such that we have an affine open cover $\bigcup_{i\in I} \Spec(A_i)=X$ with each $A_i$ a Nagata ring in the sense of \cite[\href{https://stacks.math.columbia.edu/tag/032R}{Tag 032R}]{stacks-project}. The Nagata property is preserved under morphisms essentially of finite type by \cite[\href{https://stacks.math.columbia.edu/tag/032U}{Tag 032U}]{stacks-project} and \cite[\href{https://stacks.math.columbia.edu/tag/0334}{Tag 0334}]{stacks-project}. Moreover, the normalization morphism of $X$ is finite by \cite[\href{https://stacks.math.columbia.edu/tag/035S}{Tag 035S}]{stacks-project}.

    \item A morphism $X\to Y$ between schemes is \emph{embeddable} if it can be written as a composition $X\hookrightarrow P\to Y$, where $P\to Y$ is a smooth morphism and $X\hookrightarrow P$ is a closed embedding. In this case, we also say that $X$ is embeddable over $Y$. 


    \item A variety over a field $\kk$ is a geometrically integral scheme of finite type over $\kk$. A threefold over $\kk$ is a variety over $\kk$ of dimension $3$.

    \item For a normal variety $X$ over a field of characteristic $0$, we say $X$ has \emph{rational singularities} if the natural map $\cO_X\to Rf_*\cO_Y$ is an isomorphism for a (hence any) resolution of singularities $f\colon Y\to X$. By \cite[Theorem 4.20, Corollary 5.24]{kollar-mori}, if $X$ has rational Gorenstein singularities and $\dim X=2$, then $X$ is a local complete intersection.

    \item We denote by $\KK(\cD)$ the K-group of a triangulated category or an abelian category $\cD$. 

    \item For a lattice $\Lambda$ and any subset $S\subset \Lambda$, we denote by $\langle S\rangle$ the saturation of the subgroup generated by $S$ in $\Lambda$. For any $\ZZ$-module $\A$, we denote by $\A_{\mathbb{F}}\coloneqq\A\otimes_{\ZZ}\mathbb{F}$, where $\mathbb{F}$ is any field extension of $\QQ$.

\end{itemize}

\subsection{Derived categories}

For a scheme $X$, we consider the following derived categories:

\begin{itemize}

   \item the unbounded derived category $\D(X)$ of sheaves of $\cO_X$-modules; for any object $E\in \D(X)$, we denote by $\cH^i(E)\in \mathrm{Mod}(\cO_X)$ the $i$-th cohomology sheaf,

   \item the derived category $\Dqc(X)$ (resp. $\Dqc^{+}(X)$, $\Dqc^-(X)$) of $\cO_X$-modules with (resp. bounded below, bounded above) quasi-coherent cohomology sheaves,

\item the category $\D_{\mathrm{pc}}(X)$ of pseudo-coherent complexes; here, we say a complex in $\D(X)$ is pseudo-coherent if affine locally, it is quasi-isomorphic to a bounded above complex of finitely generated locally free sheaves,

    \item the derived category $\Db(X)$ of pseudo-coherent complexes on $X$ with bounded cohomology sheaves, 
    
    \item when $X$ is Noetherian, the derived categories $\D^+(X)$ and $\D^-(X)$ of $\cO_X$-modules with bounded below and bounded above coherent cohomology sheaves, respectively; and

    \item the full triangulated subcategory $\Dperf(X)\subset \D(X)$ of perfect complexes on $X$. If $X$ is Noetherian, we have $\Dperf(X)\subset \Db(X)$. When $X$ is regular and quasi-compact, we have $\Dperf(X)=\Db(X)$ (cf.~\cite[\href{https://stacks.math.columbia.edu/tag/0FDC}{Tag 0FDC}]{stacks-project}).
\end{itemize}

According to \cite[\href{https://stacks.math.columbia.edu/tag/08E8}{Tag 08E8}]{stacks-project}, if $X$ is Noetherian, then $\D_{\mathrm{pc}}(X)=\D^-(X)$ coincides with the bounded above derived category of $\cO_X$-modules with coherent cohomology sheaves. In this case, we also have $\Db(X)=\Db(\Coh(X))$, $\D^-(X)=\D^-(\Coh(X))$, and $\Dqc(X)=\D(\mathrm{QCoh}(X))$. 

We denote by $\KK_0(X)$ and $\KK^0(X)$ the K-group of $\Db(X)$ and $\Dperf(X)$, respectively.

For a morphism between schemes $f\colon X\to Y$, we have the following derived functors:

\begin{itemize}
    \item the derived pushforward $$Rf_*\colon \Dqc(X)\to \Dqc(Y)$$ when $f$ is quasi-compact and quasi-separated. It induces a functor $$Rf_*\colon \Db(X)\to \Db(Y)$$ when $f$ is proper and $Y$ is Noetherian. If $f$ is furthermore perfect in the sense of \cite[\href{https://stacks.math.columbia.edu/tag/0687}{Tag 0687}]{stacks-project}, e.g.~$Y$ is regular (cf.~\cite[\href{https://stacks.math.columbia.edu/tag/068B}{Tag 068B}]{stacks-project}), or $f$ is flat and locally of finite presentation (cf.~\cite[\href{https://stacks.math.columbia.edu/tag/068A}{Tag 068A}]{stacks-project}), or $f$ is a local complete intersection morphism (cf.~\cite[\href{https://stacks.math.columbia.edu/tag/069H}{Tag 069H}]{stacks-project}), then it restricts to $$Rf_*\colon \Dperf(X)\to \Dperf(Y)$$ by \cite[\href{https://stacks.math.columbia.edu/tag/0B6G}{Tag 0B6G}]{stacks-project}, 

    \item the derived pullback $Lf^*\colon \Dqc(Y)\to \Dqc(X)$ and $Lf^*\colon \Dperf(Y)\to \Dperf(X)$. If $f$ is perfect, then it restricts to $$Lf^*\colon \Db(Y)\to \Db(X),$$ and

    \item the upper shriek $f^!\colon \Dqc^+(Y)\to \Dqc^+(X)$ and $f^!\colon \D^+(Y)\to \D^+(X)$ if $Y$ is Noetherian and $f$ is separated and of finite type (cf.~\cite[\href{https://stacks.math.columbia.edu/tag/0AA0}{Tag 0AA0}]{stacks-project} and \cite[\href{https://stacks.math.columbia.edu/tag/0AU1}{Tag 0AU1}]{stacks-project}). If $f$ is also proper, then $f^!$ is the right adjoint of $Rf_*$ (cf.~\cite[\href{https://stacks.math.columbia.edu/tag/0F42}{Tag 0F42}]{stacks-project}). If $f$ is perfect, then it induces $f^!\colon \Db(Y)\to \Db(X)$ by \cite[\href{https://stacks.math.columbia.edu/tag/0B6U}{Tag 0B6U}]{stacks-project}. If $f$ is a local complete intersection morphism, then it induces $$f^!\colon \Dperf(Y)\to \Dperf(X)$$ by \cite[\href{https://stacks.math.columbia.edu/tag/0B6V}{Tag 0B6V}]{stacks-project}.
\end{itemize}

We also have functors $$R\cH om_X(-,-)\colon \D(X)^{\mathrm{op}}\times \D(X)\to \D(X)$$ and
$$-\otimes^{\mathrm{L}} -\colon\D(X)\times \D(X)\to \D(X).$$
When $X$ is locally Noetherian, they restrict to functors $\D^-(X)^{\mathrm{op}}\times \D^+(X)\to \D^+(X)$ and $\D^-(X)\times \D^-(X)\to \D^-(X)$, respectively. By \cite[\href{https://stacks.math.columbia.edu/tag/08DJ}{Tag 08DJ}]{stacks-project}, for any $E,F,G\in \D(X)$, we have
\[R\cH om_X(E,R\cH om_X(F,G))\cong R\cH om_X(E\otimes^{\mathrm{L}}F,G).\]

When $Y$ is Noetherian and has the resolution property in the sense of \cite[\href{https://stacks.math.columbia.edu/tag/0F86}{Tag 0F86}]{stacks-project}, any object in $\Dperf(Y)$ is quasi-isomorphic to a bounded complex of finite locally free sheaves on $Y$ and any object in $\D^-(Y)$ is quasi-isomorphic to a bounded above complex of finite locally free sheaves on $Y$. In this case, if $f\colon X\to Y$ is a morphism from another Noetherian scheme $X$, then we have
\[Lf^*R\cH om_Y(E,F)\cong R\cH om_X(Lf^*E, Lf^*F)\]
for any $E\in \D^-(Y)$ and $F\in \D^+(Y)$ if either $E\in \Dperf(Y)$ or $f$ has finite Tor-dimension (cf.~\cite[\href{https://stacks.math.columbia.edu/tag/0GM7}{Tag 0GM7}]{stacks-project}).

For a locally Noetherian scheme $X$ and an object $E\in \D^-(X)$, we define the derived dual $$\mathbb{D}^X(E)\coloneqq R\cH om_X(E, \oh_X)\in \D^+(X)$$ and its shift  $\mathbb{D}^X_n(E)\coloneqq\mathbb{D}^X(E)[n]$, and the underived dual $$E^{\vee}=\cH om_X(E, \oh_X)\coloneqq\cH^0(\mathbb{D}^X(E))\in \Coh(X).$$
If $X=\Spec(\kk)$ for a field $\kk$, then $E$ is a complex of finite-dimensional $\kk$-vector spaces, and we write $E^*\coloneqq \mathbb{D}^X(E)$ to match the usual notation of dual vector spaces.

Let $f\colon X\to Y$ be a morphism between schemes over $S$. For any morphism $T\to S$, we denote the naturally induced morphism by $$f_T\colon X_T\coloneqq X\times_S T\to Y_T\coloneqq Y\times_S T.$$ Similarly, for any $E\in \Dqc(X)$, we denote by $E_T$ the derived pullback of $E$ along the natural projection $X_T\to X$.

For $x\in X$ and $E\in \Dqc(X)$, we denote by $E|_x\in \Dqc(\Spec \oh_{X,x})$ the (derived) pullback of $E$ along the natural flat morphism $\Spec \oh_{X,x}\to X$.

\subsection{Semi-orthogonal decompositions}

We now review semi-orthogonal decompositions in both absolute and relative settings. We refer to \cite[Section 3]{BLMNPS21} for a more detailed introduction. Let $\cD$ be a triangulated category.

\begin{definition}
A \emph{semi-orthogonal decomposition} of $\cD$ is a sequence of full triangulated subcategories $\cD_1, \dots, \cD_m$ of $\cD$, called the \emph{components} of the decomposition, such that
\[\cD=\langle \cD_1, \dots, \cD_m\rangle,\]
where $\langle -\rangle$ denotes the extension closure, satisfying
\begin{itemize}
    \item $\Hom_{\cD}(\cD_i, \cD_j)=0$ for $i>j$, and

    \item for any $E\in \cD$, there is a sequence of morphisms 
    \[0=E_m\to E_{m-1}\to \cdots \to E_0=E\]
    such that $\pr_i(E)\coloneqq\cone(E_i\to E_{i-1})\in \cD_i$ for each $1\leq i\leq m$.
\end{itemize}

Then we obtain a functor $\pr_i\colon \cD\to \cD_i\hookrightarrow  \cD$ for each $1\leq i\leq m$, which is called the \emph{projection functor} onto $\cD_i$.
\end{definition}

\begin{definition}
A semi-orthogonal decomposition $\cD=\langle \cD_1, \dots, \cD_m\rangle$ is called \emph{strong} if for each $i$ the inclusion functor $\cD_i\hookrightarrow \cD$ has a right adjoint. A full subcategory $\cD'\subset \cD$ is a \emph{strong semi-orthogonal component} if it is part of a strong semi-orthogonal decomposition of $\cD$.
\end{definition}

\begin{definition}
Let $X$ be a Noetherian scheme. We say a semi-orthogonal decomposition $\cD=\langle \cD_1, \dots, \cD_m\rangle$ of a triangulated subcategory $\cD\subset \Dqc(X)$ is of \emph{finite cohomological amplitude} if there exists $p,q\in \ZZ$ such that the projection functor $\pr_i$ onto each component $\cD_i$ satisfies
\[\pr_i(\cD\cap \Dqc^{[a,b]}(X))\subset \Dqc^{[a+p, b+q]}(X)\]
for all $a,b\in \ZZ$, where $\Dqc^{I}(X)$ denotes the full subcategory of $\Dqc(X)$ consisting of objects whose cohomology sheaves are concentrated in degrees in $I\subset \ZZ$.

We say a semi-orthogonal component $\cD$ of $\Dperf(X), \Db(X)$, or $\Dqc(X)$ is of \emph{finite cohomological amplitude} if the semi-orthogonal decomposition defining $\cD$ is of finite cohomological amplitude.
\end{definition}

By \cite[Lemma 3.10, 3.13]{BLMNPS21}, if $X$ is Noetherian, then any semi-orthogonal component $\cD\subset \Db(X)$ induces semi-orthogonal components $\cD_{\perf}\subset \Dperf(X)$ and $\cD_{\mathrm{qc}}\subset \Dqc(X)$.

We have the following definition when we work over a general base scheme.

\begin{definition}
Let $f\colon X\to S$ be a morphism between schemes. A triangulated subcategory $\cD\subset \Dqc(X)$ is called \emph{$S$-linear} if for any $E\in \cD$ and $F\in \Dperf(S)$, one has $E\otimes^{\LL} Lf^*F\in \cD$. A semi-orthogonal decomposition of $\cD$ is called $S$-linear if all of its components are $S$-linear.
\end{definition}

\subsection{Serre's conditions for sheaves and morphisms}

Let $X$ be a locally Noetherian scheme. Recall that the dimension and codimension of a coherent sheaf $E$ on $X$ are the dimension and codimension of its support $\Supp(E)$, which are denoted by $\dim(E)$ and $\codim_X(E)$, respectively. We say $E$ is \emph{pure of dimension $d$} if $\dim(F)=d$ for any non-zero subsheaf $F\subset E$. This is equivalent to saying that all associated points of $E$ have the same dimension. 

We say $E$ is a \emph{torsion sheaf} if $\Supp(E)$ is nowhere dense in $X$, or equivalently, the set $\Ass(E)$ of associated points of $E$ does not contain any generic point of $X$. For any $E\in \Coh(X)$, there is a largest torsion subsheaf $\tor(E)\subset E$, which is called the torsion part of $E$. We say $E$ is \emph{torsion-free} if $\tor(E)=0$, or equivalently, every point in $\Ass(E)$ is a generic point of $X$. If $X$ is equidimensional, it follows from the definition that a non-zero sheaf $E$ is torsion-free if and only if $E$ is pure of dimension $\dim(X)$.

\begin{remark}
In some papers, torsion-freeness is only defined for sheaves on integral schemes. It is clear that when $X$ is integral, the two definitions coincide; see for example, \cite[\href{https://stacks.math.columbia.edu/tag/0AUV}{Tag 0AUV}]{stacks-project}.
\end{remark}

Now we review Serre's conditions $S_n$. Recall that for a coherent sheaf $E$ on a locally Noetherian scheme $X$, the depth $\depth_x(E)$ at a point $x\in X$ is defined as the depth of $E|_x$ as a finite $\oh_{X,x}$-module (cf.~\cite[\href{https://stacks.math.columbia.edu/tag/00LI}{Tag 00LI}]{stacks-project}).

\begin{definition}[{\cite[\href{https://stacks.math.columbia.edu/tag/0341}{Tag 0341}]{stacks-project}}]
Let $X$ be a locally Noetherian scheme, $E\in \Coh(X)$, and $n\in \ZZ_{\geq 0}$. We say $E$ is $S_n$ if
\[\depth_x(E)\geq \min\{n, \dim(E|_x)\}\]
for any $x\in X$. We say $X$ is $S_n$ if $\oh_X$ is. We say $E$ is \emph{Cohen--Macaulay} if it is $S_n$ for all $n\geq 0$.
\end{definition}

In particular, $E$ is $S_1$ if and only if it has no embedded points (cf.~\cite[\href{https://stacks.math.columbia.edu/tag/0346}{Tag 0346}]{stacks-project}). Therefore, $E$ is pure if and only if $\Supp(E)$ is equidimensional and $E$ is $S_1$. If $X$ is normal and integral, then torsion-free $S_2$ sheaves on $X$ coincide with reflexive sheaves by \cite[\href{https://stacks.math.columbia.edu/tag/0AVB}{Tag 0AVB}]{stacks-project}. In the following, we will see that torsion-free $S_2$ sheaves share many useful properties with reflexive sheaves.

\begin{lemma}\label{lem-general-S2}
Let $X$ be a locally Noetherian scheme and $E$ be a coherent sheaf on $X$.

\begin{enumerate}
    \item If $F\in \Coh(X)$, then $\mathrm{Ass}(\cH om_X(E, F))\subset \mathrm{Ass}(F)$. In particular, if $F$ is pure of dimension $d$ (resp.~torsion-free), then $\cH om_X(E, F)$ is also pure of dimension $d$ (resp.~torsion-free).

    \item If $F$ is an $S_2$ sheaf on $X$, then $\cH om_X(E,F)$ is also $S_2$.

    \item If $E$ is torsion-free and $S_2$, then the natural map $E\to j_*j^*E$ is an isomorphism for any open embedding $j\colon U\hookrightarrow X$ with $\codim_X(X\setminus j(U))\geq 2$.

    \item If $E$ is torsion-free and $S_2$, then for any $T\in \Coh(X)$ with $\codim_X (T)\geq 2$, we have $$\Hom_X(T, E[1])=0.$$
\end{enumerate}
\end{lemma}

\begin{proof}
Affine locally, we can take a surjection $\oh_X^{\oplus n}\twoheadrightarrow E$, which gives an inclusion $\cH om_X(E, F)\subset F^{\oplus n}$ and part (a) follows. Part (b) is \cite[\href{https://stacks.math.columbia.edu/tag/0AXQ}{Tag 0AXQ}]{stacks-project} and part (c) is \cite[\href{https://stacks.math.columbia.edu/tag/0E9I}{Tag 0E9I}]{stacks-project}.

For part (d), let $U\coloneqq X\setminus \mathrm{Supp}(T)$ and $j\colon U\hookrightarrow X$ be the inclusion. Let $F$ be the extension of $T$ and $E$ corresponds to an element of $\Hom_X(T, E[1])$. If $E$ is $S_2$, then $j_*j^*F\cong j_*j^*E\cong E$ by part (c). Therefore, the composition of natural maps $E\hookrightarrow F\to j_*j^*F\cong j_*j^*E$ is an isomorphism. Thus, $F$ splits and $\Hom_X(T, E[1])=0$ follows.
\end{proof}

We also have the following generalization of $S_n$ for morphisms:

\begin{definition}[{\cite[\href{https://stacks.math.columbia.edu/tag/045R}{Tag 045R}]{stacks-project}}]
Let $f\colon X\to Y$ be a morphism between schemes such that each fiber $X_y$ is locally Noetherian. We say $f$ is $S_n$ if $f$ is flat and every fiber of $f$ is $S_n$. We say $f$ is \emph{Cohen--Macaulay} if $f$ is $S_n$ for each $n\geq 0$. We say $f$ is \emph{Gorenstein} if $f$ is flat and every fiber of $f$ is Gorenstein.
\end{definition}

By \cite[\href{https://stacks.math.columbia.edu/tag/045U}{Tag 045U}]{stacks-project} and \cite[\href{https://stacks.math.columbia.edu/tag/0E0Q}{Tag 0E0Q}]{stacks-project}, Cohen--Macaulay morphisms and Gorenstein morphisms are stable under arbitrary base change. More generally, $S_n$ morphisms are also stable under base change by \cite[Proposition 6.7.1]{EGA-IV-2}.

We will also use the notion of \emph{local complete intersection (lci)} morphisms as in \cite[\href{https://stacks.math.columbia.edu/tag/069F}{Tag 069F}]{stacks-project}. By \cite[\href{https://stacks.math.columbia.edu/tag/068E}{Tag 068E}]{stacks-project}, the lci property is stable under composition and flat base change. By \cite{avramov:lci}, a flat morphism $f$ of finite type between Noetherian schemes is lci if and only if the cotangent complex $\mathbb{L}_f$ of $f$ is perfect with tor-amplitude in $[-1,0]$.

We say a morphism is \emph{syntomic} if it is flat and lci (cf.~\cite[\href{https://stacks.math.columbia.edu/tag/069K}{Tag 069K}]{stacks-project}). By \cite[\href{https://stacks.math.columbia.edu/tag/01UB}{Tag 01UB}]{stacks-project}, being syntomic is stable under compositions and arbitrary base changes. If $f$ is locally of finite presentation, then $f$ is syntomic if and only if $f$ is flat and each fiber of $f$ is lci.

We say a morphism $f\colon X\to Y$ is \emph{fiberwise lci in codimension $d$} if there exists a closed subset $Z\subset X$ with $\codim_{X_y}(Z_y)\geq d+1$ such that $X_y\setminus Z_y$ is lci over $\kappa(y)$ for each $y\in Y$. If $f$ is flat and locally of finite presentation, then by \cite[\href{https://stacks.math.columbia.edu/tag/06B8}{Tag 06B8}]{stacks-project} or \cite[\href{https://stacks.math.columbia.edu/tag/02V3}{Tag 02V3}]{stacks-project}, 
\[\mathrm{LCI}(X/Y)\coloneqq \{x\in X\colon X_{f(x)} \text{ is lci over } \kappa(f(x)) \text{ at }x\}\]
is open in $X$ and
\begin{equation}\label{eq:lci-locus-base-change}
\mathrm{LCI}(X_T/T)=g^{-1}(\mathrm{LCI}(X/Y))
\end{equation}
for any morphism $T\to Y$, where $g\colon X_T\to X$ is the induced morphism. In this case, $f$ is fiberwise lci in codimension $d$ if and only if $(X\setminus \mathrm{LCI}(X/Y))_y=X_y\setminus \mathrm{LCI}(X_y/\kappa(y))$ has codimension $\geq d+1$ in $X_y$ for each $y\in Y$. If $Y$ is the spectrum of a field $\kk$, then we simply say $X$ is \emph{lci in codimension $d$ (over $\kk$).}

\subsection{Relative perfect objects}

Next, we recall the definition of relative perfect objects, which can be viewed as a notion of families of bounded complexes of coherent sheaves over a base.

\begin{definition}[{\cite[\href{https://stacks.math.columbia.edu/tag/0DI0}{Tag 0DI0}]{stacks-project}}]
Let $f\colon X\to S$ be a morphism which is flat and locally of finite presentation. An object $E\in \Dqc(X)$ is \emph{perfect relative to $S$}, or \emph{$S$-perfect}, if $E$ is pseudo-coherent and locally of finite Tor-dimension over $f^{-1}\oh_S$. 
\end{definition}

Note that if $S=\Spec \kk$ for a field $\kk$ and $X$ is Noetherian, then $E\in \Dqc(X)$ is $S$-perfect if and only if $E\in \Db(X)$. More generally, we have the following useful lemma.

\begin{lemma}\label{lem-S-perf-lem-1}
Let $f\colon X\to S$ be a morphism which is flat and locally of finite presentation.

\begin{enumerate}
    \item If $E\in \Dqc(X)$ is $S$-perfect and $X$ is quasi-compact, then $E\in \Db(X)$.

    \item If $E\in \Dperf(X)$, then $E$ is $S$-perfect.

    \item If $S$ is regular of finite dimension, then any object $E\in \Db(X)$ is $S$-perfect.

    \item If $f$ is smooth, then any $S$-perfect object $E$ is perfect.
\end{enumerate}
\end{lemma}

\begin{proof}
Parts (a)-(c) are included in \cite[Lemma 8.3]{BLMNPS21}. 

Part (d) is already mentioned in \cite[Section 2.1]{lieblich:moduli-of-complex}. For a smooth morphism,
\(\mathcal O_X\) is perfect over \(f^{-1}\mathcal O_S\), and a
pseudo-coherent object on \(X\) of finite Tor-amplitude over \(S\) has finite Tor-amplitude over \(X\). Hence \(E\) is perfect. Equivalently, this is \cite[\href{https://stacks.math.columbia.edu/tag/09PC}{Tag 09PC}]{stacks-project}.
\end{proof}

One of the advantages of working with relative perfect objects is that they behave well under base change, proper pushforward, and derived internal hom.

\begin{lemma}\label{lem-S-perf-lem-2}
Let $f\colon X\to S$ be a morphism which is flat and locally of finite presentation and $E\in \Dqc(X)$ be a $S$-perfect object.

\begin{enumerate}
    \item If $f$ is proper, then $Rf_*E\in \Dperf(S)$ and its formation commutes with arbitrary base change.

    \item If $S'\to S$ is another morphism, then $E_{S'}\in \Dqc(X_{S'})$ is $S'$-perfect.

    \item If $F\in \Dperf(X)$, then $E\otimes^\mathrm{L} F$ is also $S$-perfect.

\end{enumerate}
\end{lemma}

\begin{proof}
Part (a) is \cite[\href{https://stacks.math.columbia.edu/tag/0DJT}{Tag 0DJT}]{stacks-project}, part (b) is \cite[\href{https://stacks.math.columbia.edu/tag/0DI5}{Tag 0DI5}]{stacks-project}, and part (c) is \cite[\href{https://stacks.math.columbia.edu/tag/0DI4}{Tag 0DI4}]{stacks-project}.
\end{proof}

We also have the following flattening stratification of relative perfect objects.

\begin{lemma}\label{lem:openness-standard-heart}
Let $X\to S$ be a flat projective morphism between Noetherian schemes and $E\in \Dqc(X)$ be an $S$-perfect object. Assume that $s_0\in S$ is a point such that $\cH^{i}(E_{s_0})=0$ for $i\notin [a,b]$, where $a\leq b$ and $a,b\in \ZZ$.

    \begin{enumerate}
        \item There is an open neighborhood $U\subset S$ of $s_0$ such that $\cH^i(E_U)=0$ and $\cH^{i}(E_{s})=0$ for all $i\notin [a,b]$ and $s\in U$.

        \item If $\cH^{i}(E_{s})=0$ for $i\notin [a,b]$ and for all $s\in S$, then we can find a finite set of locally closed subschemes $\{S_j\}_{j\in J}$ of $S$ so that $\cH^i(E_{S_j})=0$ for $i\notin [a,b]$, $\cH^i(E_{S_j})$ is flat over $S_j$ for $i\in [a,b]$, and $\bigcup_{j\in J}S_j=S$ as a set.
    \end{enumerate}
\end{lemma}

\begin{proof}
By Lemma \ref{lem-S-perf-lem-1}, we have $E\in \Db(X)$. We may assume that $S$ is affine, so by \cite[Corollary 2.1.7]{lieblich:moduli-of-complex}, $E\cong P^{\bullet}$, where $P^{\bullet}$ is a bounded complex of coherent sheaves on $X$ such that each $P^i$ is flat over $S$. For each $m$, we consider the complex $$Q_m\coloneqq [P^{m-1}\to P^m\to P^{m+1}],$$ where $P^m$ sits in degree zero. When $m\notin [a,b]$, we know that $\cH^0((Q_m)_{s_0})=0$ by our assumption. Therefore, using \cite[Lemma 2.1.4]{lieblich:moduli-of-complex} and properness of $X\to S$, we get an open neighborhood $U_m\subset S$ of $s_0$ so that $\cH^0((Q_m)_{s})=0$ for any $s\in U_m$ and $\cH^0((Q_m)_{U_m})\cong 0$. Then we set $U\coloneqq \bigcap_{m\notin [a,b]} U_m$. Note that by the boundedness of $P^{\bullet}$, this intersection is a finite intersection, hence $U$ is open and part (a) follows.

For part (b), since the conclusion only requires a locally closed stratification covering $S$ as a set, we may replace $S$ by its reduction and take all strata with their reduced induced scheme structures. By part (a) and our assumption, we have $\cH^i(E)=0$ for $i\notin [a,b]$. Applying \cite[\href{https://stacks.math.columbia.edu/tag/052B}{Tag 052B}]{stacks-project} to $\bigoplus_{i=a}^b \cH^i(E)$, we can find a non-empty open subset $S_1$ of $S$ so that $\cH^i(E_{S_1})$ is flat over $S_1$ for each $i\in [a,b]$. Using part (a) to $E_{S\setminus S_1}$, we also have $\cH^i(E_{S\setminus S_1})=0$ for $i\notin [a,b]$. Applying \cite[\href{https://stacks.math.columbia.edu/tag/052B}{Tag 052B}]{stacks-project} again, we find an open subset $S_2$ of $S\setminus S_1$ so that $\cH^i(E_{S_2})$ is flat over $S_2$ for each $i\in [a,b]$. Continuing this process, for each $k\geq 1$, we get an open subset $S_k\subset S\setminus \bigcup_{t=1}^{k-1} S_t$ so that $\cH^i(E_{S_k})=0$ for $i\notin [a,b]$, $\cH^i(E_{S_k})$ is flat over $S_k$ for $i\in [a,b]$. Then we obtain a sequence of closed subschemes
\[\cdots \subset S\setminus \bigcup_{t=1}^{k-1} S_t\subset S\setminus \bigcup_{t=1}^{k-2} S_t\subset \cdots \subset S\setminus S_1\subset S.\]
Note that by \cite[\href{https://stacks.math.columbia.edu/tag/052B}{Tag 052B}]{stacks-project}, if $S\setminus \bigcup_{t=1}^{k-1} S_t$ is non-empty, then $S_k$ is also non-empty. So by Noetherian property of $S$, such sequence stabilizes at the empty set, i.e. $S=\bigcup_{j\in J} S_j$ for a finite set $J$. This proves part (b).
\end{proof}


\subsection{Dualizing complexes}

In this section, we review dualizing complexes of schemes. We mainly follow \cite{stacks-project}. Let $X$ be a Noetherian scheme that admits a dualizing complex $\omega^{\bullet}_X$ in the sense of \cite[\href{https://stacks.math.columbia.edu/tag/0A87}{Tag 0A87}]{stacks-project}. Then we have $\omega^{\bullet}_X\in \Db(X)$. In particular, there exists $j\in \ZZ$ so that $\cH^{j}(\omega^{\bullet}_X)\neq 0$ and $\cH^{<j}(\omega^{\bullet}_X)=0$. We then define the dualizing sheaf as $\omega_X\coloneqq \cH^{j}(\omega^{\bullet}_X)$. By \cite[\href{https://stacks.math.columbia.edu/tag/0AWK}{Tag 0AWK}]{stacks-project}, up to shifting $\omega^{\bullet}_X$ on each connected component of $X$, $\omega_X$ is torsion-free and $S_2$, and $\mathrm{Supp}(\omega_X)=X$. In this paper, we always choose $\omega^{\bullet}_X$ to satisfy this property.

Assume $X$ is equidimensional. Then by \cite[\href{https://stacks.math.columbia.edu/tag/0AWF}{Tag 0AWF}]{stacks-project}, after shifting on each connected component, we can normalize $\omega^{\bullet}_{X}$ so that $(\omega^{\bullet}_{X})_x\cong \omega^{\bullet}_{\oh_{X,x}}[\dim \overline{\{x\}}]$ for each point $x\in X$ (see also \cite[Lemma 2.65]{kollar-kovas:sing}) with $\omega^{\bullet}_{\oh_{X,x}}$ a normalized dualizing complex of $\oh_{X,x}$ in the sense of \cite[\href{https://stacks.math.columbia.edu/tag/0A7M}{Tag 0A7M}]{stacks-project}. In this case, we have
\begin{equation}\label{eq-Hi-omega}
    \cH^i(\omega^{\bullet}_X)=0 \text{ for }i\notin [-\dim X, 0]
\end{equation}
 and $\omega_X=\cH^{-\dim X}(\omega^{\bullet}_{X})$. Moreover, $X$ is Cohen--Macaulay if and only if $\omega^{\bullet}_{X}\cong \omega_{X}[\dim X]$, and in this case $\omega_X$ is also Cohen--Macaulay (cf.~\cite[\href{https://stacks.math.columbia.edu/tag/0AWT}{Tag 0AWT}]{stacks-project}). Similarly, $X$ is Gorenstein if and only if $\omega^{\bullet}_{X}[-\dim X]\cong \omega_{X}$ is a line bundle.

When $X$ is equidimensional, we set
\[\mathbb{D}_X\coloneqq R\cH om_X(-,\omega^{\bullet}_X)\]
and $\mathbb{D}^n_X\coloneqq \mathbb{D}_X[n]$ for any $n\in \ZZ$. By \cite[\href{https://stacks.math.columbia.edu/tag/0A89}{Tag 0A89}]{stacks-project}, $\mathbb{D}^n_X$ is an involutive anti-equivalence of $\Db(X)$. We also define
\[E^{\mathsf{d}}\coloneqq \cH om_X(E,\omega^{\bullet}_X[-\dim X])\in \Coh(X)\]
for any $E\in \Db(X)$.

For a scheme of finite type over a fixed base field $\kk$, we also denote by $\omega^{\bullet}_X$ the dualizing complex of $X$ normalized relative to $\kk[0]$ in the sense of \cite[\href{https://stacks.math.columbia.edu/tag/0AUA}{Tag 0AUA}]{stacks-project}. When $X$ is equidimensional, this coincides with the notion defined in the paragraph above. If $X$ is equidimensional and lci over $\kk$ with the resolution property, then $\mathbb{L}_{X/\kk}\in \Dperf(X)$ and $\det (\mathbb{L}_{X/\kk})\in \Pic(X)$ can be defined. In this case, we have
\begin{equation}\label{eq:det-LX-omega}
    \det (\mathbb{L}_{X/\kk})\cong \omega_X.
\end{equation}

More generally, for any morphism $f\colon X\to Y$ of finite type between Noetherian schemes, we denote by $\omega^{\bullet}_{X/Y}=\omega^{\bullet}_{f}$ the dualizing complex of $f$. By \cite[\href{https://stacks.math.columbia.edu/tag/0E2S}{Tag 0E2S}]{stacks-project}, $\omega^{\bullet}_{f}$ exists when $f$ is separated or flat. If $f$ is proper, then we can take $\omega^{\bullet}_{f}=f^!(\oh_Y)$. If $f$ is flat and proper, by \cite[\href{https://stacks.math.columbia.edu/tag/0C08}{Tag 0C08}]{stacks-project} and \cite[\href{https://stacks.math.columbia.edu/tag/0FPQ}{Tag 0FPQ}]{stacks-project}, we know that $f$ is Gorenstein if and only if $\omega^{\bullet}_f$ is invertible, i.e.~$\omega^{\bullet}_f \in \Dperf(X)$ and $-\otimes^{\mathrm{L}}\omega^{\bullet}_f$ is an auto-equivalence of $\Db(X)$.

The following result will be used repeatedly.

\begin{lemma}\label{lem-serre-duality}
Let $S$ be a Noetherian scheme with a dualizing complex $\omega^{\bullet}_{S}$ and $f\colon X\to S$ be a proper morphism. Then for any $E\in \D^-(X)$, we have
\[Rf_*R\cH om_X(E, \omega^{\bullet}_{X})\cong R\cH om_S(Rf_*E, \omega^{\bullet}_{S}),\]
where we take $\omega^{\bullet}_{X}=f^!\omega^{\bullet}_{S}$. If, in addition, $S$ is regular and $f$ is Gorenstein, then for any $E\in \Db(X)$ and $F\in \Dperf(X)$, we have
\[Rf_*R\cH om_X(E,F)\cong \mathbb{D}^S(Rf_*R\cH om_X(F, E\otimes^{\LL}\omega^{\bullet}_{f}))\in \Dperf(S)\]
and
\[Rf_*R\cH om_X(F,E)\cong \mathbb{D}^S(Rf_*R\cH om_X(E, F\otimes^{\LL}\omega_{f}^{\bullet}))\in \Dperf(S).\]
\end{lemma}

\begin{proof}
See (2) and (4) in \cite[\href{https://stacks.math.columbia.edu/tag/0AU3}{Tag 0AU3}]{stacks-project}.
\end{proof}

\begin{lemma}\label{lem-S-perf-lem-3}
Let $f\colon X\to S$ be a Gorenstein separated morphism between locally Noetherian schemes which is locally of finite presentation and $E\in \Dqc(X)$ be a $S$-perfect object. If $F\in \Dperf(X)$, then $R\cH om_X(E,F)$ is also $S$-perfect. 
\end{lemma}

\begin{proof}
By Lemma \ref{lem-S-perf-lem-2}(c), after replacing $E$ with $E\otimes^{\mathrm{L}} \mathbb{D}^X(F)$, we may assume that $F=\cO_X$. Since the assertion is local, we can also assume that both $X$ and $S$ are Noetherian. As $f$ is Gorenstein, by \cite[\href{https://stacks.math.columbia.edu/tag/0C08}{Tag 0C08}]{stacks-project} and \cite[\href{https://stacks.math.columbia.edu/tag/0FNT}{Tag 0FNT}]{stacks-project}, it suffices to show $R\cH om_X(E,f^!\cO_S)$ is $S$-perfect. Now, this follows from \cite[Proposition 2.3.9]{lipman:reflex-ii}.
\end{proof}

\subsection{Duals and hulls of coherent sheaves}

We now study Ext sheaves and hulls of coherent sheaves.
For any coherent sheaf $E$ on an equidimensional Noetherian scheme $X$ that admits a dualizing complex $\omega^{\bullet}_X$, we define $$E^{H}\coloneqq\cH om_X(\cH om_X(E, \omega_X),\omega_X).$$ This is called \emph{the hull of $E$}, and has a natural map $q_E\colon E\to E^H$. The hull $E^H$ plays the same role in our paper as reflexive hulls for sheaves on normal schemes. Note that if $X$ is Gorenstein, then $$E^H=E^{\vee \vee}=\cH om_X(\cH om_X(E, \cO_X),\cO_X).$$

We need the following properties. See \cite{kollar:duality} for a more general treatment without using dualizing complexes.

\begin{lemma}\label{lem-S2-hull}
Let $X$ be an equidimensional Noetherian scheme that admits a dualizing complex and $E$ be a coherent sheaf on $X$. Then

\begin{enumerate}
    \item $\cH om_X(E, \omega_X)$ and $E^H$ are torsion-free and $S_2$,

    \item if $E$ is $S_2$ and torsion-free, then $q_E$ is an isomorphism,

    \item the map $q_E\colon E\to E^H$ satisfies $\ker(q_E)=\mathrm{tor}(E)$ and $\codim_X(\mathrm{cok}(q_E))\geq 2$, and

    \item If $E$ is torsion-free, then $E$ is $S_2$ if and only if for any $T\in \Coh(X)$ with $\codim_X (T)\geq 2$, we have $\Hom_X(T, E[1])=0$.

\end{enumerate}
\end{lemma}

\begin{proof}
Part (a) follows from Lemma \ref{lem-general-S2}(a) and Lemma \ref{lem-general-S2}(b), since $\omega_X$ is $S_2$ and torsion-free.


For part (b), by \eqref{eq-Hi-omega}, we have
$$\cE xt_X^{-\dim X}(E, \omega^{\bullet}_X)\cong \cH om_X(E, \omega_X),$$
since $\cH^i(\omega^{\bullet}_X)=0$ for $i<-\dim(X)$. Moreover, as $E$ is $S_2$, by \cite[Proposition 2.66]{kollar-kovas:sing}, we have $(\cE xt_X^{-i}(E, \omega^{\bullet}_X))|_x=0$ for any $i\neq \dim(X)$ and point $x\in X$ of codimension at most one. Therefore, we get
$$(R\cH om_{X}(E, \omega^{\bullet}_{X}))|_x\cong R\cH om_{\oh_{X, x}}(E|_x, (\omega^{\bullet}_{X})|_x)\cong \cH om_{\oh_{X, x}}(E|_x, (\omega_{X})|_x)[\dim X].$$
Since $\cH om_{\oh_{X, x}}(E|_x, \omega_{\oh_{X, x}})$ is torsion-free and $S_2$ by part (a), and $(\omega^{\bullet}_{X})|_x\cong \omega^{\bullet}_{\oh_{X,x}}[\dim \overline{\{x\}}]$ by \cite[Lemma 2.65]{kollar-kovas:sing}, we get the following natural isomorphisms $$(E^H)|_x\cong(E|_x)^H\cong R\cH om_{\oh_{X, x}}(R\cH om_{\oh_{X, x}}(E|_x, \omega^{\bullet}_{\oh_{X, x}}), \omega^{\bullet}_{\oh_{X, x}})\cong E|_x.$$
In other words, $q_E$ is an isomorphism over any point $x\in X$ of codimension at most one. Since $q_E$ is a morphism between $S_2$ torsion-free sheaves, it is an isomorphism over the whole $X$ by \cite[\href{https://stacks.math.columbia.edu/tag/0AV8}{Tag 0AV8}]{stacks-project}.

Next, we prove part (c). Since $E^H$ is torsion-free, it is clear that $\mathrm{tor}(E)\subset \ker(q_E)$. So $q_E$ can be factored as $$E\to E/\mathrm{tor}(E)\to (E/\mathrm{tor}(E))^H\cong E^H.$$ Therefore, to prove $\codim_X(\mathrm{cok}(q_E))\geq 2$, we can assume that $E$ is torsion-free. In this case, $E$ is $S_1$, hence it is $S_2$ at any point of $X$ of codimension at most one, and the statement follows from part (b).

To show $\mathrm{tor}(E)= \ker(q_E)$, it remains to prove the injectivity of $q_E$ when $E$ is torsion-free. In this case, if $\ker(q_E)\neq 0$, then it is also torsion-free. Moreover, the composition of natural maps $$\ker(q_E)\to (\ker(q_E))^H\to E^H$$ is zero since it factors through $q_E$. However, by the above paragraph, $\ker(q_E)\to (\ker(q_E))^H$ and $E\to E^H$ is an isomorphism over an open subset $W$ of $X$, which implies that the zero map $$\ker(q_E)\to (\ker(q_E))^H\to E^H$$ is isomorphic to $\ker(q_E)\hookrightarrow E$ over $W$, a contradiction. This shows $\ker(q_E)=\mathrm{tor}(E)$.

Finally, the ``only if" part of (d) is proved in Lemma \ref{lem-general-S2}(d). Assume that for any $T\in \Coh(X)$ with $\codim_X (T)\geq 2$, we have $\Hom_X(T, E[1])=0$. Then by part (c), the natural exact sequence $$0\to E\to E^H\to \cok(q_E)\to 0$$ splits, which implies $\cok(q_E)\subset E^H$. Therefore, we obtain $\cok(q_E)=0$ by part (c) and the torsion-freeness of $E^H$ proved in (a). This proves that $E\cong E^H$ is $S_2$ by part (a).
\end{proof}

The following two lemmas give some bounds on the codimension of Ext-sheaves.

\begin{lemma}\label{lem-supp-ext}
Let $X$ be an equidimensional Noetherian scheme that admits a dualizing complex. Then for any $0\neq E\in \Coh(X)$ of codimension $c$, we have
$$\cE xt^{k}_X(E, \omega^{\bullet}_X[-\dim X])=0$$
for all $k<c$ and $\codim_X(\cE xt^{k}_X(E, \omega^{\bullet}_X[-\dim X]))\geq k$ for $k\geq c$. Moreover, we have $$\codim_X(\cE xt^{c}_X(E, \omega^{\bullet}_X[-\dim X]))=c.$$
\end{lemma}

\begin{proof}
Since $X$ is catenary (cf.~\cite[\href{https://stacks.math.columbia.edu/tag/0AWF}{Tag 0AWF}]{stacks-project}) and equidimensional, by \cite[Lemma 2.4, Proposition 4.1]{Heinrich:dim-formula}, we have
\[\codim_X \overline{\{x\}}+\dim \overline{\{x\}}=\dim X\]
for any point $x\in X$.

For the first vanishing part, by \cite[Proposition 2.66]{kollar-kovas:sing}, it is enough to prove
\[c-\dim X\leq -\dim E|_x-\dim \overline{\{x\}}\]
for any $x\in \Supp(E)$. Note that
$$\dim X-\dim E|_x-\dim \overline{\{x\}}=\codim_X \overline{\{x\}}-\dim E|_x=\dim(\oh_{X, x})-\dim E|_x\geq c,$$
then the result follows.

When $k\geq c$, for any point $x\in X$ of codimension $l< k$, we have
\[\dim \overline{\{x\}}= \dim X- l.\]
Hence, if $x\in \Supp(E)$, then 
\[-(k-\dim X)<\dim \overline{\{x\}}\leq \depth_x E+\dim \overline{\{x\}}.\]
Therefore, $(\cE xt^{k}_X(E, \omega^{\bullet}_X[-\dim X]))|_x=0$ by \cite[Proposition 2.66]{kollar-kovas:sing} and $$\codim_X(\cE xt^{k}_X(E, \omega^{\bullet}_X[-\dim X]))\geq k$$ follows. The last statement follows directly from the non-vanishing part of \cite[Proposition 2.66]{kollar-kovas:sing}.
\end{proof}



\begin{lemma}\label{lem-codim-Sp}
Let $X$ be an equidimensional Noetherian scheme that admits a dualizing complex. Then for any nonzero coherent sheaf $E$ pure of codimension $c$ on $X$, if $E$ is $S_{p}$ for $p\geq 0$, then we have
\[\codim_X(\cE xt^{k}_X(E, \omega^{\bullet}_X[-\dim X]))\geq k+p\]
for $k> c$.
\end{lemma}

\begin{proof}
Let $x\in X$ be a point of codimension $q<k+p$ and $x\in \Supp(E)$. Then by \cite[Proposition 2.66]{kollar-kovas:sing}, we know that $$(\cE xt^{k}_X(E, \omega^{\bullet}_X[-\dim X]))|_x=0$$ for
\[-k+\dim X<\mathrm{depth}_x E+\dim \overline{\{x\}}.\]
As $\mathrm{depth}_x E\geq \min\{p, q-c\}$ and $k>c$, we are done.
\end{proof}

\section{Intersection theory and Chern characters}\label{sec:intersect-ch}

In this section, we first review basic facts about intersection theory over a general base scheme, then we define and study Chern characters of bounded complexes of coherent sheaves on general schemes. We closely follow the treatment of Chow groups and bivariant Chow groups in \cite[\href{https://stacks.math.columbia.edu/tag/02P3}{Tag 02P3}]{stacks-project}.

\subsection{Basic notions}

We fix a locally Noetherian and universally catenary scheme $S$ and a dimension function $\delta\colon S\to \ZZ$, i.e.~we fix $(S, \delta)$ as in \cite[\href{https://stacks.math.columbia.edu/tag/02QL}{Tag 02QL}]{stacks-project}. 

For a morphism $f\colon X \to S$ of locally finite type, the \emph{$\delta$-dimension} is defined as in \cite[\href{https://stacks.math.columbia.edu/tag/02QP}{Tag 02QP}]{stacks-project}:
\[\dim_{\delta}(Z)\coloneqq \delta(f(x))+\mathrm{trdeg}_{\kappa(f(x))}\kappa(x),\]
where $Z\subset X$ is an irreducible closed subset and $x\in Z$ is the generic point. Using this, we can define the groups of cycles $\mathrm{Z}_{k}(X,\delta)$ of $\delta$-dimension $k$ and the Chow groups $\CH_k(X,\delta)$ for all $k\in \ZZ$ as in \cite[\href{https://stacks.math.columbia.edu/tag/02QQ}{Tag 02QQ}]{stacks-project} and \cite[\href{https://stacks.math.columbia.edu/tag/02RV}{Tag 02RV}]{stacks-project}. We set $\mathrm{Z}_*(X,\delta)\coloneqq \bigoplus_{k\in \ZZ} \mathrm{Z}_k(X,\delta)$ and $\CH_*(X,\delta)\coloneqq\bigoplus_{k\in \ZZ} \CH_k(X,\delta)$. 

In later sections, we will mostly work in the situations covered by the following examples, so the dimension function $\delta$ can be suppressed from the notation in most later applications.

\begin{example}\label{example-over-curve}
If $S$ is an integral Noetherian scheme of dimension $\leq 1$, then we can take $\delta(s)\coloneqq \dim \overline{\{s\}}$ for any $s\in S$ as in \cite[\href{https://stacks.math.columbia.edu/tag/02QN}{Tag 02QN}]{stacks-project}. We call this \emph{the standard dimension function}. If $X\to S$ is a closed morphism and locally of finite type, then the $\delta$-dimension of any integral closed subscheme $Z$ of $X$ is $\dim(Z)$ by \cite[\href{https://stacks.math.columbia.edu/tag/02JX}{Tag 02JX}]{stacks-project}. In this case, $\mathrm{Z}_k(X)\coloneqq \mathrm{Z}_k(X,\delta)$ and $\CH_k(X)\coloneqq \CH_k(X,\delta)$ coincide with the usual notions.
\end{example}

\begin{example}\label{rmk:dim-function-regular}
We fix an integral Noetherian Cohen--Macaulay scheme $S$ of finite Krull dimension. In this case, we always equip $S$ with the dimension function $\delta$ defined in \cite[\href{https://stacks.math.columbia.edu/tag/0F91}{Tag 0F91}]{stacks-project}, i.e.
\begin{equation}\label{eq:dim-function-regular}
\delta(s)\coloneqq -\dim (\mathcal{O}_{S,s}).
\end{equation}
By \cite[\href{https://stacks.math.columbia.edu/tag/02JX}{Tag 02JX}]{stacks-project}, we have
\[\dim_{\delta}(Z)\geq \dim(Z)-\dim(S),\]
for any morphism $X\to S$ locally of finite type and  integral closed subscheme $Z\subset X$, and equality holds if $X\to S$ is closed.

Note that if $\dim S\leq 1$, $\delta+\dim S$ coincides with the standard dimension function of Example \ref{example-over-curve}. So up to a degree shift, the different choices of $\delta$ in Example \ref{example-over-curve} and \ref{rmk:dim-function-regular}  do not affect any result.
\end{example}

As in the situation of \cite{fulton:intersection-theory}, the functor $\CH_*(-,\delta)$ is covariant with respect to proper morphisms and contravariant with respect to flat or embeddable lci morphisms between schemes of locally finite type over $S$, see e.g.~\cite[\href{https://stacks.math.columbia.edu/tag/02S0}{Tag 02S0}]{stacks-project} and \cite[\href{https://stacks.math.columbia.edu/tag/0FF3}{Tag 0FF3}]{stacks-project}.

Following \cite[\href{https://stacks.math.columbia.edu/tag/0B76}{Tag 0B76}]{stacks-project}, we can define \emph{the bivariant Chow group} $\A^p(X\to Y,\delta)$ of degree $p$ associated with every morphism $X\to Y$ of locally finite type over $S$. An element $c\in \A^p(X\to Y,\delta)$ is a rule that assigns to every morphism locally of finite type $Y'\to Y$ and every $k\in \ZZ$ a group homomorphism
\[c\cap -\colon \CH_k(Y',\delta)\to \CH_{k-p}(X',\delta),\]
where $X'\coloneqq X\times_Y Y'$, satisfying the usual compatibilities with flat pullback, proper pushforward, and Gysin maps of effective Cartier divisors. We set $$\A^p(X,\delta)\coloneqq \A^p(X\xra{\mathrm{id}} X,\delta)$$ and $\A^*(X,\delta)\coloneqq\bigoplus_{p\in \ZZ} \A^p(X,\delta)$. If $X$ is quasi-compact, then for any $E\in \Dperf(X)$, we can define the Chern classes, Chern characters, and Todd classes $$c_p(E),\ch(E),\td(E)\in \A^*(X,\delta)_{\QQ}$$ in the usual way (see \cite[\href{https://stacks.math.columbia.edu/tag/0GUE}{Tag 0GUE}]{stacks-project}).

By definition, for any morphism $f\colon Y'\to Y$ which is locally of finite type, we have a natural restriction homomorphism $$\A^p(X\to Y,\delta)\to \A^p(X'\to Y',\delta)$$ as in \cite[\href{https://stacks.math.columbia.edu/tag/0F9Z}{Tag 0F9Z}]{stacks-project}.

If $f\colon X\to Y$ is a flat or an embeddable lci morphism between schemes of locally finite type over $S$, we denote by $[f]\in \A^*(X\to Y,\delta)$ the bivariant class corresponding to flat pullback in the flat case, or to the Gysin map in the embeddable lci case.

For any two morphisms $f\colon X\to Y$ and $g\colon Y\to Z$ between schemes of locally finite type over $S$, we have a natural associative bilinear product map
\[\A^p(X\to Y,\delta)\times \A^q(Y\to Z,\delta)\to \A^{p+q}(X\to Z,\delta),\quad (c, c')\mapsto c\cdot c'.\]
If $f$ is proper, then we have a natural pushforward homomorphism $$f_*\colon \A^p(X\to Z,\delta)\to \A^p(Y\to Z,\delta)$$ as in \cite[\href{https://stacks.math.columbia.edu/tag/0EPK}{Tag 0EPK}]{stacks-project}.


\subsection{Flat base change}

Following \cite[\href{https://stacks.math.columbia.edu/tag/0FVF}{Tag 0FVF}]{stacks-project}, we recall the behavior of Chow groups under flat base change. Let $S'$ be a locally Noetherian and universally catenary scheme with a dimension function $\delta'$. We fix a flat morphism $g\colon S'\to S$ such that there exists $d\in \ZZ$ with
\[\delta'(s')=\delta(s)+d\]
for any $s\in S$ and any generic point $s'$ of any irreducible component of $g^{-1}(s)$.

\begin{example}
Assume that $S$ is an integral Noetherian scheme of dimension $\leq 1$ with the standard dimension function as in Example \ref{example-over-curve}. Let $S'$ be another integral Noetherian scheme of dimension $\leq 1$ with a flat morphism $g\colon S'\to S$ (not necessarily locally of finite type). Then by \cite[\href{https://stacks.math.columbia.edu/tag/08EI}{Tag 08EI}]{stacks-project}, $g$ is dominant and the integer $$d\coloneqq \dim S'-\dim S\in\{-1,0,1\}$$ satisfies $\dim(\overline{\{s'\}})=\dim(\overline{\{s\}})+d$ for any $s\in g(S')$ and any generic point $s'$ of $g^{-1}(\{s\})$.
\end{example}

The construction in \cite[\href{https://stacks.math.columbia.edu/tag/0FVH}{Tag 0FVH}]{stacks-project} gives a well-defined pullback homomorphism
\[g^*\colon \mathrm{Z}_{k}(X,\delta)\to \mathrm{Z}_{k+d}(X_{S'},\delta')\]
for each $k$, defined by $[Z]\mapsto [Z_{S'}]$ for each integral closed subscheme $Z\subset X$ with $\dim_{\delta}(Z)=k$. Moreover, it factors through rational equivalences and gives a homomorphism
\[g^*\colon \CH_{k}(X,\delta)\to \CH_{k+d}(X_{S'},\delta').\]

In the following, we collect some properties of $g^*$.

\begin{lemma}\label{lem-flat-base-change}
Let $h\colon Y\to X$ be a morphism which is locally of finite type.

\begin{enumerate}
    \item If $E_1,\dots, E_m$ are bounded complexes of finite locally free sheaves on $X$, then the diagram
\[\begin{tikzcd}
	{\CH_k(X,\delta)} & {\CH_{k+d}(X_{S'},\delta')} \\
	{\CH_{k-p}(X,\delta)} & {\CH_{k+d-p}(X_{S'},\delta')}
	\arrow["{g^*}", from=1-1, to=1-2]
	\arrow["{\prod^m_{i=1}c_{p_i}(E_i)\cap -}"', from=1-1, to=2-1]
	\arrow["{\prod^m_{i=1}c_{p_i}((E_i)_{S'})\cap -}", from=1-2, to=2-2]
	\arrow["{g^*}", from=2-1, to=2-2]
\end{tikzcd}\]
commutes for each $p_i\geq 0$, where $p=\sum^m_{i=1} p_i$.

\item If $h$ is proper, then the diagram
\[\begin{tikzcd}
	{\CH_k(Y,\delta)} & {\CH_{k+d}(Y_{S'},\delta')} \\
	{\CH_{k}(X,\delta)} & {\CH_{k+d}(X_{S'},\delta')}
	\arrow["{g^*}", from=1-1, to=1-2]
	\arrow["{h_*}"', from=1-1, to=2-1]
	\arrow["{(h_{S'})_*}", from=1-2, to=2-2]
	\arrow["{g^*}", from=2-1, to=2-2]
\end{tikzcd}\]
commutes.

\item If $h$ is flat of relative dimension $r$, then the diagram
\[\begin{tikzcd}
	{\CH_{k}(X,\delta)} & {\CH_{k+d}(X_{S'},\delta')} \\
	{\CH_{k+r}(Y,\delta)} & {\CH_{k+d+r}(Y_{S'},\delta')}
	\arrow["{g^*}", from=1-1, to=1-2]
	\arrow["{h^*}"', from=1-1, to=2-1]
	\arrow["{(h_{S'})^*}", from=1-2, to=2-2]
	\arrow["{g^*}", from=2-1, to=2-2]
\end{tikzcd}\]
commutes.
\end{enumerate}

\end{lemma}

\begin{proof}
Part (a) follows from \cite[\href{https://stacks.math.columbia.edu/tag/0FVn}{Tag 0FVN}]{stacks-project}. Parts (b) and (c) follow from \cite[\href{https://stacks.math.columbia.edu/tag/0FVL}{Tag 0FVL}]{stacks-project} and \cite[\href{https://stacks.math.columbia.edu/tag/0FVK}{Tag 0FVK}]{stacks-project}, respectively.
\end{proof}


We introduce the following notion for later use.

\begin{definition}\label{def-A-star}
Let $X\to B$ be a universally closed morphism of finite type. For any $p\in \ZZ$, we define $\A^p_{\star}(X/B)$ to be the abelian group such that each element $c$ in it corresponds to a rule that assigns to each morphism $D\to B$ from an integral Noetherian regular scheme $D$ of dimension $\leq 1$ a class $c_D\in \A^p(X_D)$, such that for any morphism $g\colon D'\to D$ from an integral Noetherian regular scheme $D'$ of dimension $\leq 1$, we have:

\begin{itemize}
    \item if $g$ is locally of finite type, then $c_{D'}\in \A^p(X_{D'})$ is the restriction of $c_D\in \A^p(X_{D})$, and

    \item if $g$ is a dominant morphism (hence flat), the diagram
\[\begin{tikzcd}
	{\CH_k(X_D)} & {\CH_{k+d}(X_{D'})} \\
	{\CH_{k-p}(X_D)} & {\CH_{k+d-p}(X_{D'})}
	\arrow["{g^*}", from=1-1, to=1-2]
	\arrow["{c_D\cap -}"', from=1-1, to=2-1]
	\arrow["{c_{D'}\cap -}", from=1-2, to=2-2]
	\arrow["{g^*}", from=2-1, to=2-2]
\end{tikzcd}\]
commutes, where $d\coloneqq \dim D'-\dim D$.
\end{itemize}
\end{definition}

We set $\A^*_{\star}(X/B)\coloneqq\bigoplus_{p\in \ZZ} \A^p_{\star}(X/B)$. By Lemma \ref{lem-flat-base-change}(a), it is clear that the usual characteristic classes of perfect complexes lie in $\A^*_{\star}(X/B)$. In particular, we have the following.

\begin{lemma}
Let $X\to B$ be a universally closed morphism of finite type between Noetherian schemes. Then for each $p\geq 0$ and $E\in \Dperf(X)$, we have a class
\[c_p(E)\in \A^p_{\star}(X/B).\]
Moreover, if $E_1, \dots, E_m\in \Dperf(X)$, then $\prod_{i=1}^m c_{p_i}(E_i)\in \A^{\sum^m_{i=1}p_i}_{\star}(X/B)$ for $p_i\geq 0$, so we also have a group homomorphism
\[\ch_k(-)\colon \KK^0(X)\to \A^k_{\star}(X/B)_{\QQ}\]
for each $k\geq 0$.
\end{lemma}

\begin{proof}
In this case, any perfect complex on $X_{B'}$ has well-defined Chern classes for every morphism $B'\to B$ from a Noetherian scheme $B'$ by \cite[\href{https://stacks.math.columbia.edu/tag/0GUE}{Tag 0GUE}]{stacks-project}. Hence, we get a class $c_p(E)\in \A^p_{\star}(X/B)$ by Lemma \ref{lem-flat-base-change}(a) that satisfies the desired properties.
\end{proof}

\subsection{Localized Chern characters}

We also need localized Chern characters constructed in \cite[\href{https://stacks.math.columbia.edu/tag/0FB0}{Tag 0FB0}]{stacks-project}, which generalize \cite[Section 18.1]{fulton:intersection-theory}. As above, fix a locally Noetherian and universally catenary scheme $S$ and a dimension function $\delta\colon S\to \ZZ$. 

\begin{lemma}\label{lem-local-ch}
Let $i\colon X\hookrightarrow P$ be a closed embedding between schemes locally of finite type over $S$ with $P$ quasi-compact. Then for any perfect complex $E$ on $P$ such that $E_{P\setminus X}=0$, there exists a canonical class $$\ch^P_X(E)\in \A^*(X\to P,\delta)_{\QQ}$$ such that
\begin{equation}\label{eq-local-ch}
    i_*\ch^P_X(E)=\ch(E)\in \A^*(P,\delta)_{\QQ}
\end{equation}
and satisfies the following properties.

\begin{enumerate}
    \item Assume $P'\to P$ is a morphism locally of finite type and $E'$ and $X'$ are the pullbacks of $E$ and $X$ to $P'$, respectively. Then $\ch^{P'}_{X'}(E')\in  \A^*(X'\to P',\delta)_{\QQ}$ is the restriction of $\ch^P_X(E)$.

    \item $\ch^P_X(E)\cap i_*\alpha=\ch(Li^*E)\cap \alpha$ for any $\alpha\in \CH_*(X,\delta)_{\QQ}$.

    \item If $i$ is a composition $X\hookrightarrow X'\hookrightarrow P$, where $j\colon X\hookrightarrow X'$ and $X'\hookrightarrow P$ are both closed embeddings, then
    \[\ch_{X'}^P(E)=j_*\ch^P_X(E)\in \A^*(X'\to P,\delta)_{\QQ}.\]

    \item Let $Y\to P$ be a morphism locally of finite type and $c\in \A^*(Y\to P,\delta)_{\QQ}$. Then
    \[c\cdot \ch^P_X(E)=\ch^P_X(E)\cdot c\]
    in $\A^*(Y\times_P X\to P,\delta)_{\QQ}$.

    \item Let $E_1\to E\to E_2$ be an exact triangle of perfect complexes on $P$ such that $E_i|_{P\setminus X}=0$ for $i=1,2$. Then
    \[\ch^P_X(E)=\ch^P_X(E_1)+ \ch^P_X(E_2).\]
    
    \item Let $F$ be another perfect complex on $P$ such that $F|_{P\setminus X'}=0$ for a closed subscheme $X'\subset P$. Then
    \[\ch^P_{X\cap X'}(E\otimes^{\LL} F)=\ch^P_X(E)\cdot \ch^P_{X'}(F).\]
    
\end{enumerate}
\end{lemma}

\begin{proof}
These are the contents of \cite[\href{https://stacks.math.columbia.edu/tag/0FB0}{Tag 0FB0}]{stacks-project} and \cite[\href{https://stacks.math.columbia.edu/tag/0FB9}{Tag 0FB9}]{stacks-project}.
\end{proof}

\begin{lemma}\label{lem-homotopy}
Let $X\subset P$ be a closed embedding of schemes locally of finite type over $S$ and let $E$ be a perfect complex on $P\times_S \bA^1_S$ with $P$ quasi-compact and $E|_{(P\setminus X)\times_S \bA^1_S}=0$. Let $i_0, i_1\colon S\to \bA^1_S$ be two sections of $\bA^1_S\to S$. Then for any $\alpha\in \CH_*(P,\delta)_{\QQ}$,
\[\ch^P_X(E_0)\cap \alpha=\ch^P_X(E_1)\cap \alpha,\]
where $E_t$ is the derived restriction of $E$ along $i_t\times_S \mathrm{id}_{P}$ for each $t\in \{0,1\}$.
\end{lemma}

\begin{proof}
By Lemma \ref{lem-local-ch}(a), for each $t\in \{0,1\}$ and $\alpha\in \CH_*(P,\delta)_{\QQ}$, we have
\[\ch^P_X(E_t)\cap \alpha=i_t^*(\ch^{P\times_S \bA^1_S}_{X\times_S \bA^1_S}(E)\cap p^*\alpha),\]
where $$i_t^*\colon \CH_*(X\times_S \bA^1_S,\delta)_{\QQ}\to \CH_*(X,\delta)_{\QQ}$$ and $p\colon P\times_S \bA^1_S\to P$ is the projection. Then the result follows from the fact that $i_0^*=i_1^*$ by \cite[\href{https://stacks.math.columbia.edu/tag/02TY}{Tag 02TY}]{stacks-project}.
\end{proof}

Using Lemma \ref{lem-homotopy} and the deformation to the normal cone \cite[\href{https://stacks.math.columbia.edu/tag/0FBG}{Tag 0FBG}]{stacks-project}, the following lemma follows from the same proof as \cite[Corollary 18.1.2]{fulton:intersection-theory}.

\begin{lemma}\label{lem-tau-XYZ}
Let $i\colon X\to Y$ and $j\colon Y\to Z$ be closed embeddings of schemes locally of finite type over $S$, where $j$ is a regular embedding with the normal bundle $\cN$ and $Z$ is quasi-compact. Fix $E\in \Coh(X)$ so that $i_*E$ and $j_*i_*E$ have resolutions by bounded complexes of finite rank locally free sheaves. Then
\[\ch^Z_X(j_*i_*E)=\ch^Y_X(i_*E)\cdot \td^{-1}(\cN)\cdot [j]\in \A_*(X\to Z,\delta)_{\QQ}.\]
\end{lemma}

The localized Chern characters also behave well under flat base change of $S$.

\begin{lemma}\label{lem-flat-local-ch}
Let $S'$ be a locally Noetherian and universally catenary scheme with a dimension function $\delta'$, and $g\colon S'\to S$ be a flat morphism such that there exists $d\in \ZZ$ with $\delta'(s')=\delta(s)+d$ for any $s\in S$ and any generic point $s'$ of $g^{-1}(s)$. Let $i\colon X\hookrightarrow P$ be a closed embedding between schemes locally of finite type over $S$ with $P$ and $P_{S'}$ quasi-compact.

Then for any perfect complex $E$ on $P$ such that $E|_{P\setminus X}=0$, the formation of $\ch^P_X(E)$ commutes with the base change $S'\to S$, i.e.~we have
\[g^*(\ch^P_X(E)\cap \alpha)=\ch^{P_{S'}}_{X_{S'}}(E_{S'})\cap g^*\alpha \in \CH_*(X_{S'},\delta')_{\QQ}\]
for any $\alpha\in \CH_*(P,\delta)_{\QQ}$.
\end{lemma}

\begin{proof}
This follows from the construction of $\ch^P_X(E)$. Indeed, by Lemma \ref{lem-flat-base-change}, we see that the constructions in \cite[\href{https://stacks.math.columbia.edu/tag/0F9H}{Tag 0F9H}]{stacks-project} and \cite[\href{https://stacks.math.columbia.edu/tag/0F9J}{Tag 0F9J}]{stacks-project} commute with flat base change. Therefore, the class constructed in \cite[\href{https://stacks.math.columbia.edu/tag/0F9K}{Tag 0F9K}]{stacks-project} commutes with flat base change, since the assumptions are preserved under flat base change and the formation only uses proper pushforward, \cite[\href{https://stacks.math.columbia.edu/tag/0F9H}{Tag 0F9H}]{stacks-project}, and \cite[\href{https://stacks.math.columbia.edu/tag/0F9J}{Tag 0F9J}]{stacks-project}. Moreover, the blow-up construction in \cite[\href{https://stacks.math.columbia.edu/tag/0F8Z}{Tag 0F8Z}]{stacks-project} is compatible with flat base change as well. Then we can conclude that the formation of $\ch^P_X(E)$ in \cite[\href{https://stacks.math.columbia.edu/tag/0FB2}{Tag 0FB2}]{stacks-project} commutes with flat base change.
\end{proof}


\subsection{Riemann--Roch homomorphism}

We fix an integral Noetherian regular scheme $S$ of finite Krull dimension. In this case, we always equip $S$ with the dimension function $\delta$ defined in \eqref{eq:dim-function-regular}.

Using localized Chern characters, we can construct a well-behaved Riemann--Roch homomorphism for any quasi-projective scheme over $S$. The following result is already mentioned in \cite[pp. 395]{fulton:intersection-theory}.

\begin{theorem}\label{thm-tau}
Let $X$ be a scheme quasi-projective over $S$. Then there exists a homomorphism
\[\tau_{X/S}\colon \KK_0(X)\to \CH_*(X,\delta)_{\QQ}\]
that satisfies the following properties. Let $Y$ be another scheme quasi-projective over $S$.

\begin{enumerate}
    \item If $f\colon X\to Y$ is proper, then for any class $\xi\in \KK_0(X)$, we have
    \[\tau_{Y/S}(f_*\xi)=f_*\tau_{X/S}(\xi).\]

    \item If $f\colon X\to Y$ is lci, then for any class $\xi\in \KK_0(Y)$, we have
    \[\tau_{X/S}(f^*\xi)=\td(\mathbb{T}_f)\cap f^*\tau_{Y/S}(\xi),\]
where $\mathbb{T}_f\coloneqq\mathbb{D}^X(\mathbb{L}_f)$ is the tangent complex of $f$ and $\mathbb{L}_f$ is the cotangent complex of $f$.

    \item If $\alpha\in \KK^0(X)$ and $\xi\in \KK_0(X)$, then
    \[\tau_{X/S}(\alpha\otimes \xi)=\ch(\alpha)\cap \tau_{X/S}(\xi).\]

    \item If $X$ is lci over $S$, then
    \[\tau_{X/S}(\oh_X)=\Td(X/S)\coloneqq\td(X/S)\cap [X],\]
where $\td(X/S)\coloneqq \td(\mathbb{T}_{X/S})$.

    \item If $V\subset X$ is a closed subscheme pure of $\delta$-dimension, then
    \[\tau_{X/S}(\cO_V)=[V]+\text{ terms of lower }\delta\text{-dimensions}.\]

\end{enumerate}
\end{theorem}

\begin{proof}
We factor $X\to S$ as $X\hookrightarrow P\to S$, where $i\colon X\hookrightarrow P$ is a closed embedding and $P$ is smooth over $S$. Then $P$ is also regular. Therefore, for any $E\in \Db(X)$, we have $i_*E\in \Dperf(P)$. We define 
\[\tau_{X/S}(E)\coloneqq\tau^P_X(E)\coloneqq\ch^P_X(i_*E)\cdot \td(P/S)\cap [P].\]
The proof of independence of the choice of $P$, and the proofs of (a)-(c), follow from the argument of \cite[Theorem 18.2]{fulton:intersection-theory}. Indeed, Step 1 of \cite[Theorem 18.2]{fulton:intersection-theory} follows from Lemma \ref{lem-tau-XYZ} and Step 2 is a consequence of Lemma \ref{lem-local-ch}(f). Similarly, Step 3 of \cite[Theorem 18.2]{fulton:intersection-theory} can be deduced from \eqref{eq-local-ch}. Step 4 follows from Lemma \ref{lem-local-ch}(a). Next, Step 5 follows from the same calculation as in \cite[Theorem 18.2]{fulton:intersection-theory} by using the projective bundle formula of K-groups and Chow groups (cf.~\cite[\href{https://stacks.math.columbia.edu/tag/02TX}{Tag 02TX}]{stacks-project}). Then the claim that $\tau^P_X(E)$ is independent of the choice of $P$, which is Step 6 of \cite[Theorem 18.2]{fulton:intersection-theory}, can be argued in the same way as in \cite[Theorem 18.2]{fulton:intersection-theory} using previous steps. Similarly, relying on previous steps, the arguments of Steps 7, 8 and 9, together with \cite[Lemma 18.2]{fulton:intersection-theory}, apply verbatim, which proves (a) and (c). Finally, Step 10 also follows from the same argument by using the deformation to the normal cone in \cite[\href{https://stacks.math.columbia.edu/tag/0FBG}{Tag 0FBG}]{stacks-project}, which proves (b).

For part (d), when $X$ is smooth over $S$, we see $\ch^X_X(\oh_X)=\ch(\oh_X)=1\in \A^0(X)$. Therefore, we have $\tau_{X/S}(\oh_X)=\Td(X/S)$ in this case. When $X$ is lci over $S$, we embed $X$ into a scheme smooth over $S$, and the result follows from the smooth case and part (b).

Finally, we prove (e). By part (b), we may assume that $X\to S$ is projective. When $X$ is a projective bundle over $S$, we have $\oh_V\in \Db(X)=\Dperf(X)$, hence 
$$\tau_{X/S}(\oh_V)=\ch(\oh_V)\cdot \td(X/S)\cap [X]=[V]+\text{ terms of lower }\delta\text{-dimensions}$$
as $(\ch(\oh_V)\cap [X])_{\dim_{\delta} V}=[V]$. For the general case, we can compactify $X$ over $S$ and find a surjective finite morphism $X\to P$ to a projective bundle $P$ over $S$, then the rest of the argument is the same as Step 8 of \cite[pp.359]{fulton:intersection-theory}.
\end{proof}

When $S$ is a field $\kk$, we always omit $S$ from the notation. If $X$ is also proper over $\kk$, following \cite[Example 19.1.5(b)]{fulton:intersection-theory}, we define
\[\CH^i_{\num}(X)_{\QQ}\]
as the quotient of $\CH_{\dim X-i}(X)_{\QQ}$ by the subgroup of numerically trivial classes, i.e. those $\xi\in \CH_{\dim X-i}(X)_{\QQ}$ with $\int_X \ch(\alpha)\cap \xi=0$ for any $\alpha\in \KK^0(X)$. By Theorem \ref{thm-tau}(a) and (c), we know that $\tau_{X}$ factors through the natural surjection $\KK_0(X)\to \Knum(X)$, where $\Knum(X)$ is the numerical K-group of $\Db(X)$ (cf.~Section \ref{subsec:num-k}).

The Riemann--Roch homomorphism commutes with flat base change of $S$.

\begin{lemma}\label{lem-tau-flat-base-change}
Let $S'$ be an integral Noetherian regular scheme of finite Krull dimension with the dimension function $\delta'$ as in \eqref{eq:dim-function-regular}, and $g\colon S'\to S$ be a flat morphism such that there exists $d\in \ZZ$ with $\delta'(s')=\delta(s)+d$ for any $s\in S$ and any generic point $s'$ of $g^{-1}(s)$. Let $X$ be a quasi-projective scheme over $S$. Then
\[g^*\tau_{X/S}(\xi)=\tau_{X_{S'}/S'}(g^*_X\xi)\in \CH_*(X_{S'},\delta')_{\QQ}\]
for any $\xi\in \KK_0(X)$.
\end{lemma}

\begin{proof}
Let $i\colon X\hookrightarrow P$ be a closed embedding such that $P$ is smooth over $S$. Then
\[g^*\tau_{X/S}(\xi)=g^*(\ch^P_X(i_*\xi)\cdot \td(P/S)\cap [P])\]
\[=\ch^{P_{S'}}_{X_{S'}}(g^*_Pi_*\xi)\cdot \td(P_{S'}/S')\cap [P_{S'}]\]
where the first equality is the definition of $\tau_{X/S}$ and the second one follows from Lemma \ref{lem-flat-base-change}(a) and Lemma \ref{lem-flat-local-ch}. Then the result follows from $\ch^{P_{S'}}_{X_{S'}}(g^*_Pi_*\xi)=\ch^{P_{S'}}_{X_{S'}}(i_{S' *}g^*_X\xi)$ by flat base change theorem.
\end{proof}

Next, we describe the behavior of the Riemann--Roch homomorphism under the specialization. Assume that $S$ is the spectrum of a DVR with the closed point $p\in S$ and the fraction field $K$. Let $X$ be a quasi-projective scheme over $S$. Then we have an open embedding $i_K\colon X_K\hookrightarrow X$ and a closed embedding $i_p\colon X_p\hookrightarrow X$, which induce exact sequences
\[\KK_0(X_p)\xra{i_{p*}}\KK_0(X)\xra{i_K^*} \KK_0(X_K)\to 0\]
and 
\[\CH_*(X_p)\xra{i_{p*}}\CH_*(X)\xra{i_K^*} \CH_*(X_K)\to 0\]
by \cite[Exercise II.6.10]{hartshorne:gtm52} and \cite[\href{https://stacks.math.columbia.edu/tag/02RX}{Tag 02RX}]{stacks-project}. As the normal bundle of $p$ in $S$ is trivial, we see that $i_p^*\circ i_{p_*}=0$ both at the level of K-groups and Chow groups. In other words, there are canonical homomorphisms $\mathrm{sp}$, which fill the dashed arrows in the diagrams
\[\begin{tikzcd}
	{\KK_0(X_p)} & {\KK_0(X)} & {\KK_0(X_K)} & 0 \\
	& {\KK_0(X_p)}
	\arrow["{i_{p*}}", from=1-1, to=1-2]
	\arrow["{i_K^*}", from=1-2, to=1-3]
	\arrow["{i_p^*}", from=1-2, to=2-2]
	\arrow[from=1-3, to=1-4]
	\arrow["{\mathrm{sp}}", dashed, from=1-3, to=2-2]
\end{tikzcd}\]
and
\[\begin{tikzcd}
	{\CH_*(X_p)} & {\CH_*(X)} & {\CH_*(X_K)} & 0 \\
	& {\CH_*(X_p)}
	\arrow["{i_{p*}}", from=1-1, to=1-2]
	\arrow["{i_K^*}", from=1-2, to=1-3]
	\arrow["{i_p^*}", from=1-2, to=2-2]
	\arrow[from=1-3, to=1-4]
	\arrow["{\mathrm{sp}}", dashed, from=1-3, to=2-2]
\end{tikzcd}\]
which are called \emph{specialization maps}. When $X=S$, it is clear from the above definition that $$\mathrm{sp}\colon \KK_0(\Spec(K))\to \KK_0(\Spec(\kappa(p)))$$ and
$$\quad\mathrm{sp}\colon \CH_*(\Spec(K))\to \CH_*(\Spec(\kappa(p)))$$ are isomorphisms.

A standard diagram chase gives the following lemma (see also \cite[Proposition 20.3]{fulton:intersection-theory}).

\begin{lemma}\label{lem-sp-push}
If $X\to S$ is proper, then we have a commutative diagram
\[\begin{tikzcd}
	{\CH_*(X_K)} & {\CH_*(X_p)} \\
	{\CH_*(\Spec(K))} & {\CH_*(\Spec(\kappa(p)))}
	\arrow["{\mathrm{sp}}", from=1-1, to=1-2]
	\arrow[from=1-1, to=2-1]
	\arrow[from=1-2, to=2-2]
	\arrow["{\mathrm{sp}}", from=2-1, to=2-2]
\end{tikzcd}\]
where vertical arrows are proper pushforwards.
\end{lemma}

The following result shows that specialization maps are compatible with the Riemann--Roch homomorphism.

\begin{lemma}[{\cite[Example 20.3.4]{fulton:intersection-theory}}]\label{lem-tau-special}
Let $X$ be a quasi-projective scheme over $S$. Assume that $S$ is the spectrum of a DVR with the closed point $p\in S$ and the fraction field $K$. Then the diagram
\[\begin{tikzcd}
	{\KK_0(X_K)} & {\CH_*(X_K)_{\QQ}} \\
	{\KK_0(X_p)} & {\CH_*(X_p)_{\QQ}}
	\arrow["{\tau_{X_K/K}}", from=1-1, to=1-2]
	\arrow["{\mathrm{sp}}"', from=1-1, to=2-1]
	\arrow["{\mathrm{sp}}", from=1-2, to=2-2]
	\arrow["{\tau_{X_p/\kappa(p)}}", from=2-1, to=2-2]
\end{tikzcd}\]
commutes.
\end{lemma}

\subsection{Chern characters}\label{subsec-chern}

We now introduce a well-behaved notion of Chern characters for bounded coherent complexes. We fix an integral Noetherian regular scheme $S$ of finite Krull dimension with the dimension function $\delta$ as in \eqref{eq:dim-function-regular}.

Recall that for a scheme $X$ that is quasi-projective and lci over $S$, we define the \emph{virtual Todd class of $X$ over $S$} as
\[\td(X/S)\coloneqq\td(\mathbb{T}_{X/S})\in \A^*(X,\delta)_{\QQ}.\]
Since $X$ is quasi-projective and lci over $S$, the cotangent complex $\mathbb{L}_{X/S}$ and the tangent complex $\mathbb{T}_{X/S}$ are perfect complexes. Moreover, $X$ is quasi-compact, so $\td(\mathbb{T}_{X/S})\in \A^*(X,\delta)_{\QQ}$ is well-defined.

It is clear that when $X$ is smooth over $S$, the class $\td(X/S)$ is exactly the usual Todd class of $X$ over $S$ defined by the relative tangent bundle. More generally, the class $[\mathbb{T}_{X/S}]\in \KK^0(X)$ is the same as the virtual tangent bundle in the sense of \cite[B.7.6]{fulton:intersection-theory}, and $\td(X/S)$ is the same as the notion of virtual Todd class used in \cite{fulton:intersection-theory}.

Note that we can always take the formal inverse for classes in $\A^*(X,\delta)_{\QQ}$ whose degree-zero part is $1$ (cf.~\cite[\href{https://stacks.math.columbia.edu/tag/0ESY}{Tag 0ESY}]{stacks-project}). Then motivated by the Grothendieck--Riemann--Roch Theorem, we define the following notion.

\begin{definition}\label{def:ch-lci}
Let $X$ be a scheme quasi-projective and lci over $S$. For any class $\xi\in \KK_0(X)$, we define the \emph{Chern character of $\xi$ over $S$} as
\[\bch(\xi)_S\coloneqq\td^{-1}(X/S)\cap \tau_{X/S}(\xi) \in \CH_*(X,\delta)_{\QQ}.\]
We denote the component of $\bch(\xi)_S$ of $\delta$-dimension $\dim_{\delta} X-i$ by 
\[\bch_i(\xi)_S\in \CH_{\dim_{\delta} X-i}(X,\delta)_{\QQ}.\]
\end{definition}

Therefore, we get a homomorphism
\[\bch_i(-)_S\colon \KK_0(X)\to \CH_{\dim_{\delta} X-i}(X,\delta)_{\QQ}.\]
We will omit $S$ in the notation if it is clear.

Using Theorem \ref{thm-tau}, we can prove some general properties of Chern characters.

\begin{lemma}\label{lem-chM-general}
Let $X$ and $Y$ be schemes quasi-projective and lci over $S$.

\begin{enumerate}
    \item If $f\colon X\to Y$ is proper, then for any class $\xi\in \KK_0(X)$, we have
    \[f_*(\td(X/S)\cap \bch(\xi)_S)=\td(Y/S)\cap \bch(f_*\xi)_S.\]

    \item If $f\colon X\to Y$ is lci or flat, then for any class $\xi\in \KK_0(Y)$, we have
    \[\bch(f^*\xi)_S=f^*\bch(\xi)_S.\]

    \item If $\alpha\in \KK^0(X)$ and $\xi\in \KK_0(X)$, then
    \[\bch(\alpha\otimes \xi)_S=\ch(\alpha)\cap \bch(\xi)_S.\]

    \item If $E$ is a coherent sheaf, then
    \[\bch(E)_S=\sum \left(\mathrm{length}_{\cO_{X,\eta}}E|_{\eta}\right) [\overline{\{\eta\}}]+\text{ terms of lower } \delta\text{-dimensions},\]
    where the sum runs over the generic points of irreducible components of $\mathrm{Supp}(E)$ with top $\delta$-dimension.
\end{enumerate}


\end{lemma}

\begin{proof}
Parts (a)-(c) follow from the definition of $\bch$ and Theorem \ref{thm-tau}. Note that in part (b), if $f$ is flat, then it is lci by \cite[\href{https://stacks.math.columbia.edu/tag/09RL}{Tag 09RL}]{stacks-project}. Moreover, we have $[\mathbb{T}_f]+[f^*\mathbb{T}_{Y/S}]=[\mathbb{T}_{X/S}]\in \KK^0(X)$ by the standard exact triangle of cotangent complexes.

For (d), when $E=\cO_V$, the result follows from Theorem \ref{thm-tau}(e). Now, the general statement can be deduced by d\'evissage (cf.~\cite[\href{https://stacks.math.columbia.edu/tag/01YG}{Tag 01YG}]{stacks-project}).
\end{proof}

Besides Lemma \ref{lem-chM-general}, there are further compatibilities of $\bch$.

\begin{lemma}\label{lem-chM-lci}
Let $X$ be a scheme quasi-projective and lci over $S$.

\begin{enumerate}
    \item We have
    \[\bch(\oh_X)_S=[X].\]

    \item For any class $\xi\in \KK^0(X)$, we have
    \[\ch(\xi)\cap [X]=\bch(\xi)_S.\]

\end{enumerate}
\end{lemma}

\begin{proof}
By our definition of $\bch$, part (a) is equivalent to Theorem \ref{thm-tau}(d). Then part (b) follows directly from Lemma \ref{lem-chM-general}(c) and part (a). 
\end{proof}

\begin{lemma}\label{lem-chM-flat-base-change}
Let $S'$ be an integral Noetherian regular scheme of finite Krull dimension with the dimension function $\delta'$ as in \eqref{eq:dim-function-regular}, and $g\colon S'\to S$ be a flat morphism such that there exists $d\in \ZZ$ with $\delta'(s')=\delta(s)+d$ for any $s\in S$ and any generic point $s'$ of $g^{-1}(s)$. Let $X$ be a quasi-projective lci scheme over $S$. Then
\[g^*\bch(\xi)_S=\bch(g^*_X\xi)_{S'}\in \CH_*(X_{S'}, \delta')_{\QQ}\]
for any $\xi\in \KK_0(X)$.
\end{lemma}

\begin{proof}
This follows from Lemma \ref{lem-tau-flat-base-change} and \cite[\href{https://stacks.math.columbia.edu/tag/08QQ}{Tag 08QQ}]{stacks-project}.
\end{proof}

We end this subsection with the following key lemmas, which enable us to compute the Chern characters of dual sheaves. For an element $e$ in $\CH_*(X,\delta)_{\QQ}$ (resp.~$\A^*(X\to Y,\delta)_{\QQ}$), we define
\[e^{\vee}\coloneqq\sum (-1)^i e_i,\]
where $e_i$ is the component of $e$ in $\CH_{\dim_{\delta} X-i}(X,\delta)_{\QQ}$ (resp.~$\A^i(X\to Y,\delta)_{\QQ}$).

\begin{lemma}\label{lem-derived-dual-ch}
Let $X$ be an equidimensional quasi-projective lci scheme over $S$. Then
\[\bch(E)_S=(\bch(\mathbb{D}^X(E))_S)^{\vee}\in \CH_*(X,\delta)_{\QQ}\]
for any $E\in \Db(X)$.
\end{lemma}

\begin{proof}
Since $X$ is lci over $S$, it is Gorenstein. Hence $\omega_X^{\bullet}[-\dim X]\cong \omega_X$ is a line bundle. In particular, we have
\begin{equation}\label{eq:two-dual}
\mathbb{D}^X(E)\otimes^{\LL} \omega_X[\dim X]\cong \mathbb{D}_X(E).
\end{equation}

Let $j\colon X\hookrightarrow Y$ be a closed embedding into a scheme $Y$ smooth over $S$ of codimension $c$. Since $X$ is lci over $S$, \cite[\href{https://stacks.math.columbia.edu/tag/069M}{Tag 069M}]{stacks-project} implies that $j$ is also lci. Then $\mathbb{L}_j\cong \cN_{X/Y}^{\vee}[1]$, where the normal sheaf $\cN_{X/Y}$ is locally free. From \cite[Theorem III.7.11]{hartshorne:gtm52}, we see
\[\omega_X\otimes j^*\omega_Y^{-1}\cong \det(\cN_{X/Y}).\]

By definition, we have
\[\tau_{X/S}(E)=j^*\td(Y/S)\cdot \ch^Y_X(j_*E)\cap [Y].\]
Hence,
\[\bch(E)_S=\td^{-1}(X/S)\cdot j^*\td(Y/S)\cdot \ch^Y_X(j_*E)\cap [Y]\]
\[=\td^{-1}(\mathbb{T}_j)\cdot \ch^Y_X(j_*E)\cap [Y].\]

We also have
\[\tau_{X/S}(\mathbb{D}^X(E))=(-1)^{\dim X}\tau_{X/S}(\mathbb{D}_X(E)\otimes \omega_X^{-1})=(-1)^{\dim X}\ch(\omega_X^{-1})\cap \tau_{X/S}(\mathbb{D}_X(E))\]
\[=(-1)^{\dim X}\ch(\omega_X^{-1})\cdot j^*\td(Y/S)\cdot \ch^Y_X(j_*\mathbb{D}_X(E))\cap [Y]\]
\[=(-1)^{\dim X}\ch(\omega_X^{-1})\cdot j^*\td(Y/S)\cdot \ch^Y_X(\mathbb{D}_Y(j_*E))\cap [Y]\]
\[=(-1)^c\ch(\omega_X^{-1}\otimes j^*\omega_Y)\cdot j^*\td(Y/S)\cdot \ch^Y_X(\mathbb{D}^Y(j_*E))\cap [Y],\]
where the first equality follows from \eqref{eq:two-dual}, the second equality follows from Theorem \ref{thm-tau}(c), the fourth one follows from \cite[\href{https://stacks.math.columbia.edu/tag/0GEW}{Tag 0GEW}]{stacks-project}, and the last one follows from \cite[\href{https://stacks.math.columbia.edu/tag/0FBF}{Tag 0FBF}]{stacks-project} (see also \cite[Proposition 18.1(c)]{fulton:intersection-theory}). In particular, we see
\[\bch(\mathbb{D}^X(E))_S=(-1)^c\td^{-1}(X/S)\cdot \ch(\omega_X^{-1}\otimes j^*\omega_Y)\cdot j^*\td(Y/S)\cdot \ch^Y_X(\mathbb{D}^Y(j_*E))\cap [Y]\]
\[=(-1)^c\td^{-1}(\mathbb{T}_j)\cdot \ch(\omega_X^{-1}\otimes j^*\omega_Y)\cdot \ch^Y_X(\mathbb{D}^Y(j_*E))\cap [Y]\]
\[=(-1)^c(\td^{-1}(\mathbb{T}_j))^{\vee}\cdot \ch^Y_X(\mathbb{D}^Y(j_*E))\cap [Y],\]
where the last equality follows from Lemma \ref{lem-formula-td} and $\det(\mathbb{T}_j)\cong \omega_X^{-1}\otimes j^*\omega_Y$. Thus, using \cite[Example 18.1.2]{fulton:intersection-theory}, we can compute
\[\bch_i(\mathbb{D}^X(E))_S=(-1)^c\sum_{k+l=c+i} (-1)^k\td^{-1}_k(\mathbb{T}_j)\cdot (\ch^Y_X(\mathbb{D}^Y(j_*E)))_l\cap [Y]\]
\[=(-1)^c\sum_{k+l=c+i} (-1)^k\td^{-1}_k(\mathbb{T}_j)\cdot (-1)^l(\ch^Y_X(j_*E))_l\cap [Y]\]
\[=(-1)^i\sum_{k+l=c+i}\td^{-1}_k(\mathbb{T}_j)\cdot(\ch^Y_X(j_*E))_l\cap [Y]\]
\[=(-1)^i\bch_i(E)_S\]
and the result follows.
\end{proof}

\begin{lemma}\label{lem-formula-td}
Let $X$ be a quasi-projective scheme over $S$ and $E$ be a perfect complex on $X$. Then
\[\td^{-1}(E)\cdot \ch(\det(E))=(\td^{-1}(E))^{\vee}\in \A^*(X,\delta)_{\QQ}.\]
\end{lemma}

\begin{proof}
Without loss of generality, we can assume that $E$ is a locally free sheaf. By the splitting principle \cite[\href{https://stacks.math.columbia.edu/tag/02UL}{Tag 02UL}]{stacks-project}, we can assume that $E$ is a direct sum of line bundles on $X$ with Chern roots $c_1,...,c_r$. Then
\[\td^{-1}(E)\cdot \ch(\det(E))=\exp(\sum^r_{i=1} c_i)\cdot\prod^r_{i=1} \frac{1-\exp(-c_i)}{c_i}=\prod^r_{i=1} \frac{\exp(c_i)-1}{c_i}\]
\[=\prod^r_{i=1} \frac{1-\exp(-(-c_i))}{-c_i}=(\td^{-1}(E))^{\vee}\]
and the result follows.
\end{proof}

\subsection{A generalization}\label{subsec:generalization-ch}

If $X \to S$ is a quasi-projective flat morphism with equidimensional fibers which is fiberwise lci in codimension $d$, then using the exact sequence
\[\CH_k(X\setminus \mathrm{LCI}(X/S),\delta)\to \CH_k(X,\delta)\to \CH_k(\mathrm{LCI}(X/S),\delta)\to 0\]
from \cite[\href{https://stacks.math.columbia.edu/tag/02RX}{Tag 02RX}]{stacks-project}, we have a natural identification $$\CH_k(X,\delta)=\CH_k(\mathrm{LCI}(X/S),\delta)$$ for each $k\geq \dim_{\delta} X-d$. Therefore, for \(i\le d\), we define \(\bch_i(\xi)_S\) to be the
unique class in \(\CH_{\dim_\delta X-i}(X,\delta)_{\mathbb Q}\) whose
restriction to \(\mathrm{LCI}(X/S)\) is
\[
\bch_i(\xi|_{\mathrm{LCI}(X/S)})_S\in \CH_{\dim_\delta X-i}(\mathrm{LCI}(X/S),\delta)_{\mathbb Q}.
\]


From the definition, for any $\xi\in \CH_{\dim_{\delta} X-j}(X,\delta)_{\QQ}$ with $i+j\leq d$, if we denote by $\td_i(X/S)\cap\xi$ the unique lift of $$\td_i(\mathrm{LCI}(X/S)/S)\cap (\xi)|_{\mathrm{LCI}(X/S)}\in \CH_{\dim_{\delta} X-i-j}(\mathrm{LCI}(X/S),\delta)_{\QQ},$$ then we have
\begin{equation}\label{eq:tau_k}
[\tau_{X/S}(-)]_{\dim_{\delta} X-k}=\sum_{i+j=k}\td_i(X/S)\cap\bch_j(-)_S\in \CH_{\dim_{\delta} X-k}(X,\delta)_{\QQ}
\end{equation}
for any $k\leq d$, where $[-]_t$ denotes the component of $\delta$-dimension $t$.

In this case, we have the following general properties.

\begin{lemma}\label{lem:chi-property}
Assume that $Y \to S$ is also a quasi-projective flat morphism with equidimensional fibers that is fiberwise lci in codimension $d$. Let $f\colon X\to Y$ be a morphism over $S$.

\begin{enumerate}
    \item If $f\colon X\to Y$ is proper and $f^{-1}(\mathrm{LCI}(Y/S))=\mathrm{LCI}(X/S)$, then for any class $\xi\in \KK_0(X)$, we have
    \[f_*\left(\sum_{i+j=k}\td_i(X/S)\cap \bch_j(\xi)_S\right)=\sum_{i+j=k+\dim_{\delta}Y-\dim_{\delta}X}\td_i(Y/S)\cap \bch_j(f_*\xi)_S\]
for any $k\leq d$.

\item If $S=\Spec(\kk)$ for a field $\kk$, $\kk_1$ an extension of $\kk$, and $g\colon X_{\kk_1}\to X$ is the base change morphism, then for any class $\xi\in \KK_0(X)$, we have
    \[\bch_{\leq d}(g^*\xi)=g^*\bch_{\leq d}(\xi).\]

    \item If $X$ and $Y$ are equidimensional and $f\colon X\to Y$ is either lci or flat of a fixed relative dimension, then for any class $\xi\in \KK_0(Y)$, we have
    \[\bch_{\leq d}(f^*\xi)_S=f^*\bch_{\leq d}(\xi)_S.\]

    \item If $\alpha\in \KK^0(X)$ and $\xi\in \KK_0(X)$, then
    \[\bch_{\leq d}(\alpha\otimes \xi)_S=\sum_{i+j\leq d} \ch_{i}(\alpha)\cap \bch_j(\xi)_S.\] 

    \item If $E$ is a coherent sheaf with $\dim_{\delta}(\mathrm{Supp}(E))\geq \dim_{\delta} X-d$, then
    \[\bch_{\leq d}(E)_S=\sum \left(\mathrm{length}_{\cO_{X,\eta}}E|_{\eta}\right) [\overline{\{\eta\}}]+\text{ terms of lower } \delta\text{-dimensions},\]
    where the sum runs over the generic points of irreducible components of $\mathrm{Supp}(E)$ with top $\delta$-dimension.
    

\end{enumerate}

\end{lemma}

\begin{proof}
Part (a) follows from \eqref{eq:tau_k} and Theorem \ref{thm-tau}(a). For part (b), note that
\[g^{-1}(\mathrm{LCI}(X/\kk))=\mathrm{LCI}(X_{\kk_1}/\kk_1)\]
by \eqref{eq:lci-locus-base-change}, so the result follows from Lemma \ref{lem-chM-flat-base-change}. Parts (d) and (e) can be deduced directly from Lemma \ref{lem-chM-general}(c) and (d).

Finally, we prove part (c). Note that $f$ is perfect, so we have an induced morphism $$f'\colon \mathrm{LCI}(X/S)\to  \mathrm{LCI}(Y/S)$$ by \cite[\href{https://stacks.math.columbia.edu/tag/09RL}{Tag 09RL}]{stacks-project} and \cite[\href{https://stacks.math.columbia.edu/tag/069J}{Tag 069J}]{stacks-project}. Moreover, by our assumption, we have $$f^*\xi\in \CH_{i+\dim_{\delta} X-\dim_{\delta} Y}(X, \delta)$$ for any $i$ and any $\xi\in \CH_{i}(Y, \delta)$. Therefore, the result follows from the corresponding result for $f'$ in Lemma \ref{lem-chM-general}(b).
\end{proof}

\begin{lemma}\label{lem-chM-special}
Assume that $S$ is the spectrum of a DVR with the closed point $p\in S$ and the fraction field $K$. Then the diagram
\[\begin{tikzcd}
	{\KK_0(X_K)} & {\CH_{\dim X-1-i}(X_K)_{\QQ}} \\
	{\KK_0(X_p)} & {\CH_{\dim X-1-i}(X_p)_{\QQ}}
	\arrow["{\bch_{i}(-)}", from=1-1, to=1-2]
	\arrow["{\mathrm{sp}}"', from=1-1, to=2-1]
	\arrow["{\mathrm{sp}}", from=1-2, to=2-2]
	\arrow["{\bch_{i}(-)}", from=2-1, to=2-2]
\end{tikzcd}\]
commutes for any $i\leq d$.
\end{lemma}

\begin{proof}
By \eqref{eq:lci-locus-base-change}, we have
\[(\mathrm{LCI}(X/S))_c=\mathrm{LCI}(X_c/\kappa(c))\]
for any point $c\in S$. So we may assume that $X\to S$ is lci. In this case, the result follows from Lemma \ref{lem-tau-special}, as $(\mathbb{L}_{X/S})_K\cong \mathbb{L}_{X_K/K}$ and $(\mathbb{L}_{X/S})_p\cong \mathbb{L}_{X_p/\kappa(p)}$ by \cite[\href{https://stacks.math.columbia.edu/tag/08QQ}{Tag 08QQ}]{stacks-project}.
\end{proof}

The following lemma shows the local constancy of Chern characters of relative perfect objects.

\begin{lemma}\label{lem-constant-chM}
Let $f\colon X\to S$ be a projective flat morphism between Noetherian schemes and $E\in \Dqc(X)$ be a $S$-perfect object. Assume that $S$ has finite Krull dimension and $f$ is fiberwise lci in codimension $d$ with equidimensional fibers. Then for any $Z\in \A_{\star}^*(X/S)_{\QQ}$ and $0\leq i\leq d$, the function
\[s\mapsto \int_{X_s} Z_s\cap \bch_i(E_s)\]
is a locally constant function on $S$. Moreover, 
\begin{equation}\label{eq:base-change}
\int_{X_s} Z_s\cap \bch_{i}(E_s)=\int_{X_t} Z_t\cap \bch_{i}(E_t)
\end{equation}
for any point $t\to S$ over $s$.
\end{lemma}


\begin{proof}
The equality \eqref{eq:base-change} follows from Definition \ref{def-A-star} and Lemma \ref{lem:chi-property}(b).

For the local constancy part, without loss of generality, we can assume that 
$S$ is connected. Since $S$ is of finite Krull dimension, it is enough to show that the value of the above function at any two points $s,s'\in S$ with $s\in \overline{\{s'\}}$ and $\codim_{\overline{\{s'\}}} \overline{\{s\}}=1$ is the same. Since the claim is local, by Lemma \ref{lem-S-perf-lem-2}(b), after restricting to closed subschemes and localization, we can assume that $S$ is a one-dimensional integral scheme. By \eqref{eq:base-change}, we can replace $S$ by its normalization and then localize further. In particular, we may assume that $S$ is the spectrum of a DVR and $Z\in \A^*(X)_{\QQ}$. Therefore, the constancy follows from Lemma \ref{lem-chM-special} and Lemma \ref{lem-sp-push}.
\end{proof}

\begin{lemma}\label{lem-dual-ch}
Let $X$ be an equidimensional quasi-projective scheme over a field that is lci in codimension $d$. If $E$ is a torsion-free $S_k$ sheaf (hence $k\geq 1$), then
\[\mathbf{ch}_{i}(E)=(-1)^i\cdot\mathbf{ch}_{i}(E^{\vee})\]
for any $i\leq \min\{k,d\}$.
\end{lemma}

\begin{proof}
This immediately follows from applying Lemma \ref{lem-derived-dual-ch} and Lemma \ref{lem-codim-Sp} to $\mathrm{LCI}(X/\kk)$ and $E|_{\mathrm{LCI}(X/\kk)}$.
\end{proof}


\subsection{Mumford intersection numbers}\label{subsec:mumford}

We now review Mumford's intersection theory of Weil divisors on normal varieties constructed in \cite{langer:intersection-normal} and \cite{enokizono}.

\begin{theorem}[{\cite[Theorem 0.1]{langer:intersection-normal}}]\label{thm-mumford-intersection}
Let $X$ be a geometrically normal geometrically integral projective variety of dimension $n$ over a field $\kk$. There is a unique $\ZZ$-multilinear form
\[\CH^1(X)\times \CH^1(X)\times\Pic(X)^{\times (n-2)}\to \QQ,\quad 
(D_1,D_2,L_1,\dots, L_{n-2})\mapsto L_1.\dots L_{n-2}.D_1.D_2\]
that satisfies:

\begin{enumerate}
    \item It is symmetric in $D_1$ and $D_2$.

    \item It is symmetric in $L_1,\dots, L_{n-2}$.

    \item If $D_1$ and $D_2$ are Cartier divisors, then
    \[L_1.\dots L_{n-2}.D_1.D_2=\int_X D_1\cdot D_2\cdot c_1(L_1)\cdot\dots \cdot c_1(L_{n-2})\cap [X].\]

    \item If $D_2$ is Cartier, then
    \[L_1.\dots L_{n-2}.D_1.D_2=\int_X D_2\cdot c_1(L_1)\cdot\dots \cdot c_1(L_{n-2})\cap [D_1].\]

    \item If $\kk$ is algebraically closed and $L_1,\dots, L_{n-2}$ are all very ample, then for a general complete intersection surface $S$ cut out by divisors in  $L_1,\dots, L_{n-2}$, we have
    \[L_1.\dots L_{n-2}.D_1.D_2=(D_1)_S.(D_2)_S,\]
    where the right-hand side is the Mumford intersection number on normal surfaces as in \cite[Example 8.3.11]{fulton:intersection-theory} and \cite{mumford:normal-surface}.

    \item The multilinear form is compatible with base change to any field extension of $\kk$.
\end{enumerate}
\end{theorem}

The construction can be described as follows. Recall that for any Weil divisor $D$ on a geometrically normal variety $X$ over $\kk$, we have an associated rank one reflexive sheaf $\cO_X(D)$. We define a class $$\cO_{D}\coloneqq \cO_X-\cO_X(-D)\in \KK_0(X).$$
If $D$ is a prime Weil divisor, then $\cO_X(-D)=\cI_D$ and $\cO_D$ is the same as the class of the structure sheaf of $D$. As in \cite[Section 2.2]{langer:intersection-normal}, if $X$ is also projective, then we define
\begin{equation}\label{eq:def-pair}
    L_1.\dots L_{n-2}.D^2\coloneqq \lim_{m\to +\infty} 2\frac{\chi(X, c_1(L_1) \cdots c_1(L_{n-2})\cdot \cO_{mD})}{m^2}
\end{equation}
\[=-\lim_{m\to +\infty} 2\frac{\chi(X, c_1(L_1) \cdots c_1(L_{n-2})\cdot \cO_X(-mD))}{m^2},\]
where $c_1(L)$ is the endomorphism of $\KK_0(X)$ given by $[E]\mapsto [E]-[L^{-1}\otimes E]$. Note that using \cite[Proposition 2.3]{langer:intersection-normal} after base change to $\overline{\kk}$, this limit converges to a rational number. If $D_1, D_2$ are two Weil divisors, then we set
\[L_1.\dots L_{n-2}.D_1.D_2\coloneqq \frac{1}{2}(L_1.\dots L_{n-2}.(D_1+D_2)^2-L_1.\dots L_{n-2}.D_1^2-L_1.\dots L_{n-2}.D_2^2).\]
In particular, if $X$ is a geometrically normal projective surface, then
\begin{equation}\label{eq:rr-normal-surface}
    D^2=-\lim_{m\to +\infty} 2\frac{\chi(X, \cO_X(-mD))}{m^2}.
\end{equation}
By Lemma \ref{lem-flat-base-change}, it is clear that the above construction is compatible with any field extension of $\kk$.

We also have the Hodge index theorem on normal varieties.

\begin{lemma}[{\cite[Corollary 1.21]{langer:higgs-normal}}]\label{lem-hodeg-index}
Let $X$ be a geometrically normal projective variety of dimension $n$. Fix a collection $L_1,\dots, L_{n-2}$ of nef line bundles on $X$. Assume that
there exists a nef line bundle $L$ such that $L.L_1\dots L_{n-2}$ is \emph{numerically non-trivial},
i.e., there exists a Weil divisor $D$ such that $$L.L_1\dots L_{n-2}.D\neq 0.$$

If $H$ is an ample line bundle on $X$, then 
\[H.L.L_1\dots L_{n-2}>0.\]
Moreover, if $L.L_1\dots L_{n-2}.D=0$ for a Weil divisor $D$, then
\[L_1.\dots L_{n-2}.D.D\leq 0.\]
\end{lemma}

It is not known in general whether such an intersection pairing commutes with lci pullback. However, the following lemma is enough for our purpose.

\begin{lemma}\label{lem:intersect-pairing-restrict}
Let $X$ be a geometrically normal projective variety of dimension $n\geq 3$. Let $L_1,\dots, L_{n-2}\in \Pic(X)$. If $D$ is a Weil divisor on $X$ and $T\in |L_{n-2}|$ is a normal integral divisor such that $T$ intersects $D$ properly and $\cO_X(mD)|_T$ is reflexive for each $m\in \ZZ$. Then we have
\[L_1.\cdots.L_{n-2}.D^2=(L_1)|_T.\cdots.(L_{n-3})|_T.D|_T^2.\]
\end{lemma}

\begin{proof}
After base change to the algebraic closure, we may assume that the base field is algebraically closed. By writing $L_1,\dots, L_{n-3}$ as differences of very ample line bundles and the multilinearity of the intersection pairing, we may assume that $L_1,\dots, L_{n-3}$ are all very ample. Therefore, using \cite[Corollary 2.8]{langer:intersection-normal}, after restricting to a general complete intersection of divisors in $L_1,\dots, L_{n-3}$, we may also assume $n=3$. In this case, $T$ is a normal surface, and we need to show
\[L_1.D^2=(D|_T)^2.\]
Since $T$ intersects properly with $D$, $D|_T$ is also a Weil divisor on $T$. Therefore, $\cO_X(-mD)|_T$ and $\cO_T(-mD|_T)$ coincide at all codimension one points of $T$. By the assumption that $\cO_X(-mD)|_T$ is reflexive, we have $\cO_X(-mD)|_T=\cO_T(-mD|_T)$. Therefore, we have an exact sequence
\[0\to \cO_X(-mD)\otimes L_1^{-1}\to \cO_X(-mD)\to \cO_T(-mD|_T)\to 0.\]
This together with \eqref{eq:rr-normal-surface} implies
\[L_1.D^2=-\lim_{m\to +\infty} 2\frac{\chi(X, c_1(L_1)\cdot \cO_X(-mD))}{m^2}=-\lim_{m\to +\infty} 2\frac{\chi(T,\cO_T(-mD|_T))}{m^2}=(D|_T)^2\]
as desired.
\end{proof}

\subsection{Alternative definitions of top Chern characters}\label{subsec:chn}

Although normal surfaces are only lci in codimension $1$, we can use the above intersection numbers and the Riemann--Roch homomorphism to formally define $\bch_2$ as follows.

\begin{definition}[{\cite{langer:normal-surface}}]\label{def:ch2-surface}
Let $X$ be a geometrically normal projective surface over a field $\kk$. For any $\xi\in \KK_0(X)$, we define $\bch_2(\xi)\in \QQ$ as 
\[\bch_2(\xi)\coloneqq \chi(\xi)+\frac{1}{2}K_X.\bch_1(\xi)-\rk(\xi)\chi(\cO_X).\]
\end{definition}

One may expect to define $\bch_n(-)$ for any $n$-dimensional variety which is lci in codimension $n-1$ analogously to Definition \ref{def:ch2-surface}. However, this involves the intersection of higher codimension cycles, which is not well-defined in general. But in the case of threefolds, we can define it as follows. Recall that for a quasi-projective threefold $X$ over a field which is lci in codimension $2$, the cycle $$\Td_2(X)\coloneqq [\tau_{X}(\cO_X)]_1\in \CH_1(X)_{\QQ}$$ is the image of
\[\td_2(\mathrm{LCI}(X/\kk))\cap [\mathrm{LCI}(X/\kk)]\in \CH_1(\mathrm{LCI}(X/\kk))_{\QQ}\]
under the identification $\CH_1(\mathrm{LCI}(X/\kk))_{\QQ}=\CH_1(X)_{\QQ}$. By abuse of notation, we also write $\td_2(X)$ for $\Td_2(X)$.

\begin{definition}\label{def:ch3}
Let $X$ be a geometrically normal projective threefold over a field $\kk$ which is $\QQ$-factorial and lci in codimension $2$. We define a homomorphism
\[\bch_3(-)\colon \KK_0(X)\to \QQ\]
by
\[\bch_3(E)\coloneqq \chi(E)+\frac{1}{2}K_X.\bch_2(E)-\td_2(X).\bch_1(E)-\chi(\cO_X)\rk(E).\]
\end{definition}

Note that the intersection numbers $K_X.\bch_2(E)$ and $\td_2(X).\bch_1(E)$ make sense by the $\QQ$-factorial assumption of $X$.

\begin{lemma}\label{lem:property-ch3}
Let $X$ be a projective threefold over a field $\kk$ such that $X_{\overline{\kk}}$ is normal, $\QQ$-factorial, and lci in codimension $2$. Let $E\in \Db(X)$. We have

\begin{enumerate}
    \item $\bch_3(E\otimes \cL)=\bch_3(E)+\bch_2(E).\cL+\frac{1}{2}\cL^2.\bch_1(E)+\frac{1}{6}\cL^3\rk(E)$ for any $\cL\in \Pic(X)$, and 

    \item if $X$ is also Gorenstein, then $\bch_i(E)=(-1)^i\bch_i(\mathbb{D}^X(E))$ for any $0\leq i\leq 3$.
\end{enumerate}

\end{lemma}

\begin{proof}
After base change to $\overline{\kk}$, we may assume that $\kk$ is algebraically closed. By writing $\cL$ as a difference of two very ample line bundles, we may assume that $\cL$ is very ample. Let $D\in |\cL|$ be a general divisor, so it is a normal projective lci surface. From the definition, Lemma \ref{lem:chi-property}(d), and $\chi((E\otimes \cL)|_D)=\chi(E\otimes \cL)-\chi(E)$, we have
\[\bch_3(E\otimes \cL)-\bch_3(E)=\chi((E\otimes \cL)|_D)-\rk(E)\td_2(X).\cL+\frac{1}{2}K_X.\left(\bch_1(E).\cL+\frac{1}{2}\rk(E)\cL^2\right).\]
Moreover, since $D$ is lci, by Lemma \ref{lem-chM-general}, we have $$\bch((E\otimes \cL)|_D)=\ch(\cL_D)\cap \bch(E_D)\in \CH_*(D)_{\QQ}$$
and
\[\chi((E\otimes \cL)|_D)=\td(D)\cap \bch((E\otimes \cL)|_D).\]
Therefore, part (a) follows from a direct computation using Lemma \ref{lem:chi-property}(a).

Applying Lemma \ref{lem-derived-dual-ch} to $E|_{\mathrm{LCI}(X/\kk)}$, we have $\bch_i(E)=(-1)^i\bch_i(\mathbb{D}^X(E))$ for $0\leq i\leq 2$. Moreover, using Lemma \ref{lem-serre-duality}, we obtain $$\chi(E\otimes \omega_X)=-\chi(E,\cO_X)=-\chi(\mathbb{D}^X(E)).$$ These together with part (a) imply part (b).
\end{proof}

For later convenience, we define a class of morphisms as follows.

\begin{definition}\label{def:adm}
We say that a projective flat morphism $f\colon X\to S$ of relative dimension $n\geq 1$ between Noetherian schemes is \emph{admissible} if it satisfies one of the following assumptions.

\begin{enumerate}[(1)]
    \item $f$ is lci.

    \item $n=2$, each geometric fiber of $f$ is either lci or a normal $\QQ$-Gorenstein surface, and for any morphism $D\to S$ essentially of finite type from the spectrum $D$ of a DVR, the total space $X_D$ is $\QQ$-Gorenstein.

    \item $n=2$, $S=\Spec(\kk)$ for a field $\kk$, and $X_{\overline{\kk}}$ is a normal surface.

    \item $n=3$, each geometric fiber of $f$ is a normal $\QQ$-factorial threefold that is lci in codimension $2$, and for any morphism $D\to S$ essentially of finite type from the spectrum $D$ of a DVR, the total space $X_D$ is $\QQ$-factorial.

\end{enumerate}

\end{definition}

By definition, being admissible is stable under base change essentially of finite type. In the following, we give some non-trivial examples of admissible morphisms.

\begin{example}\label{ex:admissible}
Given a projective flat morphism $f\colon X\to S$ of relative dimension $n\geq 1$ between Noetherian schemes.
\begin{enumerate}
    \item If $f$ is Gorenstein, or more generally, naively $\QQ$-Gorenstein in the sense of \cite[Definition 7.1]{nakayama:q-gorenstein}, and each geometric fiber of $f$ is lci or a normal surface, then $f$ is admissible. Indeed, this follows from \cite[Lemma 7.20]{nakayama:q-gorenstein}.

    \item If $S=\Spec(\kk)$ for an algebraically closed field $\kk$ which is not the algebraic closure of a finite field and $X$ is a normal $\QQ$-factorial projective threefold over $\kk$ that is lci in codimension $2$, then $X\to S$ is admissible by \cite[Théorème 6.5]{gabber:q-factor}.

    \item Using \cite[Théorème 6.5]{gabber:q-factor} and \cite[Theorem 2.91]{kollar:book-general-type}, it follows that if $S$ has characteristic $0$ and each geometric fiber of $f$ is a normal projective threefold with isolated $\QQ$-factorial canonical singularities, then $f$ is admissible.
\end{enumerate}
\end{example}

A key property of admissible morphisms is the constancy of the top Chern character in a flat family.


\begin{lemma}\label{lem:constant-generalization}
Let $f\colon X\to S$ be an admissible morphism and $E\in \Dqc(X)$ be a $S$-perfect object. Assume that $S$ is Nagata and has finite Krull dimension. Then the function
\[s\mapsto \bch_n(E_s)\in \QQ\]
is a locally constant function on $S$ and $\bch_n(E_s)=\bch_n(E_t)$ for any point $t\to S$ with the image $s\in S$.
\end{lemma}

\begin{proof}
When $f$ is lci, the result follows from Lemma \ref{lem-constant-chM}. In the remaining cases, it is clear that $\chi(E_s)$, $\rk(E_s)$, $\chi(\cO_{X_s})$, and the Mumford intersection numbers are constant and compatible with field extensions. Therefore, the last statement follows from the definition, Lemma \ref{lem-constant-chM}, and Lemma \ref{lem-flat-base-change}.

To prove the local constancy, as in the proof of Lemma \ref{lem-constant-chM}, we can base change $S$ to the spectrum $D$ of a DVR that $D\to S$ is essentially of finite type by \cite[Lemma 11.19]{BLMNPS21}.

The result for case (3) is obvious. In the case (2) in Definition \ref{def:adm}, by the $\QQ$-Gorenstein assumption, we can take $m\in \ZZ_{>0}$ so that $m\bch_1(\omega^{\bullet}_{X_D/D})$ is Cartier. Hence, the result follows from Lemma \ref{lem-sp-push}. In the case (4) in Definition \ref{def:adm}, by the $\QQ$-factorial assumption, we can take $m\in \ZZ$ so that $m\bch_1(\omega^{\bullet}_{X_D/D})$ and $m\bch_1(E_D)$ are both Cartier. Then the result also follows from Lemma \ref{lem-sp-push}.
\end{proof}

We do not expect Mumford intersection numbers to be locally constant for all flat families of Weil divisors in normal projective surfaces, and This is the reason for imposing the $\QQ$-Gorenstein condition in the surface case of Definition \ref{def:adm}.

\section{Stability conditions}\label{sec:stab}

In this section, we recall definitions and basic properties of (weak) stability conditions on triangulated categories. We follow \cite{bridgeland:stability,BLMNPS21}. Fix a triangulated category $\cD$.

\begin{definition}
A \emph{slicing} $\cP$ on $\cD$ is a collection $\{\cP(\phi)\}_{\phi\in \mathbb{R}}$ of full additive subcategories of $\cD$ that satisfies

\begin{itemize}
    \item $\cP(\phi+1)=\cP(\phi)[1]$ for all $\phi\in \mathbb{R}$,

    \item $\Hom_{\cD}(\cP(\phi_1), \cP(\phi_2))=0$ for $\phi_1>\phi_2$, and

    \item {(Harder--Narasimhan filtrations)} for any non-zero $E\in \cD$, there exists a finite sequence of morphisms
    \[0=E_0\xra{s_1}E_1\to \cdots \xra{s_m} E_m=E\]
    such that the cone $\mathrm{cone}(s_i)$ is in $\cP(\phi_i)$ for a sequence $\phi_1>\phi_2>\cdots>\phi_m$ of real numbers.
\end{itemize}
\end{definition}

If $0 \neq E \in \mathcal{D}$, we set $$\HN_{\cP}^-(E)\coloneqq \cone(s_m), \quad \HN_{\cP}^+(E)\coloneqq \cone(s_1)$$ and $$\phi_{\cP}^+(E)\coloneqq\phi_1, \quad \phi_{\cP}^-(E)\coloneqq\phi_m.$$ For an interval $I\subset \mathbb{R}$, we define $$\cP(I)\coloneqq\{E \colon \phi_{\cP}^+(E), \phi_{\cP}^-(E)\in I\}\cup \{0\}=\langle \cP(\phi) \rangle_{\phi\in I},$$ where $\langle -\rangle$ denotes the extension closure.

\begin{definition}\label{def:stability-condition}
Let $\Lambda$ be a finite rank lattice and $\bv\colon \KK(\cD)\to \Lambda$ be a group homomorphism.

\begin{itemize}
    \item A \emph{weak pre-stability condition} on $\cD$ (with respect to $\Lambda$) is a pair $\sigma=(Z, \cP)$ where $\cP$ is a slicing of $\cD$ and $Z\colon \Lambda\to \CC$ is a group homomorphism, called the \emph{central charge}, such that for all non-zero objects $E\in \cP(\phi)$, we have $Z(\bv(E))\in \mathbb{R}_{> 0}\cdot e^{\mathfrak{i}\pi \phi}$ for $\phi\notin \ZZ$ and $Z(\bv(E))\in \mathbb{R}_{\geq 0}\cdot e^{\mathfrak{i}\pi \phi}$ for $\phi\in \ZZ$.

    If, moreover, for all $\phi\in \RR$ and non-zero objects $E\in \cP(\phi)$, we have $Z(\bv(E))\neq 0$, then $\sigma$ is called a \emph{pre-stability condition}.

    By abuse of notation, we write $Z(E)$ for $Z(\bv(E))$. The non-zero objects of $\cP(\phi)$ are called \emph{$\sigma$-semistable of phase $\phi$}. Simple objects of $\cP(\phi)$ are called \emph{$\sigma$-stable}.

    \item A weak pre-stability condition $\sigma=(Z, \cP)$ satisfies the \emph{support property} if there exists a quadratic form $Q$ on $(\Lambda/\Lambda^{Z})_{\RR}$ such that $Q|_{(\ker(Z)/\Lambda^{Z})_{\RR}}$ is negative definite and $Q(\overline{\bv}(E))\geq 0$ for any $\sigma$-semistable object $E\in \cD$. Here, $$\Lambda^{Z}\coloneqq\langle \bv(E) \colon E\in \cP(0,1], Z(E)=0 \rangle\subset \Lambda,$$
where $\langle S \rangle$ denotes the saturated sublattice generated by a subset $S\subset \Lambda$, and $\overline{\bv}$ is the induced homomorphism $\overline{\bv}\colon \KK(\cD)\to \Lambda/\Lambda^{Z}$.

    \item A (weak) pre-stability condition $\sigma=(Z, \cP)$ is called a \emph{(weak) stability condition} if $\sigma$ satisfies the support property.
\end{itemize}
\end{definition}

\begin{remark}
Given a weak stability condition $\sigma=(Z, \cA)$, by \cite[Lemma 11.4]{bayer2016space}, the support property is equivalent to saying that there is a constant $C>0$ such that for all $\sigma$-semistable objects $E\in \cP(0,1]$, we have
\[\|\overline{\bv}(E)\|\leq C|Z(E)|,\]
where $\|-\|$ is a norm on $(\Lambda/\Lambda^Z)_{\RR}$. In fact,  we only need to take
\[
    Q(v)=C^2\left|Z(v)\right|^{2}-\|v\|^2.
\]
This implies that for any constant $D>0$, there exist only finitely many classes $v\in \Lambda/\Lambda^Z$ such that $|Z(v)|\leq D$ and $Q(v)\geq 0$, as these lattice points are contained in a compact subset in $(\Lambda/\Lambda^Z)_{\RR}$.
\end{remark}

\begin{remark}
Our definition of weak stability conditions is the same as \cite[Definition 2.3]{piyaratne2019moduli}, which differs slightly from \cite{BLMNPS21}. The definition in \cite{BLMNPS21} requires the quadratic form $Q$ in the definition of support property to be defined on the bigger space $\Lambda_{\RR}$. Note that a weak stability condition $\sigma=(Z,\cP)$ with respect to $\bv\colon \KK(\cD)\to \Lambda$ can also be regarded as a weak stability condition with respect to $\overline{\bv}\colon \KK(\cD)\to \Lambda/\Lambda^{Z}$.
\end{remark}

If $\sigma=(Z, \cP)$ is a weak pre-stability condition on $\cD$ and $0 \neq E \in \mathcal{D}$, we set $$\HN_{\sigma}^-(E)\coloneqq \HN_{\cP}^-(E)\quad \text{ and }\quad \HN_{\sigma}^+(E)\coloneqq \HN_{\cP}^+(E).$$
Similarly, we set
$$\phi_{\sigma}^+(E)\coloneqq \phi_{\cP}^+(E)\quad \text{ and }\quad \phi_{\sigma}^-(E)\coloneqq\phi_{\cP}^-(E).$$

We denote by $\Stab^{\mathsf{wp}}_{\Lambda}(\cD)$ the set of all weak pre-stability conditions on $\cD$ with respect to $\Lambda$. In the following, we equip $\Stab^{\mathsf{wp}}_{\Lambda}(\cD)$ with the coarsest topology such that the maps
\[\cZ\colon \Stab^{\mathsf{wp}}_{\Lambda}(\cD)\to \Hom_{\ZZ}(\Lambda, \CC), \quad \sigma \mapsto Z\]
and
\[\Stab^{\mathsf{wp}}_{\Lambda}(\cD)\to \RR^2, \quad \sigma \mapsto (\phi_{\sigma}^+(E), \phi_{\sigma}^-(E))\]
are continuous for any $0\neq E\in \cD$. The corresponding subspaces consisting of weak stability conditions and stability conditions are written as $\Stab^{\mathsf{w}}_{\Lambda}(\cD)$ and $\Stab_{\Lambda}(\cD)$, respectively. One of the most important results in the theory of stability conditions is the following.

\begin{theorem}[\cite{bridgeland:stability}]
The continuous map \[\cZ\colon \Stab_{\Lambda}(\cD)\to \Hom_{\ZZ}(\Lambda, \CC), \quad \sigma \mapsto Z\]
is a local homeomorphism, so $\Stab_{\Lambda}(\cD)$ naturally has the structure of a complex manifold of complex dimension $\rk(\Lambda)$.
\end{theorem}

In the following, we recall an alternative description of stability conditions.

\begin{definition}
A \emph{stability function (resp.~weak stability function)} $Z$ on an abelian category $\cA$ is a group homomorphism $Z\colon \KK(\cA)\to \CC$ such that for all non-zero $E\in \cA$, we have $\Im(Z(E))\geq 0$, and if $\Im(Z(E))=0$ then $\Re(Z(E))<0$ (resp.~$\Re(Z(E))\leq 0$).

For any non-zero $E\in \cA$, we define its \emph{phase} by $$\phi(E)\coloneqq\phi(Z(E))=\frac{1}{\pi} \arg Z(E)$$ if $Z(E)\neq 0$ and $\phi(E)=1$ otherwise. An object $E\in \cA$ is called \emph{$Z$-semistable} if for any non-trivial subobject $F\subset E$, we have $\phi(F)\leq \phi(E/F)$.
\end{definition}

\begin{definition}
Given any homomorphism $Z\colon \KK(\cA)\to \CC$, we define $\cA^{Z}\subset \cA$ as the subcategory of objects $E\in \cA$ with $Z(E)=0$.
\end{definition}

It is clear from the definition that a weak stability function $Z$ on $\cA$ is a stability function if and only if $\cA^{Z}=0$.

\begin{definition}
We say a weak stability function $Z$ on an abelian category $\cA$ satisfies the \emph{HN property} if every non-zero object $E\in \cA$ admits an HN filtration: a sequence 
\[0=E_0\hookrightarrow E_1\hookrightarrow E_2\hookrightarrow\cdots \hookrightarrow E_m=E\]
such that the factor $E_i/E_{i-1}$ is $Z$-semistable for each $1\leq i\leq m$ with
\[\phi(E_1/E_0)>\phi(E_2/E_1)>\cdots >\phi(E_m/E_{m-1}).\]
\end{definition}

Given a (weak) pre-stability condition $\sigma=(Z, \cP)$, we know that $Z$ is a (weak) stability function on the heart $\cP(0,1]$. Conversely, we also have:

\begin{lemma}[{\cite{bridgeland:stability}}]
To give a (weak) pre-stability condition on $\cD$ is equivalent to giving the heart of a bounded t-structure $\cA\subset \cD$ and a (weak) stability function on $\cA$ with the HN property.
\end{lemma}
In particular, we also denote by $(\cA, Z)$ for a weak pre-stability condition $\sigma=(Z, \cP)$, where $\cA=\cP(0,1]$.

For any weak pre-stability condition $\sigma=(\cA, Z)$ and $0\neq E\in \cD$, we set
\[\mu_{\sigma}(E)\coloneqq -\frac{\Re Z(E)}{\Im Z(E)}\]
if $\Im Z(E)\neq 0$, and $+\infty$ otherwise. This is called the slope of $E$ associated with $\sigma$. For $0\neq E\in \cA$, we define
$$\mu_{\sigma}^+(E)\coloneqq \mu_{\sigma}(\HN_{\sigma}^+(E))\quad \text{ and }\quad \mu_{\sigma}^-(E)\coloneqq\mu_{\sigma}(\HN_{\sigma}^-(E)).$$

The support property can be checked at the level of the abelian category with the same definition. Thus, the above lemma can be rephrased in terms of (weak) stability conditions as well.

\subsection{Tilting property}\label{subsec-tilting-property}

We now discuss the tilting property, which allows us to rotate a weak stability condition. We follow \cite[Section 14, 19]{BLMNPS21}.

\begin{definition}
Let $\cB \subset \cA$ be an abelian subcategory of an abelian category $\cA$. We say that $\cB$ is a \emph{Noetherian torsion subcategory} of $\cA$ if $\cB$ is a Noetherian abelian category, and if there exists a torsion pair $(\cB, \cB^{\perp})$ in $\cA$.
\end{definition}

\begin{remark}\label{rmk:serre-torsion}
    It is straightforward to check that a Serre subcategory \(\mathcal{B} \subset \mathcal{A}\) is a Noetherian torsion subcategory if and only if for any \(E \in \mathcal{A}\), there is no infinite sequence in \(\mathcal{A}\):
    \[
        F_{1} \subsetneq \cdots \subsetneq F_{n} \subsetneq \cdots \subsetneq E,
    \] where \(F_{n} \in \mathcal{B}\) for all \(n \in \mathbb{N}^{+}\).
\end{remark}

\begin{definition}\label{def-tilting}
A weak pre-stability condition $\sigma=(\cA, Z)$ (with respect to $\Lambda$) has the \emph{tilting property} if

\begin{enumerate}[label=(t\arabic*)]
    \item\label{t1} $\cA^{Z}\subset \cA$ is a Noetherian torsion subcategory, and

    \item\label{t2} for every $E\in \cA$ with $\mu_{\sigma}^+(E)<+\infty$, there exists  a short exact sequence $$0\to E\to \sT(E)\to \sT(E)/E\to 0$$ in $\cA$ with $\sT(E)/E\in \cA^{Z}$ and $\Hom_{\cD}(\cA^{Z}, \sT(E)[1])=0$.
\end{enumerate}
\end{definition}

For any $b\in \RR$, we define a torsion pair $(\cT^{b},\cF^{b})$ as
\[\cT^{b}\coloneqq\langle E\in \cA\colon E \text{ is } \sigma\text{-semistable with }\mu_{\sigma}(E)>b \rangle\]
and
\[\cF^{b}\coloneqq\langle E\in \cA\colon E \text{ is } \sigma\text{-semistable with }\mu_{\sigma}(E)\leq b \rangle.\]
The corresponding \emph{tilted heart} is defined by $\cA^{b}\coloneqq\langle \cF^{b}[1],\cT^{b} \rangle$. We also define
\begin{equation}\label{eq:Zb-def}
Z^{b}\coloneqq \Im Z+\mathfrak{i}(-\Re Z-b\Im Z) \colon \Lambda\to \CC.
\end{equation}

\begin{lemma}[{\cite[Proposition 14.16]{BLMNPS21}}]\label{lem-tilt-weak-stab}
Let $\sigma=(\cA, Z)$ be a weak stability condition on $\cD$ with respect to $\Lambda$. If it has the tilting property, then for any $b\in \RR$, $\sigma^{b}\coloneqq (\cA^{b}, Z^{b})$ is a weak stability condition on $\cD$  with respect to $\Lambda$ and the support property is given by the same quadratic form as for $\sigma$. Moreover, $(\cA^{b})^{Z^{b}}=\cA^Z\subset \cA^{b}$ is a Noetherian torsion subcategory.
\end{lemma}

\begin{lemma}\label{lem:tilt-property-no-infty-seq}
Let $\sigma=(\cA, Z)$ be a weak pre-stability condition on $\cD$. If it satisfies \ref{t2} and $\cA^Z$ is Noetherian, then there is no infinite sequence
	\[
		G_{1}  \subsetneq \cdots \subsetneq G_{n}  \subsetneq \cdots
	\]
    in \(\mathcal{A}\) with \(\mu_{\sigma}^+(G_{n})<+\infty\) and \(0 \neq G_{n+1}/G_{n} \in \mathcal{A}^{Z}\) for any \(n \in \mathbb{N}^{+}\).
\end{lemma}

\begin{proof}
Assume such a sequence exists. Since $\Hom_{\cD}(\cA^Z, \sT(G_1)[1])=0$ and $G_i/G_1\in \cA^Z$, the inclusion $G_1\hookrightarrow \sT(G_1)$ factors via $G_i$, hence $G_i/G_1$ is an increasing sequence of subobjects of $\sT(G_1)/G_1\in \cA^Z$, which contradicts the Noetherian property of $\cA^Z$.
\end{proof}

\subsection{Jordan--H\"older filtrations}

\begin{definition}
Let $\sigma=(\cA, Z)$ be a weak pre-stability condition on $\cD$. A \emph{Jordan--H\"older (JH) filtration} of a non-zero $\sigma$-semistable object $E\in \cA$ is a filtration
\[0=E_0\subset E_1\subset E_2\subset \cdots \subset E_n=E\]
of subobjects of $E$ such that the factor $E_{i+1}/E_i$ is $\sigma$-stable with $\phi_{\sigma}(E_{i+1}/E_i)=\phi_{\sigma}(E)$ for each $0\leq i\leq n-1$.
\end{definition}

Unlike HN filtrations, the existence of JH filtrations is not part of the definition. However, it follows from the support property, as explained below.

In the following, we will freely use the notion and properties of \emph{quasi-abelian categories} and \emph{strict subobjects/quotient objects}, see \cite[Section 4]{bridgeland:stability} for an overview. A typical example of a quasi-abelian category is $\cP(I)$ for any slicing $\cP$ on $\cD$ and any interval $I\subset \RR$ of length $<1$. We say that a quasi-abelian category is of \emph{finite length} if, for every object, every ascending chain of strict subobjects and every descending chain of strict quotients stabilizes. In an abelian category this is equivalent to being both Noetherian and Artinian.

\begin{lemma}[{\cite[Lemma 4.4]{bridgeland:stab-on-K3}}]\label{lem:supp-property-finite-length}
Let $\sigma=(Z, \cP)$ be a weak stability condition. Fix $0<\epsilon<\frac{1}{2}$. Then for any interval $I\subset \RR$ of length $\leq 2\epsilon$ and any $0\neq E\in \cP(I)$, the set 
\begin{equation}\label{eq:class-subquotient}
\{\overline{\bv}(A)\in \Lambda/\Lambda^Z\colon A \text{ is a strict subquotient of }E \text{ in }\cP(I)\}
\end{equation}
is finite. 

In particular, if $I\cap \ZZ=\varnothing$ or $\sigma$ is a stability condition, the quasi-abelian category $\cP(I)$ is a finite length category.
\end{lemma}

\begin{proof}
Let $\phi$ be the middle point of $I$. Then we have
\[Z(F)\in S_{I}\coloneqq \{r\cdot \exp{(\mathfrak{i}\pi\psi)}\colon r\in \RR_{\geq 0}, \phi-\epsilon\leq \psi \leq \phi+\epsilon\}\]
for any non-zero object $F\in \cP(I)$. Set $f(z)\colon \CC\to \RR$, where $f \coloneqq \Re (\exp(-\mathfrak{i}\pi \phi)z)$. Then from $0<\epsilon<\frac{1}{2}$, we obtain $f(F)\geq 0$ for any non-zero object $F\in \cP(I)$. Moreover, $f(-)$ is additive with respect to strict exact sequences in $\cP(I)$. Therefore, for an object $0\neq E\in \cP(I)$, the central charge of any strict subquotient of $E$ in $\cP(I)$ lies in the set
\[\{z\in \CC\colon f(z)\leq f(E)\}.\]
Note that for $z\in S_I$, we have
\[|z|\geq f(z)\geq \cos(\pi\epsilon)|\exp(-\mathfrak{i}\pi \phi)z|=\cos(\pi\epsilon)|z|.\]
Moreover, for any $z_i,z\in S_I$ with $\sum_i z_i=z$, we have
\begin{equation}\label{eq:sum-SI}
    \sum_i |z_i|\leq \frac{1}{\cos(\pi\epsilon)}\sum_i f(z_i)=\frac{1}{\cos(\pi\epsilon)} f(z) \leq \frac{1}{\cos(\pi\epsilon)}|z|.
\end{equation} 
Therefore, for any non-zero strict subquotient $A$ of $E$ in $\cP(I)$, we get
\[|Z(A)|=|\exp(-\mathfrak{i}\pi \phi)Z(A)|\leq \frac{1}{\cos(\pi\epsilon)}|\exp(-\mathfrak{i}\pi \phi)Z(E)|=\frac{1}{\cos(\pi\epsilon)}|Z(E)|.\]
If $A_1,\cdots, A_m$ are the factors of a filtration of $A$ with each $A_i$ semistable, then by the support property, we obtain
\[\|\overline{\bv}(A)\|\leq \sum_{i=1}^m \|\overline{\bv}(A_i)\|\leq \sum_{i=1}^m C|Z(A_i)|\]
\[=\sum_{i=1}^m C|\exp(-\mathfrak{i}\pi \phi)Z(A_i)|\leq C\frac{1}{\cos(\pi\epsilon)}|\exp(-\mathfrak{i}\pi \phi)Z(A)|=C\frac{1}{\cos(\pi\epsilon)}|Z(A)|,\]
where the last inequality follows from \eqref{eq:sum-SI}, $\|-\|$ is a norm on $(\Lambda/\Lambda^Z)_{\RR}$, and $C>0$ is a constant. So we finally get
\begin{equation}\label{eq:bound-norm-JH}
\|\overline{\bv}(A)\|\leq C\frac{1}{(\cos(\pi\epsilon))^2}|Z(E)|.
\end{equation}
Since $\Lambda/\Lambda^Z$ is discrete, the set \eqref{eq:class-subquotient} is finite. 

When $I\cap \ZZ=\varnothing$ or $\cA^{Z}=0$, any two strict subobjects $E_1\subset E_2\subset E$ in $\cP(I)$ with $E_1\neq E_2$ satisfy $Z(E_2/E_1)\neq 0$, so $\overline{\bv}(E_1)\neq  \overline{\bv}(E_2)$. This finishes the proof.
\end{proof}

The category $\cP(1)$ does not have finite length in general if $\cA^{Z}$ is not of finite length.

JH filtrations may not be unique. If $\sigma$ is a stability condition, the set of isomorphism classes of graded objects of JH filtrations of a fixed $\sigma$-semistable object is unique.

In the following, we establish a weak uniqueness property of JH factors for weak stability conditions. We need to consider the following stronger version of \ref{t2}: 

\begin{definition}
A weak pre-stability condition $\sigma=(\cA, Z)$ satisfies \ref{t3} if

\begin{enumerate}[label=(t3)]
    \item\label{t3} for every $E\in \cA$ with $\mu_{\sigma}^+(E)<+\infty$, there exists  a short exact sequence $$0\to E\to \sT(E)\to \sT(E)/E\to 0$$ in $\cA$ with $\sT(E)/E\in \cA^{Z}$ and $\Hom_{\cD}(\cA^{Z}, \sT(E))=\Hom_{\cD}(\cA^{Z}, \sT(E)[1])=0$.
\end{enumerate}
\end{definition}

The object $\sT(E)$ behaves similarly to the double-dual of a torsion-free sheaf. Note that if \(\sT(E)\) exists, then it is unique up to a unique isomorphism. Moreover, if $F\to E$ is a morphism between objects in $\cA$ with $\mu_{\sigma}^+(F),\mu_{\sigma}^+(E)<+\infty$, then we have a unique induced morphism $\sT(F)\to \sT(E)$. We begin with a collection of general properties.

\begin{lemma}\label{lem:tilt-property-basic}
Let $\sigma=(\cA, Z)$ be a weak pre-stability condition on $\cD$ that satisfies \ref{t3}. Fix $E\in \cA$ with $\mu^+_{\sigma}(E)<+\infty$.

\begin{enumerate}
    \item We have $\mu^+_{\sigma}(\sT(E))<+\infty$ and $\Hom_{\cD}(\cA^Z, \sT(E)[i])=0$ for $i\leq 1$.

    \item The injection $\sT(E)\hookrightarrow \sT(\sT(E))$ is an isomorphism.

    \item If $F\subset E$ such that $E/F\in \cA^Z$, then $\Hom_{\cD}(E, \sT(F))=\Hom_{\cD}(F, \sT(F))$. In particular, the injection $F\hookrightarrow \sT(F)$ can be factored as $F\hookrightarrow E\to \sT(F)$ and the second map is unique and injective.

    \item If $F\subset E$ such that $\sT(E)\cong E$ and $E/F\in \cA^Z$. Then the object $\sT(F)\cong E$ and the injection $F\hookrightarrow \sT(F)$ is unique up to unique isomorphism of $\sT(F)$.

\end{enumerate}

\end{lemma}

\begin{proof}
For part (a), it suffices to show $\mu^+_{\sigma}(\sT(E))<+\infty$. For any subobject $F\subset \sT(E)$ with $F\notin \cA^Z$, the kernel $F'\subset F$ of $F\to \sT(E)/E$ has the same non-zero central charge as $F$. Moreover, by the snake lemma, we have $F'\subset E$. Therefore, the result follows from $\mu^+_{\sigma}(E)<+\infty$.

For part (b), we have a commutative diagram
\[\begin{tikzcd}
	0 & E & {\sT(E)} & {\sT(E)/E} & 0 \\
	0 & \sT(E) & {\sT(\sT(E))} & {\sT(\sT(E))/\sT(E)} & 0
	\arrow[from=1-1, to=1-2]
	\arrow[from=1-2, to=1-3]
	\arrow[hook, from=1-2, to=2-2]
	\arrow[from=1-3, to=1-4]
	\arrow[hook, from=1-3, to=2-3]
	\arrow[from=1-4, to=1-5]
	\arrow[from=1-4, to=2-4]
	\arrow[from=2-1, to=2-2]
	\arrow[from=2-2, to=2-3]
	\arrow[from=2-3, to=2-4]
	\arrow[from=2-4, to=2-5]
\end{tikzcd}\]
with rows exact and $\sT(\sT(E))/\sT(E)\in \cA^Z$. Then from $\Hom_{\cD}(\cA^Z, \sT(E)[1])=0$, we get $$\sT(\sT(E))=\sT(E)\oplus \sT(\sT(E))/\sT(E),$$ which gives $\sT(\sT(E))/\sT(E)\cong 0$ by $\mu^+_{\sigma}(\sT(\sT(E)))<+\infty$ proved in (a). Therefore, the injection $$\sT(E)\hookrightarrow \sT(\sT(E))$$ is an isomorphism.

Applying $\Hom_{\cD}(-,\sT(F))$ to the exact sequence $0\to F\to E\to E/F\to 0$ and using (a), we get a natural identification $$\Hom_{\cD}(E, \sT(F))=\Hom_{\cD}(F, \sT(F)).$$ For the injectivity of $E\to \sT(F)$, note that its kernel is contained in $E/F$, which is in $\cA^Z$, hence is zero by $\mu^+_{\sigma}(E)<+\infty$. This proves (c).

For part (d), by (c), we have an injection $E\to \sT(F)$. Note that the induced morphism $$\sT(F)/F\twoheadrightarrow \sT(F)/E$$ is surjective, hence $\sT(F)/E\in \cA^Z$. Then from $E\cong \sT(E)$, we get $\Hom_{\cD}(\cA^Z, E[1])=0$ and $$\sT(F)\cong E\oplus \sT(F)/E.$$ Since $\mu^+_{\sigma}(\sT(F))<+\infty$ by (a), we obtain $\sT(F)/E\cong 0$.
\end{proof}

\begin{lemma}\label{lem:wt-inj}
Let $\sigma=(\cA, Z)$ be a weak pre-stability condition on $\cD$ that satisfies \ref{t3}. Let $E, F\in \cA$ be objects with $\mu^+_{\sigma}(E)<+\infty$ and $0\neq F\subset E$. 

\begin{enumerate}
    \item We have $\sT(F)\subset \sT(E)$, which is an isomorphism if $E/F\in \cA^Z$.

    \item If $\Hom_{\cD}(\cA^Z, E/F)=0$, we have $E/F\subset \sT(E)/\sT(F)$. If we also have $\mu^+_{\sigma}(E/F)<+\infty$, then $\sT(\sT(E)/\sT(F)) \cong \sT(E/F)$.

    \item If $E\hookrightarrow \sT(E)$ is an isomorphism and $\Hom_{\cD}(\cA^Z, E/F)=0$, then $\sT(F)\cong F$.

    \item If $\mu_{\sigma}^+(E/F)<+\infty$ and $E/F\hookrightarrow \sT(E/F)$ is an isomorphism, then the induced sequence 
    \[0\to \sT(F)\to \sT(E)\to \sT(E/F)\to 0\]
    is exact.
\end{enumerate}

\end{lemma}

\begin{proof}
For part (a), let $K$ be the kernel of $\sT(F)\to \sT(E)$. Applying the snake lemma to 
\[\begin{tikzcd}
	0 & F & {\sT(F)} & {\sT(F)/F} & 0 \\
	0 & E & {\sT(E)} & {\sT(E)/E} & 0
	\arrow[from=1-1, to=1-2]
	\arrow[from=1-2, to=1-3]
	\arrow[hook, from=1-2, to=2-2]
	\arrow[from=1-3, to=1-4]
	\arrow[from=1-3, to=2-3]
	\arrow[from=1-4, to=1-5]
	\arrow[from=1-4, to=2-4]
	\arrow[from=2-1, to=2-2]
	\arrow[from=2-2, to=2-3]
	\arrow[from=2-3, to=2-4]
	\arrow[from=2-4, to=2-5]
\end{tikzcd}\]
and using $\sT(F)/F \in \cA^Z$, we get $K\in \cA^Z$, which contradicts Lemma \ref{lem:tilt-property-basic}(a). Thus, we obtain an injection $\sT(F)\to \sT(E)$. When $E/F\in \cA^Z$, we also have $\sT(E)/\sT(F)\in \cA^Z$. Thus, $\sT(E)/\sT(F)=0$ and part (a) follows from Lemma \ref{lem:tilt-property-basic}(a).

For part (b), we consider the exact sequence $0\to F\to E\to E/F\to 0$. By (a), we get a commutative diagram
\[\begin{tikzcd}
	0 & F & E & {E/F} & 0 \\
	0 & {\sT(F)} & {\sT(E)} & {\sT(E)/\sT(F)} & 0
	\arrow[from=1-1, to=1-2]
	\arrow[from=1-2, to=1-3]
	\arrow[hook, from=1-2, to=2-2]
	\arrow[from=1-3, to=1-4]
	\arrow[hook, from=1-3, to=2-3]
	\arrow[from=1-4, to=1-5]
	\arrow[from=1-4, to=2-4]
	\arrow[from=2-1, to=2-2]
	\arrow[from=2-2, to=2-3]
	\arrow[from=2-3, to=2-4]
	\arrow[from=2-4, to=2-5]
\end{tikzcd}\]
with rows exact. Since $\Hom_{\cD}(\cA^Z, E/F)=0$ and $\sT(F)/F,\sT(E)/E\in \cA^Z$, we see that the induced map in the above diagram $E/F\to \sT(E)/\sT(F)$ is injective and its cokernel is in $\cA^Z$. If we also have $\mu^+_{\sigma}(E/F)<+\infty$, the result follows from part (a).

For part (c), using the diagram in the proof of part (b), the snake lemma, and the injectivity of $E/F\to \sT(E)/\sT(F)$ in (b), we see that $\sT(F)/F=0$ and the result follows.

Finally, part (d) follows from the diagram in the proof of part (b) and the statement of (b), as in this case, $E/F\cong \sT(E)/\sT(F)\cong \sT(E/F)$.
\end{proof}

Analogously to the double-dual of torsion-free sheaves, the operator $\sT(-)$ also preserves stability.

\begin{lemma}\label{lem:wtE-stable}
Let $\sigma=(\cA, Z)$ be a weak pre-stability condition on $\cD$ such that for every $E\in \cA$ with $\mu_{\sigma}^+(E)<+\infty$, there exists a short exact sequence $$0\to E\to \sT(E)\to \sT(E)/E\to 0$$ in $\cA$ with $\sT(E)/E\in \cA^{Z}$ and $\Hom_{\cD}(\cA^{Z}, \sT(E))=0$. Then $E$ is $\sigma$-(semi)stable if and only if $\sT(E)$ is $\sigma$-(semi)stable.
\end{lemma}

\begin{proof}
Let $F\subset\sT(E)$ be a non-zero subobject and $F_1\in \cA^Z$ be the image of $F\to \sT(E)/E$. Set $F_2\coloneqq \ker(F\twoheadrightarrow F_1)$. Then we have a diagram 
\[\begin{tikzcd}
	0 & {F_2} & F & {F_1} & 0 \\
	0 & E & {\sT(E)} & {\sT(E)/E} & 0
	\arrow[from=1-1, to=1-2]
	\arrow[from=1-2, to=1-3]
	\arrow[hook, from=1-2, to=2-2]
	\arrow[from=1-3, to=1-4]
	\arrow[hook, from=1-3, to=2-3]
	\arrow[from=1-4, to=1-5]
	\arrow[hook, from=1-4, to=2-4]
	\arrow[from=2-1, to=2-2]
	\arrow[from=2-2, to=2-3]
	\arrow[from=2-3, to=2-4]
	\arrow[from=2-4, to=2-5]
\end{tikzcd}\]
with rows exact. Thus, from the construction, we have $Z(F_2)=Z(F)$ and $Z(E/F_2)=Z(\sT(E)/F)$, and it is direct to get the statement.
\end{proof}

Now, we can state the result: the set of JH factors is unique up to applying $\sT(-)$.

\begin{proposition}\label{prop:unique-JH}
Let $\sigma=(\cA, Z)$ be a weak stability condition on $\cD$ that satisfies \ref{t3}. Let $E\in \cA$ be a $\sigma$-semistable object with $\mu_{\sigma}(E)<+\infty$. Then the isomorphism class of the object $$\bigoplus_{0\leq i\leq n-1} \sT(E_{i+1}/E_i)$$ is independent of the choice of JH filtrations $$0=E_0\subset E_1\subset \cdots \subset E_n=E.$$
\end{proposition}

\begin{proof}
Consider the statement: if $E\in \cA$ is a $\sigma$-semistable object with $\mu_{\sigma}(E)<+\infty$ and has a JH filtration $$0=E_0\subset E_1\subset \cdots \subset E_n=E$$ of length $n$, then the isomorphic class of the object $\bigoplus_{0\leq i\leq n-1}\sT(E_{i+1}/E_i)$ is independent of the choice of JH filtrations of $E$.

We proceed by induction on $n$ to prove this statement. When $n=1$, $E$ is $\sigma$-stable, and the result is clear, so we assume that $n\geq 2$ and the statement holds for $\sigma$-semistable objects with $\mu_{\sigma}<+\infty$ and have a JH filtration of length $\leq n-1$. By Lemma \ref{lem:wt-inj}(b) and \ref{lem:wtE-stable}, the filtration
$$0\subset \sT(E_1)\subset \sT(E_2)\subset \cdots \subset \sT(E_n)=\sT(E)$$
is a JH filtration of $\sT(E)$ of length $n$, so we may assume that $\sT(E_i)=E_i$ for each $1\leq i\leq n$.

Let $$0=F_0\subset F_1\subset \cdots \subset F_m=E$$ be another JH filtration of $E$ of length $m$. Then $\sT(F_j)=F_j$ for $1\leq j\leq m$ by Lemma \ref{lem:wt-inj}(c). Let $1\leq k\leq n$ be the smallest integer that $E_k$ contains $F_1$. Then we have a non-zero map $F_1\to E_{k}/E_{k-1}$, which is injective and has a cokernel in $\cA^Z$ as they are $\sigma$-stable with the same finite slope. Since $\sT(F_1)\cong F_1$, we know that $F_1\to E_{k}/E_{k-1}$ is an isomorphism, hence we get a splitting 
\[E_k\cong E_{k-1}\oplus F_1.\]
Therefore, 
$$0=F_1/F_1\subset F_2/F_1\subset \cdots \subset F_{m-1}/F_1 \subset F_m/F_1=E/F_1$$
is a JH filtration of $E/F_1$ of length $m-1$. If we define $E_j'\coloneqq E_j$ for $1\leq j\leq k-1$, and $E_j'\coloneqq E_j/F_1$ for $k\leq j\leq n$, then by the minimality of $k$, we have $E'_k\cong E_{k-1}=E'_{k-1}$ and
\[0=E_0'\subset E_1'\subset \cdots \subset E'_{k-1}\subset E'_{k+1} \subset \cdots \subset  E_{n-1}'\subset E'_{n}=E/F_1\]
is a JH filtration of $E/F_1$ of length $n-1$. Therefore, the induction hypothesis applies to $E/F_1$ and we have
\[\bigoplus_{0\leq i\leq n-1,i\neq k-1}\sT(E'_{i+1}/E_i')\cong\bigoplus_{1\leq j\leq m-1}\sT((F_{j+1}/F_1)/(F_j/F_1)).\]
Therefore,
\[\bigoplus_{0\leq i\leq n-1, i\neq k-1}\sT(E_{i+1}/E_i)\cong\bigoplus_{1\leq j\leq m-1}\sT(F_{j+1}/F_j).\]
As $E_k/E_{k-1}\cong F_1$, we get
\[\bigoplus_{0\leq i\leq n-1}\sT(E_{i+1}/E_i)\cong\bigoplus_{0\leq j\leq m-1}\sT(F_{j+1}/F_j)\]
and the result follows.
\end{proof}

\section{Stability conditions in families}\label{sec:stab-family}

In this section, we collect definitions and properties needed in the relative setting. We mainly follow \cite{BLMNPS21}.

\subsection{Numerical K-groups}\label{subsec:num-k}

Let $X$ be a projective scheme over a field $\kk$ and $\cD\subset \Db(X)$ be a strong semi-orthogonal component. Recall that we have the Euler pairing
\[\chi_{\kk}(-,-)\colon \KK(\cD_{\perf})\times \KK(\cD)\to \ZZ,\]
where $\cD_{\perf}\coloneqq\cD\cap \Dperf(X)$. When the base field is clear from the context, we will always omit it from the notation. We define the \emph{numerical K-group} of $\cD$ to be the quotient of the usual K-group $\KK(\cD)$ by the null space of $\chi$ on the right, which is denoted by $\Knum(\cD)$. Similarly, we denote by $\Knum(\cD_{\perf})$ the quotient of $\KK(\cD_{\perf})$ by the null space of $\chi$ on the left. By \cite[Lemma 12.7]{BLMNPS21}, $\Knum(\cD)$ and $\Knum(\cD_{\perf})$ are finite rank lattices.

Now given a field extension $\kk\subset \kk_1$, the pullback functor $\cD\to \cD_{\kk_1}$ induces homomorphisms $\KK(\cD)\to \KK(\cD_{\kk_1})$ and $\KK(\cD_{\perf})\to \KK(\cD_{\kk_1,\perf})$. By \cite[Lemma 12.14]{BLMNPS21}, it also induces
\[\eta_{\kk_1/\kk}\colon \Knum(\cD_{\perf})\to \Knum(\cD_{\kk_1,\perf}).\]
Therefore, as explained in \cite[Proposition and Definition 12.15]{BLMNPS21}, dualizing this map gives the following pushforward map
\[\eta^{\vee}_{\kk_1/\kk}\colon \Knum(\cD_{\kk_1})\hookrightarrow\Hom(\Knum(\cD_{\kk_1,\perf}),\ZZ)\to \Hom(\Knum(\cD_{\perf}),\ZZ)\to \Knum(\cD)\otimes \QQ,\]
such that the image $\Knum(\cD)_{\kk_1}\coloneqq\im(\eta^{\vee}_{\kk_1/\kk})$ contains $\Knum(\cD)$ as a subgroup of finite index; moreover, it is contained in $\Knum(\cD)_{\overline{\kk}}$.

Next, we are going to discuss numerical K-groups in the relative setting. Following \cite{BLMNPS21}, we will work in the following situation in the rest of this section.

\begin{Assum}\label{assum-stab}
The morphism $f\colon X\to S$ is a flat projective morphism between Noetherian schemes such that $X$ has finite Krull dimension and $S$ is a Nagata scheme which is quasi-projective over a Noetherian affine scheme. The subcategory $\cD\subset \Db(X)$ is a $S$-linear strong semi-orthogonal component of finite cohomological amplitude.
\end{Assum}

\begin{definition}[{\cite[Definition 21.1]{BLMNPS21}}]
Fix a morphism $f\colon X\to S$ and a semi-orthogonal component $\cD\subset \Db(X)$ satisfying Assumption \ref{assum-stab}. We define the \emph{relative numerical K-group} $\Knum(\cD/S)$ as the quotient of $\bigoplus_{s\in S} \Knum(\cD_s)_{\overline{s}}$ by the saturation of the subgroup generated by the elements of the form
\[\eta^{\vee}_{t_1/g(t_1)}(E_{t_1})-\eta^{\vee}_{t_2/g(t_2)}(E_{t_2})\]
for all tuples $(g,E,t_1,t_2)$ where $g\colon T\to S$ is a morphism from a connected scheme $T$, $E\in \Dqc(X_T)$ is a $T$-perfect object such that $E_t\in \cD_t$ for all $t\in T$, and $t_1,t_2\in T$.
\end{definition}

\begin{remark}
By \cite[Remark 21.2]{BLMNPS21}, it is enough to consider morphisms $f\colon T\to S$ of finite type from a connected affine scheme $T$.
\end{remark}

In particular, for $E\in \Dqc(X_T)$ as in the definition, we see that the image of $[E_t]\in \Knum(\cD_t)$ under the composition
\[\Knum(\cD_t)\to \Knum(\cD_{f(t)})_t\hookrightarrow \Knum(\cD_{f(t)})_{\overline{f(t)}}\to \Knum(\cD/S)\]
is independent of $t\in T$. In this case, we denote this image by $[E]\in \Knum(\cD/S)$.


\subsection{Base change}

Let $f\colon X\to S$ be a morphism of schemes that are quasi-compact with affine diagonal, where $X$ is Noetherian of finite Krull dimension. Let
\[\langle \cD_1,\dots ,\cD_m\rangle=\Db(X)\]
be a strong $S$-linear semi-orthogonal decomposition. Then by \cite[Theorem 3.17]{BLMNPS21}, for any $g\colon T\to S$ from a scheme $T$ which is quasi-compact with affine diagonal such that $g$ is faithful with respect to $f$ (e.g.~$f$ and $g$ are Tor-independent), there is an induced $T$-linear semi-orthogonal decomposition
\[\langle (\cD_{1,\mathrm{qc}})_T,\dots ,(\cD_{m, \mathrm{qc}})_T\rangle=\Dqc(X_T).\]
Furthermore, if the starting decomposition has finite cohomological amplitude, then it induces a $T$-linear semi-orthogonal decomposition
\[\langle (\cD_1)_T,\dots ,(\cD_m)_T\rangle=\Db(X_T).\]
Moreover, if $g$ is proper, the pushforward functor induces $(\cD_i)_T\to \cD_i$; if $g$ has finite Tor-dimension, the pullback functor induces $\cD_i\to (\cD_i)_T$.

\begin{definition}
We say a t-structure on a semi-orthogonal component $\cD\subset \Db(X)$ is \emph{$S$-local} if for every quasi-compact open subset $U\subset S$, there exists a t-structure on $\cD_U$ such that the restriction functor $\cD\to \cD_U$ is t-exact. 

Similarly, we say a slicing $\cP$ of $\cD$ is \emph{$S$-local} if for every quasi-compact open subset $U\subset S$, there exists a slicing $\cP_U$ of $\cD_U$ such that the restriction sends $\cP(\phi)$ to $\cP_U(\phi)$ for each $\phi\in \mathbb{R}$.
\end{definition}

Note that by \cite[Theorem 4.13]{BLMNPS21}, if $S$ is affine, then any bounded t-structure on $\cD\subset \Db(X)$ is $S$-local.

By \cite[Theorem 5.3]{BLMNPS21}, for any $S$-linear strong semi-orthogonal component $\cD\subset \Db(X)$ and any $g\colon T\to S$ from a scheme $T$ which is quasi-compact with affine diagonal such that $g$ is faithful with respect to $f$, a t-structure $(\cD^{\leq 0}, \cD^{\geq 0})$ on $\cD$ induces a t-structure $((\cD_{\mathrm{qc}})_T^{\leq 0}, (\cD_{\mathrm{qc}})_T^{\geq 0})$ on $(\cD_{\mathrm{qc}})_T$. If $\cA$ is the heart of $(\cD^{\leq 0}, \cD^{\geq 0})$, then we denote by $(\cA_{\mathrm{qc}})_T$ the heart of $((\cD_{\mathrm{qc}})_T^{\leq 0}, (\cD_{\mathrm{qc}})_T^{\geq 0})$. In the case of change of base fields, we have:

\begin{lemma}[{\cite[Theorem 5.3, Proposition 5.7]{BLMNPS21}}]\label{lem:base-change-t-structure}
Assume that $S=\Spec(\kk)$ and $T=\Spec(\kk_1)$ for two fields $\kk\subset \kk_1$. Given a $S$-linear strong semi-orthogonal component $\cD\subset \Db(X)$ that has finite cohomological amplitude with a bounded t-structure $(\cD^{\leq 0}, \cD^{\geq 0})$ on $\cD$ and can be obtained from a Noetherian t-structure on $\cD$ by tilting. Denote by $\pi\colon X_{\kk_1}\to X$ the base change morphism.

\begin{itemize}
    \item  The t-structure $((\cD_{\mathrm{qc}}^{\leq 0})_{\kk_1}, (\cD_{\mathrm{qc}}^{\geq 0})_{\kk_1})$ on $(\cD_{\mathrm{qc}})_{\kk_1}$ restricts to a bounded t-structure $(\cD_{\kk_1}^{\leq 0}, \cD_{\kk_1}^{\geq 0})$ on $\cD_{\kk_1}$.

    \item Functors $L\pi^*\colon \cD_{\mathrm{qc}}\to (\cD_{\mathrm{qc}})_{\kk_1}$ and $R\pi_*\colon (\cD_{\mathrm{qc}})_{\kk_1}\to \cD_{\mathrm{qc}}$ are t-exact with respect to t-structures $((\cD_{\mathrm{qc}}^{\leq 0})_{\kk_1}, (\cD_{\mathrm{qc}}^{\geq 0})_{\kk_1})$ and $(\cD_{\mathrm{qc}}^{\leq 0}, \cD_{\mathrm{qc}}^{\geq 0})$.

    \item $(\cD_{\mathrm{qc}}^{\leq 0})_{\kk_1}$ is the smallest full subcategory of $(\cD_{\mathrm{qc}})_{\kk_1}$ that contains $L\pi^*(\cD^{\leq 0})$ and is closed under extensions and small colimits.

    \item If $\kk=\kk_1$, then we have $(\cD_{\kk_1}^{\leq 0}, \cD_{\kk_1}^{\geq 0})=(\cD^{\leq 0}, \cD^{\geq 0})$.

    \item  For any $a,b\in \ZZ\cup\{\pm\infty\}$, we have
\[(\cD_{\mathrm{qc}})_{\kk_1}^{[a,b]}=\{E\in (\cD_{\mathrm{qc}})_{\kk_1}\colon R\pi_*E\in (\cD_{\mathrm{qc}})^{[a,b]}\}.\]
\end{itemize}
\end{lemma}

If $\cA$ is the heart of $(\cD^{\leq 0}, \cD^{\geq 0})$, then we denote by $(\cA_{\mathrm{qc}})_{\kk_1}$ and $\cA_{\kk_1}$ the heart of $((\cD_{\mathrm{qc}}^{\leq 0})_{\kk_1}, (\cD_{\mathrm{qc}}^{\geq 0})_{\kk_1})$ and $(\cD_{\kk_1}^{\leq 0}, \cD_{\kk_1}^{\geq 0})$, respectively.

Now we discuss the base change of numerical K-groups. Let $X$ be a projective scheme over a field $\kk$ and $\cD\subset \Db(X)$ be a strong semi-orthogonal component.

\begin{definition}
We say a weak pre-stability condition $\sigma=(Z,\cP)$ on $\cD$ with respect to $\bv\colon \KK(\cD)\to \Lambda$ is \emph{numerical} if $\bv$ factors through the natural surjection $\KK(\cD)\twoheadrightarrow \Knum(\cD)$. We still denote by the induced homomorphisms $\Knum(\cD)\to \Lambda$ and $\Knum(\cD)\to \CC$ by $\bv$ and $Z$, respectively.
\end{definition}

Given a numerical weak pre-stability condition $\sigma=(Z,\cP)$ on $\cD$ with respect to $\bv\colon \KK(\cD)\to \Lambda$ and a field extension $\kk\subset \kk_1$, we write $Z_{\kk_1}$ for the composition $Z\circ \eta^{\vee}_{\kk_1/\kk}$. We set 
\[\Lambda_{\kk_1}\subset \Lambda_{\QQ}\]
to be the subgroup generated by $\Lambda$ and the image of $\bv$ extended to $\Knum(\cD)_{\kk_1}$. Therefore, $\Lambda_{\kk_1}$ contains $\Lambda$ as a subgroup of finite index. We denote by 
\[\bv_{\kk_1}\colon \Knum(\cD_{\kk_1})\to \Lambda_{\kk_1}\]
the induced map. Hence, $Z_{\kk_1}$ can be factored as the composition $\Knum(\cD_{\kk_1})\to \Knum(\cD)_{\kk_1}\to \Lambda_{\kk_1}$.

Let $\sigma$ be a numerical weak pre-stability condition on $\cD$. Then for any field extension $\kk\subset \kk_1$, by Lemma \ref{lem:base-change-t-structure} and the discussion above, we get a pair $\sigma_{\kk_1}=(Z_{\kk_1}, \cA_{\kk_1})$, which is called \emph{the base change of $\sigma$ to $\kk_1$}.

\begin{definition}\label{def:geo-stable}
Assume that $\sigma_{\overline{\kk}}=(Z_{\overline{\kk}}, \cA_{\overline{\kk}})$ is also a weak pre-stability condition. We say $E\in \cD$ is \emph{geometrically $\sigma$-stable} if $E_{\overline{\kk}}$ is $\sigma_{\overline{\kk}}$-stable.
\end{definition}

\subsection{Torsion and torsion-free objects}

We fix a morphism $f\colon X\to S$ and a semi-orthogonal component $\cD\subset \Db(X)$ satisfying Assumption \ref{assum-stab}. Moreover, we assume that $C\coloneqq S$ is a Dedekind scheme, i.e. $C$ is integral, regular, and one-dimensional.

Recall that for a Dedekind scheme $C$, we write $p\in C$ for a closed point, $c\in C$ for an arbitrary point, and $K$ for its fraction field.

The following notions of torsion and torsion-free objects relative to $C$ are crucial.

\begin{definition}
An object $E\in \cD$ is called \emph{$C$-torsion} if it is the pushforward of an object in $\cD_Z$ for a proper closed subschemes $Z\subset C$. We denote by $\cD_{C\text{-}\tor}$ the subcategory of $C$-torsion objects in $\cD$.
\end{definition}

According to \cite[Lemma 6.4]{BLMNPS21}, an object $E\in \cD$ is $C$-torsion if and only if $E_K=0$. Moreover, we have an exact triangle of triangulated categories:
\[\cD_{C\text{-}\tor}\to \cD\to \cD_{K}.\]

\begin{definition}
Let $\cA_C\subset \cD$ be the heart of a $C$-local t-structure. We say $E\in \cA_C$ is \emph{$C$-torsion-free} if it does not contain any non-zero $C$-torsion subobject. We denote by $\cA_{C\text{-}\tor}\subset \cA_C$ the subcategory of $C$-torsion objects, and by $\cA_{C\text{-}\mathrm{tf}}\subset \cA_C$ the subcategory of $C$-torsion-free objects.
\end{definition}

The following result proved in \cite[Lemma 6.6]{BLMNPS21} will be useful.

\begin{lemma}
The subcategory $\cA_{C\text{-}\mathrm{tf}}\subset \cA_C$ is closed under subobjects and extensions. The subcategory $\cA_{C\text{-}\tor}\subset \cA_C$ is closed under subobjects, quotients, and extensions.
\end{lemma}

Recall that for the heart $\cA_C\subset \cD$ of a $C$-local t-structure, we have an induced heart $\cA_c$ for any $c\in C$ by \cite[Theorem 5.6, 5.7]{BLMNPS21}. 

\begin{lemma}[{\cite[Lemma 6.12]{BLMNPS21}}]\label{lem:6.12}
If $E\in \cA_C$ is $C$-torsion-free, then $E_c\in \cA_c$ for each $c\in C$.
\end{lemma}

\begin{definition}
We say $\cA_C$ has a \emph{$C$-torsion theory} if the pair of subcategories $(\cA_{C\text{-}\tor},\cA_{C\text{-}\mathrm{tf}})$ forms a torsion pair.
\end{definition}

\begin{remark}\label{rmk:noether-C-torsion-pair}
By \cite[Remark 6.16]{BLMNPS21}, if $\cA_C$ is Noetherian, then it has a $C$-torsion theory.
\end{remark}

To check the existence of a $C$-torsion theory, we need the following criterion.

\begin{lemma}[{\cite[Theorem 17.1]{BLMNPS21}}]\label{lem:17.1}
Let $\cA_C$ be the heart of a $C$-local t-structure. Assume that 

\begin{enumerate}
    \item $\cA_C$ universally satisfies openness of flatness, i.e. for every morphism of finite presentation $T\to S$ from an affine scheme $T$ and every $T$-perfect object $E\in \Db(X_T)$, the set
    \[\{t\in T \colon E_t\in \cA_t\}\]
    is open, and

    \item for every closed point $p\in C$, there exists a weak stability condition $\sigma_p=(\cA_p, Z_p)$ so that $\cA_p^{Z_p}\subset \cA_p$ is a Noetherian torsion subcategory.
\end{enumerate}

Then $\cA_C$ admits a $C$-torsion theory.
\end{lemma}

\begin{proof}
By \cite[Theorem 17.1]{BLMNPS21}, it suffices to show that the condition (a) implies the property in \cite[Definition 10.4]{BLMNPS21}, i.e. for every morphism $T\to S$ from a scheme $T$ and every $T$-perfect object $E$ on $X_T$, the set
\[\{t\in T \colon E_t\in (\cA_{\mathrm{qc}})_t\}\]
is open. Indeed, by \cite[Lemma 10.6]{BLMNPS21}, it suffices to check it for all morphisms of finite presentation $T\to S$ from affine schemes. By Assumption \ref{assum-stab}, we know that $T$ is Noetherian and $X_T\to T$ is flat projective. Therefore, any $T$-perfect object $E$ on $X_T$ lies in $\Db(X_T)$ by Lemma \ref{lem-S-perf-lem-1}. Moreover, $E_t\in \Db(X_t)$, so we have an identification
\[\{t\in T \colon E_t\in (\cA_{\mathrm{qc}})_t\}=\{t\in T \colon E_t\in (\cA_{\mathrm{qc}})_t\cap \Db(X_t)\}.\]
Now the result follows from $(\cA_{\mathrm{qc}})_t\cap \Db(X_t)=\cA_t$ by the construction of $(\cA_{\mathrm{qc}})_t$ in \cite[Theorem 5.3]{BLMNPS21}.
\end{proof}

\subsection{Harder--Narasimhan structures over curves}\label{sec:HN-structure}

Next, we review the theory of Harder--Narasimhan (HN) structures over curves introduced in \cite{BLMNPS21} in order to do semistable reduction for objects.

We fix a morphism $f\colon X\to S$ and a semi-orthogonal component $\cD\subset \Db(X)$ satisfying Assumption \ref{assum-stab}. Moreover, we assume that $C\coloneqq S$ is a Dedekind scheme.

\begin{definition}\label{def-central-charge-over-C}
A \emph{central charge on $\cD$ over $C$} is a pair $Z_C=(Z_K, Z_{C\text{-}\tor})$ where
\[Z_K\colon \KK(\cD_K)\to \CC\quad \text{and}\quad Z_{C\text{-}\tor}\colon \KK(\cD_{C\text{-}\tor})\to \CC\]
are group homomorphisms such that for each $E\in \cD$ and each proper closed subscheme $i_W\colon W\hookrightarrow C$, we have
\[Z_K(E_K)=\frac{1}{\mathrm{length} (W)}Z_{C\text{-}\tor}(i_{W*}E_W).\]
\end{definition}

For any $E\in \cD$, we set $Z_C(E)$ to be $Z_K(E_K)$ if $E_K\neq 0$, and $Z_{C\text{-}\tor}(E)$ otherwise.

\begin{definition}
A \emph{weak Harder--Narasimhan (HN) structure} on $\cD$ over $C$ is a triple $\sigma_C=(Z_K, Z_{C\text{-}\tor}, \cP)$ where
\begin{itemize}
    \item $\cP$ is a $C$-local slicing over $C$, and

    \item $Z_C=(Z_K, Z_{C\text{-}\tor})$ is a central charge on $\cD$ over $C$,
\end{itemize}

satisfying

\begin{itemize}
    \item for any $\phi\in \mathbb{R}$ and any non-zero $E\in \cP(\phi)$, we have either 
    
    \begin{itemize}
        \item  $E_K\neq 0$ and $Z_K(E_K)\in \mathbb{R}_{>0}\cdot e^{\mathfrak{i}\pi\phi}$ (for $\phi\notin \ZZ$) or $Z_K(E_K)\in \mathbb{R}_{\geq 0}\cdot e^{\mathfrak{i}\pi\phi}$ (for $\phi\in \ZZ$), or

        \item $0\neq E\in \cD_{C\text{-}\tor}$ and $Z_{C\text{-}\tor}(E)\in \mathbb{R}_{>0}\cdot e^{\mathfrak{i}\pi\phi}$ (for $\phi\notin \ZZ$) or $Z_{C\text{-}\tor}(E)\in \mathbb{R}_{\geq 0}\cdot e^{\mathfrak{i}\pi\phi}$ (for $\phi\in \ZZ$).
    \end{itemize}

\end{itemize}

We say $\sigma_C$ is a \emph{Harder--Narasimhan (HN) structure} on $\cD$ over $C$ if it further satisfies

\begin{itemize}
    \item for any $\phi\in \mathbb{R}$ and any non-zero $E\in \cP(\phi)$, we have either

    \begin{itemize}
        \item $E_K\neq 0$ and $Z_K(E_K)\in \mathbb{R}_{>0}\cdot e^{\mathfrak{i}\pi\phi}$, or

        \item $0\neq E\in \cD_{C\text{-}\tor}$ and $Z_{C\text{-}\tor}(E)\in \mathbb{R}_{>0}\cdot e^{\mathfrak{i}\pi\phi}$.
    \end{itemize}
    
\end{itemize}
\end{definition}

In the above setting, we set $\cA_C\coloneqq\cP(0,1]$, hence $\cA_C$ is a heart of a bounded $C$-local t-structure. We also denote $\sigma_C$ by $(Z_C, \cP)$ or $(\cA_C, Z_C)$ for simplicity, where $Z_C=(Z_K, Z_{C\text{-}\tor})$ is the central charge. We say a central charge $Z_C$ on $\cA_C$ is a \emph{(weak) stability function} if $Z_K$ and $Z_{C\text{-}\tor}$ are (weak) stability functions on $\cA_K$ and $\cA_{C\text{-}\tor}$, respectively.

Given a weak HN structure $\sigma_C=(Z_K, Z_{C\text{-}\tor}, \cP)$, we can define the slope $\mu_{\sigma_C}(E)$ for any $E\in \cA_C$ by using $Z_K(E_K)$ if $E_K\neq 0$ and by using $Z_{C\text{-}\tor}(E)$ if $E_K=0$ as in \cite[Definition 13.7]{BLMNPS21}. Therefore, we can also define \emph{$\sigma_C$-(semi)stability} and phase analogously. Note that $\sigma_C$-(semi)stability here is the same as $Z_C$-(semi)stability defined in \cite[Definition 13.9]{BLMNPS21}. The maximal and minimal slopes of HN factors of an object $E\in \cA_C$ with respect to $\sigma_C$ are denoted by $\mu^+_{\sigma_C}(E)$ and $\mu^-_{\sigma_C}(E)$.

\begin{lemma}\label{lem-HN-induce}
A (weak) HN structure $\sigma_C$ on $\cD$ over $C$ induces a (weak) pre-stability condition $\sigma_c=(Z_c, \cP_c)$ on $\cD_c$ for every $c\in C$.
\end{lemma}

\begin{proof}
This is a combination of \cite[Lemma 13.11]{BLMNPS21} and \cite[Lemma 15.6]{BLMNPS21}.
\end{proof}

\subsection{Tilting property for HN structures}\label{subsec-tilting-property-HN}

The strategy in Section \ref{subsec-tilting-property} also works for weak HN structures.

\begin{definition}
Given a weak HN structure $\sigma_C=(\cA_C, Z_K, Z_{C\text{-}\tor})$ on $\cD$ over $C$, we write $\cA^{Z_C}_C\subset \cA_C$ for the subcategory of objects $E$ with $Z_C(F)=0$ for every \emph{subquotient} $F$ of $E$ in $\cA_C$.
\end{definition}

\begin{definition}\label{def-tilting-hn}
A weak HN structure $\sigma_C=(\cA_C, Z_K, Z_{C\text{-}\tor})$ has the \emph{tilting property} if

\begin{enumerate}[label=(tc\arabic*)]
    \item\label{t'1} $\cA^{Z_C}_C\subset \cA_C$ is a Noetherian torsion subcategory, and

    \item\label{t'2} for every $E\in \cA_C$ with $\mu^+_{\sigma_C}(E)<+\infty$, there exists a short exact sequence $0\to E\to \wt{E}\to E^0\to 0$ in $\cA_C$ with $E^0\in \cA^{Z_C}_C$ and $\Hom_{\cD}(\cA^{Z_C}_C, \wt{E}[1])=0$.
\end{enumerate}

\end{definition}

Similarly, starting with a weak HN structure $\sigma_C=(\cA_C, Z_K, Z_{C\text{-}\tor})$, we can define a torsion pair $(\cT^{b}_C,\cF^{b}_C)$ as
\[\cT^{b}_C\coloneqq\langle E\in \cA_C\colon E \text{ is } \sigma_{C}\text{-semistable with }\mu_{\sigma_C}(E)>b \rangle\]
and
\[\cF^{b}_C\coloneqq\langle E\in \cA_C \colon E \text{ is } \sigma_{C}\text{-semistable with }\mu_{\sigma_C}(E)\leq b \rangle.\]
Then we get the tilted heart $\cA^{b}_C\coloneqq\langle \cF^{b}_C[1],\cT^{b}_C \rangle$. We also define 
\[Z^b_C\coloneqq (Z^{b}_K, Z^{b}_{C\text{-}\tor}),\]
where $$Z^{b}_K\coloneqq \Im Z_K+\mathfrak{i}(-\Re Z_K-b\Im Z_K)$$ 
and 
$$Z^{b}_{C\text{-}\tor}\coloneqq\Im Z_{C\text{-}\tor}+\mathfrak{i}(-\Re Z_{C\text{-}\tor}-b\Im Z_{C\text{-}\tor}).$$

\begin{lemma}[{\cite[Proposition 19.5]{BLMNPS21}}]\label{lem-tilt-weak-hn}
Let $\sigma_C=(\cA_C, Z_K, Z_{C\text{-}\tor})$ be a weak HN structure on $\cD$ over $C$. If it has the tilting property, and the induced $\sigma_c$ has the tilting property for each $c\in C$, then for any $b\in \RR$, $\sigma^{b}_C=(\cA^{b}_C, Z^{b}_C)$ is a weak HN structure on $\cD$ over $C$ and $(\cA^{b}_C)^{Z^{b}_C}\subset \cA^{b}_C$ is a Noetherian torsion subcategory.
\end{lemma}

\subsection{Flat families of fiberwise stability conditions}

Now, we can define the notion of flat families of fiberwise stability conditions. We fix a morphism $f\colon X\to S$ and a semi-orthogonal component $\cD\subset \Db(X)$ satisfying Assumption \ref{assum-stab}.

\begin{definition}\label{def-flat-collection}
A \emph{flat family of fiberwise stability conditions} on $\cD$ over $S$ is a collection of numerical stability conditions $\underline{\sigma}=(\sigma_s=(Z_s, \cP_s))_{s\in S}$ on $\cD_s$ for every point $s\in S$ such that:

\begin{enumerate}[label=(c\arabic*)]
    \item\label{c1} $\underline{\sigma}$ \emph{universally has locally constant central charges}, i.e.~for every morphism $T\to S$ and every $T$-perfect object $E\in \Dqc(X_T)$ such that $E_t\in \cD_t$ for all $t\in T$, the function $T\to \CC$ given by $t\mapsto Z_t(E_t)$ is locally constant.

    \item\label{c2} $\underline{\sigma}$ \emph{universally satisfies openness of geometric stability}, i.e.~for every morphism $T\to S$ and every $T$-perfect object $E\in \Dqc(X_T)$, the set 
    \[\{t\in T \colon E_t\in \cD_t \text{ and is geometrically }\sigma_t\text{-stable}\}\]
    is open in $T$. 

    \item\label{c3} For any morphism $C\to S$ which is essentially of finite type from a Dedekind scheme $C$, the stability condition $\sigma_c$ for any $c\in C$ is induced by an HN structure $\sigma_C$ on $\cD_C$ over $C$ in the sense of Lemma \ref{lem-HN-induce}.
\end{enumerate}
\end{definition}

\begin{remark}
By \cite[Lemma 20.3]{BLMNPS21}, to check \ref{c1} and \ref{c2}, we only need to consider finite type morphisms $T\to S$ from an affine scheme $T$.
\end{remark}

\begin{definition}\label{def-flat-collection-weak}
A \emph{flat family of fiberwise weak stability conditions} on $\cD$ over $S$ is a collection of numerical weak stability conditions $\underline{\sigma}=(\sigma_s=(Z_s, \cP_s))_{s\in S}$ on $\cD_s$ for every point $s\in S$ such that $\underline{\sigma}$ satisfies \ref{c1} and \ref{c2} with the following additional assumptions:

\begin{enumerate}[label=(w\arabic*)]
    \item\label{wc1} For each $s\in S$, the central charge $Z_s$ is defined over $\QQ[\mathfrak{i}]$ and $\cA^{Z}_s\subset \cA_s$ is a Noetherian torsion subcategory.

    \item\label{wc2} For any morphism $T\to S$ essentially of finite type with $T$ integral and any $T$-perfect $E\in \cD_T$ whose generic fiber $E_{K(T)}$ is $\sigma_{K(T)}$-semistable, there exists a nonempty open subset $U\subset T$ such that $E_t$ is $\sigma_t$-semistable for all $t\in U$.

    \item\label{wc3} For any morphism $C\to S$ which is essentially of finite type from a Dedekind scheme $C$, the weak stability condition $\sigma_c$ for any $c\in C$ is induced by a weak HN structure $\sigma_C=(\cA_C, Z_C)$ on $\cD_C$ over $C$ in the sense of Lemma \ref{lem-HN-induce} such that $\cA^{Z_C}_C\subset \cA_C$ is a Noetherian torsion subcategory.
\end{enumerate}
\end{definition}

By the universal property of $\Knum(\cD/S)$, for any flat family of fiberwise weak stability conditions, there exists a central charge $Z\colon \Knum(\cD/S)\to \CC$ such that for any $s\in S$, the central charge $Z_s$ factors as $Z_s\colon \Knum(\cD_s)\to \Knum(\cD/S)\xra{Z} \CC$.

We will always restrict our attention to flat families of fiberwise weak stability conditions where $Z\colon \Knum(\cD/S)\to \CC$ factors via a group homomorphism $\bv\colon \Knum(\cD/S)\to \Lambda$ to a finite rank lattice $\Lambda$, which is called a \emph{relative Mukai homomorphism}.

\begin{definition}
Given a flat family of fiberwise weak stability conditions $\underline{\sigma}$, for which $Z$ factors via a relative Mukai homomorphism $\bv$, we let $\Lambda^{Z}$ be the saturated subgroup of $\Lambda$ generated by $\bv([E_t])$ for all $E_t\in \cA^{Z}_t$ and all points $t$ over $S$. We write $\overline{\bv}$ for the composition of $\bv$ with the quotient map $\Lambda\twoheadrightarrow \Lambda/\Lambda^{Z}$.
\end{definition}

To check openness properties, we need the following lemma.

\begin{lemma}[{\cite[Proposition 20.8]{BLMNPS21}}]\label{lem:openness-flat}
Let $\underline{\sigma}=(\sigma_s=(Z_s, \cP_s))_{s\in S}$ be a collection of numerical weak stability conditions on $\cD_s$ for every point $s\in S$. Assume that 

\begin{enumerate}
    \item for any field $t\colon \Spec(\kk_1)\to S$, the induced pair $(Z_t, \cA_t)$ is a weak stability condition and the pullback $\cA_s\to \cA_t$ preserves semistability, and

    \item $\underline{\sigma}$ satisfies \ref{c1}, \ref{wc2}, \ref{wc3}, and $\cA^Z_s\subset \cA_s$ is a Noetherian torsion subcategory for each $s\in S$.
\end{enumerate}

Then for any morphism $T\to S$ of finite type and any $T$-perfect object $E$ on $X_T$, the functions $$\phi^+_E\colon T\to \RR\cup\{-\infty\},\quad t\mapsto \phi^+_{\sigma_t}(E_t)$$ and $$\phi^-_E\colon T\to \RR\cup\{+\infty\},\quad t\mapsto \phi^-_{\sigma_t}(E_t)$$ are, respectively, upper and lower semicontinuous constructible functions on $T$. Here, we set $\phi^{\pm}$ of the zero object to be $\mp\infty$ for convenience.
\end{lemma}

\begin{proof}
The statement can be found in \cite[Proposition 20.8]{BLMNPS21} under additional assumptions \ref{c2} and the central charges $Z_s$ are defined over $\QQ[\mathfrak{i}]$, i.e. $\underline{\sigma}$ is a flat family of fiberwise weak stability conditions on $\cD$ over $S$. However, in the proof of \cite[Proposition 20.8]{BLMNPS21}, the rationality of $Z_s$ is only used to apply \cite[Proposition 14.20]{BLMNPS21}, whose statement is part of our assumption (a). Moreover, \ref{c2} is never used. Therefore, the same proof as \cite[Proposition 20.8]{BLMNPS21} works in our setting.
\end{proof}

\subsection{Stability conditions in families}

Finally, we can define stability conditions over a base scheme.

\begin{definition}
Let $f\colon X\to S$ be a flat, proper, finitely presented morphism of schemes. An $S$-perfect object $E\in \Dqc(X)$ is \emph{universally gluable} if for every $s\in S$ we have
\[\Ext^i_{X_s}(E_s, E_s)=0\]
for every $i<0$. We denote by $\Dpug(X/S)\subset \Dqc(X)$ the full subcategory of universally gluable $S$-perfect objects.
\end{definition}

From now on, we fix a morphism $f\colon X\to S$ and a semi-orthogonal component $\cD\subset \Db(X)$ satisfying Assumption \ref{assum-stab}.

\begin{definition}
We denote by
\[\cM_{\mathrm{pug}}(\cD/S)\colon (\mathrm{Sch}/S)^{\mathrm{op}}\to \mathrm{Grpd}\]
the functor whose value on $T\in (\mathrm{Sch}/S)$ consists of all $E\in \Dpug(X_T/T)$ such that $E_t\in \cD_t$ for all $t\in T$.
\end{definition}

\begin{definition}
A subfunctor $\cM \subset \cM_{\mathrm{pug}}$ is \emph{bounded} if there is a pair $(B,\cE)$ where $B$ is a scheme of finite type over $S$ and $\cE\in \cM(B)$ is an object such that for every geometric point $\bar{s}$ over $S$ and $E\in \cM(\kappa(\bar{s}))$, there exists a $\kappa(\bar{s})$-rational point $b$ of $B\times_S \Spec(\kappa(\bar{s}))$ such that $\cE_b\cong E$.
\end{definition}

\begin{definition}
Let $\underline{\sigma}=(\sigma_s)_{s\in S}$ be a flat family of fiberwise weak stability conditions on $\cD$ over $S$ with respect to a relative Mukai homomorphism $\bv\colon \Knum(\cD/S)\to \Lambda$. Fix a vector $v\in \Lambda$ and $\psi\in \mathbb{R}$ such that $\phi(Z(v))=\psi$. We denote by 
\[\cM^{\st}_{\underline{\sigma}}(v)\colon (\mathrm{Sch}/S)^{\mathrm{op}}\to \mathrm{Grpd}\]
the functor whose value on $T\in (\mathrm{Sch}/S)$ consists of all $T$-perfect objects $E\in \Dpug(X_T/T)$ such that for all $t\in T$, we have $E_t\in \cD_t$, $E_t$ is geometrically $\sigma_t$-stable of phase $\psi$, and $\bv([E_t])=v$ in $\Lambda$.

Similarly, we define $\cM_{\underline{\sigma}}(v)$ to be the corresponding moduli functor parameterizing $\sigma_t$-semistable of phase $\psi$ and class $v$.
\end{definition}

\begin{definition}
Let $\underline{\sigma}$ be a flat family of fiberwise weak stability conditions on $\cD$ over $S$. We say $\underline{\sigma}$ \emph{satisfies boundedness (with respect to $\Lambda$)} if $\cM^{\st}_{\underline{\sigma}}(v)$ is bounded for every $v\in \Lambda$.
\end{definition}

\begin{definition}\label{def-stab-family}
Let $\underline{\sigma}$ be a flat family of fiberwise (weak) stability conditions on $\cD$ over $S$ with respect to a relative Mukai homomorphism $\bv\colon \Knum(\cD/S)\to \Lambda$. We say $\underline{\sigma}$ satisfies the \emph{support property with respect to $\Lambda$} if:

\begin{enumerate}[label=(b\arabic*)]
    \item\label{b1} There exists a quadratic form $Q$ on $(\Lambda/\Lambda^Z)_{\mathbb{R}}$ such that $Q|_{(\ker(Z_s)/\Lambda^{Z})_{\RR}}$ is negative definite for every $s\in S$, and for every $\sigma_s$-semistable object $E\in \cD_s$, we have $Q(\overline{\bv}(E))\geq 0$.

    \item\label{b2} $\underline{\sigma}$ satisfies boundedness with respect to $\Lambda$.
\end{enumerate}
In this case, we call $\underline{\sigma}$ a \emph{(weak) stability condition on $\cD$ over $S$ (with respect to $\Lambda$)}.
\end{definition}

\section{Tilt-stability on triangulated categories}\label{sec:general-tilt}

In this section, we describe the construction of a family of weak stability conditions from a given weak stability condition. The proofs in this section are rather technical, and we suggest skipping this section on a first reading since we will only apply the results of Theorem \ref{thm:tilt-stability} and \ref{thm:wall-chamber-abstract} in later sections.

We fix a triangulated category $\cD$ and a group homomorphism $\bv\colon \KK(\cD)\to \Lambda$, where \(\Lambda\) is a finite rank lattice. Let \(\sigma=(\mathcal{A}), Z\) be a weak pre-stability condition on $\cD$ with respect to \(\bv\). Denote by \(\mu(-)\) the slope function associated with $\sigma$. Throughout this section, we denote by $\cH^i(E)\in \cA$ the $i$-th cohomology object with respect to the heart $\cA$.

Recall that $$\mathcal{A}^{Z} =\{E \in \mathcal{A} \colon Z(E)=0\}$$ is the full subcategory of \(\mathcal{A}\) consisting of objects with zero central charge. Also recall that for \(0 \neq E \in \mathcal{A}\), we denote by $\HN^-_{\sigma}(E)$ and $\HN^+_{\sigma}(E)$ the Harder--Narasimhan factors of \(E\) of minimal slope $\mu^-(E)$ and maximal slope $\mu^+(E)$, respectively.

For any \(b \in \mathbb{R}\), we have the following full subcategories of \(\mathcal{A}\):
\begin{align*}
	&\mathcal{T}^{b}=\{ E\in \mathcal{A} \colon \mu^-(E)>b \}\cup\{0\}, \\
	&\mathcal{F}^{b}=\{ E\in \mathcal{A} \colon \mu^+(E) \leqslant b \}\cup\{0\}.
\end{align*} 
As discussed in Section \ref{subsec-tilting-property}, we have the heart of a bounded t-structure on \(\mathcal{D}\) given by 
\[ 
    \mathcal{A}^{b}= \langle\mathcal{F}^{b}[1],\mathcal{T}^{b}\rangle.
\]

In the rest of this section, we fix \(W \in \mathrm{Hom}_{\mathbb{Z}}(\Lambda,\mathbb{R})\) such that

\begin{itemize}
    \item \(W(E) \leqslant 0\) for any \(E \in \mathcal{A}^{Z}\),

    \item \(W(E)<0\) for some \(E \in \mathcal{A}^{Z}\), and

    \item the group homomorphisms \(\Im Z,\Re Z,W \in \mathrm{Hom}_{\mathbb{Z}}(\Lambda,\mathbb{R})\) are \(\mathbb{R}\)-linearly independent.
\end{itemize}

For any \(b,w \in \mathbb{R}\), we define
\begin{equation*}
Z^{b,w}\coloneqq W+w\Im Z+\mathfrak{i}(-\Re Z-b\Im Z) \colon \Lambda \to \mathbb{C}.
\end{equation*}
Let $\Phi_{W} \colon \mathbb{R} \to [-\infty,+\infty]$ be a function defined by
\[\Phi_{W}(x) \coloneqq \limsup_{t\to x} \left\{-\frac{W(E)}{\Im Z(E)} \colon E\text{ is }\sigma\text{-semistable and }\mu(E)=t\right\}.\]
Here, we define $\sup$ of an empty set as $-\infty$. Thus, $\Phi_{W}(x)$ is well-defined and is upper semicontinuous near the points with finite values.

We set
\[\Lambda_W\coloneqq\langle \bv(E) \in \Lambda \colon E \in \mathcal{A}^{Z}, W(E)=0\rangle.\]

The main result of this section is the following.

\begin{theorem}\label{thm:tilt-stability}
Assume that $\sigma=(\cA, Z)$ is a weak stability condition defined over $\QQ[\mathfrak{i}]$\footnote{This means $Z$ factors through $\QQ\oplus \QQ \mathfrak{i}\hookrightarrow \CC$.} that has the tilting property. If
\[(\Lambda_W)_{\RR}=(\ker \Im Z)_{\mathbb{R}} \cap (\ker \Re Z)_{\mathbb{R}} \cap (\ker W)_{\mathbb{R}},\]
then 
\[U_W\coloneqq\{(b, w)\in \RR^2\colon w>\Phi_W(b)\}\to \Stab^{\mathsf{w}}_{\Lambda}(\cD),\quad (b,w)\mapsto \sigma^{b,w}=(\cA^{b}, Z^{b,w})\]
defines a continuous family of weak stability conditions on $\cD$.
\end{theorem}
Moreover, we have a wall-chamber structure in this case, as described in Theorem \ref{thm:wall-chamber-abstract}. 

For simplicity, we denote the associated slope function by $\nu_{b,w}(-)$, and $\sigma^{b,w}$-(semi)stable objects are also called \emph{$\nu_{b,w}$-(semi)stable objects}.

The rest of this section is devoted to proving Theorem \ref{thm:tilt-stability} and Theorem \ref{thm:wall-chamber-abstract}. We start with two easy lemmas.

\begin{lemma}\label{lem:Tb-Az}
If $T\in \cT^b$ with $\Re Z(T)+b\Im Z(T)=0$, then $T\in \cA^Z$.
\end{lemma}

\begin{proof}
Since $T\in \cT^b$, we know that $T\in \cA$ and $\Im Z(T)\geq 0$. However, if $\Im Z(T)>0$, then $b=\mu(T)\geq \mu^-(T)$ and $T \notin \cT^b$, a contradiction. Hence $\Im Z(T)=\Re Z(T)=0$, i.e., $T\in \cA^Z$.
\end{proof}

\begin{lemma}\label{imaginarypart}
    For any non-zero object \(E \in \mathcal{A}^{b}\), we have \(\Re Z(E)+b\Im Z(E) \leqslant 0\). The equality holds if and only if there exists a short exact sequence in \(\mathcal{A}^{b}\):
    \[
        0 \to F[1] \to E \to T \to 0,
    \]where \(T \in \mathcal{A}^{Z}\) and \(F \in \mathcal{A}\) is either a \(\sigma\)-semistable object with slope \(\mu(F)=b\) or \(F=0\). 
\end{lemma}

\begin{proof}
    By the definition of \(\mathcal{A}^{b}\), for any non-zero object \(E \in \mathcal{A}^{b}\), there exists a short exact sequence
    \[
        0 \to F[1] \to E \to T \to 0
    \]
    in \(\mathcal{A}^{b}\), where \(T \in \mathcal{T}^{b}\) and \(F \in \mathcal{F}^{b}\). 
    
    If $F\neq 0$, then $F \in \mathcal{F}^{b}$ implies that
    \[\mu(F) \leqslant \mu(\HN^+_{\sigma}(F)) \leqslant b,\]
    which gives \( \Re Z(F)+b\Im Z(F) \geqslant 0\), and with equality if and only if \(F\) is a \(\sigma\)-semistable object with slope \(\mu(F)=b\). Similarly, if \(T \in \mathcal{T}^{b}\), then we have \(\Re Z(T)+b\Im Z(T) \leqslant 0\), with equality if and only if \(T \in \mathcal{A}^{Z}\) by Lemma \ref{lem:Tb-Az}. This completes the proof.
\end{proof}

\subsection{Large volume limit}

The following results classify semistable objects at the large volume limit.

Fix a weak stability condition $\sigma=(\cA, Z)$ on $\cD$ with respect to $\Lambda$ such that $\sigma$ is defined over $\QQ[\mathfrak{i}]$ and satisfies the tilting property. Recall that for any $b\in \RR$,
\[Z^{b}\coloneqq \Im Z+\mathfrak{i}(-\Re Z-b\Im Z) \colon \Lambda\to \CC\]
is a weak stability function on $\cA^b$. By Lemma \ref{lem-tilt-weak-stab}, we know that $\sigma^b=(\cA^b, Z^b)$ is a weak stability condition on $\cD$ with respect to $\Lambda$.

\begin{lemma}\label{lem:classify-rotate-stable-abstract}
Fix an object $E\in \cA^b$ with $Z(E)\neq 0$. Then $E$ is $\sigma^b$-semistable if and only if either

\begin{enumerate}
    \item $\Im Z(E)\geq 0$, and either

\begin{itemize}
    \item $E$ is a $\sigma$-semistable object with $\Im Z(E)>0$ and $\mu(E)>b$, or 
    
    \item $\Im Z(E)=0$ with no subobjects in $\cA^Z$;
\end{itemize}
     or

    \item $\Im Z(E)<0$ and we have a short exact sequence $$0\to A[1]\to E\to B\to 0$$ in $\cA^b$ such that $A$ is a $\sigma$-semistable object with $\Im Z(A)>0$, $\mu(A)\leq b$, and $B\in \cA^Z$, so that $\Hom_{\cD}(\cA^Z, E)=0$ when $\mu(A)<b$.
\end{enumerate}

Similarly, $E$ is $\sigma^b$-stable if and only if either

\begin{enumerate}[(1)]
    \item $\Im Z(E)\geq 0$, and

    \begin{itemize}
        \item $E$ is a $\sigma$-stable object with $\Im Z(E)>0$ and $\mu(E)>b$; or

        \item $\Im Z(E)=0$ with no subobjects in $\cA^Z$ and no proper quotient $E\twoheadrightarrow F$ in $\cA$ with $\Im Z(F)=0$ and $F\notin \cA^Z$;
        
    \end{itemize}
    or

    \item $E$ satisfies (b), the object $A$ in (b) is $\sigma$-stable with $\Hom(\cA^Z, A[1])=0$, and either $\mu(A)<b$ or $\mu(A)=b$ and $B=0$.
\end{enumerate}
\end{lemma}

\begin{proof}
The statement for semistability follows from \cite[Lemma 2.19]{piyaratne2019moduli} or \cite[Lemma 14.17]{BLMNPS21}. It remains to consider the statement for stability. Let $0\to F\to E\to G\to 0$ be any exact sequence of $\sigma^b$-semistable objects in $\cA^b$. Then they are either in case (a) or (b). 

If $E$ is in case (a), then $\cH^{-1}(F)=\cH^{-1}(E)=0$ and we have an exact sequence
\[0\to \cH^{-1}(G)\to F\to E\to \cH^0(G)\to 0.\]
When $\cH^{-1}(G)\neq 0$, $G$ is in case (b), so $\cH^{-1}(G)$ is a $\sigma$-semistable object with slope $\leq b$ and $\Im Z(\cH^{-1}(G))>0$, and $\cH^0(G)\in \cA^Z$. In particular, $\mu(G)\leq b$. In this case, both $\mu(E)$ and $\mu(F)$ are bigger than $b$, hence we cannot have $\mu_{\sigma^b}(E)=\mu_{\sigma^b}(F)=\mu_{\sigma^b}(G)$. So $\cH^{-1}(G)=0$ and we get an exact sequence $$0\to F\to E\to G\to 0$$ of objects in $\cA$. 

If $E$ is a $\sigma$-semistable object with $\Im Z(E)>0$, then $\mu_{\sigma^b}(E)$ only depends on $\mu(E)$. So from the above discussion, it is clear that $E$ is $\sigma^b$-stable if and only if it is $\sigma$-stable. If $\Im Z(E)=0$ and has no subobjects in $\cA^Z$, then $\mu_{\sigma^b}(E)=0$ does not depend on $E$. Therefore, we see that $E$ is $\sigma^b$-stable if and only if it has no quotient $F$ in $\cA$ with $\Im Z(F)=0$ and $\Re Z(F)\neq 0$.




If $E$ is in case (b), then we have an exact sequence
\[0\to \cH^{-1}(F)\to A\to \cH^{-1}(G)\to \cH^0(F)\to B\to \cH^0(G)\to 0.\]
Assume that $E$ satisfies (2). If $G$ is in case (a), then $\cH^{-1}(G)=0$ and $\cH^0(G)\in \cA^Z$. In this case, the only possibility for $\mu_{\sigma^b}(F)=\mu_{\sigma^b}(E)=\mu_{\sigma^b}(G)=+\infty$ is $\mu(A)=b$, which implies $G=\cH^0(G)=0$ by $B=0$. If $F$ is in case (a), then we have an exact sequence
\[0\to A\to \cH^{-1}(G)\to F\to B\to \cH^0(G)\to 0.\]
Since $\Hom(\cA^Z, A[1])=0$, we see that $F=0$ or $F\notin \cA^Z$. If $F\neq 0$, then $\mu(F)>b\geq \mu(G)$, so $\mu_{\sigma^b}(F)\neq \mu_{\sigma^b}(E)$. Therefore, we have $F=0$. If $F$ and $G$ are both in case (b), then the $\sigma$-stability of $A$ implies $\mu(F)< \mu(G)\leq b$, so $\mu_{\sigma^b}(F)<\mu_{\sigma^b}(G)$. Therefore, in any case, $E$ is $\sigma^b$-stable.

Finally, assume that $E$ is in case (b) and $\sigma^b$-stable. Then $\Hom(\cA^Z, A[1])=0$. Moreover, $A$ is $\sigma$-stable, otherwise there exists a subobject $A'\subset A$ with $\mu(A')=\mu(A)$, so $A'[1]$ is a subobject of $E$ and $\mu_{\sigma^b}(A')=\mu_{\sigma^b}(E)$. If $\mu(A)=b$, then $\mu_{\sigma^b}(A)=+\infty$. So $A[1]$ is a subobject of $E$ with $\mu_{\sigma^b}(A)=\mu_{\sigma^b}(E)=+\infty$. Then the $\sigma^b$-stability of $E$ gives $B=0$ as desired.
\end{proof}

\begin{lemma}\label{lem:lvl}
Let $E\in \cA^{b}$ be a $\nu_{b, w}$-semistable object for all sufficiently large $w \gg 0$. Then $E$ is $\sigma^b$-semistable and we have the following possibilities:

\begin{itemize}
    \item $\cH^{-1}(E)=0$ and $E$ is $\sigma$-semistable with $\Im Z(E)>0$, 

    \item $\cH^{-1}(E)=0$ and $E$ is $\sigma$-semistable with $\Im Z(E)=0$, such that $\Re Z(E)<0$ and $E$ has no subobject in $\cA^{Z}$,

    \item $\cH^{-1}(E)=0$ and $E\in \cA^Z$, or

    \item $\cH^{-1}(E)\neq 0$ is $\sigma$-semistable with $\Im Z(\cH^{-1}(E))>0$, $\mu(\cH^{-1}(E))\leq b$, and $\cH^0(E)\in \cA^Z$.
\end{itemize}

\end{lemma}

\begin{proof}
If $\Im Z^b(E)=0$, then it is clear that $E$ is $\sigma^b$-semistable. If $\Im Z^b(E)\neq 0$, then the $\nu_{b,w}$-semistability of $E$ implies that any non-zero subobject $F$ of $E$ satisfies $\Im Z^b(F)\neq 0$ and
\[\lim_{w\to +\infty}\frac{\nu_{b,w}(F)}{w}=\frac{1}{b-\mu(F)}=\mu_{\sigma^b}(F).\]
So $E$ is also $\sigma^b$-semistable. Now, the remaining statement follows from Lemma \ref{lem:classify-rotate-stable-abstract}.
\end{proof}

\begin{lemma}\label{lem:slope-stable-lvl}
Assume furthermore that $\Phi_W<+\infty$ and $E\in \cA^b$ is a $\sigma^b$-stable object. Then $E$ is $\nu_{b,w}$-stable for any $w\gg 0$.
\end{lemma}

\begin{proof}
If $E\in \cA^Z$, then it is straightforward to see that $E$ is a simple object in $\cA^b$. In this case, it is also $\nu_{b,w}$-stable for any $w$. Therefore, in the following, we may assume that $E\notin \cA^Z$.

We first assume that $\Im Z(E)> 0$. By Lemma \ref{lem:classify-rotate-stable-abstract}, $E$ is a $\sigma$-stable object with $\mu(E)>b$. For any subobject $0\neq F\subsetneq E$ in $\cA^b$, we have an exact sequence
\[0\to \cH^{-1}(E/F)\to F\to E\to \cH^0(E/F)\to 0\]
in $\cA$. It is clear that $\Im Z(F)>0$ and $\mu^-(F)>b$. Moreover, if we denote by $F'$ the image of $F\to E$ in $\cA$, since $E$ is $\sigma$-stable, we get $\mu^+(F')\leq \mu(E)$. Then from $\mu^+(\cH^{-1}(E/F))\leq b<\mu(E)$, we obtain $b<\mu^-(F)\leq \mu^+(F)<\mu(E)$. If $F'=$ As $\Phi_W$ is upper semicontinuous, we know that
\[-\frac{W(F)}{\Im Z(F)}\leq X_E\coloneqq \sup\{\Phi_W(x)\colon x\in [b, \mu(E)]\}<+\infty.\]
We also have 
\[\mu(F)<b+\frac{(\mu(E)-b)\Im Z(F')}{\Im Z(F)}.\]
From $\mu(F)\in (b, \mu(E)]$ and the assumption that $\sigma$ is defined over $\QQ[\mathfrak{i}]$, we know that the set of $\mu(F)$ with $\Im Z(F)\leq D$ for any fixed constant $D>0$ is finite.
Since $\Im Z(F')\leq \Im Z(E)$, we get
\begin{align*}
\mu(F)&<Y_E  \\
&\coloneqq \max\left\{\frac{b+\mu(E)}{2},\max\{\mu(G)\colon \Im Z(G)\leq 2\Im Z(E), G\subsetneq E\in \cA^b\}\right\}\\
&<\mu(E).
\end{align*}
Therefore, for any subobject $0\neq F\subsetneq E$ in $\cA^b$ and any $w> X_E$, we see that
\[\nu_{b,w}(F)<\frac{X_E-w}{Y_E-b}.\]
By $Y_E-b<\mu(E)-b$, we get $\nu_{b,w}(F)<\nu_{b,w}(E)$ for $w\gg 0$. Therefore, $E$ is $\nu_{b,w}$-stable for $w\gg 0$.

Now, we assume that $\Im Z(E)=0$. If $\Im Z(F)=0$ as well, then $\cH^{-1}(E/F)=0$ and we have an exact sequence $0\to F\to E\to E/F\to 0$ in both $\cA$ and $\cA^b$. In this case, we have $\nu_{b,w}(F)<\nu_{b,w}(E)< \nu_{b,w}(E/F)$ by Lemma \ref{lem:classify-rotate-stable-abstract}(1). If $\Im Z(F)>0$, then
\[b<\mu^-(F)\leq \mu^+(F)\leq b+\frac{-\Re Z(E)}{D},\]
where $D\coloneqq \min \{\Im Z(G)\colon G\in \cA, \Im Z(G)>0\}>0$, which exists since $\sigma$ is defined over $\QQ[\mathfrak{i}]$.
As $\Phi_W$ is upper semicontinuous, we know that
\[-\frac{W(F)}{\Im Z(F)}\leq X'_E\coloneqq \sup\left\{\Phi_W(x)\colon x\in \left[b, b+\frac{-\Re Z(E)}{D}\right]\right\}<+\infty.\]
Therefore, there exists a constant $C>0$ so that for $w\geq C$, we obtain
\[\nu_{b,w}(F)\leq \frac{D(X'_E-w)}{-\Re Z(E)}<\nu_{b,w}(E)=\frac{W(E)}{\Re Z(E)}.\]
Hence, $E$ is $\nu_{b,w}$-stable for $w\gg 0$ in this case as well.

Finally, we assume that $\Im Z(E)<0$. Then it satisfies Lemma \ref{lem:classify-rotate-stable-abstract}(2). For any exact sequence $0\to K\to E\to F \to 0$ in $\cA^b$ with $F\notin \cA^Z$, we get a long exact sequence
\[0\to \cH^{-1}(K)\to A\to \cH^{-1}(F)\to \cH^0(K)\to B\to \cH^{0}(F)\to 0.\]
Since $B\in \cA^Z$, we also have $\cH^0(F)\in \cA^Z$. Therefore, $\cH^{-1}(F)\neq 0$. If we set $F_1\coloneqq A/\cH^{-1}(K)$ and $F_2\coloneqq \im (\cH^{-1}(F)\to \cH^0(K))$, then we have an exact sequence $$0\to F_1\to \cH^{-1}(F)\to F_2\to 0.$$ By the $\sigma$-stability of $A$, we have $\mu^-(F_1)\geq \mu(A)=\mu(E)$. From $B\in \cA^Z$, we also get $\mu^+(F_2)=\mu^+(\cH^0(K))>b$. Thus, from $W(\cH^0(F))\leq 0$, we obtain
\[-\frac{W(F)}{\Im Z(F)}\leq-\frac{W(\cH^{-1}(F))}{\Im Z(\cH^{-1}(F))}\leq X''_E\coloneqq \sup\{\Phi_W(x)\colon x\in [\mu(E), b]\}<+\infty\]
as above. Moreover, we have
\[\mu(F)=\mu(\cH^{-1}(F))>b-\frac{\Im Z(F_1)(b-\mu(E))}{\Im Z(\cH^{-1}(F))}>\mu(E)\]
and $0<\Im Z(F_1)\leq \Im Z(A)$. So
\begin{align*}
\mu(F)>& Y''_E\\ 
\coloneqq &\min\left\{\frac{b+\mu(E)}{2},\min\{\mu(G)\colon -\Im Z(G)\leq -2\Im Z(E),E\twoheadrightarrow G\in \cA^b \text{ with non-zero kernel}\}\right\}\\
>&\mu(E).
\end{align*}
Now, we can conclude that
\[\nu_{b,w}(F)>\nu_{b,w}(E)\]
for any $w\gg 0$, hence $E$ is $\nu_{b,w}$-stable.
\end{proof}

\subsection{Noetherian property}

Our goal is to define some weak stability conditions on \(\mathcal{D}\) with hearts \(\mathcal{A}^{b}\). To achieve this, we need to define some weak stability functions \(Z'\) on \(\mathcal{A}^{b}\) and then verify that they satisfy the Harder--Narasimhan property and support property. Among these, verifying the Harder--Narasimhan property is not straightforward. However, if \(\mathcal{A}^{b}\) is Noetherian and the image of the imaginary part of \(Z'\) is discrete in \(\mathbb{R}\), then \(Z'\) has the Harder--Narasimhan property. In the next subsection, we will discuss the properties of \(Z'\), while in this subsection, we prove that under the following conditions, \(\mathcal{A}^{b}\) is Noetherian, following \cite{piyaratne2019moduli,bayer2011space}.

For any \(b \in \mathbb{R}\), we set
\[
    \mathcal{I}^{b}\coloneqq \{E \in \mathcal{A}^{b} \colon \Re Z(E)+b\Im Z(E)=0\}.
\]

\begin{lemma}\label{Noetherianlem}
    Let \(b \in \mathbb{R}\). Assume that
        \begin{enumerate}
            \item the abelian category \(\mathcal{A}\) is Noetherian, and
            \item there is no infinite sequence in \(\mathcal{A}\):
	           \[
		            G_{1} \subsetneq \cdots \subsetneq G_{n} \subsetneq \cdots
	           \] where \(G_{n} \in \mathcal{F}^{b}\) and \(G_{n+1}/G_{n} \in \mathcal{A}^{Z}\) for any \(n \in \mathbb{N}^{+}\). 
        \end{enumerate}
    Then \(\mathcal{I}^{b}\) is a Noetherian torsion subcategory of \(\mathcal{A}^{b}\).
\end{lemma}

\begin{proof}
    By Lemma \ref{imaginarypart}, \(\mathcal{I}^{b}\) is a Serre subcategory of \(\mathcal{A}^{b}\). Assume that \(\mathcal{I}^{b}\) is not a Noetherian torsion subcategory. Then by Remark \ref{rmk:serre-torsion}, there is an infinite sequence in \(\mathcal{A}^{b}\):
    \[
        F_{1} \subsetneq \cdots \subsetneq F_{n} \subsetneq \cdots \subsetneq E,
    \]where \(\Re Z(F_{n})+b\Im Z(F_{n})=0\) for any \(n \geq 1\).
	
    Applying the cohomology functor with respect to the heart \(\mathcal{A}\), we obtain an infinite sequence in \(\mathcal{A}\):
    \[
        \mathcal{H}^{-1}(F_{1}) \subseteq \cdots \subseteq \mathcal{H}^{-1}(F_{n}) \subseteq \cdots \subseteq \mathcal{H}^{-1}(E).
	\] 
    Since \(\mathcal{A}\) is Noetherian, this sequence stabilizes. Therefore, we may assume that
    \[\mathcal{H}^{-1}(F_{n}) = \mathcal{H}^{-1}(F_{n+1}) \subseteq \mathcal{H}^{-1}(E)\]
    for any \(n \in \mathbb{N}^{+}\). Thus, by taking the cohomology long exact sequence associated to the triangle 
    \[F_{n}\to F_{n+1}\to F_{n+1}/F_n,\]
    we get an exact sequence in \(\mathcal{A}\) for any \(n \geq 1\):
    \begin{equation}\label{eq:long-cA}
        0 \to \mathcal{H}^{-1}(F_{n+1}/F_{n}) \to  \mathcal{H}^{0}(F_{n}) \to \mathcal{H}^{0}(F_{n+1}) \to \mathcal{H}^{0}(F_{n+1}/F_{n}) \to 0.
    \end{equation}
    By Lemma \ref{imaginarypart} and \(\Re Z(F_{n})+b\Im Z(F_{n})=0\), we see that \(\mathcal{H}^{-1}(F_{n})\) is either a \(\sigma\)-semistable object with slope \(b\) or \(\mathcal{H}^{-1}(F_{n})=0\), and \(\mathcal{H}^{0}(F_{n}) \in \mathcal{A}^{Z}\). Combining this with \eqref{eq:long-cA}, we know that the object \(\mathcal{H}^{-1}(F_{n+1}/F_{n})\) also belongs to \(\mathcal{A}^{Z}\). Therefore, we must have \(\mathcal{H}^{-1}(F_{n+1}/F_{n})=0\) for any \(n \geq 1\) as \(\mathcal{H}^{-1}(F_{n+1}/F_{n}) \in \mathcal{F}^{b}\). This together with Lemma \ref{imaginarypart} implies $F_{n+1}/F_{n}\in \cA^Z$.

    For any \(n \geq 1\), let \(E_{n}\) be the cokernel of \(F_{n} \to E\) in \(\mathcal{A}^{b}\). Then we have an infinite sequence of epimorphisms in \(\mathcal{A}^{b}\):
    \[
        E \twoheadrightarrow E_{1} \twoheadrightarrow \cdots \twoheadrightarrow E_{n}  \twoheadrightarrow \cdots
    \] with \(\ker (E_{n} \twoheadrightarrow E_{n+1}) \cong F_{n+1}/F_{n}\). Similarly, after taking the cohomology functor with respect to the heart \(\mathcal{A}\), we obtain an infinite sequence in \(\mathcal{A}\):
    \[
        \mathcal{H}^{0}(E) \twoheadrightarrow \mathcal{H}^{0}(E_{1}) \twoheadrightarrow  \cdots \twoheadrightarrow \mathcal{H}^{0}(E_{n}) \twoheadrightarrow \cdots.
    \]Because \(\mathcal{A}\) is Noetherian, the above sequence stabilizes. Thus, after discarding finitely many terms, we may also assume that \(\mathcal{H}^{0}(E_{n}) = \mathcal{H}^{0}(E_{n+1})\). Combining this with \(\mathcal{H}^{-1}(F_{n+1}/F_{n})=0\), we get an exact sequence
    \[0\to \mathcal{H}^{-1}(E_{n}) \to \mathcal{H}^{-1}(E_{n+1})\to F_{n+1}/F_n\to 0\]
in $\cA$ for each \(n \geq 1\).
    
    Now, we set \(G_{n}\coloneqq\mathcal{H}^{-1}(E_{n}) \in \mathcal{F}^{b}\). We then get an infinite sequence in \(\mathcal{A}\):
	\[
		G_{1} \subseteq \cdots \subseteq G_{n} \subseteq \cdots
	\] where \(G_{n} \in \mathcal{F}^{b}\) and \(G_{n+1}/G_{n} \cong F_{n+1}/F_{n} \in \mathcal{A}^{Z} \setminus \{0\}\) for any \(n \geq 1\). This contradicts our second assumption on \(\mathcal{A}\). Thus, \(\mathcal{I}^{b}\) is a Noetherian torsion subcategory of \(\mathcal{A}^{b}\).
\end{proof}

\begin{corollary}\label{cor:Noetherian}
    Let \(b \in \mathbb{R}\) and we assume that
    \begin{enumerate}
        \item the set \(\{\Re Z(E)+b\Im Z(E) \in \mathbb{R} \colon E \in \mathcal{A}^{b}\}\) is discrete in \(\mathbb{R}\),
        \item the abelian category \(\mathcal{A}\) is Noetherian, and
        \item there is no infinite sequence in \(\mathcal{A}\):
	\[
		            G_{1} \subsetneq \cdots \subsetneq G_{n} \subsetneq \cdots
	           \] where \(G_{n} \in \mathcal{F}^{b}\) and \(G_{n+1}/G_{n} \in \mathcal{A}^{Z}\) for any \(n \in \mathbb{N}^{+}\). 
    \end{enumerate}
    Then the abelian category \(\mathcal{A}^{b}\) is Noetherian and \(\mathcal{I}^{b}\) is a Noetherian torsion subcategory of \(\mathcal{A}^{b}\).
\end{corollary}

\begin{proof}
    Suppose that there is an infinite sequence of epimorphisms in \(\mathcal{A}^{b}\):
    \[
        E_{0} \twoheadrightarrow E_{1} \twoheadrightarrow \cdots \twoheadrightarrow E_{n}  \twoheadrightarrow \cdots
    \] By Lemma \ref{imaginarypart}, we have
	\begin{align*}
		-\Re Z(E_{0})-b\Im Z(E_{0}) \geqslant -\Re Z(E_{1})-b\Im Z(E_{1}) \geqslant \cdots \geqslant 0.
	\end{align*}Since \(\{-\Re Z(E)-b\Im Z(E) \in \mathbb{R} \mid E \in \mathcal{A}^{b}\}\) is discrete in \(\mathbb{R}\), we may assume that \(-\Re Z(E_{n})-b\Im Z(E_{n})\) is constant after reindexing. 
    
    For any \(n \in \mathbb{N}^{+}\), let \(F_{n}\) be the kernel of \(E_{0} \to E_{n}\). Therefore, we have an infinite sequence in \(\mathcal{A}^{b}\):
    \[
        F_{1} \subseteq \cdots \subseteq F_{n} \subseteq \cdots \subseteq E_{0}.
    \]
    It is clear that \(F_{n} \in \mathcal{I}^{b}\) and \(F_{n+1}/F_{n} \cong \ker(E_{n} \twoheadrightarrow E_{n+1})\).
    Therefore, by Lemma \ref{Noetherianlem}, we know that the above infinite sequence stabilizes. Thus, \(\mathcal{A}^{b}\) is Noetherian. 
\end{proof}

If $b\in \QQ$ and $\sigma$ satisfies the tilting property, we know that $\cA$ is Noetherian by \cite[Lemma 14.8]{BLMNPS21}. Combining this with Corollary \ref{cor:Noetherian} and Lemma \ref{lem:tilt-property-no-infty-seq}, we get

\begin{corollary}\label{cor:tilt-noetherian}
Assume that \(b \in \mathbb{Q}\), and $\sigma$ is a weak stability condition defined over $\QQ[\mathfrak{i}]$ and satisfies the tilting property. Then $\cA^{b}$ is Noetherian and \(\mathcal{I}^{b}\) is a Noetherian torsion subcategory of \(\mathcal{A}^{b}\).
\end{corollary}

\subsection{Central charge}

In the following, we aim to determine for which pairs \((b,w) \in \mathbb{R}^{2}\) the homomorphism \(Z^{b,w}\) is a weak stability function on the heart \(\mathcal{A}^{b}\).

\begin{lemma}\label{lem:weakstabilityfunction}
    The homomorphism \(Z^{b,w}\) is a weak stability function on \(\mathcal{A}^{b}\) if \(w \geqslant \Phi_W(b)\). Moreover, when \(w>\Phi_W(b)\), we have
	\[
            (\cA^{b})^{Z^{b, w}}=\{E \in  \mathcal{A}^{Z} \colon W(E)=0\},
	\]
so $\Lambda^{Z^{b,w}}=\Lambda_W$.
\end{lemma}

\begin{proof}
    By Lemma \ref{imaginarypart}, we know that \(\Im Z^{b,w}(E)=-(\Re Z(E)+b\Im Z(E)) \geqslant 0\) for any \(E \in \mathcal{A}^{b}\). Moreover, if \(\Re Z(E)+b\Im Z(E)=0\), there exists a short exact sequence in \(\mathcal{A}^{b}\):
	\[
		0 \to F[1] \to E \to T \to 0,
	\]where \(F,T \in \mathcal{A}\) with \(Z(T)=0\) and \(F\) is a \(\sigma\)-semistable object with slope \(b\) or \(F=0\). Therefore, \(Z^{b,w}\) is a weak stability function on \(\mathcal{A}^{b}\) if and only if \(\Re Z^{b,w}(F) \geqslant 0\) for all \(\sigma\)-semistable objects with slope \(b\). This is equivalent to
    \[
        w \geqslant -\frac{W(F)}{\Im Z(F)} 
    \] for any \(\sigma\)-semistable object \(F\) with slope \(b\). Moreover, if \(w>\Phi_W(b)\) and \(Z^{b,w}(E)=0\), then \(F=0\) and \(E \cong T \in \mathcal{A}^{Z}\) with \(W(E)=0\).
\end{proof}

\subsection{Support property}

Now, we discuss the support property of $\sigma^{b,w}$. Unlike the known approach \cite[Section 12]{bayer2016space}, which first proves the support property for $b\in \mathbb{Q}$ and then uses a deformation argument to extend it to all $b\in \mathbb{R}$, we use an argument that proves the support property for all $\sigma^{b,w}$ simultaneously. This is motivated by \cite[Section 4]{bayer2011space}.

Keep the settings in Theorem \ref{thm:tilt-stability}. Recall that
\[U_W= \{(b,w)\in \RR^2\colon w>\Phi_W(b)\}\]
is an open subset of $\RR^2$. For any $(b,w)\in \RR^2$, we define a ray
\[R^{b,w}\coloneqq \{(b,w+t)\in \RR^2\colon t\geq 0\}.\]
We define a map
\[\Pi_W\colon \KK(\cD)\setminus \{v\colon \Im Z(v)=0\}\to \RR^2\]
\[v\mapsto \left( -\frac{\Re Z(v)}{\Im Z(v)}, -\frac{W(v)}{\Im Z(v)} \right)\]
and set
\[S_W\coloneqq \left\{\Pi_W(E)\in \RR^2 \colon E \text{ is }\sigma\text{-semistable with }\Im Z(E)\neq 0 \right\}.\]

\begin{lemma}\label{lem:support-property}
Assume that
\[(\Lambda_W)_{\RR}=(\ker \Im Z)_{\mathbb{R}} \cap (\ker \Re Z)_{\mathbb{R}} \cap (\ker W)_{\mathbb{R}}.\]
Then the map
	\[
	(\Lambda/\Lambda_{W})_{\mathbb{R}} \cong \mathbb{R}^{3}, \quad v \to (\Im Z(v),\Re Z(v),W(v))
	\] 
is an isomorphism of \(\mathbb{R}\)-vector spaces. Moreover, there exists a continuous function $C_{b,w}\colon U_W\to \RR_{>0}$ such that for any $\nu_{b,w}$-semistable object $E\in \cA^b$, we have
\[\|\overline{\bv}(E)\|_{\max}\leq C_{b,w} |Z^{b,w}(E)|\]
for any $(b,w)\in U_W$, where the norm $\|-\|_{\max}$ on $(\Lambda/\Lambda_{W})_{\mathbb{R}}$ is defined by
\[\|v\|_{\max}\coloneqq \max\{|\Im Z(v)|, |\Re Z(v)|, |W(v)|\}.\]
\end{lemma}

\begin{proof}
The isomorphism $(\Lambda/\Lambda_{W})_{\mathbb{R}} \cong \mathbb{R}^{3}$ is clear from the assumption.

Let \(M_{b,w}\) be the matrix
	\[
	   \begin{pmatrix}
		1 & b & w\\
		0 & 1 & 0\\
		0 & 0 & 1\\
	   \end{pmatrix}
        \]
    and \(N_{b,w}\) be the operator norm of $M_{b,w}^{-1}\colon \RR^3\to \RR^3$, where $\RR^3$ is equipped with the norm $\|-\|_{\max}$. Then for any \(v=(v_{0},v_{1},v_{2}) \in \mathbb{R}^{3}\), we have:
	\[
		(v_{0}, v_{1}, v_{2})=(v_{0}, v_{1}+b v_{0}, v_{2}+w v_{0})M_{b,w}^{-1}.
	\]
	Then
	\begin{align*}
		\|\overline{\bv}(E)\|_{\max}&=\max\{|\Im Z(E)|, |\Re Z(E)|, |W(E)|\} \\
        & \leq N_{b,w}\max\{|\Im Z(E)|, |\Re Z(E)+b \Im Z(E)|, |W(E)+w \Im Z(E)|\} \\
		& \leqslant N_{b,w}\max\{|\Im Z(E)|,|Z^{b,w}(E)|\}
	\end{align*} for any \(E \in \mathcal{D}\). If $\Im Z(E)=0$, then it is clear that
    \[\|\overline{\bv}(E)\|_{\max}\leq N_{b,w}|Z^{b,w}(E)|\leq \frac{N_{b,w}}{\min\{1, d(R^{b,w}, S_W)\}}|Z^{b,w}(E)|.\]

We claim that for any $(b,w)\in U_W$ and $\nu_{b,w}$-semistable object $E\in \cA^b$ with $\Im Z(E)\neq 0$, we have
\[\left|\frac{Z^{b,w}(E)}{\Im Z(E)}\right|\geq d(R^{b,w}, S_W)>0,\]
where $d(-,-)$ denotes the Euclidean distance. Indeed, by definition, we have
\[\left|\frac{Z^{b,w}(E)}{\Im Z(E)}\right|=d((b,w), \Pi_W(E)).\]
So it suffices to prove 
\[d((b,w), \Pi_W(E))\geq d(R^{b,w}, S_W)>0.\]
Note that
\[d(R^{b,w}, S_W)\geq d(R^{b,w}, \{y\leq \Phi_W(x)\})\]
which is strictly positive since $\Phi_W$ is upper semicontinuous and $w>\Phi_W(b)$.

To prove $d((b,w), \Pi_W(E))\geq d(R^{b,w}, S_W)$, we consider $T \coloneqq \cH^0(E)$ and $F\coloneqq \cH^{-1}(E)$. Assume that $\nu_{b,w}(E)<+\infty$. If $\Im Z(E)>0$, then $\mu(E)>b$ and $T\neq 0$. Since $T\in \cA^b$ and $\mu^{-}(T)>b$ by definition, we know that $\HN_{\sigma}^-(T)\in \cA^b$ is a quotient of $E$ in $\cA^b$. In particular, $\nu_{b,w}(E)\leq \nu_{b,w}(\HN_{\sigma}^-(T))$ and $\mu(E)\geq \mu(\HN_{\sigma}^-(T))>b$. Therefore, we have
\[d((b,w), \Pi_W(E))\geq d(R^{b,w}, \Pi_W(\HN_{\sigma}^-(T))) \geq d(R^{b,w}, S_W) \]
where the first inequality follows from $\nu_{b,w}(E)\leq \nu_{b,w}(\HN_{\sigma}^-(T))$ and $\mu(E)\geq\mu(\HN_{\sigma}^-(T))>b$, and the last one follows from the definition of $S_W$. If $\Im Z(E)<0$ or $\nu_{b,w}(E)=+\infty$, then by Lemma \ref{imaginarypart}, $\mu(E)\leq b$ and $F\neq 0$. A similar argument applies to $F$, giving the result.

By the above claim, we get 
\[\|\overline{\bv}(E)\|_{\max}\leq C_{b,w}|Z^{b,w}(E)|\coloneqq \frac{N_{b,w}}{\min\{1, d(R^{b,w}, S_W)\}}|Z^{b,w}(E)|\leq \frac{1+|b|+|w|}{\min\{1, d(R^{b,w}, S_W)\}}|Z^{b,w}(E)|\]
for any $(b,w)\in U_W$ and $\nu_{b,w}$-semistable object $E\in \cA^b$. Since \(N_{b,w}\) and \(d(R^{b,w},S_W)\) are continuous functions of $(b,w)$, we reach our conclusion.
\end{proof}

\subsection{Skewed weak stability functions}

In this subsection, we review necessary definitions and results from \cite{bridgeland:stability}. We will freely use the notions and properties of \emph{quasi-abelian categories}, \emph{strict subobjects, strict quotient objects}, and \emph{strict morphisms}, see \cite[Section 4]{bridgeland:stability} for an overview. A typical example of a quasi-abelian category is $\cP(I)$ for any slicing $\cP$ on $\cD$ and any interval $I\subset \RR$ of length $<1$.

\begin{definition}
Let \(\mathcal B\subset\mathcal D\) be a quasi-abelian subcategory and fix
\(\alpha\in\mathbb R\). A \emph{skewed weak stability function} on \(\mathcal B\)
with phase interval \((\alpha,\alpha+1]\) is a group homomorphism
\[
P\colon \KK(\mathcal B)\to\mathbb C
\]
such that for any $0\neq E\in \cB$, we have
\[P(E)\in \RR_{>0}\cdot\exp(\mathfrak{i}\pi\psi), \quad\alpha<\psi\leq \alpha+1\]
if $\psi \notin \ZZ$ and
\[P(E)\in \RR_{\geq 0}\cdot\exp(\mathfrak{i}\pi\psi), \quad\alpha<\psi\leq \alpha+1\]
if $\psi\in \ZZ$. The number $\psi=\psi(E)$ is called the \emph{phase} of $E$. 

We say $E$ is \emph{$P$-semistable} if $\psi(F)\leq \psi(E)\leq \psi(E/F)$ for any non-zero strict proper subobject $F\hookrightarrow E$. An \emph{HN filtration of $0\neq E\in \cB$ with respect to $P$} is a filtration of strict subobjects
\[0=E_0\subset E_1\subset E_2\subset \cdots \subset E_m=E\]
so that $E_{i}/E_{i-1}$ is $P$-semistable with $\psi(E_1/E_0)>\psi(E_2/E_1)>\cdots >\psi(E_m/E_{m-1})$.
\end{definition}

We will apply the theory of skewed weak stability functions in the following setting. Fix a weak stability condition $\sigma=(\cP, Z)$ on $\cD$. Fix a number $0<\epsilon_0<\frac{1}{8}$ and a group homomorphism $P\colon \KK(\cD)\to \CC$ such that
\[|P(E)-Z(E)|\leq \sin(\pi\epsilon_0)|Z(E)|\]
for any $\sigma$-semistable object $E\notin \cA^Z$ and $Z(F)=0$ implies $P(F)=0$ for any $F\in \cA$.

\begin{lemma}[{\cite[Lemma 7.3]{bridgeland:stability}}]
Assume that $E\in \cP(a,b)$ is $P$-semistable with $0<b-a<1-2\epsilon_0$, then $E\in \cP(\psi(E)-\epsilon_0,\psi(E)+\epsilon_0)$.
\end{lemma}

\begin{definition}
Suppose $0<b-a<1-2\epsilon_0$. A non-zero object $E\in \cP(a,b)$ is said to be \emph{enveloped by $\cP(a,b)$} if $a+\epsilon_0\leq \psi(E)\leq b-\epsilon_0$.

A \emph{thin subcategory} of $\cD$ is a full subcategory of the form $\cP(a,b)$ where $0<b-a<1-2\epsilon_0$.
\end{definition}


The following property is important in our later proof.

\begin{lemma}[{\cite[Lemma 7.5]{bridgeland:stability}}]\label{lem:envelop}
Suppose an object $0\neq E\in \cD$ is enveloped by both thin categories $\cP(a,b)$ and $\cP(a',b')$. Then $E$ is $P$-semistable in $\cP(a,b)$ if and only if $E$ is $P$-semistable in $\cP(a',b')$.
\end{lemma}

\begin{definition}
A \emph{maximal destabilizing quotient (mdq)} of an object $0\neq E\in \cP(a,b)$ is a non-zero strict quotient $E\twoheadrightarrow B$ in $\cP(a,b)$ such that any non-zero strict quotient $E\twoheadrightarrow B'$ satisfies $\psi(B')\geq \psi(B)$, with equality only if $E\twoheadrightarrow B'$ factors via $E\twoheadrightarrow B$.
\end{definition}

\subsection{Proof of Theorem \ref{thm:tilt-stability}}

Now, we prove the main theorem of this section, which is divided into several steps and lemmas. We will prove the result by replacing the assumption
\[(\Lambda_W)_{\RR}=(\ker \Im Z)_{\mathbb{R}} \cap (\ker \Re Z)_{\mathbb{R}} \cap (\ker W)_{\mathbb{R}}\]
in Theorem \ref{thm:tilt-stability} by a more general assumption:

\begin{itemize}
    \item Assume that for any \((b,w) \in U_W\), the pair (\(Z^{b,w},\mathcal{A}^{b})\) satisfies the support property such that there exists a continuous function \(C \colon U_W \to \mathbb{R}_{>0}\) with the property that for any \((b,w) \in U_W\) and any \(\nu_{b,w}\)-semistable object \(E \in \mathcal{A}^{b}\) with \(Z^{b,w}(E) \neq 0\), we have
    \[
        \left\|\overline{\bv}(E)\right\| \leqslant C_{b,w}|Z^{b,w}(E)|.
    \]Here \(\left\|\cdot\right\|\) is a fixed norm on the vector space \((\Lambda/\Lambda_{W})_{\mathbb{R}}\).
\end{itemize}
Under the hypotheses of Theorem \ref{thm:tilt-stability}, this condition is ensured by Lemma \ref{lem:support-property}.

By Lemma \ref{lem:weakstabilityfunction}, $Z^{b, w}$ is a weak stability function on $\cA^{b}$.

If $b\in \QQ$, by Corollary \ref{cor:tilt-noetherian}, we know that $\cA^{b}$ is Noetherian. Then \cite[Proposition 4.10]{macri:lecture-bridgeland-stability} shows that $Z^{b, w}$ satisfies the HN property. Moreover, the pair $\sigma^{b,w}$ satisfies the support property by Lemma \ref{lem:support-property}. This means $\sigma^{b, w}$ is a weak stability condition when $b\in \QQ$. 

To treat the case $b\in \RR$ and glue them into a continuous family, we use a similar argument as in \cite[Section 12]{bayer2016space} and \cite[Section 7]{bridgeland:stability}. 

Fix $(b,w)\in U_W$ with $b\in \QQ$. Denote by $\cP_{b,w}$ the associated slicing of $\sigma^{b,w}$ and by $\phi_{b,w}(-)$ the phase function. We take $0<\epsilon<\frac{1}{20}$ so that $\epsilon\leq d(R^{b,w}, S_W)$ and we have

\begin{itemize}

    \item for any $(b',w')\in U_W$ in the open ball $B_{\epsilon}((b,w))$ and any \(\nu_{b',w'}\)-semistable object \(E \in \cA^{b'}\) with \(Z^{b',w'}(E) \neq 0\), we have \(\|\overline{\bv}(E)\| \leqslant C|Z^{b',w'}(E)|\), where $C>0$ is the maximum of $C_{b,w}$ on the closed ball $\overline{B}_{\epsilon}((b,w))$.
\end{itemize}
Note that the existence of $C$ comes from the continuity of $C_{b,w}$ by our assumption. In the rest of the proof, we fix a point $(b',w')\in U_W\cap B_{\epsilon}((b,w))$ so that

\begin{itemize}
    \item $|Z^{b',w'}-Z^{b,w}|<\frac{1}{C}\sin{(\pi\epsilon)}$.
\end{itemize}
In particular, the set of such $(b',w')$ forms an open neighborhood of $(b,w)$. We further assume that $b'\notin \QQ$.

Note that we always equip $\Hom_{\ZZ}(\Lambda, \CC)$ with the operator norm. So the assumption above and Lemma \ref{lem:weakstabilityfunction} imply that for any \(\nu_{b,w}\)-semistable object \(E \in \cA^{b}\), we have
\begin{equation}\label{eq:diff-phase-epsilon}
    |Z^{b',w'}(E)-Z^{b,w}(E)|\leq |Z^{b',w'}-Z^{b,w}|\cdot \|\overline{\bv}(E)\|\leq \sin{(\pi\epsilon)}|Z^{b,w}(E)|.
\end{equation}

In the following, we may assume that $b'<b$, since the argument for $b'>b$ is similar. Therefore, we have
\begin{equation}\label{eq:b-b'}
b-d(R^{b,w}, S_W)<b'<b.
\end{equation}

We will use the following result without mentioning it.

\begin{lemma}[{\cite[Lemma 1.1.2]{polishchuk2007constant}}]
The abelian category $\cA^{b}$ has a torsion pair $(\cA^{b'}[1]\cap \cA^b, \cA^{b'}\cap \cA^b)$. The abelian category $\cA^{b'}$ has a torsion pair $(\cA^{b}\cap \cA^{b'}, \cA^{b}[-1]\cap \cA^{b'})$.
\end{lemma}

The following lemma is useful.

\begin{lemma}\label{lem:deformation-1}
Let $0\neq E\in \cD$.

\begin{enumerate}
    \item If $E\in \cP_{b,w}(0,\frac{1}{2}]$, then $E\in \cA^{b'}$ and $\Im Z^{b',w'}(E)>0$.

    \item If $E\in \cA^{b'}[1]\cap \cA^b$, then $\phi^-_{b,w}(E)>1-\epsilon$ and $\Im Z^{b',w'}(E)<0$.

\end{enumerate}

\end{lemma}

\begin{proof}
To prove (a), it suffices to assume that $E$ is $\nu_{b,w}$-semistable with $\nu_{b,w}(E)\leq 0$. In this proof, we set $F\coloneqq \cH^{-1}(E)$. Then $\HN^+_{\sigma}(F)[1]$ is a subobject of $E$ in $\cA^b$. In particular, by the $\nu_{b,w}$-semistability, we have $\nu_{b,w}(\HN^+_{\sigma}(F))\leq \nu_{b,w}(E)$ and $\mu(\HN^+_{\sigma}(F))<b$. Therefore, we get
\[-\frac{W(\HN^+_{\sigma}(F))}{\Im Z(\HN^+_{\sigma}(F))}\geq \nu_{b,w}(E)(\mu(\HN^+_{\sigma}(F))-b)+w\geq w,\]
which implies $d(R^{b,w}, \Pi_W(\HN^+_{\sigma}(F)))=b-\mu(\HN^+_{\sigma}(F))$. In particular, we have
\[d(R^{b,w}, S_W)\leq d(R^{b,w}, \Pi_W(\HN^+_{\sigma}(F)))=b-\mu(\HN^+_{\sigma}(F))\]
and hence
\[b-b'< b-\mu(\HN^+_{\sigma}(F))\]
by \eqref{eq:b-b'} and gives $F[1]\in \cA^{b'}$. Since $b'<b$, we also have $\cH^0(E)\in \cA^{b'}$, which together implies $E\in \cA^{b'}$.

If $\Im Z^{b',w'}(E)=0$, then by Lemma \ref{imaginarypart}, $Z(\cH^0(E))=0$ and either $F=0$ or $F$ is $\sigma$-semistable with $\mu(F)=b'$, which is impossible since $\mu(\HN^+_{\sigma}(F))<b'$ as we proved above. This completes the proof of (a).

To prove part (b), note that $\cA^b$ has a torsion pair $(\cA^{b'}[1] \cap \cA^b, \cA^{b'} \cap \cA^b)$, so $\cA^{b'}[1] \cap \cA^b$ is closed under the quotient of objects in $\cA^b$. In particular, $\HN^-_{\sigma^{b,w}}(E)\in \cA^{b'}[1] \cap \cA^b$. Note that $Z^{b',w'}(\HN^-_{\sigma^{b,w}}(E))\neq 0$, otherwise $\HN^-_{\sigma^{b,w}}(E)\in \cA^Z$ by Lemma \ref{lem:weakstabilityfunction} and contradicts $\HN^-_{\sigma^{b,w}}(E)\in \cA^{b'}[1]$. Therefore, by \eqref{eq:diff-phase-epsilon} and $\HN^-_{\sigma^{b,w}}(E)\in \cA^{b'}[1] \cap \cA^b$, we get
\[\phi^-_{b,w}(E)=\phi_{b,w}(\HN^-_{\sigma^{b,w}}(E))>1-\epsilon\]
and the result follows.
\end{proof}

By \eqref{eq:diff-phase-epsilon}, $Z^{b',w'}$ is a skewed weak stability function on $\cP_{b,w}(I)$ for any interval $I\subset \RR$ of length $< 1-2\epsilon$. The phase of $E\in \cP_{b,w}(I)$ under $Z^{b',w'}$ is denoted by $\psi_{b',w',I}(E)$. Note that if $J$ is another interval of length $< 1-2\epsilon$ containing $I$, then $$\psi_{b',w',I}(E)=\psi_{b',w',J}(E)\in I+(-\epsilon, \epsilon)$$ for $E\in \cP_{b,w}(I)$. We will drop $I$ in the subscript of $\psi_{b',w',I}$ if it is clear.

We need a crucial technical result.

\begin{lemma}\label{lem:deformation-2}
Suppose that $x,y$ are real numbers with $-\frac{1}{2}<x\leq -\epsilon$, $0<y\leq \frac{1}{2}$, and $2\epsilon<y-x<1-2\epsilon$. Then for any $0\neq E\in \cP_{b,w}(x,y)$, there is a strict subobject $E_1\hookrightarrow E$ in $\cP_{b,w}(x,y)$ such that $E_1\in \cA^{b'}$ and $E/E_1\in \cP_{b,w}(x,0]\cap \cA^{b'}[-1]$. Moreover, if $E_1\neq 0$, then $\Im Z^{b',w'}(E_1)>0$.

Moreover, if $x', y'$ is another pair of real numbers satisfying the same numerical condition and $E\in \cP_{b,w}(x',y')$, then the resulting strict subobject remains the same in \(\mathcal{P}_{b,w}(x',y')\).
\end{lemma}

\begin{proof}
Since $\cP_{b,w}$ is a slicing on $\cD$, we have a strict exact sequence
\[0\to A\to E\to B\to 0\]
in $\cP_{b,w}(x,y)$ such that $A\in \cP_{b,w}(0,y)$ and $B\in \cP_{b,w}(x,0]$. By Lemma \ref{lem:deformation-1}(a), if $A\neq 0$, then $A\in \cA^{b'}$ and $\Im Z^{b',w'}(A)>0$. Therefore, if the statement holds for $B$, then by pullback to $E$, the statement also holds for $E$. So we may assume that $E=B\in \cP_{b,w}(x,0]\subset \cA^b[-1]$.

Note that $\cA^b[-1]$ has the torsion pair $(\cA^{b'} \cap \cA^b[-1], \cA^{b'}[-1] \cap \cA^b[-1])$. Therefore, we have an exact sequence
\[0\to E_1\to E\to E_2\to 0 \]
in $\cA^b[-1]$ with $E_1\in \cA^{b'} \cap \cA^b[-1]$ and $E_2\in \cA^{b'}[-1] \cap \cA^b[-1]$. From the construction and $E\in \cP_{b,w}(x,0]$, we have
\[x<\phi_{b,w}^-(E)\leq \phi^-_{b,w}(E_2)\leq \phi^+_{b,w}(E_2)\leq 0,\]
which implies $E_2\in \cP_{b,w}(x,0]$. Moreover, by Lemma \ref{lem:deformation-1}(b), we also have $$E_1\in \cP_{b,w}(-\epsilon, 0]\subset  \cP_{b,w}(x,0].$$ Hence, the exact sequence above is also a strict exact sequence in $\cP_{b,w}(x,y)$, and the result follows.

The independence statement follows because all these objects are constructed from the slicing $\cP_{b,w}$ and the torsion pair $(\cA^{b}[-1] \cap \cA^{b'}, \cA^{b}[-1] \cap \cA^{b'}[-1])$, which does not depend on $x,y$.
\end{proof}

Now, we prove an abstract generalization of \cite[Lemma 12.4]{bayer2016space}.

\begin{lemma}\label{lem:deformation-bms}
Suppose that $x,y$ are real numbers with $-\frac{1}{2}<x\leq -\epsilon$, $0<y\leq \frac{1}{2}$, and $2\epsilon<y-x<1-2\epsilon$. Then for any $0\neq E\in \cP_{b,w}(x,y)$, there is a strict filtration of subobjects $$0=E_0\subset E_1\subset E_2\subset E_3=E$$ in $\cP_{b,w}(x,y)$ so that
\begin{enumerate}
		\item we have \(E_{1} \in \mathcal{A}^{b'}\), and \(\Im Z^{b',w'}(N)>0\) for any nonzero strict quotient \(N\) of \(E_{1}\) in \(\mathcal{P}_{b,w}(x,y)\);
		\item \(E_{2}/E_{1} \in  \mathcal{P}_{b,w}(-\epsilon,0] \cap \mathcal{A}^{b'}[-1]\) is \(Z^{b',w'}\)-semistable in \(\mathcal{P}_{b,w}(x,y)\), and \(\Im Z^{b',w'} (E_{2}/E_{1})=0\);
		\item \(E_{3}/E_{2} \in  \mathcal{P}_{b,w}(x,0)\cap \mathcal{A}^{b'}[-1]\), and for any nonzero strict subobject \(M\) of \(E_{3}/E_{2}\) in \(\mathcal{P}_{b,w}(x,y)\), we have \(\Im Z^{b',w'}(M)<0\).
	\end{enumerate}
    Moreover, if $x', y'$ is another pair of real numbers satisfying the same numerical condition and $E\in \cP_{b,w}(x',y')$, then the resulting filtration remains the same in \(\mathcal{P}_{b,w}(x',y')\).
\end{lemma}

\begin{proof} We divide the proof into several steps.

\textbf{Step 1.}

Since $\cP_{b,w}$ is a slicing on $\cD$, we have a strict exact sequence
\[0\to A\to E\to B\to 0\]
in $\cP_{b,w}(x,y)$ such that $A\in \cP_{b,w}(0,y)$ and $B\in \cP_{b,w}(x,0]$. By Lemma \ref{lem:deformation-1}(a), if $A\neq 0$, then $A\in \cA^{b'}$ and $\Im Z^{b',w'}(A)>0$. Moreover, any non-zero strict quotient \(N\) of \(A\) in \(\mathcal{P}_{b,w}(x,y)\) is contained in $\cP_{b,w}(0,y)$, so $\Im Z^{b',w'}(N)>0$ by Lemma \ref{lem:deformation-1}(a) again. Therefore, if the statement holds for $B$, then by pulling back the filtration to $E$, the statement also holds for $E$. So we may assume that $E=B\in \cP_{b,w}(x,0]\subset \cA^b[-1]$.

We take $E_1\subset E$ as in Lemma \ref{lem:deformation-2}. Therefore, we only need to prove the statement to $E/E_1$. So we reduce to the case $E_1=0$ and $E\in \mathcal{P}_{b,w}(x, 0] \cap \mathcal{A}^{b'}[-1]$.

\textbf{Step 2.}

Recall that $\cA^{b'}$ has a torsion pair $(\cI^{b'}, (\cI^{b'})^{\perp})$ by the tilting property of $\sigma$, Lemma \ref{Noetherianlem}, and Lemma \ref{lem:tilt-property-no-infty-seq}. Therefore, we have an exact sequence
\[0\to F\to E\to G\to 0\]
in $\cA^{b'}[-1]$ so that $F\in \cI^{b'}[-1]$ and $G\in (\cI^{b'})^{\perp}[-1]$. In this step, we aim to show $F\in \mathcal{P}_{b,w}(-\epsilon, 0]$ and $G\in \mathcal{P}_{b,w}(x, 0)$. We may assume that both $F$ and $G$ are non-zero.

First, we show $F\in \cA^b[-1]\cap \cA^{b'}[-1]$. We know that there exists an exact sequence
\[0\to F_1\to F\to F_2\to 0\]
in $\cA^{b'}[-1]$ such that $F_1\in \cA^b[-1]\cap \cA^{b'}[-1]$ and $F_2\in \cA^b[-2]\cap \cA^{b'}[-1]$. Since $F\in \cI^{b'}[-1]$, we get
\[\Im Z^{b',w'}(F_1)=\Im Z^{b',w'}(F)=\Im Z^{b',w'}(F_2)=0.\]
If $F_2\neq 0$, then $\Im Z^{b',w'}(F_2)<0$ by Lemma \ref{lem:deformation-1}(b), which makes a contradiction. So $F_2=0$ and $F_1=F\in \cA^b[-1]$ as desired.

Next, we prove that $F\in \mathcal{P}_{b,w}(-\epsilon, 0]$, i.e.~$\phi^-_{b,w}(F)>-\epsilon$. Using the torsion pair $(\cA^b[-1]\cap \cA^{b'}, \cA^b[-1]\cap \cA^{b'}[-1])$ of $\cA^b[-1]$, we have an exact sequence
\[0\to F_1'\to \HN^-_{\sigma^{b,w}}(F)\to F_2'\to 0\]
in $\cA^b[-1]$, where $F_1'\in \cA^b[-1]\cap \cA^{b'}$ and $F_2'\in \cA^b[-1]\cap \cA^{b'}[-1]$. If $F_1'\neq 0$, then by Lemma \ref{lem:deformation-1}(b), we have
\[-\epsilon<\phi^-_{b,w}(F_1')\leq \phi^+_{b,w}(F_1')\leq \phi_{b,w}(\HN^-_{\sigma^{b,w}}(F))\]
and $\phi^-_{b,w}(F)>-\epsilon$ follows. If $F_1'=0$, then $\HN^-_{\sigma^{b,w}}(F)=F_2'\in \cA^b[-1]\cap \cA^{b'}[-1]$. By the property of the torsion pair $(\cA^b[-1]\cap \cA^{b'}, \cA^b[-1]\cap \cA^{b'}[-1])$ of $\cA^b[-1]$, we know that the kernel of $F\to \HN^-_{\sigma^{b,w}}(F)$ is in $\cA^b[-1]\cap \cA^{b'}[-1]$, so $F\to \HN^-_{\sigma^{b,w}}(F)$ is also surjective in $\cA^{b'}[-1]$. Therefore, from $F\in \cI^{b'}[-1]$, we get $$\Im Z^{b',w'}(F)=\Im Z^{b',w'}(\HN^-_{\sigma^{b,w}}(F))=0.$$ Combining this with $\HN^-_{\sigma^{b,w}}(F)\in \cA^b[-1]\cap \cA^{b'}[-1]$ and \eqref{eq:diff-phase-epsilon}, we get $\phi^-_{b,w}(F)>-\epsilon$ as desired.

Finally, we prove $G\in \cP_{b,w}(x,0)$. Since $E\in \cA^b[-1]\cap \cA^{b'}[-1]$, by the property of the torsion pair $(\cA^b[-1]\cap \cA^{b'}[-1], \cA^b[-2]\cap \cA^{b'}[-1])$ of $\cA^{b'}[-1]$, the surjection $E\to G$ in $\cA^{b'}[-1]$ is also surjective in $\cA^b[-1]$. Therefore, we obtain
\[x<\phi^-_{b,w}(E)\leq \phi^-_{b,w}(G)\]
and thus $G\in \cP_{b,w}(x,0]$. Note that if $G$ has a subobject $G'$ in $\cA^b[-1]$ with $\Im Z^{b,w}(G')=0$, by the property of the torsion pair $(\cA^b[-1]\cap \cA^{b'}, \cA^b[-1]\cap \cA^{b'}[-1])$ of $\cA^b[-1]$, we have $G'\in \cA^b[-1]\cap \cA^{b'}[-1]$. However, this contradicts Lemma \ref{imaginarypart} and $b'<b$. This proves $G\in \cP_{b,w}(x,0)$.

\textbf{Step 3.} In this step, we prove that $F$ is \(Z^{b',w'}\)-semistable in \(\mathcal{P}_{b,w}(x,y)\).

If $F$ is not \(Z^{b',w'}\)-semistable in \(\mathcal{P}_{b,w}(x,y)\), then there is a non-trivial strict exact sequence
\[0\to F_3\to F\to F_4\to 0\]
in \(\mathcal{P}_{b,w}(x,y)\) so that
\begin{equation}\label{eq:psi-inequal}
    \psi_{b',w'}(F_3)>\psi_{b',w'}(F)=0>\psi_{b',w'}(F_4).
\end{equation}
Since $F\in \cA^b[-1]\cap \cA^{b'}[-1]$, we have $\Hom_{\cD}(\cA^{b'}, F)=0$. Therefore, by applying Lemma \ref{lem:deformation-2} to $F_3$, we get $F_3\in \cP_{b,w}(x,0]\cap \cA^{b'}[-1]$. But this contradicts \eqref{eq:psi-inequal} as $\Im Z^{b',w'}(F)=0$.

\textbf{Step 4.} Finally, we prove that for any non-zero strict subobject $M$ of $G$ in $\cP_{b,w}(x,y)$, we have $\Im Z^{b',w'}(M)<0$.

Indeed, since $G\in \cP_{b,w}(x,0)\cap\cA^{b'}[-1]$, by applying Lemma \ref{lem:deformation-2} to $M$, we get $M\in \cP_{b,w}(x,0]\cap \cA^{b'}[-1]$, so $\Im Z^{b',w'}(M)\leq 0$. If $\Im Z^{b',w'}(M)= 0$, then $M\in \cI^{b'}[-1]$, which contradicts $G\in (\cI^{b'})^{\perp}$. This proves $\Im Z^{b',w'}(M)<0$ and completes the proof of the lemma.
\end{proof}

For any \(\psi \in \mathbb{R}\), we define:
\[
	\mathcal{Q}(\psi)\coloneqq \{E \in \mathcal{D} \colon E \text{ is } Z^{b',w'}\text{-semistable in } \mathcal{P}_{b,w}(\psi-\epsilon,\psi+\epsilon) \text{ with phase }\psi\} \bigcup \{0\}.
\] 
By Lemma \ref{lem:envelop}, we can alternatively write
\[
	\mathcal{Q}(\psi)= \{E \in \mathcal{D} \colon E \text{ is } Z^{b',w'}\text{-semistable in a thin subcategory with phase }\psi\} \bigcup \{0\}.
\] 

In the following, we will show that $\cQ$ is a slicing with $\cQ(0,1]=\cA^{b'}$.

\begin{lemma}\label{lem:slicing-Q-}
We have \(\mathcal{Q}(\psi)[1]=\mathcal{Q}(\psi+1)\) and \(\mathrm{Hom}_{\mathcal{D}}(\mathcal{Q}(\psi_{1}),\mathcal{Q}(\psi_{2}))=0\) for \(\psi_{1}>\psi_{2}\).
\end{lemma}

\begin{proof}
This follows from the same proof as \cite[Lemma 7.6]{bridgeland:stability}.
\end{proof}

\begin{lemma}\label{lem:Q-slicing}
The collection of full subcategories $\{\cQ(\psi)\}_{\psi\in \RR}$ is a slicing of $\cD$.
\end{lemma}

\begin{proof}
The proof is similar to the arguments in \cite[Section 7]{bridgeland:stability} and \cite[Appendix 2]{bayer2016space}. By Lemma \ref{lem:slicing-Q-}, it remains to prove that for any $0\neq E\in \cD$ and any $t\in \RR$, we have a triangle
\[A\to E \to B\]
with $A\in \cQ(>t)$ and $B\in \cQ(\leq t)$. By \eqref{eq:diff-phase-epsilon}, we get $\cP_{b,w}(\geq t+\epsilon)\subset \cQ(>t)$ and $\cP_{b,w}(\leq t-\epsilon)\subset \cQ(\leq t)$. Therefore, we may assume that $E\in \cP_{b,w}(t-\epsilon, t+\epsilon)$. Up to shift, we can assume that $|t|\leq \frac{1}{2}$. Then $Z^{b',w'}$ is a skewed weak stability function on $\cP_{b,w}(t-\epsilon, t+\epsilon)$. To prove the statement, it suffices to prove the existence of HN filtrations of objects in $\cP_{b,w}(t-\epsilon, t+\epsilon)$ with respect to $Z^{b',w'}$ such that each factor is enveloped by a thin subcategory. If $(t-3\epsilon, t+5\epsilon)\cap \ZZ=\varnothing$, then $\cP_{b,w}(t-3\epsilon, t+5\epsilon)$ is of finite length by Lemma \ref{lem:supp-property-finite-length}. Therefore, the result follows from \cite[Lemma 7.7]{bridgeland:stability}. From now on, we assume that $(t-3\epsilon, t+5\epsilon)\cap \ZZ\neq \varnothing$. Hence, $0\in (t-3\epsilon, t+5\epsilon)$ and $-5\epsilon<t<3\epsilon$.

\textbf{Step 1.} By our assumption on $\epsilon$, we can consider $E$ as an object in a larger category $\cP_{b,w}(J)$, where
\[J\coloneqq \left(-\frac{1}{2}+4\epsilon, \frac{1}{2}-5\epsilon\right).\]
It suffices to prove the existence of HN filtration of any non-zero $E\in \cP_{b,w}(J)$ with respect to $Z^{b',w'}$ such that each factor is enveloped by a thin subcategory. Let $$E_1\subset E_2\subset E$$ be the strict filtration of $E$ constructed in Lemma \ref{lem:deformation-bms}. Then $E_1\in \cA^{b'}$, $E_2/E_1\in \cA^{b'}[-1]$ is $Z^{b',w'}$-semistable with $\psi_{b',w'}(E_2/E_1)=0$, and $$E/E_2\in \cA^{b'}[-1]\cap \cP_{b,w}(-\frac{1}{2}+4\epsilon,0).$$ Therefore, we only need to show $E_1$ and $E/E_2$ have HN filtrations with respect to $Z^{b',w'}$, regarded as a skewed weak stability function on $\cP_{b,w}(J)$, such that each factor is enveloped by a thin subcategory. In the following, we only deal with $E_1$, as the argument for $E/E_2$ is completely the same after switching the roles of quotient objects and subobjects in the following steps.

\textbf{Step 2.} Set \[J'\coloneqq \left(-\frac{1}{2}+2\epsilon, \frac{1}{2}-\epsilon\right).\]
By Lemma \ref{lem:supp-property-finite-length}, we know that there are no infinite sequences of strict subobjects 
\[\cdots\subset E^{j+1}\subset E^j\subset \cdots \subset E^2\subset E^1\]
in $\cP_{b,w}(J')$ with $\psi_{b',w'}(E^{j+1})>\psi_{b',w'}(E^j)$ for any $j$ and no infinite sequences of strict quotient objects 
\[E^1\twoheadrightarrow E^2\twoheadrightarrow\cdots \twoheadrightarrow E^j\twoheadrightarrow E^{j+1}\twoheadrightarrow \cdots\]
in $\cP_{b,w}(J')$ with $\psi_{b',w'}(E^{j+1})<\psi_{b',w'}(E^j)$ for any $j$. Also note that 
\begin{equation}\label{eq:phase-psi}
    \psi_{b',w'}(E)\in \left( -\frac{1}{2}+\epsilon,\frac{1}{2} \right)
\end{equation}
for any $0\neq E\in \cP_{b,w}(J')$ by \eqref{eq:diff-phase-epsilon}.

We define a class of objects
\[\cC\coloneqq \{0\neq E\in \cP_{b,w}(J')\colon \Im Z^{b',w'}(N)>0\text{ for any strict quotient } N \text{ of } E \text{ in }\cP_{b,w}(J')\}.\]
Then it is clear that $\cC$ is closed under strict quotient in $\cP_{b,w}(J')$.  Note that by definition, for any $E\in \cC$ and any strict subobject $A\subset E$ in $\cP_{b,w}(J')$ with $\psi_{b',w'}(A)>\psi_{b',w'}(E/A)$, we have
\[\psi_{b',w'}(A)>\psi_{b',w'}(E)>\psi_{b',w'}(E/A)>0.\]
Since $\cC$ is closed under strict quotient, the non-existence of quotient filtrations as above shows that for any $0\neq E\in \cC$ which is not $Z^{b',w'}$-semistable, there always exists a $Z^{b',w'}$-semistable strict quotient $E\twoheadrightarrow B$ in $\cP_{b,w}(J')$ with $\psi_{b',w'}(E)>\psi_{b',w'}(B)$.

Similarly, we claim that for any $0\neq E\in \cC$ which is not $Z^{b',w'}$-semistable, we can find a $Z^{b',w'}$-semistable strict subobject $A\subset E$ in $\cP_{b,w}(J')$ with $\psi_{b',w'}(A)>\psi_{b',w'}(E)$ and $A\in \cC$. Indeed, by the non-existence of filtrations by subobjects as above, it suffices to prove that there exists a strict subobject $A\subset E$ with $\psi_{b',w'}(A)>\psi_{b',w'}(E)$ and $A\in \cC$. To this end, let $$0\to F\to E\to G \to 0$$ be a strict exact sequence with $\psi_{b',w'}(F)>\psi_{b',w'}(E)>\psi_{b',w'}(G)$. If $F\in \cC$, then we are done. Otherwise, let $A\subset F$ be the strict subobject constructed in Lemma \ref{lem:deformation-bms}(a). Hence, we have $A\in \cC$ and either $\Im Z^{b',w'}(F/A)<0$ or $Z^{b',w'}(F/A)\in \RR_{\geq 0}$. In both cases, we get $$\psi_{b',w'}(A)\geq \psi_{b',w'}(F)>\psi_{b',w'}(E),$$ and the claim follows.

\textbf{Step 3.} Similar to \cite[Lemma 7.7]{bridgeland:stability}, we define $\cH$ to be the class of objects $E$ in $\cP_{b,w}(J')$ with $\psi_{b',w'}(E)<\frac{1}{2}-4\epsilon$ and every non-zero strict quotient $B$ of $E$ in $\cP_{b,w}(J')$ satisfies $\psi_{b',w'}(B)>0>-\frac{1}{2}+3\epsilon$. In particular, $\cH\subset \cC$. It is clear that if $E\in \cH$ and $E\twoheadrightarrow E'$ is a strict quotient with $\psi_{b',w'}(E)\geq \psi_{b',w'}(E')$, then $E'\in \cH$.

Now, we show that any $0\neq E\in \cH$ has an mdq. Consider a strict exact sequence $$0\to A\to E\to E^1\to 0$$ with $A\in \cP_{b,w}(\geq \psi_{b',w'}(E)+\epsilon)$ and $E^1\in \cP_{b,w}(< \psi_{b',w'}(E)+\epsilon)$. Note that by $E\in \cH$, we have $E^1\in \cH$ and $\psi_{b',w'}(A)>\psi_{b',w'}(E)>\psi_{b',w'}(E^1)$. Then as in the proof of \cite[Lemma 7.7]{bridgeland:stability}, the mdq of $E^1$ is also the mdq of $E$.

If $E^1$ is $Z^{b',w'}$-semistable, then we are done. Otherwise, from $E^1\in \cH$ and $E^1\in \cP_{b,w}(< \psi_{b',w'}(E)+\epsilon)$, we know that every $Z^{b',w'}$-semistable strict subobject or quotient object of $E^1$ is enveloped by $\cP_{b,w}(J')$. Therefore, by Step 2, Lemma \ref{lem:slicing-Q-}, and the argument in \cite[Proposition 2.4]{bridgeland:stability}, we can find a strict exact sequence $0\to A^1\to E^1\to E^2\to 0$ such that $A^1$ is $Z^{b',w'}$-semistable with $$\psi_{b',w'}(A^1)>\psi_{b',w'}(E^1)>\psi_{b',w'}(E^2)$$ so that the mdq of $E^2$ is the mdq of $E^1$. Consider a strict exact sequence $0\to A^2\to E^2\to E^3\to 0$ with $A^2\in \cP_{b,w}(\geq \psi_{b',w'}(E^2)+\epsilon)$ and $E^3\in \cP_{b,w}(< \psi_{b',w'}(E^2)+\epsilon)$ as above, if $E^3$ is not $Z^{b',w'}$-semistable, then we continue this process. By Step 2, this will stop in finite steps. Therefore, we get the existence of the mdq of $E$ as desired.

\textbf{Step 4.} Similar to \cite[Lemma 7.7]{bridgeland:stability}, we define $\cG$ to be the class of objects $E$ in $\cC$ with $\phi_{b,w}^+(E)<\frac{1}{2}-5\epsilon$.

We claim that if $E\in \cG$ and $E\twoheadrightarrow B$ is the mdq, then its strict kernel $E'$ is in $\cG$ as well. It is clear that $\phi_{b,w}^+(E')<\frac{1}{2}-5\epsilon$. If $E'\twoheadrightarrow B'$ is the mdq, we have a commutative diagram
\[\begin{tikzcd}
	0 & {E'} & E & B & 0 \\
	0 & {B'} & Q & B & 0
	\arrow[from=1-1, to=1-2]
	\arrow[from=1-2, to=1-3]
	\arrow[from=1-2, to=2-2]
	\arrow[from=1-3, to=1-4]
	\arrow[from=1-3, to=2-3]
	\arrow[from=1-4, to=1-5]
	\arrow[shift left, no head, from=1-4, to=2-4]
	\arrow[no head, from=1-4, to=2-4]
	\arrow[from=2-1, to=2-2]
	\arrow[from=2-2, to=2-3]
	\arrow[from=2-3, to=2-4]
	\arrow[from=2-4, to=2-5]
\end{tikzcd}\]
whose rows are strict exact sequences. By definition, we have $\psi_{b',w'}(Q)>\psi_{b',w'}(B)>0$. Moreover, either $\Im Z^{b',w'}(B')\neq 0$ or $Z^{b',w'}(B')\in \RR_{\geq 0}$ holds. Therefore, using \eqref{eq:phase-psi}, we get $$\psi_{b',w'}(B')> \psi_{b',w'}(Q)> \psi_{b',w'}(B)>0$$ and the claim follows.


\textbf{Step 5.} Note that $\cG\subset \cH$. Therefore, for any $0\neq E\in \cG$ which is not $Z^{b',w'}$-semistable, we know that the mdq $E\twoheadrightarrow E'$ exists by Step 3 and satisfies $\psi_{b',w'}(E)>\psi_{b',w'}(E')>0.$ Thus, by Step 4, we get 
\[\psi_{b',w'}(A)>\psi_{b',w'}(E)>\psi_{b',w'}(E')>0,\]
where $A\in \cG$ is the strict kernel of $E\twoheadrightarrow E'$. Therefore, the remaining proof of \cite[Proposition 2.4]{bridgeland:stability} applies in this case, and we can conclude that any $0\neq E\in \cG$ has an HN filtration with respect to $Z^{b',w'}$, regarded as a skewed weak stability function on $\cP_{b,w}(J')$, such that each factor is enveloped by the thin subcategory $\cP_{b,w}(J')$. In particular, this applies to $E_1$ in Step 1.
\end{proof}

From Lemma \ref{lem:Q-slicing}, we know that $(Z^{b',w'},\cQ)$ is a weak pre-stability condition on $\cD$. By the following lemma, we get $(Z^{b',w'},\cQ)=\sigma^{b',w'}$, which finishes the proof of Theorem \ref{thm:tilt-stability}.

\begin{lemma}\label{lem:compare-Q-Ab'}
We have $\cQ(0,1]=\cA^{b'}$.
\end{lemma}

\begin{proof}
By the property of t-structures, it suffices to show that $\cQ(\psi)\subset \cA^{b'}$ for any $\psi\in (0,1]$.

We first assume that $0<\psi< \epsilon$. Let $0\neq E\in \cQ(\psi)$. Then we have an exact sequence
\[0\to F\to E\to G\to 0\]
in $\cP_{b,w}(\psi-\epsilon, \psi+\epsilon)$ with $F\in \cP_{b,w}(0,\psi+\epsilon)$ and $G\in \cP_{b,w}(\psi-\epsilon,0]$. Since $(\cA^b\cap \cA^{b'}[1], \cA^b\cap \cA^{b'})$ is a torsion pair of $\cA^b$, we have an exact sequence
\[0\to F_1\to F\to F_2\to 0\]
in $\cA^b$ with $F_1\in \cA^b\cap \cA^{b'}[1]$ and $F_2\in \cA^b\cap \cA^{b'}$. Note that $F_1=0$, otherwise by Lemma \ref{lem:deformation-1}(b), we have $$\psi+\epsilon<1-\epsilon <\phi^-_{b,w}(F_1)\leq \phi^+_{b,w}(F_1)\leq \phi^+_{b,w}(F),$$
contradicts $F\in \cP_{b,w}(0,\psi+\epsilon)$. Thus $F=F_2\in \cA^b\cap \cA^{b'}$. Similarly, we have an exact sequence
\[0\to G_1\to G\to G_2\to 0\]
in $\cA^b[-1]$ with $G_1\in \cA^b[-1]\cap \cA^{b'}$ and $G_2\in \cA^b[-1]\cap \cA^{b'}[-1]$. By Lemma \ref{lem:deformation-1}(b), we have $$G_1\in \cP_{b,w}(-\epsilon, 0]\subset \cP_{b,w}(\psi-\epsilon, 0].$$ If $G_2=0$, then it is clear that $G\in \cA^b[-1]\cap \cA^{b'}[-1]$, and we get $E\in \cA^{b'}$. Therefore, we may assume that $G_2\neq 0$. Since $G_2\in \cP_{b,w}(\psi-\epsilon, 0]$, the above sequence is also a strict exact sequence in $\cP_{b,w}(\psi-\epsilon, 0]\subset \cP_{b,w}(\psi-\epsilon, \psi+\epsilon)$. Therefore, $G_2$ is a strict quotient of $E$ in $\cP_{b,w}(\psi-\epsilon, \psi+\epsilon)$ and its kernel $F'$ is an extension of $F$ and $G_1$, which is in $\cA^{b'}$. However, we then have $$\psi_{b',w'}(G_2)\leq 0 < \psi=\psi_{b',w'}(E)$$ and $\psi_{b',w'}(F')>0$, which by \eqref{eq:diff-phase-epsilon} and together $Z^{b',w'}(E)=Z^{b',w'}(F')+Z^{b',w'}(G_2)$ implies
\[\psi_{b',w'}(G_2)< \psi=\psi_{b',w'}(E)<\psi_{b',w'}(F')\]
and contradicts the $Z^{b',w'}$-semistability of $E$.

Next, we assume that $\epsilon\leq \psi\leq 1$. We claim that $E\in \cA^b$. If $\epsilon\leq \psi\leq 1-\epsilon$, then it is clear since $E\in \cP^{b,w}(\psi-\epsilon, \psi+\epsilon)$. If $1-\epsilon<\psi \leq 1$, then we have a strict exact sequence
\[0\to F\to E\to G\to 0\]
in $\cP_{b,w}(\psi-\epsilon, \psi+\epsilon)$ such that $F\in \cP_{b,w}(1, \psi+\epsilon)$ and $G\in \cP_{b.w}(\psi-\epsilon, 1]$. If $F\neq 0$, then by Lemma \ref{lem:deformation-1}(a), we have $\Im Z^{b',w'}(F)<0$, hence $\psi_{b',w'}(F)>1$. However, by \eqref{eq:diff-phase-epsilon}, we can check $\psi_{b',w'}(F)>1\geq \psi=\psi_{b',w'}(E)>\psi_{b',w'}(G)$, contradicts the $Z^{b',w'}$-semistability of $E$.

Therefore, by the torsion pair of $\cA^b$ used above, we have an exact sequence
\[0\to F\to E\to G\to 0\]
in $\cA^b$ with $F\in \cA^b\cap \cA^{b'}[1]$ and $G\in \cA^b\cap \cA^{b'}$. If $F=0$, then $E=G$ and we are done. If $F\neq 0$, then by Lemma \ref{lem:deformation-1}(b), we know that
$$\psi-\epsilon\leq 1-\epsilon <\phi^-_{b,w}(F),$$
which gives $F\in \cP_{b,w}(\psi-\epsilon, 1]$. Since $E,G\in  \cP_{b,w}(\psi-\epsilon, 1]$, we know that the above sequence is a strict exact sequence in $\cP_{b,w}(\psi-\epsilon, 1]\subset \cP_{b,w}(\psi-\epsilon, \psi+\epsilon)$. However, this contradicts the $Z^{b',w'}$-semistability of $E$ since $\psi_{b',w'}(F)>1\geq \psi=\psi_{b',w'}(E)$.

This finishes the proof of $\cQ(\psi)\subset \cA^{b'}$ for $\psi\in (0,1]$. 
\end{proof}

\subsection{Wall-chamber structure}

We end this section by proving the wall-chamber structure result for weak stability conditions constructed in Theorem \ref{thm:tilt-stability}.


For a function $\mathsf{g}\colon \RR\to \RR$, we set
\[U_{\mathsf{g}}\coloneqq \{(b,w)\in \RR^2 \colon w>\mathsf{g}(b)\}.\]
We also define 
\[\cU\coloneqq \left\{(b,w)\in \RR^2\colon w>\frac{1}{2}b^2\right\}.\]



A similar argument as in \cite[Proposition 12.5]{bayer2016space} and \cite[Proposition 4.1]{fey:application-to-NL} gives the following wall-chamber structure.

\begin{theorem}\label{thm:wall-chamber-abstract}
Keep the settings in Theorem \ref{thm:tilt-stability}. Assume furthermore that $\Phi_W<+\infty$. 
Fix a class $v\in \Lambda$ such that $(\Im Z(v), \Re Z(v), W(v))\neq (0,0,0)$. Then there is a locally finite collection \(\mathcal{W}_v\) of connected
components of intersections of affine lines with \(U_W\), called ``\emph{walls}",
such that 
\begin{enumerate}
	    \item If $\Im Z(v)\ne0$, then all lines generated by $\ell_i\in \mathcal{W}_v$ pass through $\Pi_W(v)$.
	    \item If $\Im Z(v)=0$ and $\Re Z(v)\neq 0$, then all $\ell_i\in \mathcal{W}_v$ are parallel of slope $\frac{W(v)}{\Re Z(v)}$.
	   		\item The $\nu_{b,w}$-(semi)stability of any object $E$ with $\bv(E)=v$ is unchanged as $(b,w)$ varies within any connected component (called a ``\emph{chamber}") of  $U_W \setminus \bigcup_{\ell_i \in \mathcal{W}_v}\ell_i$.
		\item For any wall $\ell_i\in \mathcal{W}_v$, there is a map $f\colon F\to E$ in $\cA^b$ such that
\begin{itemize}
\item $E$ is $\nu_{b,w}$-semistable of class $v$ with $\nu_{b,w}(E)=\nu_{b,w}(F)=\,\mathrm{slope}\,(\ell_i)$ constant for any  $(b,w)\in \ell_i$, and
\item $f$ is an injection $F\hookrightarrow E $ in $\cA^b$ which strictly destabilizes $E$ for $(b,w)$ in at least one of the two chambers adjacent to the wall $\ell_i$.
\end{itemize} 
	\end{enumerate}

Moreover, if $\mathsf{g}$ is a convex function with $\Phi_W\leq \mathsf{g}$, then $\Pi_W(E)\notin U_{\mathsf{g}}$ for any $(b,w)\in U_{\mathsf{g}}$ and any $\nu_{b,w}$-semistable object $E$ with $\Im Z(E)\neq 0$.
\end{theorem}

\begin{proof}
For a class $v'\in \Lambda$ such that $(\Im Z(v'), \Re Z(v'), W(v'))$ is not proportional to $(\Im Z(v), \Re Z(v), W(v))$, we define
\[\ell(v,v')\coloneqq \{(b,w)\in \RR^2\colon \nu_{b,w}(v)=\nu_{b,w}(v')\}.\]
Then it is clear that $\ell(v,v')$ is a line that satisfies (a) and (b). In the following, we only deal with the case $\Im Z(v)\neq 0$, as the remaining cases can be treated analogously, but easier.

\textbf{Step 1}. Let $\ell\subset U_W$ be a connected line segment and fix a point $(b_0,w_0)\in \ell$. We claim that if $E\in \cA^{b_0}$ is an object with class $v$ and $\ell\subset \ell(v,v')$ for some $v'\in \Lambda$, then $E$ is $\nu_{b_0,w_0}$-semistable if and only if $E\in \cA^b$ is $\nu_{b,w}$-semistable for any point $(b,w)\in \ell$; moreover, $\Pi_W(v)\notin \ell$.

First, we assume that $\nu_{b_0,w_0}(E)=+\infty$. Then $\Pi_W(E)\notin U_W$ by Lemma \ref{imaginarypart}, which implies that $\ell(v,v')$ is a vertical line by the assumption $\ell\subset U_W\cap \ell(v,v')$. In this case, the claim is clear.

Now, we assume that $\nu_0\coloneqq \nu_{b_0,w_0}(E)<+\infty$, so $\ell(v,v')$ has a finite slope. As $E\in \cA^{b_0}$, we see that $\mu^+(\cH^{-1}(E))\leq b_0$ and $\mu^-(\cH^0(E))>b_0$. If there exists $(b,w)\in \ell\setminus \Pi_W(v)$ so that $\cH^{-1}(E)[1]\notin \cA^b$, then 
\[b<\mu^+(\cH^{-1}(E))\leq b_0.\]
In particular, the point $(\mu^+(\cH^{-1}(E)),\nu_{0}(\mu^+(\cH^{-1}(E))-b_0)+w_0)$ lies on $\ell$. So we obtain
\[\Phi_W(\mu^+(\cH^{-1}(E)))<\nu_{0}(\mu^+(\cH^{-1}(E))-b_0)+w_0.\]
But $\HN_{\sigma}^+(\cH^{-1}(E))[1]$ is a subobject of $E$ in $\cA^{b_0}$, so $$\nu_{b_0,w_0}(\HN_{\sigma}^+(\cH^{-1}(E))[1])\leq \nu_0\quad\text{ and }\quad\Pi_W(\HN_{\sigma}^+(\cH^{-1}(E)))\notin U_W,$$ which makes a contradiction. Therefore, we get $\cH^{-1}(E)[1]\in \cA^b$. Similarly, we have $\cH^0(E)\in \cA^b$ for any $(b,w)\in \ell\setminus \Pi_W(v)$. This shows that $E\in \cA^b$. In particular, by continuity, $\mu(E)-b$ does not change the sign when $(b,w)$ varies in $\ell$, so $\ell\setminus \Pi_W(v)$ is connected and $\overline{\ell\setminus \Pi_W(v)}=\ell$.

Applying Theorem \ref{thm:tilt-stability}, we see that the set
\[\{(b,w)\in \ell\setminus \Pi_W(v)\colon E \text{ is } \nu_{b,w}\text{-semistable}\}\]
is closed in $\ell \setminus \Pi_W(v)$, so to show $E$ is semistable at any point in $\ell\setminus \Pi_W(v)$, it remains to show it is open. For any $(b_1,w_1)\in \ell\setminus \Pi_W(v)$ so that $E$ is $\nu_{b_1,w_1}$-semistable, we may take $\epsilon>0$ small enough so that $E\in \cP_{b_1,w_1}(\phi)$ and $(\phi-2\epsilon,\phi+2\epsilon)\subset (0,1)$. Therefore, by Theorem \ref{thm:tilt-stability}, there exists an open subset $I\subset \ell\setminus \Pi_W(v)$ containing $(b_1,w_1)$ so that $E\in \cP_{b,w}(\phi-\epsilon, \phi+\epsilon)$ and $\cP_{b,w}(\phi-\epsilon, \phi+\epsilon)\subset \cP_{b_1,w_1}(\phi-2\epsilon, \phi+2\epsilon)$ for any $(b,w)\in I$ . If $E$ is not $\nu_{b',w'}$-semistable for some $(b',w')\in I$, then the first piece of the HN filtration of $E$ with respect to $\nu_{b',w'}$ is an exact sequence $0\to F\to E \to G\to 0$ in $\cA^{b'}$ so that $$F,G\in \cP_{b',w'}(\phi-\epsilon, \phi+\epsilon)\subset \cP_{b_1,w_1}(\phi-2\epsilon, \phi+2\epsilon)\subset \cA^{b_1}.$$
Thus, it is also an exact sequence in $\cA^{b_1}$. Using $\nu_{b',w'}(F)>\nu_0$, it is direct to check that $\nu_{b_1,w_1}(F)>\nu_0$ using $w_1-w'=\nu_0(b_1-b')$, contradicts the $\nu_{b_1,w_1}$-semistability of $E$. Thus $E$ is $\nu_{b,w}$-semistable for any $(b,w)\in I$, and the openness follows. Since $E$ is semistable along $\ell\setminus \Pi_W(v)$, by the closedness of semistability, if $\Pi_W(v)\in \ell$, it is also semistable at the point $\Pi_W(v)$. However, this contradicts Lemma \ref{lem:weakstabilityfunction}. Hence, $\Pi_W(v)\notin \ell$.

\textbf{Step 2.} Define
\begin{align*}
\mathcal{W}_v\coloneqq \Big\{&\ell\subset U_W \colon \ell \text{ is a connected component of } \ell(v, v')\cap U_W \text{ such that exists }  (b,w)\in \ell  \text{ and } \\
 &\nu_{b,w}\text{-semistable } E,F\in \cA^b \text{ with } F\subset E, \nu_{b,w}(E)=\nu_{b,w}(F), \bv(E)=v, \bv(F)=v'\Big\}.
\end{align*}
We claim that $\mathcal{W}_v$ is locally finite, i.e., for any compact subset $K\subset U_W$, there are only finitely many $\ell\in \mathcal{W}_v$ so that $\ell\cap K\neq \varnothing$. 

By Lemma \ref{lem:support-property}, there exists a constant $C_K>0$ so that
\[\|\overline{\bv}(G)\|_{\max}\leq C_K|Z^{b,w}(G)|\]
for any $\nu_{b,w}$-semistable object $G$ and $(b,w)\in K\cap U_W$. We also set $$M_K\coloneqq \sup\{|Z^{b,w}(v)|\colon (b,w)\in K\},$$ which exists since $K$ is compact. If $(b,w)\in K$ and $F\hookrightarrow E$ is an injection in $\cA^b$ with $\nu_{b,w}(E)=\nu_{b,w}(F)$ so that $E,F$ are $\nu_{b,w}$-semistable with $\bv(E)=v$, then $Z^{b,w}(F)=tZ^{b,w}(E)$ for some $0<t<1$. Therefore, we have
\[\|\overline{\bv}(F)\|_{\max}\leq C_K|Z^{b,w}(F)|=tC_K|Z^{b,w}(E)|\leq C_K M_K.\]
Therefore, the set 
\begin{align*}
\Big\{(\Im Z(F), \Re Z(F), W(F))\in \RR^3 & \colon  (b,w)\in K, \text{ there exists } \nu_{b,w}\text{-semistable } \\
 &E,F\in \cA^b \text{ with } F\subset E, \nu_{b,w}(E)=\nu_{b,w}(F), \bv(E)=v\Big\}
\end{align*}
is finite, and the claim follows.

\textbf{Step 3.} Now, we show that if $\cC$ is a connected component of $U_W \setminus \bigcup_{\ell_i \in \mathcal{W}_v}\ell_i$, then for any object $E$ with $\bv(E)=v$, $E$ is $\nu_{b,w}$-semistable for some $(b,w)\in \cC$ if and only if it is $\nu_{b,w}$-semistable for all $(b,w)\in \cC$. In particular, $\Pi_W(v)\notin \cC$.

Let $\gamma\colon [0,1]\to \cC$ be any continuous path. Write  $(b_t,w_t)\coloneqq \gamma(t)$. Assume that $E$ is a $\nu_{b_0,w_0}$-semistable object with $\bv(E)=v$. By Lemma \ref{lem:weakstabilityfunction}, to prove the claim in this step, it suffices to show $E$ is also $\nu_{b_1,w_1}$-semistable.

If $E$ is not $\nu_{b_1,w_1}$-semistable, we set
\[c\coloneqq \sup\{t\in [0,1]\colon E \text{ is }\nu_{b_t,w_t}\text{-semistable}\}.\]
By the closeness of semistability, we see that $E$ is $\nu_{b_c,w_c}$-semistable, so $c<1$. If $E$ is $\nu_{b_c,w_c}$-stable, from the local finiteness of $\mathcal{W}_v$, we know that there exists an open interval $c\in J\subset [0,1]$ such that $E$ is $\nu_{b_t,w_t}$-stable for any $t\in J$, contradicting the maximality of $c$. Therefore, $E$ is strictly $\nu_{b_c,w_c}$-semistable, but this also contradicts $\cC\subset U_W \setminus \bigcup_{\ell_i \in \mathcal{W}_v}\ell_i$. Thus, $E$ is $\nu_{b_1,w_1}$-semistable and the claim follows.

\textbf{Step 4.} By previous steps, the collection $\mathcal{W}_v$ satisfies all properties except the second one in (d). To show this last property, we need to prove that for any $\ell\in \mathcal{W}_v$, $(b_0,w_0)\in \ell$, $\nu_{b_0,w_0}$-semistable object $E$ with $\bv(E)=v$, and a subobject $F\subset E$ in $\cA^{b_0}$ with $\nu_{b_0,w_0}(F)=\nu_{b_0,w_0}(E)$ and $(\Im Z(F), \Re Z(F), W(F))$ is not proportional to $(\Im Z(v), \Re Z(v), W(v))$, we have $\nu_{b',w'}(F)>\nu_{b',w'}(E)$ for $(b',w')$ in at least one of the two chambers adjacent to the wall $\ell$.

If $\ell$ is a vertical line, then we have $\mu(E)=\mu(F)=b$, and the claim is easy to check. In the following, we may assume that $\ell$ has a finite slope. In this case, the solution of
\begin{align*}
0=f_{E,F}(b,w)\coloneqq & \Im Z^{b,w}(E)\Im Z^{b,w}(F)(\nu_{b,w}(F)-\nu_{b,w}(E))\\
=&W(E)\Re Z(F)-W(F)\Re Z(E)+(W(E)\Im Z(F)-W(F)\Im Z(E))b\\
+&(\Im Z(E)\Re Z(F)-\Im Z(F) \Re Z(E))w
\end{align*}
is a line that contains $\ell$. Since $\ell$ has a finite slope, the coefficient of $w$ in $f_{E,F}(b,w)$ is non-zero. Moreover, by Step 1, points $\Pi_W(v)$ and $\Pi_W(F)$ do not lie on $\ell$. So $\Im Z^{b,w}(F)$ and $\Im Z^{b,w}(E)$ are always positive for $(b,w)\in \ell$. Therefore, $\nu_{b,w}(F)>\nu_{b,w}(E)$ holds for any point $(b,w)$ in the chamber below (resp. above) $\ell$ when $\Im Z(E)\Re Z(F)-\Im Z(F) \Re Z(E)>0$ (resp. $<0$). This ends the proof of properties (a)-(d).


\textbf{Step 5}. It remains to prove that if $\mathsf{g}$ is a convex function with $\Phi_W\leq \mathsf{g}$, then $\Pi_W(E)\notin U_{\mathsf{g}}$ for any $(b_0,w_0)\in U_{\mathsf{g}}$ and any $\nu_{b_0,w_0}$-semistable object $E$ with $\Im Z(E)\neq 0$.

Indeed, if $\Pi_W(E)\in U_{\mathsf{g}}$, then the convexity of $\mathsf{g}$ implies that the line segment $\ell$ connecting $\Pi_W(E)$ and $(b_0,w_0)$ lies entirely in $U_{\mathsf{g}}\subset U_W$. But this is impossible by Step 1, since $\Pi_W(E)\in \ell$.
\end{proof}

\section{Slope-stability on varieties}\label{sec:slope-bg}

In this section, we introduce the notation used in later sections. Then we discuss the classical slope-stability, its variants, and the relative version (cf.~Proposition \ref{prop-slope-family}). 

\subsection{Slope-stability}\label{subsec:slope}

Let $X$ be an $n$-dimensional equidimensional projective scheme over a field $\kk$. We first fix some notation. If $c_1,\cdots, c_l\in \A^*(X)_{\QQ}$ and $Z\in \CH_*(X)_{\QQ}$, then we set
\[c_1.c_2\dots c_l.Z\coloneqq\int_X c_1\cdot c_2\cdots c_l\cap Z\]
and
\[c_1.c_2\dots c_l\coloneqq\int_X c_1\cdot c_2\cdots c_l\cap [X].\]
We write $\CH^k(X)\coloneqq \CH_{\dim X-k}(X)$. If $Z\in \CH_0(X)_{\QQ}$, by abuse of notation, we denote $\int_X Z\in \QQ$ by $Z$ as well. If $\cL$ is a line bundle on $X$ and $D\in |\cL|$, then we also denote $c_1(\cL)\in \A^1(X)$ by $\cL$ or $D$ by abuse of notation.

As discussed in \cite{bayer:icm}, we consider the following setup of a triple $(X,H, \gamma)$:

\begin{Setup}\label{setup-gamma}
We fix a triple $(X,H, \gamma)$ as follows. The scheme $X$ is an $n$-dimensional, equidimensional, projective scheme over a field $\kk$ that is lci in codimension $d$ and $1\leq d\leq n$. We fix an ample line bundle $H$ on $X$ and a class
\begin{equation}\label{eq-gamma}
    \gamma=e^{-B_{\gamma}}\cdot (1,0,\gamma_2,\gamma_3,\dots,\gamma_n)\in \A^*(X)_{\QQ},
\end{equation}
where $B_{\gamma}\in \Pic(X)_{\QQ}$ and $\gamma_i\in \A^i(X)_{\QQ}$ for $2\leq i\leq n$. 
\end{Setup}

\begin{definition}\label{def:v}
In Setup \ref{setup-gamma}, we define a homomorphism
\[\bv^{\gamma}_{H,\leq d}\colon \KK_0(X)\to \QQ^{d+1}\]
given by
\[\bv^{\gamma}_{H,\leq d}(\alpha)\coloneqq\big(H^n.\bch_0^{\gamma}(\alpha), H^{n-1}.\bch_1^{\gamma}(\alpha),\dots, H^{n-d}.\bch_{d}^{\gamma}(\alpha)\big)\]
for any class $\alpha\in \KK_0(X)$, where $\bch_i^{\gamma}(\alpha)$ denotes the component of  $\gamma\cap \bch_{\leq d}(\alpha)$ of dimension $\dim X -i$ for any $0\leq i\leq d$. If $d=n$, then we set $\bv^{\gamma}_{H}\coloneqq \bv^{\gamma}_{H,\leq d}$.

Moreover, when $X$ is a geometrically normal projective surface or a geometrically normal projective $\QQ$-factorial threefold that is lci in codimension $2$, we can define $\bv^{\gamma}_{H}$ using  $\bch_2$ or $\bch_3$ defined in Section \ref{subsec:chn}, respectively.
\end{definition}

If $\gamma=1$, we omit $\gamma$ from the above notation.

For any $\alpha\in \KK_0(X)$, the associated \emph{$\mu^{\gamma}_H$-slope} is
\begin{equation}\label{eq:def-slope}
\mu^{\gamma}_H(\alpha)\coloneqq\frac{\bv^{\gamma}_{H,1}(\alpha)}{\bv^{\gamma}_{H,0}(\alpha)}=\frac{H^{n-1}.\bch_1(\alpha)}{H^n.\bch_0(\alpha)}-\frac{H^{n-1}.B_{\gamma}.\bch_0(\alpha)}{H^n.\bch_0(\alpha)}
\end{equation}
when $\bv^{\gamma}_{H,0}(\alpha)\neq 0$, and $\mu^{\gamma}_H(\alpha)\coloneqq+\infty$ otherwise. We will omit the superscript $\gamma$ and set $\mu_H\coloneqq \mu^{\gamma}_H$ if $B_{\gamma}=0$.

We say a sheaf $E\in \Coh(X)$ is \emph{$\mu^{\gamma}_H$-(semi)stable} if for any non-trivial proper subsheaf $F\subset E$, we have $$\mu^{\gamma}_H(F)(\leq) \mu^{\gamma}_H(E/F),$$ where $(\leq)$ denotes $<$ for stability and $\leq$ for semistability. Therefore, a $\mu^{\gamma}_H$-semistable sheaf is either a torsion sheaf or a torsion-free sheaf, and if it is torsion, then it is $\mu^{\gamma}_H$-stable if and only if it is the structure sheaf of a closed point.

\begin{remark}\label{rmk-different-stability}
We give some comments on the above notions.
\begin{itemize}

\item We will see later that the term $\gamma_2$ in Setup \ref{setup-gamma} should be regarded as a correction of the classical Bogomolov–Gieseker inequality for $\mu_H$-semistable sheaves, which is caused by the singularities of $X$ and $\mathrm{char}(\kk)$. The other terms $\gamma_3,\dots,\gamma_n$ play no essential role in our paper.

    \item If $B_{\gamma}=0$ or $X$ is irreducible, the term $$\frac{H^{n-1}.B_{\gamma}.\bch_0(\alpha)}{H^n.\bch_0(\alpha)}$$ does not depend on $\alpha$. In these cases, $\mu^{\gamma}_H$-stability is the same as $\mu_H$-stability by setting $B_{\gamma}=0$. 

\end{itemize}

\end{remark}

Let $\Lambda^{\gamma}_{H,\leq d}\subset \QQ^{d+1}$ be the image of $\bv^{\gamma}_{H,\leq d}$. Note that $\bv^{\gamma}_{H,\leq d}(\alpha)$ factors through $\KK_0(X)\twoheadrightarrow \Knum(X)$ and $\Knum(X)$ is a finite rank lattice\footnote{To see this, we may assume that $H$ is very ample. By Lemma \ref{lem:chi-property}(c), after restricting to a complete intersection of $n-d$ general divisors in $|H|$, we can assume that $X$ is lci, so $\bv_{H,\leq d}^{\gamma}(\xi)=0$ for any numerically trivial $\xi\in \KK_0(X)$.}, so $\Lambda^{\gamma}_{H,\leq d}$ is a lattice of rank $d+1$. We set
\[\bv^{\gamma}_{H,i}(\alpha)\coloneqq(\bv^{\gamma}_{H,\leq d}(\alpha))_i=H^{n-i}.\bch_i^{\gamma}(\alpha)\in \QQ\]
and $\Lambda^{\gamma}_{H,i}\coloneqq \im (\bv^{\gamma}_{H,i})\subset \QQ$.

The following observation is useful.

\begin{lemma}\label{lem:supp-codim-and-bv}
Let $0\neq E\in \Coh(X)$ and $1\leq i\leq d+1$. Then $\codim_X(E)\geq i$ if and only if $\bv^{\gamma}_{H,\leq i-1}(E)=0$. Moreover, if $i\leq d$, then $\codim_X(E)=i$ if and only if $\bv^{\gamma}_{H,\leq i-1}(E)=0$ and $\bv^{\gamma}_{H, i}(E)\neq 0$.
\end{lemma}

\begin{proof}
By Lemma \ref{lem:chi-property}(e), we know that $\codim_X(E)=i$ if and only if $\bch_{\leq i-1}(E)=0$ and $\bch_{i}(E)$ is a non-zero effective cycle corresponding to the scheme-theoretic support of $E$. Then the result follows from
\[\bch^{\gamma}_{l}(E)=\sum_{0\leq k\leq l} \frac{(-B_{\gamma})^k}{k!}.\bch_{l-k}(E)+\sum_{2\leq s\leq l,0\leq t\leq l-s} \frac{(-B_{\gamma})^t}{t!}.\gamma_s.\bch_{l-s-t}(E)\]
for any $l\leq d$.
\end{proof}

For any lattice $\Lambda\subset \QQ^{m+1}$, we denote by $v_i$ the $i$-th component of an element $v\in \Lambda$. We define a homomorphism
\[Z\colon \Lambda\to \CC\]
by $Z(v)=-v_1+\mathfrak{i} v_0$. So if we take $\Lambda=\Lambda^{\gamma}_{H,\leq 1}$, then $$Z^{\gamma}_H(\alpha)\coloneqq Z(\bv^{\gamma}_{H,\leq 1}(\alpha))=-\bv_{H,1}^{\gamma}(\alpha)+\mathfrak{i}\bv_{H,0}^{\gamma}(\alpha)$$
for $\alpha\in \KK_0(X)$.

\begin{proposition}\label{prop-slope-weak-stab}
Fix $(X,H, \gamma)$ as in Setup \ref{setup-gamma}. The pair $\sigma^{\gamma}_H\coloneqq(\Coh(X), Z)$ is a weak stability condition on $\Db(X)$ with respect to the lattice $\Lambda^{\gamma}_{H, \leq 1}$. Moreover, 

\begin{itemize}
    \item $(\Coh(X))^{Z^{\gamma}_H}$ is the same as the category of coherent sheaves supported in dimension $\leq n-2$, and

    \item $\sigma^{\gamma}_H$ satisfies the tilting property and \ref{t3}.
\end{itemize}
\end{proposition}

\begin{proof}
By Lemma \ref{lem:supp-codim-and-bv}, $Z^{\gamma}_H=Z\circ \bv^{\gamma}_{H,\leq 1}$ is a weak stability function on $\Coh(X)$. Since $\Lambda^{\gamma}_{H, \leq 1}$ is of rank two, the central charge $Z$ is injective, hence it satisfies the support property with respect to the trivial quadratic form $Q\equiv 0$. The HN property is the same as the classical situation \cite[Theorem 1.3.4]{HL10}, or we can apply \cite[Proposition 4.10]{macri:lecture-bridgeland-stability}. So we can conclude that $\sigma^{\gamma}_H$ is a weak stability condition.

By Lemma \ref{lem:supp-codim-and-bv}, $(\Coh(X))^{Z^{\gamma}_H}$ is the same as the subcategory of coherent sheaves on $X$ of codimension $\geq 2$, so it is a Noetherian torsion subcategory of $\Coh(X)$. Finally, by Lemma \ref{lem-S2-hull}(a) and (d), the property \ref{t3} of $\sigma^{\gamma}_H$ is given by associating a torsion-free sheaf $E$ to $E^{H}$.
\end{proof}

The following lemma is useful in practice.

\begin{lemma}\label{lem:dual-slope-stable}
Fix $(X,H, \gamma)$ as in Setup \ref{setup-gamma} such that $B_{\gamma}=0$, and $E\in \Coh(X)$ be a torsion-free sheaf. Then we have 

\begin{enumerate}
    \item $\mu^{\gamma}_{H}(E^{\vee})=-\mu^{\gamma}_{H}(E)$,

    \item $E$ is $\mu^{\gamma}_{H}$-(semi)stable if and only if $E^{\vee \vee}$ is $\mu^{\gamma}_{H}$-(semi)stable,

    \item $E$ is $\mu^{\gamma}_{H}$-(semi)stable if and only if $E^{H}$ is $\mu^{\gamma}_{H}$-(semi)stable, and

    \item $E$ is $\mu^{\gamma}_{H}$-(semi)stable if and only if $E^{\vee}$ is $\mu^{\gamma}_{H}$-(semi)stable.
\end{enumerate}
\end{lemma}

\begin{proof}
Since $E$ is torsion-free and $B_{\gamma}=0$, part (a) follows from Lemma \ref{lem-dual-ch}. By Lemma \ref{lem-general-S2}(a), both $E^H$ and $E^{\vee \vee}$ are torsion-free. Since the natural maps $E\to E^H$ and $E\to E^{\vee \vee}$ coincide over $\mathrm{LCI}(X/\kk)$, using Lemma \ref{lem-S2-hull}(c), we see that they are both injective with cokernels supported in codimension $\geq 2$. Therefore, parts (b) and (c) both follow from Lemma \ref{lem:wtE-stable}.

Finally, we prove part (d). We first assume that $E$ is $\mu^{\gamma}_{H}$-(semi)stable. Let $0\to F\to E^{\vee}\to G\to 0$ be an exact sequence of torsion-free sheaves such that $$\mu^{\gamma}_{H}(F)(\geq)\mu^{\gamma}_{H}(E^{\vee})(\geq) \mu^{\gamma}_{H}(G).$$ By dualizing, we get an exact sequence
\[0\to G^{\vee}\to E^{\vee \vee}\to F_1\to 0, \]
where $F_1\coloneqq \im(E^{\vee \vee}\to F^{\vee})$. Since $\mathrm{LCI}(X/\kk)$ is Gorenstein, applying Lemma \ref{lem-codim-Sp} to $G|_{\mathrm{LCI}(X/\kk)}$, we see that $F^{\vee}/F_1\subset \cE xt^1_X(G,\cO_X)$ is supported in codimension $\geq 2$. Therefore, Lemma \ref{lem:supp-codim-and-bv} gives $\mu^{\gamma}_H(F_1)=\mu^{\gamma}_H(F^{\vee})$.
Using parts (a) and (b), we can conclude that $E^{\vee\vee}$ is $\mu^{\gamma}_{H}$-(semi)stable and $$\mu^{\gamma}_{H}(G^{\vee})(\geq)\mu^{\gamma}_{H}(E^{\vee \vee})(\geq) \mu^{\gamma}_{H}(F_1),$$ a contradiction. The case that $E^{\vee}$ is $\mu^{\gamma}_{H}$-(semi)stable can be treated analogously.
\end{proof}

\subsection{Slope-stability in families}\label{subsec:slope-family}

Now, we extend Proposition \ref{prop-slope-weak-stab} to the relative setting. We mostly work in the following relative version of Setup \ref{setup-gamma}.

\begin{Setup}\label{setup-gamma-relative}
We fix a triple $(f\colon X\to S, \cL, \gamma)$ as follows. The morphism $f\colon X\to S$ is a projective flat morphism between Noetherian schemes which is fiberwise lci in codimension $d$ with $1\leq d\leq n$ and satisfies

\begin{itemize}
    \item $X$ and $S$ have finite Krull dimension,

    \item $S$ is a connected Nagata scheme and is quasi-projective over a Noetherian affine scheme, and

    \item each fiber of $f$ is equidimensional of dimension $n$.
    
\end{itemize}

We also fix an $f$-ample line bundle $\cL$ on $X$ and a class
\begin{equation}
    \gamma=e^{-B_{\gamma}}\cdot (1,0, \gamma_2,\gamma_3,\dots,\gamma_n)\in \A_{\star}^*(X/S)_{\QQ},
\end{equation}
where $B_{\gamma}\in \Pic(X)_{\QQ}$ and $\gamma_i\in \A^i_{\star}(X/S)_{\QQ}$ for $2\leq i\leq n$. Here, $\A^*_{\star}(X/S)$ is defined in Definition \ref{def-A-star}.
\end{Setup}

Note that both $X$ and $S$ have the resolution property by \cite[\href{https://stacks.math.columbia.edu/tag/0FDD}{Tag 0FDD}]{stacks-project}.

Similar to the absolute setting in Definition \ref{def:v}, we can define a relative Mukai homomorphism for $\Db(X)$ over $S$ by
\[\bv^{\gamma}_{\cL,\leq d}\colon \Knum(X/S) \twoheadrightarrow \Lambda^{\gamma}_{\cL,\leq d}\subset \QQ^{d+1}, \quad \alpha\mapsto \oplus_{i=0}^{d} \bv^{\gamma_s}_{\cL_s,i}(\alpha_s)\]
for a point $s\in S$. Note that by Lemma \ref{lem-constant-chM} and the connectedness of $S$, such a homomorphism and lattice $\Lambda^{\gamma}_{\cL,\leq d}$ are both independent of the choice $s\in S$. When $f$ is admissible in the sense of Definition \ref{def:adm}, we also get well-defined $\Lambda^{\gamma}_{\cL}\coloneqq \Lambda^{\gamma}_{\cL,\leq n}$ and $\bv^{\gamma}_{\cL}\coloneqq \bv^{\gamma}_{\cL,\leq n}$ by Lemma \ref{lem:constant-generalization}.


The following is a relative version of Proposition \ref{prop-slope-weak-stab}.

\begin{proposition}\label{prop-slope-family}
Fix $(f\colon X\to S, \cL, \gamma)$ as in Setup \ref{setup-gamma-relative}. Then the collection $\underline{\sigma}^{\gamma}_{\cL}\coloneqq(\sigma^{\gamma_s}_{\cL_s})_{s\in S}$ satisfies \ref{c1}, \ref{c2}, \ref{wc1}, \ref{wc2}, \ref{wc3}, and \ref{b1}. Moreover, if $f$ is admissible in the sense of Definition \ref{def:adm}, then $\underline{\sigma}^{\gamma}_{\cL}$ is a weak stability condition on $\Db(X)$ over $S$ with respect to $\Lambda^{\gamma}_{\cL}$.
\end{proposition}

\begin{proof}
The argument is similar to \cite[Example 21.18]{BLMNPS21}. Using Lemma \ref{lem-constant-chM}, we know that $\underline{\sigma}^{\gamma}_{\cL}$ satisfies \ref{c1}. By Proposition \ref{prop-slope-weak-stab}, we see $\sigma^{\gamma_s}_{\cL_s}$ is a weak stability condition on $\Db(X_s)$ for each $s\in S$. For \ref{b1}, as in Proposition \ref{prop-slope-weak-stab}, we can take $Q\equiv 0$. When $f$ is admissible, the condition \ref{b2} follows from \cite[Theorem 4.4]{langer:positive-char}.

It is clear that $Z^{\gamma_s}_{\cL_s}$ is defined over $\QQ[\mathfrak{i}]$. And the corresponding category $(\Coh(X_s))^{Z^{\gamma_s}_{\cL_s}}$ is the same as the category of torsion sheaves on $X_s$ supported in codimension $\geq 2$, hence it is a Noetherian torsion subcategory of $\Coh(X_s)$ and \ref{wc1} follows. Next, conditions \ref{c2} and \ref{wc2} follow from the classical argument as in \cite[Proposition 2.3.1]{HL10}.

Finally, we verify \ref{wc3}. We assume that $C\to S$ is a morphism essentially of finite type from a Dedekind scheme $C$. By Proposition \ref{prop-slope-weak-stab}, we can define a weak stability function $(Z^{\gamma_K}_{\cL_K}, Z^{\gamma}_{\cL, C\text{-}\tor})$ on $\Coh(X_C)$ over $C$ by
\[Z^{\gamma}_{\cL, C\text{-}\tor}(E)\coloneqq Z(\bv^{\gamma_p}_{\cL_p,\leq 1}(F))=-\bv^{\gamma_p}_{\cL_p, 1}(F)+\mathfrak{i}\bv^{\gamma_p}_{\cL_p, 0}(F)\]
for any $E=i_{p*}F$ and $F\in \Coh(X_p)$ and extend it to the whole $\Coh(X_C)_{C\text{-}\tor}$ by \cite[Lemma 6.11]{BLMNPS21}, and 
\[Z^{\gamma_K}_{\cL_K}(E)=Z(\bv^{\gamma_K}_{\cL_K,\leq 1}(E))=-\bv^{\gamma_K}_{\cL_K, 1}(E)+\mathfrak{i}\bv^{\gamma_K}_{\cL_K, 0}(E)\]
for any $E\in \Coh(X_K)$. We denote the corresponding slope function by $\mu^{\gamma}_{\cL, C}$.

By \cite[Proposition 15.9]{BLMNPS21}, it is enough to prove that the weak stability function $(Z^{\gamma_K}_{\cL_K}, Z^{\gamma}_{\cL, C\text{-}\tor})$ on $\Coh(X_C)$ over $C$ satisfies the HN property in the sense of \cite[Definition 13.10]{BLMNPS21}. To this end, we apply \cite[Corollary 16.5]{BLMNPS21}. Since $\Coh(X_C)$ is Noetherian, we know that $\Coh(X_C)$ admits a $C$-torsion theory by Remark \ref{rmk:noether-C-torsion-pair}. Moreover, generic openness of semistability in the sense of \cite[Definition 16.3]{BLMNPS21} holds by \cite[Proposition 2.3.1]{HL10}. Therefore, by Proposition \ref{prop-slope-weak-stab} and \cite[Corollary 16.5]{BLMNPS21}, it suffices to verify assumptions in \cite[Theorem 16.1]{BLMNPS21}. The verification of \cite[Theorem 16.1(0), (1), (2)]{BLMNPS21} is the same as \cite[Proposition 16.6]{BLMNPS21}.

It remains to verify \cite[Theorem 16.1(3)]{BLMNPS21}, i.e.~$(Z^{\gamma}_{\cL, C\text{-}\tor}, \Coh(X_C)_{C\text{-}\tor})$ as a weak stability condition on $\Db(X_C)_{C\text{-}\tor}$ has the tilting property.  Note that by Lemma \ref{lem:supp-codim-and-bv}, $$(\Coh(X_C)_{C\text{-}\tor})^{Z^{\gamma}_{\cL, C\text{-}\tor}}\subset \Coh(X_C)_{C\text{-}\tor}$$ consists of sheaves supported in codimension $\geq 2$ in fibers, which verifies \ref{t1}. For any $E\in \Coh(X_C)_{C\text{-}\tor}$ with $\mu^{\gamma, +}_{\cL, C}(E)<+\infty$, by \cite[Lemma 6.11]{BLMNPS21}, let us assume without loss of generality that $E=i_{p*}F$ for $F\in \Coh(X_p)$. Then $F$ is a torsion-free sheaf on $X_p$. Therefore, by Lemma \ref{lem-S2-hull}, we have an injection $F\hookrightarrow F^H$ such that $F^H$ is a torsion-free $S_2$ sheaf on $X_p$ and $F^H/F$ is supported at a codimension $\geq 2$ locus in $X_p$. Hence, we get $i_{p*}(F^H/F)\in (\Coh(X_C)_{C\text{-}\tor})^{Z^{\gamma}_{\cL, C\text{-}\tor}}$. By Lemma \ref{lem-S2-hull}(d), we have $$\Hom_{X_C}((\Coh(X_C)_{C\text{-}\tor})^{Z^{\gamma}_{\cL, C\text{-}\tor}}, i_{p*}F^H[1])=0$$ and the result follows.
\end{proof}

For later use, we denote the weak HN structure used in the above proof by 
\begin{equation}\label{eq:hn-structure-slope}
    \sigma^{\gamma}_{\cL, C}=(\Coh(X_C), Z^{\gamma}_{\cL, C}),
\end{equation}
where $Z^{\gamma}_{\cL, C}\coloneqq (Z^{\gamma_K}_{\cL_K}, Z^{\gamma}_{\cL, C\text{-}\tor})$.

\begin{remark}
Similar to \cite[Example 21.18]{BLMNPS21}, if we define the slope of a sheaf and the lattice using coefficients of Hilbert polynomials, then we get a weak stability condition on $\Db(X)$ over $S$ for any projective flat family $X\to S$. We will not use this construction in our paper. See also Appendix \ref{appendix:approach}.
\end{remark}

\section{Bogomolov--Gieseker inequalities for semistable sheaves}\label{sec:bg}

In this section, we first introduce the Le Potier function in the absolute and relative setting. Then we prove Theorem \ref{thm:exist-bg-function}, which gives a quadratic upper bound of (relative) Le Potier function. We end this section with a sharper inequality on normal varieties (cf.~Theorem \ref{thm-bg-normal}).

\subsection{Bogomolov--Gieseker functions}

One of the most important properties of slope-stability on smooth projective varieties is the Bogomolov--Gieseker (BG) inequality (cf.~\cite{HL,langer:positive-char}). Therefore, we introduce the following notions.

\begin{definition}\label{assum-bg}
Fix a triple $(X, H, \gamma)$ as in Setup \ref{setup-gamma} such that either 
\begin{itemize}
    \item $d\geq 2$, or

    \item $X_{\overline{\kk}}$ is a normal surface.
\end{itemize}
We define the \emph{Le Potier function} by 
$$\Phi^{\gamma}_{X, H}\colon \RR \to [-\infty, +\infty]$$ by
\[\Phi^{\gamma}_{X, H}(x)\coloneqq \limsup_{b\to x} \left\{\frac{\bv^{\gamma}_{H, 2}(E)}{\bv^{\gamma}_{H, 0}(E)} \colon E\text{ is }\mu^{\gamma}_H\text{-semistable and }\mu^{\gamma}_H(E)=b\right\},\]
where the supremum of the empty set is $-\infty$.Set
\[U^{\gamma}_{X, H}\coloneqq \{(b,w)\in \RR^2\colon w>\Phi^{\gamma}_{X, H}(b)\}.\]
\end{definition}

We say $(X, H, \gamma)$ \emph{has a Bogomolov--Gieseker (BG) function $\mathsf{g}$} if there is a convex function $\mathsf{g}\colon \RR\to \RR$ such that $\Phi^{\gamma}_{X, H}\leq \mathsf{g}$, i.e. for any $\mu^{\gamma}_{H}$-semistable torsion-free sheaf $E\in \Coh(X)$, we have
    \[\frac{\bv^{\gamma}_{H, 2}(E)}{\bv^{\gamma}_{H, 0}(E)}\leq \Phi^{\gamma}_{X, H}(\mu^{\gamma}_{H}(E))\leq \mathsf{g}(\mu^{\gamma}_{H}(E)).\]

\begin{remark}
If $\kk\subset \kk_1$ is a field extension, then it is easy to see $\Phi^{\gamma}_{X, H}(x)=\Phi^{\gamma_{\kk_1}}_{X_{\kk_1}, H_{\kk_1}}(x)$ for any $x\in \RR$.
\end{remark}

In the relative setting, we can define the Le Potier function analogously.

\begin{definition}\label{def:relative-LP}
If $(f\colon X\to S, \cL, \gamma)$ is a triple as in Setup \ref{setup-gamma-relative} with $d\geq 2$ or $f$ is admissible in the sense of Definition \ref{def:adm}, then we define the \emph{relative Le Potier function} by
\[\Phi^{\gamma}_{X/S, \cL}(x)\coloneqq \limsup_{b\to x} \left\{\frac{\bv^{\gamma}_{\cL_s, 2}(E_s)}{\bv^{\gamma}_{\cL_s, 0}(E_s)} \colon E_s\text{ is }\mu^{\gamma_s}_{\cL_s}\text{-semistable and }\mu^{\gamma_s}_{\cL_s}(E_s)=b\text{ for some }s\in S\right\}.\]
We set \[U^{\gamma}_{X/S, \cL}\coloneqq \{(b,w)\in \RR^2\colon w>\Phi^{\gamma}_{X/S, \cL}(b)\}.\]
\end{definition}

It is clear that both $\Phi^{\gamma}_{X, H}$ and $\Phi^{\gamma}_{X/S, \cL}$ are upper semicontinuous and
\[\Phi^{\gamma}_{X/S, \cL}\geq \sup_{s\in S}\Phi^{\gamma_s}_{X_s, \cL_s}.\]



We say $(X, H, \gamma)$ \emph{has the standard BG function} if it has a BG function $\frac{x^2}{2}$. In other words, if we define a quadratic form $\bDelta$ on the graded lattice $\Lambda^{\gamma}_{H,\leq 2}$ by
\[\bDelta(v)\coloneqq v_1^2-2v_0v_2\]
for any $v\in \Lambda^{\gamma}_{H,\leq 2}$, then having the standard BG function means
\[\bDelta(\bv^{\gamma}_{H,\leq 2}(E))=(\bv^{\gamma}_{H, 1}(E))^2-2\bv^{\gamma}_{H, 0}(E)\bv^{\gamma}_{H, 2}(E)\geq 0\]
for any $\mu^{\gamma}_{H}$-semistable torsion-free sheaf $E\in \Coh(X)$.

We say a triple $(f\colon X\to S, \cL, \gamma)$ as in Setup \ref{setup-gamma-relative} has the standard BG function if 
\[\Phi^{\gamma}_{X/S, \cL}(x)\leq \frac{x^2}{2}\]
for any $x\in \RR$.

The main result in this section is the following existence result of BG functions.

\begin{theorem}\label{thm:exist-bg-function}
Fix a triple $(f\colon X\to S, \cL, \gamma)$ as in Setup \ref{setup-gamma-relative} with $d\geq 2$. Then there exists a constant $\mathsf{D}_{X/S, \cL, \gamma}\in \QQ$ so that
\[\Phi^{\gamma}_{X/S, \cL}(x)\leq \frac{x^2}{2} +\mathsf{D}_{X/S, \cL, \gamma}\]
for any $x\in \RR$.

In particular, if we fix $(X, H, \gamma)$ as in Setup \ref{setup-gamma} with $\gamma=1$ and set $\mathsf{D}_{X, H}\coloneqq \mathsf{D}_{X/\kk, H,\gamma}$, then for any $\mu_{H}$-semistable torsion-free sheaf $E\in \Coh(X)$, we have
\[(\bch_1(E).H^{n-1})^2-2\bch_0(E).H^n \bch_2(E).H^{n-2}\geq -2\mathsf{D}_{X, H} (\bch_0(E).H^n)^2.\]
\end{theorem}

\begin{proof}
We divide the proof into several steps.

\medskip

\noindent\textbf{Step 1}. We first prove the corresponding result in the absolute setting with $B_{\gamma}=0$. 

Fix $(X, H, \gamma)$ as in Setup \ref{setup-gamma} with $B_{\gamma}=0$. By Lemma \ref{lem-constant-chM}, we may assume that $\overline{\kk}=\kk$. We may also assume that $H$ is very ample, and let $X\hookrightarrow \PP(\mathrm{H}^0(\cO_X(H)))$ be the corresponding embedding. By \cite[\href{https://stacks.math.columbia.edu/tag/0B1P}{Tag 0B1P}]{stacks-project}, we can find a general projection $\pi\colon X\to \PP^n$ which is finite and surjective. Moreover, if we denote by $h$ the hyperplane class on $\PP^n$, then $\pi^*h=H$.

Since $X$ is lci in codimension $2$, we can find a closed subscheme $Z\subset \PP^n$ with $\codim_{\PP^n} Z\geq 3$ such that $\pi_U\colon\pi^{-1}(U)\to U$ is lci, where $U\coloneqq \PP^n\setminus Z$. Hence, $\pi^{-1}(U)$ is also lci. Note that when $n=2$, we have $U=\PP^2$.


For any cycle $\xi\in \CH_i(X)_{\QQ}$ with $0\leq i\leq 1$, we define $\tau_1.\xi\in \CH_{i+1}(X)_{\QQ}$ as the unique class corresponding to 
\[\td_1(\pi^{-1}(U)/U)\cap \xi|_{\pi^{-1}(U)}\in \CH_{i+1}(\pi^{-1}(U))_{\QQ}\]
under the identification $\CH_{i+1}(\pi^{-1}(U))_{\QQ}=\CH_{i+1}(X)_{\QQ}$ for $0\leq i\leq 1$. Similarly, if $\xi\in \CH_0(X)_{\QQ}$, we define $\tau_2.\xi\in \CH_{2}(X)_{\QQ}$ as the unique class corresponding to 
\[\td_2(\pi^{-1}(U)/U)\cap \xi|_{\pi^{-1}(U)}\in \CH_{2}(\pi^{-1}(U))_{\QQ}.\]

Applying Lemma \ref{lem:chi-property}(a) to $\pi^{-1}(U)\to U$, for any $E\in \Coh(X)$, we have
\[\bch_0(\pi_*E).h^{n}=\bch_0(E).H^n,\]
\[\bch_1(\pi_*E).h^{n-1}=\bch_1(E).H^{n-1}+(\tau_1.\bch_0(E)).H^{n-1},\]
and
\[\bch_2(\pi_*E).h^{n-2}=\bch_2(E).H^{n-2}+(\tau_1.\bch_1(E)).H^{n-2}+(\tau_2.\bch_0(E)).H^{n-2}.\]
Note that $\bch_0(E)=\sum^k_{i=1} x_i[X_i]$ for $x_i\in \ZZ_{\geq 0}$, where the sum runs over all irreducible components $X_i$ of $X$ with reduced structure. Therefore, by setting
\[t_j^+\coloneq \max\left\{\frac{\tau_j.[X_i].H^{n-j}}{[X_i].H^{n}}\right\}_{1\leq i\leq k}\]
and
\[t_j^-\coloneq \min\left\{\frac{\tau_j.[X_i].H^{n-j}}{[X_i].H^{n}}\right\}_{1\leq i\leq k}\]
for each $j=1,2$, we see that
\begin{equation}\label{eq:diff-bv1}
\bch_1(\pi_*E).h^{n-1}-\bch_1(E).H^{n-1}\in [t_1^- \bch_0(E).H^n, t_1^+ \bch_0(E).H^n]
\end{equation}
and
\begin{equation}\label{eq:diff-bv2}
\bch_2(\pi_*E).h^{n-2}-\bch_2(E).H^{n-2}-(\tau_1.\bch_1(E)).H^{n-2}\in \left[t_2^- \bch_0(E).H^n, t_2^+ \bch_0(E).H^n\right].
\end{equation}
In particular, 
\begin{equation}\label{eq:diff-slope}
\mu_h(\pi_*E)-\mu_H(E)\in [t_1^-,t_1^+].
\end{equation}

Now, assume that $E\in \Coh(X)$ is a torsion-free $\mu_H$-semistable sheaf. The next step is to bound the HN slopes of $\pi_*E$. Let $F\subset \pi_*E$ be the first HN factor. By adjunction, we have a non-zero map $\pi^*F\to E$, whose image is denoted by $0\neq F'\subset E$. Applying $\pi_*$, we then get morphisms
\[F\otimes \pi_*\cO_X\twoheadrightarrow \pi_*F'\hookrightarrow \pi_*E.\]
If we fix a surjection $\cO_{\PP^n}(-N)^{\oplus m}\twoheadrightarrow \pi_*\cO_X$, then by the $\mu_h$-semistability of $F$, we have $$\mu_h(\pi_*F')\geq \mu_h(F)-N.$$ Since $E$ is $\mu_H$-semistable, we also get $\mu_H(F')\leq \mu_H(E)$. Combining with \eqref{eq:diff-slope}, we obtain
\[\mu_h^+(\pi_*E)-N=\mu_h(F)-N\leq \mu_H(F')+t_1^+\leq \mu_H(E)+t_1^+\leq \mu_h(\pi_*E)+t_1^+-t_1^-.\]
Similarly, if $\pi_*E\twoheadrightarrow Q$ is the last HN factor, then from the adjunction, we get a non-zero map $E\to \pi^!Q$. Note that $\pi$ is finite, so
\[\pi_*\pi^!Q=R\cH om_{\PP^n}(\pi_*\cO_X, Q)\]
by (6) of \cite[\href{https://stacks.math.columbia.edu/tag/0AU3}{Tag 0AU3}]{stacks-project}. Therefore, we have $\cH^i(\pi^!Q)=0$ for $i<0$. In particular, we obtain a non-zero map $E\to \cH^0(\pi^!Q)$, whose image is denoted by $0\neq Q'\subset \cH^0(\pi^!Q)$. Applying $\pi_*$, we get morphisms
\[\pi_*E\twoheadrightarrow \pi_*Q'\hookrightarrow \pi_*\cH^0(\pi^!Q)\cong \cH om_{\PP^n}(\pi_*\cO_X, Q).\]
Using the surjection $\cO_{\PP^n}(-N)^{\oplus m}\twoheadrightarrow \pi_*\cO_X$ again, we obtain injection
\[\pi_*Q'\hookrightarrow \cH om_{\PP^n}(\cO_{\PP^n}(-N)^{\oplus m}, Q)=Q(N)^{\oplus m}.\]
Together with the $\mu_h$-semistability of $Q$, we then see that
\[\mu_h(\pi_*Q')\leq \mu_h(Q)+N,\]
which combines with \eqref{eq:diff-slope} and the $\mu_H$-semistability of $E$ implies
\[\mu_h(\pi_*E)+t_1^--t_1^+\leq \mu_H(E)+t_1^-\leq \mu_H(Q')+t_1^-\leq \mu_h(\pi_*Q')\leq \mu_h^-(\pi_*E)+N.\]
Therefore,
\begin{equation}\label{eq:diff-hn-slope}
0\leq \mu_h^+(\pi_*E)-\mu_h^-(\pi_*E)\leq 2(t_1^+-t_1^-+N).
\end{equation}

Now, we are ready to establish the desired inequality when $\gamma=1$. Since the classical Bogomolov--Gieseker inequality holds for $\mu_h$-semistable sheaves on $\PP^n$ (cf.~\cite{langer:positive-char}), applying it to each HN factor of $\pi_*E$ and using \eqref{eq:diff-hn-slope} and Lemma \ref{lem-key-lem}, we can conclude that
\begin{equation}\label{eq:bg-h}
(\bch_1(\pi_*E).h^{n-1})^2-2\bch_0(\pi_*E).h^n \bch_2(\pi_*E).h^{n-2}\geq -(\bch_0(\pi_*E).h^n)^2(t_1^+-t_1^-+N)^2.
\end{equation}
To prove a BG-type inequality on $X$, we may first assume that $\mu_H(E)\in [-1,1]$. Set $t_j\coloneqq \max\{|t_j^-|, |t_j^+|\}$ for each $j=1,2$. By \eqref{eq:diff-bv1}, we have
\begin{align*}
(\bch_1(E).H^{n-1})^2\geq& (\bch_1(\pi_*E).h^{n-1})^2-2t_1|\bch_1(\pi_*E).h^{n-1}|\bch_0(E).H^n\\
\geq& (\bch_1(\pi_*E).h^{n-1})^2-2t_1(t_1+1)(\bch_0(E).H^n)^2,
\end{align*}
where the last inequality uses $\mu_H(E)\in [-1,1]$ and \eqref{eq:diff-bv1}. Similarly, from \eqref{eq:diff-bv2}, we have
\begin{align*}
-&\bch_0(E).H^n \bch_2(E).H^{n-2}\\
&\geq -\bch_0(\pi_*E).h^n \bch_2(\pi_*E).h^{n-2}-t_2(\bch_0(E).H^n)^2\\
\,\,&+(\bch_0(E).H^n)(\tau_1.\bch_1(E)).H^{n-2}
\end{align*}
Combining these two inequalities with \eqref{eq:bg-h}, we see that
\begin{align*}
&(\bch_1(E).H^{n-1})^2-2\bch_0(E).H^n \bch_2(E).H^{n-2}\\
&\geq -(2t_1(t_1+1)+2t_2+(t_1^+-t_1^-+N)^2)(\bch_0(E).H^n)^2\\
&+2(\bch_0(E).H^n)(\tau_1.\bch_1(E)).H^{n-2}.
\end{align*}
Note that $E^{\vee}$ is $\mu_H$-semistable with $\bch_0(E^{\vee})=\bch_0(E)$ and $\bch_1(E^{\vee})=-\bch_1(E)$ by Lemma \ref{lem:dual-slope-stable}(d) and \ref{lem-dual-ch}. Since $\mathrm{LCI}(X/\kk)$ is Gorenstein, we also have
\[(E^H)^{\vee}|_{\mathrm{LCI}(X/\kk)}\cong E^{\vee \vee \vee}|_{\mathrm{LCI}(X/\kk)}\cong E^{\vee}|_{\mathrm{LCI}(X/\kk)}\]
by Lemma \ref{lem-S2-hull}, which implies $\bch_2(E^{\vee})=\bch_2(E^H)$ via Lemma \ref{lem-dual-ch} and $\codim_X(X\setminus \mathrm{LCI}(X/\kk))\geq 3$. Combining with Lemma \ref{lem-S2-hull}(c), we see that $\bch_2(E^{\vee})-\bch_2(E)$ is an effective cycle and get 
\[\bch_2(E^{\vee}).H^{n-2}\geq \bch_2(E).H^{n-2}.\]
Applying the above argument by replacing $E$ with $E^{\vee}$ gives
\begin{align*}
&(\bch_1(E).H^{n-1})^2-2\bch_0(E).H^n \bch_2(E).H^{n-2}\\
&\geq -(2t_1(t_1+1)+2t_2+(t_1^+-t_1^-+N)^2)(\bch_0(E).H^n)^2\\
&-2(\bch_0(E).H^n)(\tau_1.\bch_1(E)).H^{n-2}.
\end{align*}
Averaging the above two inequalities cancels the term involving $\tau_1.\bch_1(E)$, and we deduce that
\begin{equation}\label{eq:gamma=1-bg}
(\bch_1(E).H^{n-1})^2-2\bch_0(E).H^n \bch_2(E).H^{n-2}\geq -2\mathsf{D}_{X, H}(\bch_0(E).H^n)^2,
\end{equation}
where
\[\mathsf{D}_{X, H}\coloneqq t_1(t_1+1)+t_2+\frac{1}{2}(t_1^+-t_1^-+N)^2.\]
The same inequality also holds for any $\mu_H$-semistable $E\in \Coh(X)$, as tensoring a multiple of $\cO_X(H)$ does not change the $\mu_H$-semistability and the above quadratic inequality \eqref{eq:gamma=1-bg} by Lemma \ref{lem:chi-property}. Therefore, after defining
\[\mathsf{D}'_{X, H, \gamma}\coloneqq \mathsf{D}_{X, H}+\max\{0,s_{X,H,\gamma}\},\]
where
\[s_{X, H, \gamma}\coloneqq \max\left\{\frac{\gamma_2.[X_i].H^{n-2}}{[X_i].H^{n}}\right\}_{1\leq i\leq k},\]
we get
\[(\bch_1(E).H^{n-1})^2-2\bch_0(E).H^n( (\bch_2(E)+\gamma_2.\bch_0(E)).H^{n-2})\geq -2\mathsf{D}'_{X, H, \gamma}(\bch_0(E).H^n)^2.\]

\bigskip

\noindent\textbf{Step 2}. Next, we prove the general case where $B_{\gamma}\neq 0$ in the absolute setting.

By \eqref{eq:def-slope} we know that
\[0\leq \mu^{+}_H(E)-\mu^{-}_H(E)\leq 2t_3\]
for any $\mu^{\gamma}_H$-semistable sheaf $E\in \Coh(X)$, where 
\[t_3\coloneq \max\left\{\left|\frac{B_{\gamma}.[X_i].H^{n-1}}{[X_i].H^{n}}\right|\right\}_{1\leq i\leq k}.\]
Fix $l\in \ZZ_{\geq 2}$ so that $lB_{\gamma}\in \Pic(X)$. Then we also have
\[0\leq \mu_H^+(E(-lB_{\gamma}))-\mu_H^-(E(-lB_{\gamma}))\leq 2(l+1)t_3.\]
Applying Lemma \ref{lem-key-lem} and the result in Step 1 to the HN factors of $E$ and $E(-lB_{\gamma})$ with respect to $\mu_H$-semistability, we obtain
\begin{align*}
&(\bch_1(E).H^{n-1})^2-2\bch_0(E).H^n ((\bch_2(E)+\gamma_2.\bch_0(E)).H^{n-2})\\
&\geq -(2\mathsf{D}'_{X, H, \gamma}+t_3^2)(\bch_0(E).H^n)^2
\end{align*}
and
\begin{align*}
&(\bch_1(E(-lB_{\gamma})).H^{n-1})^2-2\bch_0(E(-lB_{\gamma})).H^n ((\bch_2(E(-lB_{\gamma}))+\gamma_2.\bch_0(E(-lB_{\gamma}))).H^{n-2})\\
&\geq -(2\mathsf{D}'_{X, H, \gamma}+(l+1)^2t_3^2)(\bch_0(E).H^n)^2.
\end{align*}
Therefore, using interpolation of quadratic polynomials, we get
\begin{align*}
\bDelta(\bv_{H,\leq 2}^{\gamma}(E))\geq &-\left(\frac{l-1}{l}(2\mathsf{D}'_{X, H, \gamma}+t_3^2)+\frac{1}{l}(2\mathsf{D}'_{X, H, \gamma}+(l+1)^2t_3^2)\right)(\bv_{H,0}^{\gamma}(E))^2\\
&-(l-1)\left((B_{\gamma}.\bch_0(E).H^{n-1})^2-\bv_{H,0}^{\gamma}(E)B_{\gamma}^2.\bch_0(E).H^{n-2}\right). 
\end{align*}
Thus, if we set 
\[q_{X, H, \gamma}\coloneqq \min\left\{\frac{B_{\gamma}^2.[X_i].H^{n-2}}{[X_i].H^{n}}\right\}_{1\leq i\leq k},\]
then we obtain
\[\bDelta(\bv_{H,\leq 2}^{\gamma}(E))\geq -2\mathsf{D}_{X, H, \gamma} (\bv_{H,0}^{\gamma}(E))^2,\]
where
\begin{align*}
\mathsf{D}_{X, H, \gamma}\coloneqq & \frac{1}{2}\left(\frac{l-1}{l}(2\mathsf{D}'_{X, H, \gamma}+t_3^2)+\frac{1}{l}(2\mathsf{D}'_{X, H, \gamma}+(l+1)^2t_3^2)+(l-1)(t_3^2-q_{X, H, \gamma})\right)\\
=& \mathsf{D}'_{X, H, \gamma}+(l+1)t_3^2-\frac{l-1}{2}q_{X, H, \gamma}.
\end{align*}
This completes the proof of the theorem in the absolute setting.

\medskip

\noindent \textbf{Step 3}. Finally, we establish the quadratic inequality in the relative setting.

After replacing $\cL$ by its multiples and stratifying $S$, we may assume that there exists a finite surjective morphism $\pi_S\colon X\to \PP^n_S$ so that $\pi^*_S\cO_{\PP^n_S}(1)=\cL$. Therefore, the number $N$ in the first step can be chosen uniformly for each geometric fiber of $X\to S$.

By Chow’s finiteness theorem (cf.~\cite[Exercise I.3.28, Theorem I.6.3]{kollar:book-rational-curve}), there are only finitely many possibilities of the Hilbert polynomials among all irreducible components (with reduced structure) of all geometric fibers of $X\to S$. Using the boundedness of Hilbert schemes, we can find a morphism $T\to S$ of finite type with finitely many closed subschemes $Y_1,\cdots,Y_r$ of $X_T$ flat over $T$ such that every irreducible component of the fiber over every geometric point $t\to T$ lies in the set $\{(Y_1)_t,\cdots,(Y_r)_t\}$. Therefore, from Lemma \ref{lem-constant-chM}, we know that the constants $t_1, t_2, t_3$, $s_{X_t, \cL_t, \gamma_t}$, and $q_{X_t, \cL_t, \gamma_t}$ used in the definitions of $\mathsf{D}_{X_t, \cL_t, \gamma_t}$ and $\mathsf{D}'_{X_t, \cL_t, \gamma_t}$ are all uniformly bounded when geometric points $t\to S$ vary. Thus, there exists a constant $\mathsf{D}_{X/S, \cL, \gamma}\geq 0$ so that 
\[\sup\{\mathsf{D}_{X_t, \cL_t, \gamma_t}\colon t\to S\text{ is a geometric point}\}\leq \mathsf{D}_{X/S, \cL, \gamma}.\]
This finishes the proof.
\end{proof}

\begin{lemma}\label{lem-key-lem}
Let $r_i,c_i,t_i\in \RR$ for $1\leq i\leq m$. We set $r\coloneqq\sum^m_{i=1} r_i$, $c\coloneqq\sum^m_{i=1} c_i$, and $t\coloneqq\sum^m_{i=1} t_i$. Assume that $r_i\neq 0$ for each $1\leq i\leq m$, then
\[c^2-2rt=\sum^m_{i=1} \frac{r}{r_i}(c_i^2-2r_it_i)-\sum_{j<k}r_jr_k\left(\frac{c_j}{r_j}-\frac{c_k}{r_k}\right)^2.\]
Assume furthermore that $r_i>0$ for each $1\leq i\leq m$ and $\mu_1\geq \mu_2\geq \dots\geq \mu_{m}$, where $\mu_l\coloneqq\frac{c_l}{r_l}$. Then
\[c^2-2rt\geq\sum^m_{i=1} \frac{r}{r_i}(c_i^2-2r_it_i)-\frac{1}{4}r^2(\mu_1-\mu_m)^2.\]
\end{lemma}

\begin{proof}
The first equality follows from a straightforward checking. For the inequality, it suffices to prove
\[\sum_{j<k}r_jr_k(\mu_j-\mu_k)^2\leq \frac{1}{4}r^2(\mu_1-\mu_m)^2.\]
Note that
\[\sum_{j<k}r_jr_k(\mu_j-\mu_k)^2=r^2\sum^m_{i=1}\frac{r_i}{r}(\mu_i-\mu)^2,\]
where $\mu\coloneqq \frac{c}{r}$, then the result can be deduced from Popoviciu’s inequality on weighted variance:
\[\sum^m_{i=1}\frac{r_i}{r}(\mu_i-\mu)^2\leq \frac{1}{4}(\mu_1-\mu_m)^2.\]
\end{proof}

\subsection{BG inequalities for normal varieties}


Although Theorem \ref{thm:exist-bg-function} gives a BG-type inequality, it is not sharp. In practice, the following inequality for normal varieties is useful.

\begin{theorem}\label{thm-bg-normal}
Fix $(X, H, \gamma)$ as in Setup \ref{setup-gamma}. Assume that $n\geq 2$ and $X$ is geometrically normal and geometrically integral. If $n=2$ or $d\geq 2$, then there exists a constant $C_{X, H}\geq 0$ such that if $-H^{n-2}.\gamma_{2}\geq C_{X,H}$, then
\[\bDelta(\bv^{\gamma}_{H,\leq 2}(E))\geq 0\]
for any $\mu^{\gamma}_{H}$-semistable torsion-free sheaf $E\in \Coh(X)$. If $\mathrm{char}(\kk)=0$ and $X_{\overline{\kk}}$ has rational singularities, then we can take $C_{X,H}=0$.
\end{theorem}

When $\kk=\overline{\kk}$ is of characteristic zero and $X$ has rational singularities, the result corresponding to $C_{X, H}=0$ can be easily deduced from the restriction theorem \cite[Theorem 0.5]{langer:higgs-normal} and the result for surfaces \cite[Theorem 0.1]{langer:normal-surface} (see also Lemma \ref{lem-bg-normal-surface}). In the following, we closely follow the argument in \cite{langer:positive-char,langer:on-boundedness} to prove a more general inequality.

\begin{remark}
Note that if $B_{\gamma}=0$ and $X$ is integral, then we have
\[\bDelta(\bv^{\gamma}_{H,\leq 2}(E))=(\bch_1(E).H^{n-1})^2-2\rk(E)H^n (\bch_2(E).H^{n-2})-2(\rk(E))^2H^n(\gamma_2.H^{n-2})\]
for any coherent sheaf $E$ on $X$. In this situation, the result in Theorem \ref{thm-bg-normal} becomes
\[(\bch_1(E).H^{n-1})^2-2\rk(E)H^n (\bch_2(E).H^{n-2})+2H^n C_{X, H}(\rk(E))^2\geq 0.\]
So the twist $\gamma_2$ in Setup \ref{setup-gamma} should be regarded as a correction term of classical BG inequality in the singular setting.
\end{remark}


Before proving BG inequalities, we need the notion of slope-stability for multi-polarizations.

\begin{definition}
Let $X$ be a normal projective variety of dimension $n\geq 2$ over an algebraically closed field $\kk$ with a collection $L_{\bullet}=(L_1,\dots, L_{n-1})$ of nef line bundles on $X$. If $L_1=\dots=L_{n-1}=H$, we write simply $H$ for this collection.

We fix a torsion-free sheaf $E$ on $X$ of rank $r$. We define the \emph{$L_{\bullet}$-slope} of $E$ by
\[\overline{\mu}_{L_{\bullet}}(E)\coloneqq\frac{L_1.\dots L_{n-1}.\bch_1(E)}{r}.\]
We say $E$ is \emph{$\overline{\mu}_{L_{\bullet}}$-(semi)stable} if for any non-trivial subsheaf $F\subset E$ of rank $<r$, we have $$\overline{\mu}_{L_{\bullet}}(F)(\leq) \overline{\mu}_{L_{\bullet}}(E).$$
If $X$ is lci in codimension $\geq 2$, we also define the \emph{$L_{\bullet}$-discriminant} by
\[L_2\dots L_{n-1}.\Delta(E)\coloneqq L_2\dots L_{n-1}.(\bch_1(E))^2-2r\cdot L_2\dots L_{n-1}.\bch_2(E) \in \QQ.\]
For a constant $C\in \mathbb{R}$, we say that \emph{$C$-Bogomolov--Gieseker (BG) inequality holds} for the collection of nef line bundles $L_{\bullet}=(L_1,\dots, L_{n-1})$ if $L_2\dots L_{n-1}.\Delta(E)+C r^2\geq 0$ for every $\overline{\mu}_{L_{\bullet}}$-semistable sheaf $E$ on $X$ of rank $r$.
\end{definition}
Note that $\overline{\mu}_H(E)=H^n\mu_H(E)$, so $\overline{\mu}_H$-(semi)stability is the same as $\mu_H$-(semi)stability.

Using a similar calculation as in \cite[Proposition 3.1]{langer:on-boundedness}, we know that changing the first line bundle in a polarization does not affect the BG inequality.

\begin{lemma}\label{lem-change-polarization}
Let $X$ be a normal projective variety of dimension $n\geq 2$ over an algebraically closed field $\kk$ which is lci in codimension $\geq 2$. Fix a collection $L_{\bullet}=(L_1,\dots, L_{n-1})$ of nef line bundles on $X$ such that the product $L_1\dots L_{n-1}$ is numerically nontrivial. If $C$-BG inequality holds for $\overline{\mu}_{L_{\bullet}}$-semistable sheaves of rank $<r$ for a constant $C\in \mathbb{R}$, then it holds for $\overline{\mu}_{(M, L_2,\dots, L_{n-1})}$-semistable sheaves of rank $<r$ such that $M$ is nef and $M.L_2\dots L_{n-1}$ is numerically nontrivial.
\end{lemma}

Our starting point is the following BG inequality, proved by Langer. By Lemma \ref{lem-hodeg-index}, it implies the case $n=2$ of Theorem \ref{thm-bg-normal}.

\begin{lemma}[{\cite[Theorem 0.1]{langer:normal-surface}}]\label{lem-bg-normal-surface}
Let $X$ be a normal projective surface over an algebraically closed field $\kk$. Then for any ample line bundle $H$, there exists a constant $C_{H}\geq 0$ such that $C_{H}$-BG inequality holds for $H$. If $\mathrm{char}(\kk)=0$ and $X$ has rational singularities, then we can take $C_{H}=0$.
\end{lemma}


To generalize this to higher dimensions, we need some lemmas.

\begin{lemma}\label{lem:intersect-pullback}
Let $X$ be a normal projective variety of dimension $n\geq 2$ over an algebraically closed field $\kk$ and $Z\hookrightarrow X$ be a complete intersection subscheme cut out by two divisors in a very ample linear series $|H|$. Denote by $p\colon Y\to X$ the blow-up morphism of $X$ along $Z$. If $D_1, D_2$ are Weil divisors on $X$, then for general $Z$, $Y$ is normal and we have
\[H^{n-2}.D_1.D_2=p^*H^{n-2}.p^*D_1.p^*D_2.\]
Here and in the following, the pullback of Weil divisors means pullback as cycles.
\end{lemma}

\begin{proof}
From the definition of the intersection form, it suffices to consider the case when $D=D_1=D_2$. If $Z$ is defined by $s_1, s_2\in \mathrm{H}^0(\cO_X(H))$, then $Y$ is a general divisor in $X\times \PP^1$ cut out by $xs_1-ys_2\in \mathrm{H}^0(\cO_X(H)\boxtimes \cO_{\PP^1}(1))$, where $x,y$ are homogeneous coordinates of $\PP^1$. In particular, by Bertini's theorem, $Y$ is normal and $p$ is lci. Denote by $p_i$ the projection to the $i$-th factor of $X\times \PP^1$.

Since $\cO_X(H)\boxtimes \cO_{\PP^1}(1)$ is very ample and $(p_1^{*}D)|_ Y=p^*D$ for general choices of $s_1,s_2$, by \cite[Corollary 2.8]{langer:intersection-normal}, we have
\[p^*H^{n-2}.p^*D^2=p_1^*H^{n-2}.\cO_X(H)\boxtimes \cO_{\PP^1}(1).p_1^*D^2.\]
By the multilinearity of intersection form, we obtain
\[p^*H^{n-2}.p^*D^2=p_1^*H^{n-1}.p_1^*D^2+p_1^*H^{n-2}.p_2^*\cO_{\PP^1}(1).p_1^*D^2.\]
Note that 
\[p_1^*H^{n-1}.p_1^*D^2=\lim_{m\to \infty} 2\frac{\chi(X\times \PP^1, c_1(p_1^*\cO_X(H))^{n-1}\cdot \cO_{mp_1^{*}(D)})}{m^2}=0,\]
so we get
\[p^*H^{n-2}.p^*D^2=p_1^*H^{n-2}.p_2^*\cO_{\PP^1}(1).p_1^*D^2.\]
Applying Lemma \ref{lem:intersect-pairing-restrict} to $X\times \mathrm{pt}\in |p_2^*\cO_{\PP^1}(1)|$, we finally get 
\[p^*H^{n-2}.p^*D^2=p_1^*H^{n-2}.p_2^*\cO_{\PP^1}(1).p_1^*D^2=H^{n-2}.D^2.\]
\end{proof}

\begin{lemma} \label{lem:blow-up-1}
Let $X$ be a normal projective variety of dimension $n\geq 2$ over an algebraically closed field $\kk$ and $Z\hookrightarrow X$ be a regular closed embedding. Denote by $p\colon Y\to X$ the blow-up morphism of $X$ along $Z$ such that $Z$ is equidimensional and $\codim_X(Z)=m\geq 1$. Then the morphism $p$ is lci.

Moreover, we have

\begin{enumerate}
    \item for any $E\in \Coh(Y)$, we have $R^i p_*E=0$ for $i<0$ and $i>m-1$; moreover, $R^i p_*E$ is supported on $Z$ for each $i>0$,

    \item if $E\in \Coh(Y)$ is torsion-free, then so is $p_*E$,

    \item if $E\in \Coh(X)$ is torsion-free and reflexive, then $E'\coloneqq (p^*E)^{\vee \vee}$ satisfies $p_*E'\subset E$, and

    \item if $X$ is lci in codimension $d\geq 1$ and $$\codim_X ( Z\cap (X\setminus \mathrm{LCI}(X/\kk)))\geq d+1+m-1,$$ then $Y$ is also lci in codimension $d$.
\end{enumerate}

\end{lemma}

\begin{proof}
It is clear from the construction that $p$ is lci. Now, we may assume that $m\geq 2$. Part (a) follows from the fact that $\dim p^{-1}(x)\leq m-1$ for any $x\in X$. If $E\in \Coh(Y)$ is torsion-free, and there exists a non-zero torsion subsheaf $F\subset p_*E$, then $p^*F$ is also torsion and we have a non-zero map $p^*F\to E$, which is a contradiction. Therefore, $p_*E$ is torsion-free. Part (c) follows from part (b) and Lemma \ref{lem-general-S2}(c), since $p_*E'|_{X\setminus Z}=E|_{X\setminus Z}$. 

Finally, for part (d), by \cite[\href{https://stacks.math.columbia.edu/tag/09RL}{Tag 09RL}]{stacks-project}, \cite[\href{https://stacks.math.columbia.edu/tag/069J}{Tag 069J}]{stacks-project}, and the fact that $p$ is lci, we have $p^{-1}(\mathrm{LCI}(X/\kk))=\mathrm{LCI}(Y/\kk)$. Now our assumption implies $\codim_Y(Y\setminus \mathrm{LCI}(Y/\kk))\geq d+1$, so $Y$ is also lci in codimension $d$.
\end{proof}

\begin{lemma}\label{lem:blow-up-2}
Let $X$ be a normal projective variety of dimension $n\geq 2$ over an algebraically closed field $\kk$ which is lci in codimension $1$ and $Z\hookrightarrow X$ be a complete intersection subscheme cut out by two general divisors in a very ample linear series $|H|$. Denote by $p\colon Y\to X$ the blow-up morphism of $X$ along $Z$. Then for any $E\in \Coh(Y)$, we have

\begin{enumerate}
    \item $\bch_0(Rp_*E)=\bch_0(p_*E)=p_*\bch_0(E)$,

    \item $H^{n-1}.\bch_1(Rp_*E)=H^{n-1}.\bch_1(p_*E)=(p^*H)^{n-1}.\bch_1(E)$,

    \item if $E$ is torsion-free and $\overline{\mu}_{p^*H}$-semistable, then $p_*E$ is torsion-free and $\overline{\mu}_{H}$-semistable, and

    \item if $F\in \Coh(X)$ is torsion-free and $S_2$, then $p^*F=Lp^*F$ is torsion-free.
\end{enumerate}

\end{lemma}

\begin{proof}
Since $Z$ is general, $Y$ is normal by Lemma \ref{lem:intersect-pullback}. Let $D$ be the exceptional divisor of $p$. Note that $\td_{p,1}=-\frac{1}{2}D$, so parts (a) and (b) directly follow from Lemma \ref{lem:blow-up-1}(a) and Lemma \ref{lem:chi-property}, and part (c) follows from part (a) and (b) as well as Lemma \ref{lem:blow-up-1}(b).

Assume that $Z$ is defined by $s_1, s_2\in \mathrm{H}^0(\cO_X(H))$, then $Y$ is a divisor in $X\times \PP^1$ cut out by $$xs_1-ys_2\in \mathrm{H}^0(\cO_X(H)\boxtimes \cO_{\PP^1}(1)),$$
where $x,y$ are homogeneous coordinates of $\PP^1$. Denote by $p_1$ the projection from $X\times \PP^1$ to $X$. Then $Lp_1^*F$ is also a torsion-free $S_2$ sheaf. Therefore, $(Lp_1^*F)|_Y=Lp^*F$ is a torsion-free sheaf as desired.
\end{proof}

\begin{lemma}\label{lem:blow-up-3}
Let $X$ be a normal projective variety of dimension $n\geq 3$ over an algebraically closed field $\kk$ which is lci in codimension $2$ and $Z\hookrightarrow X$ be a complete intersection subscheme cut out by two general divisors in a very ample linear series $|H|$. Denote by $p\colon Y\to X$ the blow-up morphism of $X$ along $Z$. If $C$-BG inequality holds for all $\overline{\mu}_H$-semistable sheaves of rank $r$, then $C$-BG inequality holds for all $\overline{\mu}_{p^*H}$-semistable sheaves of rank $r$.
\end{lemma}

\begin{proof}
Since $Z$ is general, $Y$ is normal by Lemma \ref{lem:intersect-pullback} and lci in codimension $2$ by Lemma \ref{lem:blow-up-1}(d). Let $E$ be a torsion-free $\overline{\mu}_{p^*H}$-semistable sheaf of rank $r$. By Lemma \ref{lem:blow-up-2}(c), we know that $R^0 p_*E$ is torsion-free of rank $r$ and $\overline{\mu}_{H}$-semistable for any $\overline{\mu}_{p^*H}$-semistable torsion-free $E\in \Coh(Y)$ of rank $r$.

Let $D$ be the exceptional divisor of $p$. By \cite[Proposition 6.7(e)]{fulton:intersection-theory}, we may assume that $\bch_1(E)=p^*L+zD$ for $z\in \ZZ$ and $L\in \CH^1(X)$. Using Lemma \ref{lem:chi-property}(d), up to twisting $E$ by a multiple of $\cO_Y(D)$, we may assume that $0\leq z <r=\rk(E)$.

According to Lemma \ref{lem:blow-up-1}, we know that $R^i p_*E=0$ for $i\notin \{0,1\}$. Moreover, $R^1 p_*E$ is supported in codimension $\geq 2$. Therefore, it suffices to show
\begin{equation*}
p^*H^{n-2}.\Delta(E)\geq H^{n-2}.\Delta(R p_*E).
\end{equation*}

A direct computation gives $\td_{p,1}=-\frac{1}{2}D$ and $\td_{p,2}=\frac{1}{6}p^*H\cdot D$. Therefore, by applying Lemma \ref{lem:chi-property}(a), we get $$\bch_1(R p_*E)=p_*\left(\bch_1(E)-\frac{r}{2}D\right)=L$$ and $$\bch_2(R p_*E)=p_*\left(\bch_2(E)-\frac{1}{2}D. \bch_1(E)+\frac{r}{6}p^*H\cdot D\right).$$ Using $p^*H^{n-2}.p^*L.D=p^*H^{n-1}.D=0$, we obtain
\[H^{n-2}.\Delta(R p_*E)=H^{n-2}.L^2+p^*H^{n-2}.(-2r\bch_2(E)+rD.\bch_1(E))\]
\[=H^{n-2}.L^2+p^*H^{n-2}.(-2r\bch_2(E)+zrD^2)\]
\[=p^*H^{n-2}.(p^*L^2+zrD^2-2r\bch_2(E))\]
\[\leq p^*H^{n-2}.(p^*L^2+z^2D^2-2r\bch_2(E))=p^*H^{n-2}.\Delta(E),\]
where the third equality follows from Lemma \ref{lem:intersect-pullback} and the last inequality follows from $0\leq z <r$ and $p^*H^{n-2}.D^2=-H^n<0$.
\end{proof}

Now we can prove a version of Theorem \ref{thm-bg-normal} using $\Delta(-)$.

\begin{proposition}\label{prop-bg-normal-over-k}
Let $X$ be a normal projective variety of dimension $n\geq 2$ over an algebraically closed field $\kk$ which is lci in codimension $2$. Then for any ample line bundle $H$, there exists a constant $C_{H}\geq 0$ such that $C_{H}$-BG inequality holds for $H$. If $\mathrm{char}(\kk)=0$ and $X$ has rational singularities, then we can take $C_{H}=0$. 
\end{proposition}

\begin{proof}
The proof is motivated by \cite[Theorem 3.4]{langer:on-boundedness}. When $n=2$, this is Lemma \ref{lem-bg-normal-surface}. Now, assume that $n\geq 3$ and the statement holds in dimension $\leq n-1$.

Up to replacing $H$ with a multiple of itself, we can assume that $H$ is very ample. We claim that we can take $C_{H}\coloneqq C_{H|_T}$, where $T$ is a general normal integral complete intersection surface by divisors in $|H|$. Note that such $T$ always exists by \cite[Corollary 3.4.14]{flenner:join-and-intersection}. If $\mathrm{char}(\kk)=0$ and $X$ has rational singularities, then $T$ also has rational singularities by \cite[Theorem 5.42]{kollar-mori}, and we have $C_{H}=C_{H|_T}=0$ by Lemma \ref{lem-bg-normal-surface}. Note that $C_{H}$ is independent of the choice of general $T$ and only depends on $|H|$ by a similar argument in \cite{koseki:bg-positive-char}. Indeed, for a family $X_1\to C$ of divisors in $|H|$ over a smooth curve $C$, we can first take a resolution $X_2$ of $X_1$ by \cite{cossart:resolution-3fold}, which is still flat over $C$ and also resolves the generic fiber of $X_1\to C$. After shrinking the base curve without changing the generic fiber, we may assume that $X_2\to C$ is smooth projective. Then the argument in \cite[Definition-Proposition 4.1]{koseki:bg-positive-char} applies verbatim.

We prove by induction on the rank of $\overline{\mu}_{H}$-semistable sheaves.  Let $E$ be a $\overline{\mu}_{H}$-semistable sheaf of rank $r$ on $X$. By Lemma \ref{lem-dual-ch} and \ref{lem:dual-slope-stable}, we may assume that $E$ is reflexive. When $r=1$, note that its restriction to a general divisor in $|H|$ remains semistable, and the result follows from the fact that the restriction does not change the value of the $H$-discriminant by Lemma \ref{lem-chM-general}(b) and Theorem \ref{thm-mumford-intersection}(e). Now, we assume that $r>1$ and the result holds for any $\overline{\mu}_{H}$-semistable sheaf of rank $<r$.

Let $L\subset |H|$ be a general pencil. Let $q\colon Y\to X$ be the blow-up of the base locus of $L$ and $p\colon Y\to L\cong \PP^1$ is the natural projection, which is the restriction of the universal family over $|H|$ to $L$. Therefore, $Y$ is also a normal projective variety over $\kk$ which is lci in codimension $2$, $q$ is lci by Lemma \ref{lem:blow-up-1}, and $p$ is flat projective. Moreover, $q^*E=Lq^*E$ is torsion-free by Lemma \ref{lem:blow-up-2} and $E'\coloneqq (q^*E)^{\vee \vee}$ is a reflexive sheaf on $Y$ of rank $r$ with $q_*E'\subset E$ by Lemma \ref{lem:blow-up-1}.

Let $0=E_0\subset E_1\subset \cdots \subset E_m=E'$ be the relative HN filtration with respect to the collection $(p^*\oh_{L}(1), q^*H,\dots,q^*H)$ and $p$. We can assume that $m>1$, otherwise $E'|_{p^{-1}(b)}$ is $\overline{\mu}_{H|_D}$-semistable for a general point $b=[D]\in L$, and the result follows from Lemma \ref{lem:chi-property}(c) and the induction hypothesis on dimensions. As explained in \cite[3.9]{langer:positive-char}, it is also the (absolute) HN filtration with respect to the collection $(p^*\oh_{L}(1), q^*H,\dots,q^*H)$. Let $F_i\coloneqq E_i/E_{i-1}$, then it is a $(p^*\oh_{L}(1), q^*H,\dots,q^*H)$-semistable torsion-free sheaf on $Y$ of rank $r>r_i>0$ with the $(p^*\oh_{L}(1), q^*H,\dots,q^*H)$-slope $\mu_i$. Moreover, $(F_i)|_D$ is $\overline{\mu}_{H|_D}$-semistable for a general divisor in $L$ as it is a factor of the relative HN filtration.

By \cite[Proposition 6.7(e)]{fulton:intersection-theory}, we can write $\bch_1(F_i)=q^*M_i+b_i N$, where $M_i$ is a Weil divisor on $X$, $N$ is the exceptional divisor of $q$, and $b_i\in \ZZ$. Note that $\sum_i b_i=0$ as $\bch_1(E')=q^*\bch_1(E)$. Then
\[\mu_i=\frac{H^{n-1}.M_i+b_iH^n}{r_i}\]
as in \cite[3.9]{langer:positive-char}. As we assume $C_{H}$-BG inequality holds for $\overline{\mu}_H$-semistable sheaves of rank $<r$, by Lemma \ref{lem:blow-up-3} and Lemma \ref{lem-change-polarization}, we get
\[q^*H^{n-2}.\Delta(F_i)+C_{H}r_i^2\geq 0.\]
We set $a_i\coloneqq H^{n-1}.M_i$ and $d\coloneqq H^n$. Then as in \cite[3.9]{langer:positive-char}, since $L^0q^*E=Lq^*E\subset E'$ has a cokernel supported in codimension $\geq 2$, we have
\[\frac{dH^{n-2}.\Delta(E)}{r}+C_{H}dr=\frac{dq^*H^{n-2}.\Delta(Lq^*E)}{r}+C_{H}dr\geq \frac{dq^*H^{n-2}.\Delta(E')}{r}+C_{H}dr\]
\[=\sum d\left(\frac{q^*H^{n-2}.\Delta(F_i)}{r_i}+C_{H}r_i\right)-\frac{d}{r}\sum_{i<j}r_ir_jq^*H^{n-2}.\left(\frac{\bch_1(F_i)}{r_i}-\frac{\bch_1(F_j)}{r_j}\right)^2\]
\[\geq \frac{d}{r}\sum_{i<j}r_ir_j \left( d\left(\frac{b_i}{r_i}-\frac{b_j}{r_j}\right)^2-\left(\frac{M_i}{r_i}-\frac{M_j}{r_j}\right)^2.H^{n-2}\right)\]
\[\geq \frac{1}{r}\sum_{i<j}r_ir_j \left( d^2\left(\frac{b_i}{r_i}-\frac{b_j}{r_j}\right)^2-\left(\frac{a_i}{r_i}-\frac{a_j}{r_j}\right)^2\right)\]
\[=2d\sum b_i\mu_i -\frac{1}{r}\sum_{i<j}r_ir_j(\mu_i-\mu_j)^2,\]
where the first equality follows from Lemma \ref{lem:intersect-pullback} and Lemma \ref{lem:chi-property}(c), and the last inequality follows from Lemma \ref{lem-hodeg-index}. As $q_*E_i\subset E$, we get
\[\frac{\sum_{j\leq i} a_j}{\sum_{j\leq i}r_j}\leq \mu\coloneqq\overline{\mu}_{H}(E).\]
Then $\sum_{j\leq i} db_j\geq \sum_{j\leq i} r_j(\mu_j-\mu)$. Using $\sum_i b_i=0$, we obtain
\[d\sum b_i\mu_i\geq \frac{1}{r}\sum_{i<j}r_ir_j(\mu_i-\mu_j)^2,\]
which implies 
\[\frac{dH^{n-2}.\Delta(E)}{r}+C_{H}dr\geq \frac{1}{r}\sum_{i<j}r_ir_j(\mu_i-\mu_j)^2\geq 0.\]
This completes the induction argument.
\end{proof}

\begin{proof}[{Proof of Theorem \ref{thm-bg-normal}}]
By Lemma \ref{lem-constant-chM} and \ref{thm-mumford-intersection}, we may assume that $\overline{\kk}=\kk$. Using Lemma \ref{lem-hodeg-index}, we have 
\[\frac{1}{H^n}\bDelta(\bv^{\gamma}_{H,\leq d}(E))\geq H^{n-2}.((\bch_1^{\gamma}(E))^2-2r\bch_2^{\gamma}(E))\]
\[=H^{n-2}.\left((\bch_1(E)-B_{\gamma}.\bch_0(E))^2-2r\left(\bch_2(E)-B_{\gamma}.\bch_1(E)+\frac{1}{2}B^2_{\gamma}.\bch_0(E)\right)\right)-2H^{n-2}.\gamma_2.r^2\]
\[=H^{n-2}.\Delta(E)-2H^{n-2}.\gamma_2.r^2\]
for any coherent sheaf $E$ on $X$ of rank $r$. Now the result follows from Remark \ref{rmk-different-stability}, and  Lemma \ref{lem-bg-normal-surface}, and Proposition \ref{prop-bg-normal-over-k} by taking 
\[C_{X, H}\coloneqq \frac{1}{2}C_H.\]
\end{proof}

\section{Tilt-stability on projective varieties}\label{sec:tilt-3}

In this section, we discuss the construction of tilt-stability from slope-stability and its family version. In particular, we generalize the results in \cite[Section 3]{bayer2016space}, \cite[Section 2]{bayer2017stability}, and \cite[Section 24, 25]{BLMNPS21} to our situation, see e.g. Proposition \ref{prop-rotate-slope-real-b}, Theorem \ref{thm-bms}, Theorem \ref{thm:tilt-HN-structure}, and Proposition \ref{prop-rotate-tilt}.


We first fix some abbreviated notation. Let $(X, H, \gamma)$ be a triple as in Setup \ref{setup-gamma}.

\begin{itemize}
    \item By abuse of notation, we write $\Lambda_{\leq i}$ and $\Lambda_i$ for $\Lambda^{\gamma}_{H, \leq i}$ and $\Lambda^{\gamma}_{H, i}$, respectively. If $X\to \Spec(\kk)$ is admissible in the sense of Definition \ref{def:adm}, then we denote by $\Lambda$ the full lattice $\Lambda^{\gamma}_{H}=\Lambda^{\gamma}_{H,\leq n}$. We also write $\bv_{\leq i}$ and $\bv_i$ for $\bv^{\gamma}_{H,\leq i}$ and $\bv^{\gamma}_{H,i}$, respectively.

    \item We denote by $\sigma$ the weak stability condition $\sigma^{\gamma}_H$. The corresponding slope function $\mu^{\gamma}_H(-)$ is denoted by $\mu(-)$. When $\bv_2$ is defined (e.g.~$d\geq 2$ or $X$ is a normal surface), we set $$\bDelta(E)\coloneqq \bDelta(\bv_{\leq 2}(E)).$$

\end{itemize}

Similarly, let $(f\colon X\to S, \cL, \gamma)$ be a triple as in Setup \ref{setup-gamma-relative}.

\begin{itemize}
    \item We write $\Lambda_{\leq i}$ and $\Lambda_i$ for $\Lambda^{\gamma}_{\cL, \leq i}$ and $\Lambda^{\gamma}_{\cL, i}$, respectively. If $f$ is admissible in the sense of Definition \ref{def:adm}, then we denote by $\Lambda$ the full lattice $\Lambda^{\gamma}_{\cL}=\Lambda^{\gamma}_{\cL,\leq n}$. We also write $\bv_{\leq i}$ and $\bv_i$ for $\bv^{\gamma}_{\cL,\leq i}$ and $\bv^{\gamma}_{\cL,i}$, respectively.

    \item We denote by $\sigma_t$ the weak stability condition $\sigma^{\gamma_t}_{\cL_t}$ on $\Db(X_t)$ for any point $t\to S$. The corresponding slope function is still denoted by $\mu(-)$, as it is independent of the choice of points.

    \item We write $\underline{\sigma}\coloneqq (\sigma_s)_{s\in S}$. 

    \item We write $\sigma_C=(\Coh(X_C), Z_C)$ for the weak HN structure $\sigma^{\gamma}_{\cL, C}$ in \eqref{eq:hn-structure-slope}, where $$Z_C= (Z_K, Z_{C\text{-}\tor})\coloneqq (Z^{\gamma_K}_{\cL_K}, Z^{\gamma}_{\cL, C\text{-}\tor}).$$
\end{itemize}

\subsection{Rotating slope-stability}

As a limiting case of tilt-stability, we start with the rotation of slope-stability.

Recall that given a weak stability condition, we can rotate its heart and central charge to get a new pair as in Section \ref{subsec-tilting-property}. In the case of $(X, H, \gamma)$ as in Setup \ref{setup-gamma}, for any $b\in \RR$, we get a pair $$\sigma^b=(\cA^b(X), Z^b)\coloneqq \sigma^{\gamma, b}_{H}=(\Coh^{b}_{H,\gamma}(X), Z^b)$$ from the weak stability condition $\sigma=\sigma^{\gamma}_{H}=(\Coh(X), Z)$ with respect to $\Lambda_{\leq 1}$, where
\[Z^{b}\coloneqq \Im Z+\mathfrak{i}(-\Re Z-b\Im Z),\]
$$\cA^b(X)\coloneqq \Coh^{b}_{H,\gamma}(X)= \langle \cF^b_{H,\gamma}[1], \cT^b_{H,\gamma} \rangle,$$
\[\cT^b\coloneqq \cT^{b}_{H,\gamma}=\langle E\in \Coh(X)\colon E \text{ is } \mu^{\gamma}_{H}\text{-semistable with }\mu^{\gamma}_{H}(E)>b \rangle,\]
and
\[\cF^b\coloneqq \cF^{b}_{H,\gamma}=\langle E\in \Coh(X)\colon E \text{ is } \mu^{\gamma}_{H}\text{-semistable with }\mu^{\gamma}_{H}(E)\leq b \rangle.\]

\begin{lemma}\label{lem-rotate}
The pair $\sigma^b=(\cA^b(X), Z^b)$ is a weak stability condition on $\Db(X)$ with respect to $\Lambda_{\leq 1}$ for any $b\in \mathbb{R}$. Moreover, $(\cA^b(X))^{Z^b}\subset \cA^b(X)$ is a Noetherian torsion subcategory.
\end{lemma}

\begin{proof}
According to Proposition \ref{prop-slope-weak-stab}, $\sigma$ has the tilting property in the sense of Definition \ref{def-tilting}. Then the result follows from Lemma \ref{lem-tilt-weak-stab}.
\end{proof}

\subsection{Rotating slope-stability in families}

Next, we arrange Lemma \ref{lem-rotate} into families and prove a stronger boundedness statement as in \cite[Proposition 25.1]{BLMNPS21}. Let $(X, H, \gamma)$ be a triple as in Setup \ref{setup-gamma}. We first describe the behavior of $\sigma^{b}$ under base field extensions.

\begin{lemma}\label{lem:rotate-base-change}
Fix a field extension $\kk\subset \mathsf{k}_1$ and let $\pi\colon X_{\kk_1}\to X$ be the base change morphism. We set $\cA^b(X_{\kk_1})\coloneqq \Coh^b_{H_{\kk_1},\gamma_{\kk_1}}(X_{\kk_1})$.

\begin{enumerate}
    \item For any $b\in \RR$ and $E\in \Db(X)$, we have $L\pi^*E\in \cA^b(X_{\kk_1})$ if and only if $E\in \cA^b(X)$. Moreover, the functor $L\pi^*\colon \cA^b(X)\to \cA^b(X_{\kk_1})$ is exact.

    \item If $E\in \cA^b(X)$, then $E$ is $\sigma^b$-semistable if and only if $E_{\kk_1}$ is $\sigma^b_{\kk_1}$-semistable.

    \item We have $\cA^b(X_{\kk_1})=(\cA^b(X))_{\kk_1}$, where $(\cA^b(X))_{\kk_1}$ is the base change of the heart constructed in Lemma \ref{lem:base-change-t-structure}. Therefore, the base change of $\sigma^b$ to $\kk_1$ coincides with $\sigma^b_{\kk_1}=(\cA^b(X_{\kk_1}), Z^b)$.

\end{enumerate}
\end{lemma}

\begin{proof}
By Lemma \ref{lem-constant-chM} and \cite[Theorem 1.3.7]{HL10}, we know that a coherent sheaf $F\in \Coh(X)$ is $\mu$-semistable if and only if $F_{\kk_1}\in \Coh(X_{\kk_1})$ is $\mu$-semistable of the same slope. Then part (a) follows from the flatness of $\pi$ and the definition of tilted hearts. Now part (b) follows easily from \cite[Lemma 14.17]{BLMNPS21} (see Lemma \ref{lem:classify-rotate-stable} below).

To prove (c), note that $(\sigma^{\gamma}_H)_{\kk_1}=\sigma^{\gamma_{\kk_1}}_{H_{\kk_1}}$. By definition, we can find an interval $I_b\subset \RR$ with length $1$ so that $\cP_{\sigma^{\gamma}_H}(I_b)=\cA^b(X)$. Therefore, by \cite[Proposition 5.9, 14.20]{BLMNPS21}, we have
\[(\cP_{\sigma^{\gamma}_H}(I_b))_{\kk_1}=\cP_{(\sigma^{\gamma}_H)_{\kk_1}}(I_b)=\cP_{\sigma^{\gamma_{\kk_1}}_{H_{\kk_1}}}(I_b)=\cA^b(X_{\kk_1})\]
as desired.
\end{proof}

Therefore, according to Lemma \ref{lem:rotate-base-change}(c) and Definition \ref{def:geo-stable}, we say an object $E\in \cA^b(X)$ is \emph{geometrically $\sigma^b$-stable} if $E_{\overline{\kk}}$ is $\sigma^b_{\overline{\kk}}$-stable.

Next, we recall the classification result of (geometrically) $\sigma^b$-(semi)stable objects.

\begin{lemma}\label{lem:classify-rotate-stable}
Fix an object $E\in \cA^b(X)$. Then $E$ is $\sigma^b$-semistable if and only if either

\begin{enumerate}
    \item $\bv_{0}(E)\geq 0$, and $E$ is a torsion-free $\mu$-semistable sheaf with slope $>b$ or a torsion sheaf that is either pure of codimension one or of codimension at least $2$; or

    \item $\bv_{0}(E)<0$ and we have a short exact sequence $$0\to A[1]\to E\to B\to 0$$ in $\cA^b(X)$ such that $A$ is a torsion-free $\mu$-semistable sheaf with slope $\leq b$ and $B$ is a torsion sheaf supported in codimension at least $2$; when $\mu(A)<b$, $\Hom_X(T, E)=0$ for every sheaf $T$ supported in codimension $\geq 2$. 
\end{enumerate}

Similarly, $E$ is (geometrically) $\sigma^b$-stable if and only if either

\begin{enumerate}[(1)]
    \item $\bv_{0}(E)\geq 0$, and

    \begin{itemize}
        \item $E$ is a torsion-free (geometrically) $\mu$-stable sheaf with slope $>b$; or

        \item $E$ is a pure torsion sheaf of rank one on its scheme-theoretic support, which is (geometrically) integral of codimension one or of dimension zero; or
    \end{itemize}

    \item $E$ satisfies (b), the sheaf $A$ in (b) is a torsion-free $S_2$ (geometrically) $\mu$-stable sheaf, and either $\mu(A)<b$ and the sequence is non-split, or $\mu(A)=b$ and $B=0$.
\end{enumerate}
\end{lemma}

\begin{proof}
By base change to $\overline{\kk}$, the result follows from Lemma \ref{lem:classify-rotate-stable-abstract}.
\end{proof}

As an immediate application of the above classification, we have:

\begin{lemma}\label{lem:rotate-geo-stable-base-change}
Fix a field extension $\kk\subset \mathsf{k}_1$ of the base field $\kk$. If $E\in\cA^b(X)$, then $E$ is geometrically $\sigma^b$-stable if and only if $E_{\kk_1}$ is geometrically $\sigma^b_{\kk_1}$-stable.
\end{lemma}

\begin{proof}
Note that $\Db(X)\to \Db(X_{\kk_1})$ is a faithful functor. Since the rank, torsion-freeness, $S_2$ property, geometric $\mu$-stability, and geometric integrality are preserved after base change to $\kk_1$, the result follows from Lemma \ref{lem:classify-rotate-stable}.
\end{proof}

We need the following generalization of \cite[Lemma 2.19(a)]{bayer2017stability}.

\begin{lemma}\label{lem-6.3}
Let $X$ be an equidimensional Noetherian scheme that admits a dualizing complex. We denote by $(\cD^{\leq 0}, \cD^{\geq 0})$ the standard t-structure on $\Db(X)$, and $\tau_{\leq i}, \tau_{\geq i}$ by the truncation functors. Let $E\in \Db(X)$ such that $\cH^{i}(E)=0$ for $i\notin \{-1,0\}$ and $\cH^{-1}(E)$ is zero or torsion-free. We define two objects $E^{\sharp}\coloneqq\tau_{\leq 0}\DD_X^{1-\dim X}(E)$ and $Q\coloneqq\tau_{\geq 1}\DD_X^{1-\dim X}(E)$.

\begin{enumerate}
    \item $\cH^j(Q)$ is a torsion sheaf supported in codimension at least $j+1$ for all $j\geq 1$ and $\cH^j(Q)=0$ for $j\leq 0$.

    \item If $\Hom_X(T, E)=0$ for any sheaf \(T\) on $X$ supported in codimension at least $2$, then $\cH^j(Q)$ is a torsion sheaf supported in codimension at least $j+2$ for all $j\geq 1$ and $\cH^j(Q)=0$ for $j\leq 0$.

    \item If $X$ is also Gorenstein and $\cH^0(E)$ is supported in codimension $\geq 2$, then there exists a triangle
\[E^{\sharp}\otimes \omega_X^{-1}\to \DD^{X}_1(E)\to Q\otimes \omega_X^{-1}\]
such that $E^{\sharp}\otimes \omega_X^{-1}\subset (\cH^{-1}(E))^{\vee}$ is a subsheaf with the quotient supported in codimension at least $2$.
\end{enumerate}
\end{lemma}

\begin{proof}
By applying $\DD_X^{1-\dim X}(-)$ to $$\cH^{-1}(E)[1]\to E\to \cH^0(E),$$ the long exact sequence of cohomology objects and the fact that $\cH^{-1}(E)$ is torsion-free show that $\cH^j(Q)$ is supported in codimension at least $j+1$ for all $j\geq 1$ by Lemma \ref{lem-supp-ext} and \ref{lem-codim-Sp}. This proves (a).

For (b), by part (a) and Lemma \ref{lem-supp-ext}, we have $$\DD_X^{1-\dim X}(\cH^j(Q)[-j])\in \cD^{\geq 0}$$ and $\cH^j(Q)$ is a torsion sheaf supported in codimension at least $j+2$ if and only if $$\DD_X^{1-\dim X}(\cH^j(Q)[-j])\in \cD^{\geq 1}.$$ 

If the statement in (b) does not hold, let $j_0\geq 1$ be an integer so that $\codim_X\cH^{j_0}(Q)=j_0+1$. We consider a spectral sequence with the second page $$E^{p,q}_2\coloneqq \cH^p(\DD_X^{1-\dim X}(\cH^{-q}(Q)))=\cE xt^p_X(\cH^{-q}(Q),\omega^{\bullet}_X[1-\dim X])$$
which converges to $\cH^{p+q}(\DD_X^{1-\dim X}(Q))$ (cf.~\cite[(3.8)]{huybrechts2006fourier}). Note that for any $r\geq 2$, $$E^{j_0-r,-j_0+r-1}_2=0$$ by part (a) and Lemma \ref{lem-supp-ext}. Similarly, $E^{j_0+r,-j_0-r+1}_2$ is supported in codimension $\geq j_0+1+r$. These two conditions together with  
$$\cH^0(\DD_X^{1-\dim X}(\cH^{j_0}(Q)[-j_0]))=\cH^{j_0}(\DD_X^{1-\dim X}(\cH^{j_0}(Q)))\neq 0$$
imply that $\codim_X(E^{j_0,-j_0}_{\infty})=j_0+1$. Thus, $\cH^0(\DD_X^{1-\dim X}(Q))$ is a non-zero torsion sheaf supported in codimension at least $j_0+1\geq 2$. Therefore, we have $$\Hom_X(\cH^0(\DD_X^{1-\dim X}(Q)), \DD_X^{1-\dim X}(E^{\sharp})[-1])=\Hom_X(\cH^0(\DD_X^{1-\dim X}(Q)), (\cH^0(E^{\sharp}))^{\mathsf{d}})=0$$
as $(\cH^0(E^{\sharp}))^{\mathsf{d}}$ is either zero or torsion-free by Lemma \ref{lem-S2-hull}. Therefore, applying $\DD_X^{1-\dim X}$ to the canonical triangle $$E^{\sharp}\to \DD_{X}^{1-\dim X}(E)\to Q,$$ we see that the composition $$\cH^0(\DD_X^{1-\dim X}(Q))\to \DD_X^{1-\dim X}(Q)\to E$$ is non-zero, where the first map exists by $\DD_X^{1-\dim X}(Q)\in \cD^{\geq 0}$. This contradicts the assumption and proves (b).

Now we prove (c). In this case, we have $\DD^X_1(-)=\DD_X^{1-\dim X}(-)\otimes\omega_X^{-1}$. By applying $\DD^X_1(-)$ to $$\cH^{-1}(E)[1]\to E\to \cH^0(E),$$ we get isomorphisms
\begin{equation*}
    (\cH^0(E))^{\vee}\cong \cH^{-1}(\DD^X_1(\cH^0(E)))\cong \cH^{-1}(\DD^X_1(E)),
\end{equation*}
which vanishes as $\cH^0(E)$ is torsion. Hence $E^{\sharp}$ is a sheaf. Moreover, we have a long exact sequence
\[0\to \cE xt^1_X(\cH^0(E), \oh_X)\to \cH^0(\DD^X_1(E))\cong E^{\sharp}\otimes\omega_X^{-1}\to (\cH^{-1}(E))^{\vee}\to \cE xt^2_X(\cH^0(E), \oh_X).\]
Since $\cH^0(E)$ is supported in codimension $\geq 2$, we see $$\cE xt^1_X(\cH^0(E), \oh_X)\cong 0$$ and $\cE xt^2_X(\cH^0(E), \oh_X)$ is supported in codimension $\geq 2$ by Lemma \ref{lem-supp-ext}. This proves that $E^{\sharp}\otimes\omega_X^{-1}\subset (\cH^{-1}(E))^{\vee}$ is a subsheaf with the quotient supported in codimension at least $2$.
\end{proof}

Now, we fix $(f\colon X\to S, \cL, \gamma)$ as in Setup \ref{setup-gamma-relative}. We define a collection $\underline{\sigma}^{b}\coloneqq (\sigma^{b}_s)_{s\in S}$ for any $b\in \RR$. The next result follows from the same proof of \cite[Proposition 25.1]{BLMNPS21} with minor modifications in our setting.

\begin{proposition}\label{prop-rotate-slope-real-b}
For any $b\in \RR$, the collection $\underline{\sigma}^{b}= (\sigma^{b}_s)_{s\in S}$ satisfies 

\begin{enumerate}
    \item conditions \ref{c1}, \ref{c2}, \ref{wc2}, \ref{wc3}, \ref{b1}, and $(\cA^b(X_s))^{Z^{b}}\subset \cA^b(X_s)$ is a Noetherian torsion subcategory for each $s\in S$,

    \item for any morphism $T\to S$ of finite type and any $T$-perfect object $E$ on $X_T$, the functions $\phi^+_E$ and $\phi^-_E$ are, respectively, upper and lower semicontinuous constructible functions on $T$,

    \item if $B_{\gamma}=0$ and $n\leq 3$, and $f$ is admissible, then \ref{b2} holds with respect to $\Lambda$; in this case, if $n=3$ and we fix $v\in \Lambda$, then $\cM^{\st}_{\underline{\sigma}^b}(v+x\bp)=\varnothing$ for $x\gg 0$, where $\bp\in \Lambda_3$ is the class of skyscraper sheaves, and

    \item in the setting of (c), if $b\in \QQ$, then $\underline{\sigma}^b= (\sigma^b_s)_{s\in S}$ is a weak stability condition on $\Db(X)$ over $S$ with respect to $\Lambda$.
\end{enumerate}
\end{proposition}

\begin{proof}
The condition \ref{b1} and the statement for $(\cA^b(X_s))^{Z^{b}}$ follow from Lemma \ref{lem-rotate}. By Lemma \ref{lem-constant-chM}, the condition \ref{c1} also holds.

Now, we verify \ref{wc3}. Let $C\to S$ be a morphism essentially of finite type from a Dedekind scheme $C$. By Proposition \ref{prop-slope-weak-stab} and Lemma \ref{lem-tilt-weak-hn}, it is enough to show that the weak HN structure $\sigma_{C}$ has the tilting property in the sense of Definition \ref{def-tilting-hn}. By Lemma \ref{lem:supp-codim-and-bv}, it is clear that $(\Coh(X_C))^{Z_C}\subset \Coh(X_C)$ is a Noetherian torsion subcategory, as it consists of sheaves whose support has codimension at least $2$ in each fiber, which verifies \ref{t'1}. By \cite[Remark 19.3]{BLMNPS21}, to verify \ref{t'2}, it is enough to assume that $E\in \Coh(X_C)$ is either torsion-free or $E=i_{p*}F$ for a torsion-free sheaf $F\in \Coh(X_p)$. In the former case, we define $\wt{E}\subset E^{H}$ to be the preimage of $G\subset E^{H}/E$, where $G$ is the maximal subsheaf whose support has codimension at least $2$ in every fiber. In the latter case, we define $\wt{E}\coloneqq i_{p*}(F^{H})$. Therefore, using Lemma \ref{lem-S2-hull} and \ref{lem-general-S2}, it is direct to verify \ref{t'2}. As a result, \ref{wc3} holds.

Next, we verify \ref{wc2} and \ref{c2}. Let $T\to S$ be a morphism essentially of finite type and $E\in \Dqc(X_T)$ be a $T$-perfect object. Let $t_0\in T$ so that $E_{t_0}\in \cA^b(X_{t_0})$ is $\sigma^{b}_{t_0}$-semistable. By our assumption, $X_T\to T$ is flat projective, and $X_T$ and $T$ are both Noetherian. Then Lemma \ref{lem-S-perf-lem-1}(a) implies $E\in \Db(X_T)$. By Lemma \ref{lem:openness-standard-heart}(a), we may replace $T$ by an open neighborhood of $t_0$ and assume that $\cH^{i}(E_t)=0$ and $\cH^{i}(E)=0$ for all $t\in T$ and $i\notin \{-1,0\}$. Applying Lemma \ref{lem:openness-standard-heart}(b), we get a locally closed stratification $\{T_i\}_{i\in I}$ of $T$ so that we have an exact triangle $\cH^{-1}(E_{T_i})[1]\to E_{T_i}\to \cH^0(E_{T_i})$ for each $i\in I$ and $\cH^{-1}(E_{T_i})$ and $\cH^{0}(E_{T_i})$ are both flat over $T_i$. Since being $S_2$, $\mu$-semistable, geometrically $\mu$-stable, pure of a fixed dimension, of a fixed rank, and geometrically integral are open in flat families of sheaves or schemes, by Lemma \ref{lem:classify-rotate-stable}, we know that the intersections of sets
\[A_E\coloneqq \{t\in T\colon E_t \text{ is }\sigma^b_t\text{-semistable}\}\]
and
\[B_E\coloneqq \{t\in T\colon E_t \text{ is geometrically }\sigma^b_t\text{-stable}\}\]
with each $T_i$ are open. Therefore, they are constructible subsets of $T$. To show they are open, it remains to prove that their complements in $T$ are stable under specialization in the sense of \cite[\href{https://stacks.math.columbia.edu/tag/0061}{Tag 0061}]{stacks-project}, i.e. if $t\in T\setminus A_E$ (resp.~$t\in T\setminus B_E$), then $t'\in T\setminus A_E$ (resp.~$t'\in T\setminus B_E$) for any $t'\in \overline{\{t\}}$. To this end, by \cite[Lemma 11.19]{BLMNPS21}, we may find a DVR $R$ with a morphism $\Spec(R)\to T$ which is essentially of finite type, that maps the generic point $\eta$ to $t$ and maps the closed point $p$ to $t'$. By Lemma \ref{lem:rotate-base-change}(b) and \ref{lem:rotate-geo-stable-base-change}, we only need to show that if $E\in \Db(X_R)$ is a $R$-perfect object with $E_{\eta}$ not $\sigma^{b}_{\eta}$-semistable or not geometrically $\sigma^b_{\eta}$-stable, then so is $E_p$. By Lemma \ref{lem:classify-rotate-stable}, the semistability or geometric stability only depends on the semistability/geometric stability/purity of cohomology sheaves of $E_K$ and $E_p$. Therefore, we can lift any destabilizing subsheaf or subsheaf with lower dimension of $\cH^i(E_{\eta})$ to $X_R$, which implies $E_p$ is not $\sigma^b_{p}$-semistable or not geometrically $\sigma^b_{p}$-stable as well. This ends the proof of \ref{wc2} and \ref{c2}, hence part (a).

Now part (b) follows from Lemma \ref{lem:rotate-base-change}(b) and Lemma \ref{lem:openness-flat}.

Note that (d) immediately follows from (c), so it remains to verify \ref{b2} and the stronger boundedness statement in (c). Note that by the property \ref{c2}, we know that $\cM^{\st}_{\underline{\sigma}^b}(v)\subset \cM_{\mathrm{pug}}(X/S)$ is an open substack for any $v\in \Lambda$.

We fix a $\sigma^b_s$-stable object $E\in \cA^b(X_s)$. When $\bv_{0}(E)> 0$, according to the classification in Lemma \ref{lem:classify-rotate-stable}, the boundedness of $\cM^{\st}_{\underline{\sigma}^b}(v)$ follows from \cite[Theorem 4.2]{langer:positive-char} and \cite[Remark 25.2]{BLMNPS21}. When $\bv_{0}(E)=0$, $E$ is either a pure rank-one sheaf on a geometrically integral codimension-one subscheme, or the skyscraper sheaf of a geometric point. In particular, $E$ is a stable sheaf on $X_s$. Then the result follows from the corresponding boundedness statements for the Hilbert scheme of $\Supp(E)$ and relative moduli space of geometrically stable sheaves over the Hilbert scheme.

When $\bv_{0}(E)<0$, by Lemma \ref{lem:classify-rotate-stable}, we have an exact sequence $A[1]\to E\to B$ in $\cA^b(X_s)$ such that $A$ is a torsion-free $S_2$ geometrically $\mu$-stable sheaf and $B$ is a torsion sheaf supported in codimension at least $2$. When $X_s$ is geometrically normal, as
\[\bv_2(A)=\bv_2(B)-\bv_2(E)\geq -\bv_2(E),\]
we get the boundedness of $A$ by the second case of \cite[Theorem 4.4]{langer:positive-char}. In particular, there are only finitely many possibilities of the Hilbert polynomial of $B$ as $\bv(E)$ is fixed. Using the first case of \cite[Theorem 4.4]{langer:positive-char}, we get the boundedness of $B$, and the corresponding result for $E$ follows. When $X_s$ is lci, applying Lemma \ref{lem-6.3}(c), we have a triangle
\[E^{\sharp}\otimes \omega^{-1}_{X_s}\to \DD^{X_s}_1(E)\to Q\otimes \omega^{-1}_{X_s}\]
such that $E^{\sharp}\otimes \omega^{-1}_{X_s}\subset A^{\vee}$ is a geometrically $\mu$-stable torsion-free sheaf by Lemma \ref{lem:dual-slope-stable} and $Q[1]$ is a sheaf supported in codimension at least $3$. Note that for any $S$-perfect object $F$ on $X$, we have $\DD^X_1(F)$ is also $S$-perfect  by Lemma \ref{lem-S-perf-lem-3} and $(\DD^X_1(F))_s=\DD^{X_s}_1(F_s)$. Therefore, it is enough to show that $\DD^{X_s}_1(E)$ is bounded, which suffices to prove the boundedness of $E^{\sharp}\otimes \omega^{-1}_{X_s}$ and $Q\otimes \omega^{-1}_{X_s}[1]$ by \cite[Lemma 9.6]{BLMNPS21}. By  Lemma \ref{lem-derived-dual-ch}, we have $\bv_{i}(E)=(-1)^{i+1}\bv_{i}(E^{\sharp}\otimes \omega^{-1}_{X_s})$ for $i\in \{0,1,2\}$ and $$\bv_{3}(E^{\sharp}\otimes \omega^{-1}_{X_s})=\bv_{3}(E)+\bv_{3}(Q\otimes \omega^{-1}_{X_s}[1])\geq \bv_{3}(E).$$ Since $E^{\sharp}\otimes \omega^{-1}_{X_s}$ is a geometrically $\mu$-stable torsion-free sheaf, the first case in \cite[Theorem 4.4]{langer:positive-char} implies that it belongs to a bounded family, and so $Q\otimes \omega^{-1}_{X_s}[1]$ also belongs to a bounded family, and the result follows.
\end{proof}

We denote by $\sigma^{b}_{C}=(\cA^b(X_C), Z^{b}_{C})$ the weak HN structure constructed in the above proof. From the above argument, we can deduce the following fact, which will be used later.

\begin{corollary}\label{cor:C-torsion-theory-real-b}
Let $C\to S$ be a morphism essentially of finite type from a Dedekind scheme $C$. Then for any $b\in \RR$, $\sigma^{b}_{C}=(\cA^b(X_C), Z^{b}_{C})$ is a weak HN structure over $C$, and $\cA^b(X_C)$ universally satisfies openness of flatness and has a $C$-torsion theory.
\end{corollary}

\begin{proof}
By Proposition \ref{prop-rotate-slope-real-b}(b), we know that $\cA^b(X_C)$ universally satisfies openness of flatness. Now the existence of a $C$-torsion theory follows from Lemma \ref{lem:17.1} and Lemma \ref{lem-rotate}.
\end{proof}

\subsection{Tilt-stability}\label{subsec-titl-from-slope}

Let $(X, H, \gamma)$ be a triple as in Setup \ref{setup-gamma} with $n\geq 2$ and either

\begin{itemize}
    \item $d\geq 2$, or

    \item $X_{\overline{\kk}}$ is a normal projective surface.
\end{itemize}
Recall that for an element $v$ in a lattice $\Lambda\subset \QQ^{m+1}$, we denote by $v_i$ the component in $\Lambda_i$. Then for any $b\in \RR$, we define \[v^{b}_i\coloneqq\sum_{j=0}^i \frac{(-b)^j}{j!}v_{i-j}\in \RR.\] For the lattice $\Lambda_{\leq 2}=\Lambda^{\gamma}_{H,\leq 2}\subset \QQ^{3}$ and $(b,w)\in \RR^2$, we define a homomorphism $Z^{b,w}\colon \Lambda_{\leq 2}\to \CC$ by
\[Z^{b,w}(v)\coloneqq -v_2+wv_0+\mathfrak{i}(v_1^b).\]
Then by Proposition \ref{prop-slope-weak-stab} and Theorem \ref{thm:tilt-stability}, the pair $$\sigma^{b,w}=\sigma^{\gamma,b,w}_{H}\coloneqq(\cA^b(X), Z^{b,w})$$ is a weak stability condition with respect to $\Lambda_{\leq 2}$ for $(b,w)\in U^{\gamma}_{X,H}$. In this setting, we denote the associated slope function by $\nu_{b,w}(-)=\nu^{\gamma}_{H,b,w}(-)$. We say an object $E\in \cA^b(X)$ is \emph{$\nu_{b,w}$-(semi)stable} if it is $\sigma^{b,w}$-(semi)stable.

By Theorem \ref{thm:tilt-stability} and \ref{thm:wall-chamber-abstract}, we have the following result.

\begin{theorem}\label{thm-bms}
We have a continuous injective map
\[U^{\gamma}_{X,H}\to \Stab_{\Lambda_{\leq 2}}^{\mathsf{w}}(\Db(X)), \quad (b,w)\mapsto \sigma^{b,w}=(\cA^b(X), Z^{b,w}).\]
Moreover, if $(X, H, \gamma)$ has a BG function $\mathsf{g}$, then for any $\nu_{b,w}$-semistable object $E$ with $\bv_{0}(E)\neq 0$, we have
\[\frac{\bv_{2}(E)}{\bv_{0}(E)}\leq \mathsf{g}(\mu(E)).\]
\end{theorem}

Note that a BG function $\mathsf{g}$ always exists by Theorem \ref{thm:exist-bg-function}.

The following wall-chamber structure is a special case of Theorem \ref{thm:wall-chamber-abstract}. For any $v\in \QQ^{d+1}$ with $v_0\neq 0$, we define $$\Pi(v)\coloneqq \left(\frac{v_1}{v_0}, \frac{v_2}{v_0}\right).$$

\begin{theorem}\label{thm:wall-chamber-tilt}
Assume furthermore that $(X, H, \gamma)$ has a BG function $\mathsf{g}$. 
Fix a class $v\in \Lambda_{\leq 2}$ such that $(v_0, v_1, v_2)\neq (0,0,0)$. Then there exists a set of lines $\{\ell_i\}_{i\in I}$ in $\RR^2$ such that $\ell_i\cap U_{\mathsf{g}}$ (called walls) are locally finite and satisfy
\begin{enumerate}
	    \item If $v_0\neq 0$, then all lines $\ell_i$ pass through $\Pi(v)$.
	    \item If $v_0= 0$ and $v_1\neq 0$, then all lines $\ell_i$ are parallel of slope $\frac{v_2}{v_1}$.
	   		\item The $\nu_{b,w}$-(semi)stability of any object $E$ with $\bv_{\leq 2}(E)=v$ is unchanged as $(b,w)$ varies within any connected component (called a ``\emph{chamber}") of  $U_{\mathsf{g}} \setminus \bigcup_{i \in I}\ell_i$.
		\item For any wall $\ell_i\cap U_{\mathsf{g}}$, there is a map $f\colon F\to E$ in $\cA^b(X)$ such that
\begin{itemize}
\item $E$ is $\nu_{b,w}$-semistable with $\bv_{\leq 2}(E)=v$ and $\nu_{b,w}(E)=\nu_{b,w}(F)=\,\mathrm{slope}\,(\ell_i)$ constant on the wall $\ell_i \cap U_{\mathsf{g}}$, and
\item $f$ is an injection $F\hookrightarrow E $ in $\cA^b(X)$ which strictly destabilizes $E$ for $(b,w)$ in one of the two chambers adjacent to the wall $\ell_i\cap U_{\mathsf{g}}$.
\end{itemize} 
\end{enumerate}

Moreover, we have $\Pi(E)\coloneqq \Pi(\bv_{\leq 2}(E))\notin U_{\mathsf{g}}$ for any $(b,w)\in U_{\mathsf{g}}$ and any $\nu_{b,w}$-semistable object $E$ with $\bv_{0}(E)\neq 0$.
\end{theorem}

In particular, applying \cite[Lemma 3.9]{bayer2016space}, we have the following generalization of \cite[Corollary 3.10]{bayer2016space}.

\begin{lemma}\label{lem-delta-JH}
Assume that $(X, H, \gamma)$ has the standard BG function. Then for any $(b,w)\in \cU$ and any strictly $\nu_{b,w}$-semistable object $E$ with $\nu^{+}_{b,w}(E)<\infty$, we have
\[\bDelta(\bv_{\leq 2}(E_i))\leq \bDelta(\bv_{\leq 2}(E))\]
for any JH factor $E_i$ of $E$. Equality holds if and only if all $\bv_{\leq 2}(E_i)$ are proportional to $\bv_{\leq 2}(E)$ and satisfy $\bDelta(\bv_{\leq 2}(E_i))=\bDelta(\bv_{\leq 2}(E))=0$. In particular, if $E$ is $\nu_{b',w'}$-stable for some $(b',w')\in \cU$, then the inequality is strict.
\end{lemma}

\subsection{Useful lemmas about tilt-stability}\label{subsec:lemma-about-tilt}

Keep the assumption at the beginning of Section \ref{subsec-titl-from-slope}. In the following, we collect some useful results for tilt-stability. The next two lemmas describe the behavior of semistable objects at the large volume limit.

\begin{lemma}\label{lem-large-volume}
If $E\in \cA^b(X)$ is $\nu_{b,w}$-semistable for $w\gg 0$, then it satisfies one of the following conditions:

\begin{enumerate}
    \item $\cH^{-1}(E)=0$ and $\cH^0(E)$ is a torsion-free $\mu$-semistable sheaf.

    \item $\cH^{-1}(E)=0$ and $\cH^0(E)$ is a torsion sheaf.

    \item $\cH^{-1}(E)\neq 0$ is a torsion-free $\mu$-semistable sheaf and $\cH^0(E)$ is either zero or supported in codimension at least $2$.
\end{enumerate}
\end{lemma}

\begin{proof}
This directly follows from Lemma \ref{lem:lvl}.
\end{proof}

\begin{lemma}\label{lem-large-volume-converse}
If $E$ is a $\mu$-stable torsion-free sheaf on $X$, then $E\in \cA^b(X)$ is $\nu_{b,w}$-stable for $b<\mu(E)$ and $w\gg 0$. If $E$ is also $S_2$, then $E[1]\in \cA^b(X)$ is $\nu_{b,w}$-stable for $b\geq \mu(E)$ and $w\gg 0$. In particular, $\mu$-stable $S_2$ sheaves do not get destabilized along the vertical wall.
\end{lemma}

\begin{proof}
This follows from Lemma \ref{lem:slope-stable-lvl}.
\end{proof}

For stable objects along the vertical wall, we have the following description.

\begin{lemma}\label{lem-vertical-wall}
Let $E\in \cA^b(X)$ be a $\nu_{b,w}$-stable object with $b=\mu(E)$ and some $(b,w)\in U_{X,H}^{\gamma}$. Then $E=F[1]$, where $F$ is a $\mu$-stable torsion-free $S_2$ sheaf on $X$.
\end{lemma}

\begin{proof}
We have $\nu_{b,w}(E)=+\infty$. Therefore, the $\nu_{b,w}$-stability implies that the cases (a) and (b) in Lemma \ref{lem-large-volume} cannot happen for $E$. So $\cH^{-1}(E)$ is a $\mu$-semistable torsion-free sheaf and $\cH^0(E)$ is supported in codimension $\geq 2$. Using $\nu_{b,w}(\cH^{-1}(E))=\nu_{b,w}(\cH^0(E))=+\infty$ and the $\nu_{b,w}$-stability of $E$ again, we have $\cH^0(E)=0$ and the $\mu$-stability of $\cH^{-1}(E)$. It remains to show that $\cH^{-1}(E)$ is $S_2$.

Note that any torsion sheaf $T$ supported in codimension $\geq 2$ is $\nu_{b,w}$-semistable with $\nu_{b,w}(T)=+\infty$. Therefore, if $\Hom_X(T,E)\neq 0$, the stability of $E$ implies that there is a surjection $T\twoheadrightarrow E$ in $\cA^b(X)$. However, its kernel $F$ is a sheaf with $\bv_{\leq 1}(F)=-\bv_{\leq 1}(E)$, hence it is a torsion-free $\mu_H$-semistable sheaf by Lemma \ref{lem-large-volume} and cannot lie in $\cA^b(X)$. This proves that $\Hom_X(T,E)= 0$ and $E[-1]=\cH^{-1}(E)$ is $S_2$ by Lemma \ref{lem-S2-hull}(d).
\end{proof}

\begin{lemma}\label{lem-H-1-S2}
Let $(b,w)\in U^{\gamma}_{X, H}$. If $E\in \cA^{b}(X)$ is an object such that $\nu_{b,w}^+(E)<+\infty$. Then $\cH^{-1}(E)$ is torsion-free and $S_2$ if it is non-zero.
\end{lemma}

\begin{proof}
By definition of $\cA^{b}(X)$, $\cH^{-1}(E)$ is torsion-free if it is non-zero. Moreover, for any torsion sheaf $T$ on $X$ supported in codimension $\geq 2$, we have $\Hom_X(T,E)=0$. Then $\Hom_X(T, \cH^{-1}(E)[1])=0$ and the result follows from Lemma \ref{lem-S2-hull}(d).
\end{proof}

\begin{lemma}\label{lem-delta-0}
Assume that $(X, H, \gamma)$ has the standard BG function and $(b,w)\in \cU$. Let $E\in \cA^{b}(X)$ be a $\nu_{b,w}$-stable object with $\bDelta(E)$ to be minimal among all tilt-semistable objects and $\bv_0(E)\neq 0$. Then $E$ is $\nu_{b,w}$-stable for any $(b,w)\in \cU$. Moreover, if $b<\mu(E)$, then $E$ is a $\mu$-stable torsion-free sheaf; if $\mu(E)\leq b$, then $\cH^{-1}(E)$ is a $\mu$-semistable torsion-free $S_2$ sheaf; if $b=\mu(E)$, then $E=\cH^{-1}(E)[1]$.
\end{lemma}

\begin{proof}
The first statement follows from Lemma \ref{lem-delta-JH} and the minimality of $\bDelta(E)$.

When $b=\mu(E)$, the result is proved in Lemma \ref{lem-vertical-wall}. When $b\neq \mu(E)$, the result follows from Lemma \ref{lem-large-volume} and \ref{lem-H-1-S2}.
\end{proof}

In some situations, it is more convenient to set $a\coloneqq\sqrt{2w-b^2}$ and a homomorphism $Z_{a,b}\colon \Lambda_{\leq 2} \to \CC$ as
\begin{equation}\label{eq:Zab}
Z_{a,b}(v)\coloneqq -v_2^{b}+\frac{a^2}{2}v_0+\mathfrak{i}(v_1^{b})=bv_1^{b}+Z^{b,w}(v).
\end{equation}
We denote by $\mu_{a,b}$ the associated slope function of $Z_{a,b}\circ \bv_{\leq 2}$. As $$\nu_{b,w}(-)=b+\mu_{a,b}(-),$$ there is no difference between $\nu_{b,w}$-(semi)stability and $\mu_{a,b}$-(semi)stability. 

The following lemma is straightforward.

\begin{lemma}\label{lem-tensor-H}
Let $E\in \cA^b(X)$ be a $\mu_{a,b}$-(semi)stable object and $\cL$ be a line bundle that is numerically equivalent to $tH$ for some $t\in \RR$. Then $E\otimes \cL\in \cA^{b+t}(X)$ is a $\mu_{a,b+t}$-(semi)stable object.
\end{lemma}

\begin{lemma}\label{lem:delta-rk1}
Assume that $X$ is geometrically normal and geometrically integral, has isolated singularities, $n\geq 3$, and $\gamma_2=B_{\gamma}=0$. Then for any rank one torsion-free $S_2$ sheaf $E$ on $X$, we have 
\[\bv_{\leq 2}(E)=\left(H^n, D.H^{n-1}, \frac{1}{2}D^2.H^{n-2}\right),\]
where $D=\bch_1(E)\in \CH^1(X)$. In particular, if $X$ is $\QQ$-factorial and of Picard number one, then we have $\bDelta(E)=0$.
\end{lemma}

\begin{proof}
We may assume that $H$ is very ample and pass to the algebraic closure of $\kk$. Then a general divisor $Y\in |H|$ is smooth and $E_Y$ is a line bundle. So the result follows from Lemma \ref{lem:chi-property}(c) and Theorem \ref{thm-mumford-intersection}.
\end{proof}

\begin{lemma}\label{lem-delta-0-rk1}
Assume that $X$ is integral, $\gamma=1$, and has the standard BG function. Let $(b,w)\in \cU$ and $E\in \cA^{b}(X)$ be a $\nu_{b,w}$-stable object with $\bDelta(E)=0$ and $\bv_0(E)\neq 0$. 

\begin{enumerate}
    \item $E$ is either a $\mu$-semistable torsion-free sheaf or $\cH^{-1}(E)$ is a $\mu$-semistable torsion-free $S_2$ sheaf and $\cH^0(E)$ is supported in codimension $\geq 3$.

    \item If $X$ is also normal with rational singularities and $\mathrm{char}(\kk)=0$, then we have $$\bv_{\leq 2}(E)=\rk(E)\left(H^n,aH^n, \frac{a^2}{2}H^n\right)$$ for some $a\in \QQ$.

    \item Assume that $\bv_{\leq 1}(E)=\pm \bv_{\leq 1}(\cO_X(aH))$ for some $a\in \ZZ$, $\chi(\oh_X)>0$, $X$ is lci or normal $\QQ$-factorial Gorenstein, $n=3$, and $\omega_X^{\vee}$ is ample and proportional to $H$. Then up to shift, $E\cong \oh_X(aH)$ or $\mathcal{I}_Z\otimes \oh_X(aH)$, where $Z\subset X$ is zero-dimensional.
\end{enumerate}
\end{lemma}

\begin{proof}
Part (a) immediately follows from Lemma \ref{lem-delta-0}. For part (b), note that by part (a), Proposition \ref{prop-bg-normal-over-k}, and Lemma \ref{lem-hodeg-index}, we have
\[0\leq H^{n-2}.(\bch_1^2(E)-2\rk(E)\bch_2(E))\leq \frac{1}{H^n}\bDelta(E)=0,\]
so $\bch_1(E)$ is numerically equivalent to a rational multiple of $H$, and the result follows.

Now, we prove part (c). Since $\bv_{\leq 1}(E)=\pm \bv_{\leq 1}(\cO_X(aH))$ and $\bDelta(E)=0$, we know that $\bv_{\leq 2}(E)=\pm \bv_{\leq 2}(\cO_X(aH))$. By Lemma \ref{lem-tensor-H}, after replacing $E$ with $E\otimes \cO_X(-aH)$, we may assume that $\bv_{\leq 2}(E)=\pm \bv_{\leq 2}(\oh_X)$. If $b<0$, then $E$ is a $\mu$-stable torsion-free sheaf by Lemma \ref{lem-delta-0}. Note that $E^{\vee \vee}$ is also a $\mu$-stable torsion-free sheaf and $0\leq \bDelta(E^{\vee \vee})\leq \bDelta(E)=0$. Therefore, $\bv_{\leq 2}(E^{\vee \vee})=\bv_{\leq 2}(\cO_X)$ as well, and we may assume that $E=E^{\vee \vee}$ and $E$ is tilt-stable everywhere by Lemma \ref{lem-delta-0}. If $E\cong \cO_X$, then we are done; otherwise, $\Hom_X(E,\oh_X)=\Hom_X(\oh_X, E)=0$ by the tilt-stability and $S_2$ property of $E$. 
Since $\omega_X^{\vee}$ is ample and proportional to $H$, we see that
$$\nu_{b,w}(E)=\nu_{b,w}(\cO_X)=\frac{w}{b}>\nu_{b,w}(\omega_X[1])$$ for $0>b>\mu_H(\omega_X[1])$ and $w>\frac{b^2}{2}$ so that $(b,w)$ is below the line connecting $\Pi(\omega_X)$ and $(0,0)$. By Lemma \ref{lem-serre-duality}, this gives $$\Ext^2_X(\oh_X, E)=\Hom_X(E, \omega_X[1])^*=0.$$ Similarly, we have $\Ext^2_X(E,\oh_X)=\Hom_X(\omega_X^{\vee}, E[1])^*=0$. Therefore, if $E\neq \cO_X$, then we get $\chi(E)=\chi(\oh_X)+\bv_3(E)\leq 0$ and $\chi(E,\oh_X)=\chi(\oh_X)-\bv_3(E)\leq 0$, which contradicts $\chi(\cO_X)>0$.

If $b=0$, then $E[-1]$ is a $\mu$-stable torsion-free $S_2$ sheaf, and the above argument applies in this case to show $E\cong \cO_X[1]$.

If $b>0$, the above argument shows that $\cH^{-1}(E)\cong \cO_X$. Since $\Hom_X(\cH^0(E), \cO_X[2])=0$, we also get $E\cong \cO_X[1]$.
\end{proof}


The following result is a generalization of \cite[Corollary 4.3]{feyzbakhsh:effective-restriction} and \cite[Lemma 4.1]{xu:gepner}.

\begin{proposition}\label{prop-restriction}
Assume that $\gamma_2=0$ and $B_{\gamma}=0$. Let $E\in \Coh(X)$ so that there exists $w>0$ and $k\in \ZZ_{>0}$ with $E,E(-kH)[1]\in \cA^{0}(X)$ are $\nu_{0,w}$-semistable. Then for any effective divisor $D\in |kH|$, $E_D\in \Coh(D)$ is torsion-free and we have
\[[\mu^-_{H_D}(E_D), \mu^+_{H_D}(E_D)]\subset \frac{k}{2}+[\min\{\nu_{0,w}(E),\nu_{0,w}(E(-kH)[1])\}, \max\{\nu_{0,w}(E),\nu_{0,w}(E(-kH)[1])\}].\]
In particular, if $\nu_{0,w}(E)=\nu_{0,w}(E(-kH)[1])$, then $E_D$ is $\mu_{H_D}$-semistable.
\end{proposition}

\begin{proof}
Note that by the $\nu_{0,w}$-semistability of $E(-kH)[1]$ and Lemma \ref{lem-H-1-S2}, $E$ is torsion-free and $S_2$, hence $E_D\in \Coh(D)$ is torsion-free. Since $X$ is lci in codimension $2$, we know that $D$ is lci in codimension $1$. Therefore, if we denote by $i\colon D\hookrightarrow X$ the inclusion, then for any $F\in \Db(D)$, we have
\[\bch_1(i_*F)=i_*(\bch_0(F))\]
and
\[\bch_2(i_*F)=i_*\left(\bch_1(F)-\frac{1}{2}kH_D.\bch_0(F)\right)\]
by Lemma \ref{lem:chi-property}(a). It follows that 
\[\nu_{0,w}(i_*F)=\frac{\bv_2(i_*F)}{\bv_1(i_*F)}=\frac{\bv_1(F)-\frac{1}{2}k\bv_0(F)}{\bv_0(F)}=\mu_{H_D}(F)-\frac{1}{2}k.\]
Now the result follows from the exact sequence
\[0\to E\to i_*E_D\to E(-kH)[1]\to 0\]
in $\cA^0(X)$ and $\nu_{0,w}$-semistability of $E$ and $E(-kH)[1]$.
\end{proof}

In particular, using Proposition \ref{prop-restriction}, the same proof of \cite[Lemma 3.3]{koseki:bg-hypersurface} gives the following.

\begin{lemma}\label{lem-naoki}
Assume that $\gamma_2=0$ and $B_{\gamma}=0$. Let $k\geq 1$ be an integer and $f(t)\colon [0,1]\to \RR$ be a star-shaped function in the sense of \cite[Definition 3.2]{koseki:bg-hypersurface} along the lines $b=0,k$ with $f(0)=0$, $f(1)=\frac{1}{2}$, and 
\[t^2-\frac{k}{2}t\leq f(t)\leq \frac{t^2}{2}.\]
Assume that there exists an object $E\in \Db(X)$ such that $E$ is either $\nu_{0,w}$-semistable for some $w>0$ or $\nu_{k,w'}$-semistable for some $w'>\frac{k^2}{2}$, $\mu_H(E)\in [0,1]$, and
\[f(\mu_H(E))\leq \frac{\bv_2(E)}{\bv_0(E)}.\]
Then we can choose such an object $E$ to further satisfy that $E_D$ is $\mu_{H_D}$-semistable for any $D\in |kH|$.
\end{lemma}

\subsection{Tilt-stability in families}

Next, we want to prove a relative version of Theorem \ref{thm-bms}. We begin with the study of tilt-stability under base change.

\begin{lemma}\label{lem:tilt-base-change}
Let $(X, H, \gamma)$ be a triple as in Setup \ref{setup-gamma} with $n\geq 2$ and either $d\geq 2$ or $X_{\overline{\kk}}$ is a normal projective surface. Fix a field extension $\kk\subset \mathsf{k}_1$ and let $\pi\colon X_{\kk_1}\to X$ be the base change morphism. Let $(b,w)\in U^{\gamma}_{X, H}$.

\begin{enumerate}
    \item We have $(\sigma^{b,w})_{\kk_1}=\sigma^{b,w}_{\kk_1}=(\cA^b(X_{\kk_1}), Z^{b,w})$.

    \item If $E_{\kk_1}\in \cA^b(X_{\kk_1})$ is $\nu_{b,w}$-semistable, then $E\in \cA^b(X)$ is $\nu_{b,w}$-semistable.

    \item If $\kk_1$ is algebraic over $\kk$, then $E$ is $\nu_{b,w}$-semistable if and only if $E_{\kk_1}$ is $\nu_{b,w}$-semistable.
\end{enumerate}

\end{lemma}

\begin{proof}
Part (a) follows from Lemma \ref{lem:rotate-base-change}(c), Lemma \ref{lem-constant-chM}, and \ref{lem:constant-generalization}. If $E_{\kk_1}$ is $\nu_{b,w}$-semistable, then by Lemma \ref{lem:rotate-base-change}(a), any subobject or quotient object $F$ of $E$ in $\cA^b(X)$ gives a subobject or quotient object $F_{\kk_1}$ of $E_{\kk_1}$ in $\cA^b(X_{\kk_1})$. So the $\nu_{b,w}$-semistability of $E$ follows.

Conversely, we assume that $\kk_1$ is algebraic over $\kk$. If $\nu_{b,w}(E)=+\infty$, the result follows from Lemma \ref{lem:rotate-base-change}(b). If $\nu_{b,w}(E)<+\infty$, from the existence of JH filtrations (cf.~Lemma \ref{lem:supp-property-finite-length}), we may assume that $E$ is $\nu_{b,w}$-stable. Then by Theorem \ref{thm:wall-chamber-abstract}, we can find an open ball $B(b,w)\subset U^{\gamma}_{X, H}$ containing $(b,w)$ so that for any $(b',w')\in B(b,w)\cap \QQ^2$, $E$ is $\nu_{b',w'}$-stable. Since $\kk_1$ is algebraic over $\kk$, by \cite[Proposition 14.20]{BLMNPS21}, we know that $E_{\kk_1}$ is $\nu_{b',w'}$-semistable for any $(b',w')\in B(b,w)\cap \QQ^2$. By the locally-finiteness of walls and density of rational numbers, we see that $E_{\kk_1}$ is $\nu_{b,w}$-semistable.
\end{proof}

Therefore, according to Lemma \ref{lem:tilt-base-change}(c) and Definition \ref{def:geo-stable}, we say $E$ is \emph{geometrically $\nu_{b,w}$-stable} if $E_{\overline{\kk}}\in \cA^b(X_{\overline{\kk}})$ is $\nu_{b,w}$-stable.

\begin{remark}
As a byproduct of Lemma \ref{lem:tilt-base-change}(a) and the condition \ref{wc2} proved in Theorem \ref{thm:tilt-HN-structure}, we know that the $\nu_{b,w}$-semistability and the geometric $\nu_{b,w}$-stability are preserved under any field extension of $\kk$ by the same argument in \cite[Proposition 14.20]{BLMNPS21}. We will not use this fact in our paper, so we omit the details.
\end{remark}

To verify the tilting property for tilt-stability, we need the following generalizations of \cite[Lemma 2.19]{bayer2017stability}.

\begin{lemma}\label{lem-blms-2.19}
Let $X$ be an equidimensional Noetherian scheme that admits a dualizing complex. Let $E\in \Db(X)$ such that $\cH^{i}(E)=0$ for $i\notin\{-1,0\}$, $\cH^{-1}(E)$ is zero or torsion-free, and $\Hom_X(K, E)=0$ for any sheaf $K$ on $X$ supported in codimension at least $2$. Then there is an exact triangle 
    \[E\to E^{\sharp \sharp}\to T,\]
    where $T$ is a torsion sheaf supported in codimension at least $3$ and $\cH^{-1}(E)\cong \cH^{-1}(E^{\sharp \sharp})$ is torsion-free $S_2$ if it is nonzero. If $X$ is also Cohen--Macaulay and has the resolution property, then $E^{\sharp \sharp}$ is quasi-isomorphic to a two-term complex $C^{-1}\to C^0$ where $C^{-1}$ is the tensor product of a locally free sheaf with $\omega_X$ (hence Cohen--Macaulay) and $C^0$ is torsion-free $S_2$.
\end{lemma}

\begin{proof}
In the proof, we denote by $(\cD^{\leq 0}, \cD^{\geq 0})$ the standard t-structure on $\Db(X)$, and $\tau_{\leq i}, \tau_{\geq i}$ by the truncation functors. Then we set $E^{\sharp}\coloneqq\tau_{\leq 0}\DD_X^{1-\dim X}(E)$ and $Q\coloneqq\tau_{\geq 1}\DD_X^{1-\dim X}(E)$ as in Lemma \ref{lem-6.3}. Similarly, we set $E^{\sharp\sharp}\coloneqq\tau_{\leq 0}\DD_X^{1-\dim X}(E^{\sharp})$ and $Q'\coloneqq\tau_{\geq 1}\DD_X^{1-\dim X}(E^{\sharp})$.

Applying $\DD_X^{1-\dim X}(-)$ to $\cH^{-1}(E)[1]\to E\to \cH^0(E)$, we get isomorphisms
\begin{equation}\label{eq:Esharp-H-1}
    (\cH^0(E))^{\mathsf{d}}\cong \cH^{-1}(\DD_X^{1-\dim X}(\cH^0(E)))\cong \cH^{-1}(\DD_X^{1-\dim X}(E))= \cH^{-1}(E^{\sharp})
\end{equation}
and a long exact sequence
\begin{equation}\label{eq-long-exact}
    0\to \cE xt^1_X(\cH^0(E), \omega^{\bullet}_X[-\dim X])\to \cH^0(E^{\sharp})\to (\cH^{-1}(E))^{\mathsf{d}}\xra{\delta} \cE xt^2_X(\cH^0(E), \omega^{\bullet}_X[-\dim X]).
\end{equation}
In particular, $\cH^{-1}(E^{\sharp})$ is zero or torsion-free by Lemma \ref{lem-general-S2}. Similarly, applying $\DD_X^{1-\dim X}(-)$ to $$\cH^{-1}(E^{\sharp})[1]\to E^{\sharp}\to \cH^0(E^{\sharp}),$$ we get isomorphisms
\begin{equation*}
    (\cH^0(E^{\sharp}))^{\mathsf{d}}\cong \cH^{-1}(\DD_X^{1-\dim X}(\cH^0(E^{\sharp})))\cong \cH^{-1}(\DD_X^{1-\dim X}(E^{\sharp}))= \cH^{-1}(E^{\sharp\sharp})
\end{equation*}
and a long exact sequence
\[0\to \cE xt^1_X(\cH^0(E^{\sharp}), \omega^{\bullet}_X[-\dim X])\to \cH^0(E^{\sharp\sharp})\to (\cH^{-1}(E^{\sharp}))^{\mathsf{d}}\xra{\delta'} \cE xt^2_X(\cH^0(E^{\sharp}), \omega^{\bullet}_X[-\dim X]).\]
From Lemma \ref{lem-supp-ext}, we know that $$\codim_X(\cE xt^1_X(\cH^0(E), \omega^{\bullet}_X[-\dim X]))\geq 1$$ and $$\codim(\cE xt^2_X(\cH^0(E), \omega^{\bullet}_X[-\dim X]))\geq 2,$$ so applying $\cH om_X(-,\omega^{\bullet}_X[-\dim X])$ to \eqref{eq-long-exact}, we obtain
\begin{equation}\label{eq-2.19}
    \cH^{-1}(E^{\sharp\sharp})=(\cH^0(E^{\sharp}))^{\mathsf{d}}\cong (\ker(\delta))^{\mathsf{d}}\cong ((\cH^{-1}(E))^{\mathsf{d}})^{\mathsf{d}}.
\end{equation}
In particular, $\cH^{-1}(E^{\sharp\sharp})$ is torsion-free and $S_2$.

Now, let $K\in \Coh(X)$ be any torsion sheaf supported in codimension $\geq 2$. From our assumption and Lemma \ref{lem-6.3}(b), we see that $Q\in \cD^{\geq 1}$ and is supported in codimension $\geq 3$. So we get $\Hom_X(K, Q)=\Hom_X(K, Q[-1])=0$. Therefore, we have $$\Hom_X(K, E^{\sharp})=\Hom_X(K, \DD_X^{1-\dim X}(E))=\Hom_X(E, \DD_X^{1-\dim X}(K)),$$ which combines with $\DD_X^{1-\dim X}(K)\in \cD^{\geq 1}$ by Lemma \ref{lem-supp-ext} and $E\in \cD^{\leq 0}$ imply $\Hom_X(K, E^{\sharp})=0$. By the fact that $\cH^{-1}(E^{\sharp})$ is torsion-free, we can apply Lemma \ref{lem-6.3} to $E^{\sharp}$ and see that $Q'$ is also supported in codimension $\geq 3$. Now, we consider the diagram
\[\begin{tikzcd}
	&& {E^{\sharp\sharp}} \\
	{\DD_X^{1-\dim X}(Q)} & E & {\DD_X^{1-\dim X}(E^{\sharp})} \\
	&& {Q'}
	\arrow[from=1-3, to=2-3]
	\arrow[from=2-1, to=2-2]
	\arrow[from=2-2, to=2-3]
	\arrow[from=2-3, to=3-3]
\end{tikzcd}\]
with $\DD_X^{1-\dim X}(Q), Q'\in \cD^{\geq 1}$ and all their cohomology sheaves are supported in codimension at least $3$. In particular, as $E\in \cD^{\leq 0}$, we have $\Hom_X(E,Q')=0$, so we have an induced morphism $E\to E^{\sharp\sharp}$. The cone $T$ of this morphism fits into an exact triangle $$Q'[-1]\to T\to \DD_X^{1-\dim X}(Q)[1]$$ and therefore has non-zero cohomology sheaves only in non-negative degrees. The long exact cohomology sequence of the exact triangle $E\to E^{\sharp\sharp}\to T$ then shows that $T\cong \cH^0(T)$ and the desired exact triangle follows. In particular, we also have $\cH^{-1}(E)=\cH^{-1}(E^{\sharp\sharp})$, hence it is $S_2$ by \eqref{eq-2.19} if it is nonzero.

For the last statement, when $X$ has the resolution property, we can choose a cochain complex $(G^{\bullet},d^{\bullet})$ of locally free sheaves with $G^i=0$ for $i>0$, which is quasi-isomorphic to $E^{\sharp}$. As $X$ is Cohen--Macaulay, we have $\omega^{\bullet}_X[-\dim X]\cong \omega_X$ and $E^{\sharp\sharp}=\tau_{\leq 0}\DD_X^{1-\dim X}(E^{\sharp})$ is quasi-isomorphic to the complex $$(G^{0})^{\mathsf{d}}\to \ker((d^{-2})^{\mathsf{d}})$$
where $(G^{0})^{\mathsf{d}}$ sits in degree $-1$. As we have an exact sequence $$0\to \ker((d^{-2})^{\mathsf{d}}) \to (G^{-1})^{\mathsf{d}}\to \im((d^{-2})^{\mathsf{d}}) \to 0,$$ from the fact that $(G^{-1})^{\mathsf{d}}\cong (G^{-1})^{\vee}\otimes \omega_X$ is torsion-free $S_2$ and $\im((d^{-2})^{\mathsf{d}})\subset (G^{-2})^{\mathsf{d}}$ is torsion-free, we see that $\ker((d^{-2})^{\mathsf{d}})$ is torsion-free and $S_2$ by \cite[Lemma 9]{kollar:duality}. Then we can define $C^0\coloneqq\ker((d^{-2})^{\mathsf{d}})$ and $C^{-1}\coloneqq(G^{0})^{\mathsf{d}}\cong (G^{0})^{\vee}\otimes \omega_X$, and the result follows.
\end{proof}

Although the following result is not needed in this section, we state it here as its proof is related to Lemma \ref{lem-blms-2.19}.

\begin{lemma}\label{lem-bmt-dual}
Let $(X, H, \gamma)$ be a triple as in Setup \ref{setup-gamma} with $X$ Gorenstein, $n\geq 2$, and either $d\geq 2$ or $X_{\overline{\kk}}$ is a normal projective surface. Assume that $B_{\gamma}=0$ and $(b,w)\in U^{\gamma}_{X,H}$. Let $E\in \cA^b(X)$ be an object with $\nu^{+}_{b,w}(E)<+\infty$. Then there is an exact triangle 
    \[E^{\sharp}\otimes \omega_X^{-1}\to \mathbb{D}^X_1(E)\to Q\otimes \omega_X^{-1},\]
where $E^{\sharp}\otimes \omega_X^{-1}\in \cA^{-b}(X)$, $\cH^j(Q)$ is a torsion sheaf supported in codimension at least $j+2$ for all $j\geq 1$, and $\cH^j(Q)=0$ for $j\leq 0$.
\end{lemma}

\begin{proof}
Since $E\in \cA^b(X)$ and $\nu^+_{b,w}(E)<+\infty$, $E$ satisfies the assumptions in Lemma \ref{lem-6.3}, so it remains to prove $E^{\sharp}\otimes \omega_X^{-1}\in \cA^{-b}(X)$.

By \eqref{eq:Esharp-H-1}, we have an isomorphism $$\cH^{-1}(E^{\sharp}\otimes \omega_X^{-1})\cong (\cH^0(E))^{\vee}.$$ As $\mu^{-}(\cH^0(E))>b$, by Lemma \ref{lem:dual-slope-stable}, we get $\mu^{+}(\cH^{-1}(E^{\sharp}\otimes \omega_X^{-1}))<-b$, so $\cH^{-1}(E^{\sharp}\otimes \omega_X^{-1})[1]\in \cA^{-b}(X)$. From \eqref{eq-long-exact}, we have an exact sequence
\[ 0\to \cE xt^1_X(\cH^0(E), \oh_X)\to \cH^0(E^{\sharp}\otimes \omega_X^{-1})\to (\cH^{-1}(E))^{\vee}\xra{\delta} \cE xt^2_X(\cH^0(E), \oh_X).\]
From Lemma \ref{lem-supp-ext}, we know that $\cE xt^1_X(\cH^0(E), \oh_X)$ is torsion, so we have $\mu^{-}(\cH^0(E^{\sharp}\otimes \omega_X^{-1}))=\mu^{-}(\ker(\delta))$. Note that $\mu^{+}(\cH^{-1}(E))\leq b$ since $E\in \cA^b(X)$, and $\mu^{+}(\cH^{-1}(E))\neq b$ as $\nu^+_{b,w}(E)<+\infty$. Therefore, using Lemma \ref{lem:dual-slope-stable}, we obtain $\mu^{-}((\cH^{-1}(E))^{\vee})>-b$. From Lemma \ref{lem-supp-ext}, $\cE xt^2_X(\cH^0(E), \oh_X)$ is supported in codimension $\geq 2$. Therefore, $$\mu^{-}(\ker(\delta))=\mu^{-}((\cH^{-1}(E))^{\vee})>-b.$$ Thus, we have $\mu^{-}(\cH^0(E^{\sharp}\otimes \omega_X^{-1}))>-b$ and $\cH^0(E^{\sharp}\otimes \omega_X^{-1})\in \cA^{-b}(X)$. So $E^{\sharp}\otimes \omega_X^{-1}\in \cA^{-b}(X)$ as desired.
\end{proof}

Now, we fix $(f\colon X\to S, \cL, \gamma)$ as in Setup \ref{setup-gamma-relative} with $n\geq 2$ and either $d\geq 2$ or $f$ is admissible (cf.~Definition \ref{def:adm}). The following result generalizes \cite[Proposition 25.3]{BLMNPS21} and \cite[Section 4.5]{piyaratne2019moduli}.

\begin{theorem}\label{thm:tilt-HN-structure}
Assume furthermore that $f$ is Cohen--Macaulay. Then for any $(b,w)\in U_{X/S, \cL}^{\gamma}$, each $\sigma^{b,w}_s$ has the tilting property and satisfies \ref{t3}, and $(\cA^b(X_s))^{Z^{b,w}}\subset \cA^b(X_s)$ is a Noetherian torsion subcategory. 

Moreover, if $B_{\gamma}=0$, $n\leq 3$, and $f$ is admissible, then the collection $\underline{\sigma}^{b,w}\coloneqq (\sigma^{b,w}_s)_{s\in S}$ satisfies \ref{c1} and

\begin{enumerate}
    \item the condition \ref{b2} holds with respect to $\Lambda$; moreover, if $n=3$ and $v\in \Lambda$, then $\cM^{\st}_{\underline{\sigma}^{b,w}}(v+x\bp)=\varnothing$ for $x\gg 0$, where $\bp\in \Lambda_{3}$ is the class of skyscraper sheaves,

    \item conditions \ref{c2} and \ref{wc2} hold,
    
    \item if $C\to S$ is a morphism essentially of finite type from a Dedekind scheme $C$, then the pair $\sigma^{b,w}_{C}\coloneqq (\cA^b(X_C), Z_{C}^{b,w})$ is a weak HN structure over $C$ and $\cA^b(X_C)$ has a $C$-torsion theory. Here, the central charge is defined by $Z_{C}^{b,w}\coloneqq(Z^{b,w}_K, Z^{b,w}_{C\text{-}\tor})$, where $$Z^{b,w}_K\colon  \KK_0(X_K) \xrightarrow{\bv_{\leq 2}} \Lambda_{\leq 2} \xrightarrow{Z^{b,w}} \CC$$ and
\[Z^{b,w}_{C\text{-}\tor}(E)\coloneqq Z^{b,w}(\bv_{\leq 2}(F))\]
for $E=i_{p*}F$ and $F\in\cA^b(X_p)$, and extend it to whole $\cA^b(X_C)_{C\text{-}\tor}$ by \cite[Lemma 6.11]{BLMNPS21}, and

    \item If $n=2$, $\underline{\sigma}^{b,w}$ is a stability condition on $\Db(X)$ over $S$ with respect to $\Lambda$. If $n=3$ and $(b,w)\in \QQ^2$, then $\underline{\sigma}^{b,w}$ is a weak stability condition on $\Db(X)$ over $S$ with respect to $\Lambda$.
\end{enumerate}

\end{theorem}

\begin{proof}
By Lemma \ref{lem:supp-codim-and-bv} and Lemma \ref{lem:weakstabilityfunction}, an object $E\in \cA^b(X_s)$ satisfies $Z^{b,w}(E)=0$ if and only if it is supported in codimension $\geq 3$, so the statement for $(\cA^b(X_s))^{Z^{b,w}}$ follows. Hence, the condition \ref{t1} for $\sigma_s^{b,w}$ is also verified. 

And by Lemma \ref{lem-blms-2.19}, \ref{t3} holds for $\sigma^{b,w}_s$, which is given by $E^{\sharp \sharp}=[C^{-1}\to C^0]$ for any object $E\in \cA^b(X_s)$ with $\nu^{+}_{b,w}(E)<+\infty$, since for any sheaf $A\in \Coh(X_s)$ supported in codimension at least $3$, we have $\Hom_{X_s}(A, C^{-1}[i])=\Hom_{X_s}(A, C^{0}[i])=0$ for $i\leq 1$ by Lemma \ref{lem-general-S2}(d). Moreover, since $C^{-1}=K\otimes \omega_{X_s}$ for a locally free sheaf $K$ on $X_s$, by Lemma \ref{lem-serre-duality}, we have
\begin{align*}
&\Hom_{X_s}(A, C^{-1}[2])=\Hom_{X_s}(A, K\otimes \omega_{X_s}[2])\\
&=\Hom_{X_s}(A\otimes K^{\vee}, \omega_{X_s}[2])=\Hom_{X_s}(\cO_{X_s},A\otimes K^{\vee}[n-2])^*=0.
\end{align*}
Thus, each $\sigma_s^{b,w}$ has the tilting property.

Now, assume that $B_{\gamma}=0$, $n\leq 3$, and $f$ is admissible. By Lemma \ref{lem-constant-chM} and \ref{lem:constant-generalization}, the condition \ref{c1} holds. In the following, we prove parts (a)-(d). We need a lemma.

\begin{lemma}\label{lem:finite-rotate}
Given a class $v\in \Lambda_{\leq 2}$. Then there is a finite set $F_v\subset \Lambda_{\leq 2}$ such that for any point $t\to S$ and any $\nu_{b,w}$-semistable object $E_t\in \cA^b(X_t)$ of class $\bv_{\leq 2}(E_t)=v$, the classes in $\Lambda_{\leq 2}$ of HN factors of $E_t$ with respect to $\sigma^b_t$ lie in $F_v$.
\end{lemma}

\begin{proof}
We may assume that $\bv^b_1(v)>0$, otherwise it is already $\sigma^b_t$-semistable. Let
\[
        0=E_0\subset E_1\subset\cdots\subset E_m=E_t
\]
be the HN filtration of \(E_t\) with respect to \(\sigma_t^b\), and set
\[
        F_i\coloneqq E_i/E_{i-1}.
\]
We define $E_-$ (resp.~$E_+$) the part of the HN filtration that all factors have negative (resp. positive) $\bv_0$. By Lemma \ref{lem:classify-rotate-stable}, when $\bv_0(F_i)\neq 0$, we have
\begin{equation}\label{eq:rotate-lp}
\frac{\bv_2(F_i)}{\bv_0(F_i)}\leq \Phi_{X/S,\cL}^{\gamma}(\mu(F_i)).
\end{equation}
Since \(F_i\in A^b(X_t)\), we also have
\begin{equation}\label{eq:eq-8.26-1}
\bv_1^b(F_i)> 0,\quad \sum_i \bv_1^b(F_i)=\bv_1^b(v)>0,\quad \sum_i \bv_2(F_i)-w\bv_0(F_i)=\bv_2(v)-w\bv_0(v).
\end{equation}
As $w>\Phi^{\gamma}_{X/S,\cL}(b)$ and $\Phi^{\gamma}_{X/S,\cL}$ is upper semicontinuous, we can find $\epsilon>0$ so that
\[\Phi^{\gamma}_{X/S,\cL}(x)\leq w-\frac{w-\Phi^{\gamma}_{X/S,\cL}(b)}{2}=:w-\delta\]
for any $|x-b|\leq \epsilon$. We set
\[N\coloneqq \max\left\{ 0, \sup_{|x-b|\leq \bv_1^b(v)}\Phi_{X/S,\cL}^{\gamma}(x)\right\},\]
which is a finite constant by Theorem \ref{thm:exist-bg-function}.

We first bound $\bv_0(F_i)$ uniformly. If $\bv_0(F_i)<0$, then $\bv_0(F_i)\leq -1$ and $\mu(F_i)<b$. Therefore, we have $\mu(F_i)\in (b-\bv_1^b(v),b)$. By the definition of $N$, we have
\begin{equation}\label{eq:low-bv2}
\bv_2(F_i)\geq N\bv_0(F_i).
\end{equation}
When $\mu(F_i)\leq b-\epsilon$, we see
\[\bv_1^b(F_i)=-\bv_0(F_i)(b-\mu(F_i))\geq -\epsilon\bv_0(F_i).\]
Combining with \eqref{eq:eq-8.26-1}, we get
\[0>\sum_{\bv_0(F_i)<0,\mu(F_i)\leq b-\epsilon}\bv_0(F_i)\geq \sum_{\bv_0(F_i)<0,\mu(F_i)\leq b-\epsilon}-\frac{1}{\epsilon}\bv_1^b(F_i)\geq -\frac{1}{\epsilon}\bv_1^b(v).\]
Hence, 
\[\sum_{\bv_0(F_i)<0,\mu(F_i)\leq b-\epsilon}\bv_2(F_i)\geq -\frac{N}{\epsilon}\bv_1^b(v).\]
When $\mu(F_i)\in (b-\epsilon,b)$, we have
\[\bv_2(F_i)\geq (w-\delta)\bv_0(F_i).\]
All together give
\[\bv_2(E_-)\geq -\frac{N}{\epsilon}\bv_1^b(E)+(w-\delta)\bv_0(E_-).\]
On the other hand, by the $\nu_{b,w}$-semistability of $E_t$, we have $\nu_{b,w}(E_-)\leq \nu_{b,w}(v)$, which together \eqref{eq:eq-8.26-1} gives
\begin{align}\label{eq:E-}
&\bv^b_1(v)|\nu_{b,w}(v)|+w\bv_0(E_-)\geq \bv^b_1(E_-)\nu_{b,w}(v)+w\bv_0(E_-)\\
&\geq \bv_2(E_-)\geq -\frac{N}{\epsilon}\bv_1^b(v)+(w-\delta)\bv_0(E_-). \notag
\end{align}
From $\delta>0$, we get a uniform bound for $\bv_0(E_-)$. Using a completely dual argument, we also get a uniform bound for $\bv_0(E_+)$, hence a uniform bound of any $\bv_0(F_i)$.
Combining these with \eqref{eq:eq-8.26-1}, we know that the set of $(\bv_0(F_i),\bv_1(F_i))\in \QQ^2$ and the number $m$ are both uniformly bounded. Together with \eqref{eq:E-} and its analog for $E_+$ again, we see that $(\bv_0(E_{\pm}),\bv_1(E_{\pm}),\bv_2(E_{\pm}))\in \QQ^3$ is uniformly bounded. In particular, \eqref{eq:low-bv2} shows that $(\bv_0(F_i),\bv_1(F_i),\bv_2(F_i))\in \QQ^3$ is uniformly bounded for $\bv_0(F_i)<0$. Using an analog of \eqref{eq:low-bv2} for $\bv_0(F_i)>0$, we also know that $(\bv_0(F_i),\bv_1(F_i),\bv_2(F_i))\in \QQ^3$ is uniformly bounded for $\bv_0(F_i)>0$. Since the class $v$ is fixed, the boundedness for $\bv_{\leq 2}(F_i)$ with $\bv_0(F_i)=0$ follows, which completes the proof of the lemma.
\end{proof}

By the above lemma, the stronger boundedness in part (a) follows from Proposition \ref{prop-rotate-slope-real-b}(c). Moreover, \ref{b2} also follows from \cite[Lemma 9.6]{BLMNPS21} and Proposition \ref{prop-rotate-slope-real-b}(c) once we know \ref{c2} holds.

Next, we prove \ref{c2}, which also completes the proof of part (a). Let $T\to S$ be a morphism essentially of finite type and $E\in \Db(X_T)$ be a $T$-perfect object. Assume that there exists $t_0\in T$ so that $E_{t_0}$ is geometrically $\nu_{b, w}$-stable of phase $\phi$. If $\phi=1$, so $\Im Z^b(E)=0$, then the geometric $\nu_{b, w}$-stability and the geometric $\sigma^{b}$-stability are the same, and the result follows from the property \ref{c2} of $\underline{\sigma}^{b}$. Therefore, we may assume that $0<\phi<1$.

We define a functor $\mathrm{Quot}_T^{\leq \phi}(E)$, which assigns any scheme $T'$ over $T$ to the set of morphisms $E_{T'}\twoheadrightarrow Q$ in $\Dqc(X_{T'})$ satisfying that $Q$ is $T'$-perfect, the morphism $E_y\to Q_y$ is surjective in $\cA^b(X_y)$ for any $y\in T'$, and  $\phi_{\sigma^{b,w}_{y}}(Q_y)\leq \phi$ for any $y\in T'$. By \ref{c1}, Proposition \ref{prop-rotate-slope-real-b}(b), and \cite[Proposition 11.6, Lemma 10.8]{BLMNPS21}, we see that $\mathrm{Quot}_T^{\leq \phi}(E)$ is an algebraic space locally of finite presentation over $T$. To prove its quasi-compactness, we need the next lemma, which is proved using a similar idea of Lemma \ref{lem:finite-rotate}.

\begin{lemma}\label{lem:finite-v2}
There is a finite set $F_E\subset \Lambda_{\leq 2}$ so that for any $t\in T$, the classes in $\Lambda_{\leq 2}$ of those $\nu_{b,w}$-semistable quotient $E_t\twoheadrightarrow Q_t$ in $\cA^b(X_t)$ with $\phi_{\sigma^{b,w}_t}(Q_t)\leq \phi$ lie in $F_E$.
\end{lemma}

\begin{proof}
By Lemma \ref{lem:openness-standard-heart}, we may assume that $\cH^i(E)=0$ for $i\notin \{-1,0\}$ and both $\cH^{-1}(E)$ and $\cH^0(E)$ are flat over $T$.

Let $v=\bv(E)$. Fix any $t\in T$ and any quotient $E_t\twoheadrightarrow Q_t$ in $\cA^b(X_t)$ with $\phi_{\sigma^{b,w}_t}(Q_t)\leq \phi$. We set $Q^i_t\coloneqq \cH^i(Q_t)$. Then we have an exact sequence
\[0\to Q^{-1}_t[1]\to Q_t\to Q^0_t\to 0\]
in $\cA^b(X_t)$, which implies
\begin{equation}\label{eq:bound-v1b}
-\bv_1^b(Q^{-1}_t),\bv^b_1(Q^0_t)\in [0,\bv_1^b(v)].
\end{equation}

We first bound $\bv_0(Q_t)$. Since $Q^0_t$ is a quotient sheaf of $\cH^0(E_t)$, we know that 
$$\bv_0(Q_t)=\bv_0(Q^0_t)-\bv_0(Q^{-1}_t)\leq \bv_0(Q^0_t)\leq  \bv_0(\cH^0(E)).$$ In particular, $\bv_0(Q_t)$ is uniformly bounded above. So to get a lower bound of $\bv_0(Q_t)$, it remains to prove an upper bound of $\bv_0(Q^{-1}_t)$. Let $A_1,\cdots,A_m$ be the HN factors of $Q^{-1}_t$ with respect to $\mu$. Then we have $\mu(A_i)\leq b$, so $\bv^b_1(A_i)\leq 0$ for each $1\leq i\leq m$. Using 
\[\bv_1^b(Q^{-1}_t)=\sum^m_{i=1} \bv^b_1(A_i) \in [-\bv_1^b(v),0],\]
we also have 
\[\bv^b_1(A_i) \in [-\bv_1^b(v),0].\]
Note that $\bv_0(A_i)>0$ implies $\bv_0(A_i)\geq 1$, so we obtain
\begin{equation}\label{eq:bound-hn-slope}
b\geq \mu(A_i)\geq b-\bv_1^b(v).
\end{equation}
If $\mu(A_i)\geq b-\epsilon$, we obtain
\begin{equation}\label{eq:v2-ai}
\bv_2(A_i)\leq (w-\delta)\bv_0(A_i).
\end{equation}
If $\mu(A_i)<b-\epsilon$, note that 
\[\epsilon\sum_{\mu(A_i)<b-\epsilon}\bv_0(A_i)<-\sum_{\mu(A_i)<b-\epsilon}\bv_1^b(A_i)\leq \bv^b_1(v),\]
so
\begin{equation}\label{eq:Me}
\sum_{\mu(A_i)<b-\epsilon}\bv_2(A_i)\leq \bv^b_1(v)\frac{M}{\epsilon},
\end{equation}
where
\[M\coloneqq \max \left\{0,\sup_{x\in [b-\bv_1^b(v), b-\epsilon]}\Phi_{X/S,\cL}^{\gamma}(x)\right\}\]
is a finite constant by Theorem \ref{thm:exist-bg-function}. Combining \eqref{eq:v2-ai} and \eqref{eq:Me}, we get
\begin{align*}
\bv_2(Q^{-1}_t)=&\sum_{\mu(A_i)\geq b-\epsilon} \bv_2(A_i)+\sum_{\mu(A_i)<b-\epsilon} \bv_2(A_i)\\ \leq& (w-\delta)\sum_{\mu(A_i)\geq b-\epsilon} \bv_0(A_i)+\bv^b_1(v)\frac{M}{\epsilon}\\
\leq & (w-\delta)\bv_0(Q^{-1}_t)+\bv^b_1(v)\frac{M}{\epsilon}.
\end{align*}
On the other hand, the $\nu_{b,w}$-semistability of $Q_t$ forces $\nu_{b,w}(Q^{-1}_t[1])\leq \nu_{b,w}(Q_t)\leq \nu_{b,w}(v)$, which together \eqref{eq:bound-v1b} implies
\[\bv_{2}(Q^{-1}_t)\geq w\bv_0(Q^{-1}_t)-|\nu_{b,w}(v)|\bv^b_1(v).\]
Hence, we see that
\begin{equation}\label{eq:bound-v2}
w\bv_0(Q^{-1}_t)-|\nu_{b,w}(v)|\bv^b_1(v)\leq \bv_{2}(Q^{-1}_t)\leq (w-\delta)\bv_0(Q^{-1}_t)+\bv^b_1(v)\frac{M}{\epsilon}.
\end{equation}
Using $\delta>0$, we obtain a uniform upper bound of $\bv_0(Q^{-1}_t)$, hence a uniform bound of $\bv_0(Q_t)$, as desired. From $0<\bv^b_1(Q_t)<\bv^b_1(v)$, we also know that the size of the set of $(\bv_0(Q_t), \bv_1(Q_t))\in \QQ^2$ is always bounded by a uniform constant. By \eqref{eq:bound-v1b}, the same holds for $(\bv_0(Q^0_t), \bv_1(Q^0_t))$ and $(\bv_0(Q^{-1}_t), \bv_1(Q^{-1}_t))$ as well.

It remains to bound $\bv_2(Q_t)$ uniformly. By the assumption, we know that
\[\bv_2(Q_t)\leq \bv^b_1(Q_t)\nu_{b,w}(v)+w\bv_0(Q_t).\]
Therefore, it suffices to get a uniform lower bound of $\bv_2(Q_t)$. Since $T$ is Noetherian, by the semi-continuity theorem and flatness of $\cH^0(E)$, we know that the Castelnuovo--Mumford regularity of $\cH^0(E)$ is bounded. Combining with the boundedness of $(\bv_0(Q^0_t), \bv_1(Q^0_t))$ and \cite[Lemma 1.7.9]{HL10}, we get a uniform bound for $\bv_2(Q^0_t)$. As
\[\bv_2(Q_t)=\bv_2(Q^0_t)-\bv_2(Q^{-1}_t),\]
the result follows from \eqref{eq:bound-v2}.
\end{proof}

In particular, there exists $\phi_0>0$ so that $\phi^-_{\sigma^{b,w}_{t}}(Q_t)\geq \phi_0$ for any point $[E_t\twoheadrightarrow Q_t]\in \mathrm{Quot}_T^{\leq \phi}(E)$, since the class in $\Lambda_{\leq 2}$ of the last HN factor of $Q_t$ lies in the finite set in Lemma \ref{lem:finite-v2}. Thus, the argument of \cite[Lemma 21.21]{BLMNPS21} and the support property imply that the set 
\begin{align*}
\Big\{ \bv_{\leq 2}(F_t)\in \Lambda_{\leq 2} \colon & [E_t\twoheadrightarrow Q_t]\in \mathrm{Quot}_T^{\leq \phi}(E) \text{ for some }  t\in T \text{ and } \\ & F_t \text{ is a JH factor } \text{ of an HN factor of } Q_t \text{ with respect to } \sigma^{b,w}_{t}\Big\}
\end{align*}
is finite. Combining this with the stronger boundedness in part (a) and \cite[Lemma 9.6]{BLMNPS21}, we see that $\mathrm{Quot}_T^{\leq \phi}(E)\to T$ is of finite type. Moreover, applying Corollary \ref{cor:C-torsion-theory-real-b} and \cite[Proposition 11.11, Lemma 11.21]{BLMNPS21}, we can conclude that $\mathrm{Quot}_T^{\leq \phi}(E)\to T$ is universally closed and \ref{c2} follows.

For \ref{wc2}, using \ref{c2} and Lemma \ref{lem:tilt-base-change}(b), the argument in \cite[Lemma 20.4]{BLMNPS21} applies in our case.

Finally, by Corollary \ref{cor:C-torsion-theory-real-b}, we know that $\cA^b(X_C)$ universally satisfies openness of flatness and has a $C$-torsion theory. Therefore, by (b) and Theorem \ref{thm-bms}, $\sigma^{b,w}_{C}$ meets all assumptions in \cite[Theorem 18.7]{BLMNPS21} and part (c) follows. Part (d) is immediate by the other parts and the definition.
\end{proof}

Combining with \cite[Theorem 21.24]{BLMNPS21}, we have the following generalization of \cite{toda:moduli-K3}.

\begin{corollary}\label{cor:moduli-surface}
Assume that $n=2$, $B_{\gamma}=0$, $S$ has characteristic $0$, and $f$ is admissible. Then for any $(b,w)\in U^{\gamma}_{X/S,\cL}$ and $v\in \Lambda$, the moduli stack $\cM_{\underline{\sigma}^{b,w}}(v)$ is an Artin stack of finite type over $S$, and admits a good moduli space which is proper over $S$.
\end{corollary}


\subsection{Rotating tilt-stability}

We end this section with a discussion on rotating tilt-stability. Let $(X, H, \gamma)$ be a triple as in Setup \ref{setup-gamma} with $n\geq 2$ and $d\geq 2$. For any $\mu\in \RR$ and $(b,w)\in U^{\gamma}_{X,H}$, we set
\[Z^{b,w,\mu}\coloneqq \frac{1}{\mathfrak{i}-\mu} Z^{b,w}\]
and let $\cA^{b,w,\mu}(X)\coloneqq \Coh^{b,w,\mu}_{H,\gamma}(X)$ be the tilted heart of $\cA^b(X)=\Coh^{b
}_{H,\gamma}(X)$ at the slope $\nu_{b,w}=\mu$ as in Section \ref{subsec-tilting-property}.

The following lemma generalizes \cite[Proposition 2.15]{bayer2017stability}.

\begin{lemma}\label{lem-rotate-tilt}
Assume that $X$ is Cohen--Macaulay. For any $(b,w)\in U^{\gamma}_{X, H}$ and $\mu\in \RR$, the pair $$\sigma^{b,w,\mu}\coloneqq (\cA^{b,w,\mu}(X), Z^{b,w,\mu})$$ is a weak stability condition on $\Db(X)$ with respect to $\Lambda_{\leq 2}$.
\end{lemma}

\begin{proof}
By Lemma \ref{lem-tilt-weak-stab}, this follows from the fact that $\sigma^{b,w}$ has the tilting property as proved in Theorem \ref{thm:tilt-HN-structure}.
\end{proof}

Its relative version follows the same proof as \cite[Proposition 26.1]{BLMNPS21}.

\begin{proposition}\label{prop-rotate-tilt}
Let $(f\colon X\to S, \cL, \gamma)$ be as in Setup \ref{setup-gamma-relative} such that $B_{\gamma}=0$, $2\leq n\leq 3$, and $f$ is Cohen--Macaulay and admissible. Then for any $(b,w)\in \QQ^2\cap U^{\gamma}_{X/S,\cL}$ and $\mu\in \QQ$, the collection $\underline{\sigma}^{b,w,\mu}\coloneqq (\sigma_{s}^{b,w,\mu})_{s\in S}$ is a weak stability condition on $\Db(X)$ over $S$ with respect to $\Lambda$.
\end{proposition}

\begin{proof}
As explained in \cite[Proposition 26.1]{BLMNPS21}, using Theorem \ref{thm:tilt-HN-structure}, Lemma \ref{lem-rotate-tilt}, and Lemma \ref{lem:classify-rotate-stable-abstract}, the same argument in Proposition \ref{prop-rotate-slope-real-b} applies verbatim.
\end{proof}

\section{Bayer–Macr\`i–Toda Conjecture}\label{sec:bmt}

We now focus on threefolds. We first discuss conjectural $\bch_3$-inequalities in Section \ref{subsec-bmt-conj}, generalizing the formulations in \cite{bayer2011bridgeland,bayer2016space}. Next, Theorem \ref{thm-stab} shows that these inequalities can be used to construct stability conditions. Finally, the behavior of $\bch_3$-inequalities under degeneration is investigated in Theorem \ref{thm-degeneration}.

Throughout this section, we follow the simplification of the notation as in Section \ref{sec:tilt-3}. We fix a triple $(X, H, \gamma)$ as in Setup \ref{setup-gamma} that has the standard BG function with $n=3$ such that either

\begin{enumerate}[(1)]
    \item $d=3$, i.e.~$X$ is lci, or

    \item $X_{\overline{\kk}}$ is normal integral $\QQ$-factorial and is lci in codimension $2$.
\end{enumerate}

\subsection{Conjectural inequalities}\label{subsec-bmt-conj}

Fix a linear map $\Gamma\in \Hom_{\QQ}(\CH^1_{\num}(X)_{\QQ}, \QQ)$ such that $\Gamma(H)\geq 0$. We set
\[\Gamma.D\coloneqq \Gamma(D)\]
for any $D\in \CH^1_{\num}(X)_{\QQ}$.

The following conjecture is first proposed by \cite{bayer2011bridgeland} in the smooth case. Recall that for any $b\in \RR$ and $v \in \Lambda$, we define
\[v^{b}_i\coloneqq\sum_{j=0}^i \frac{(-b)^j}{j!}v_{i-j}\in \Lambda_{\RR},\]
where $v_k$ is the $k$-th component of $v$.

\begin{conjecture}[{\cite[Conjecture 1.3.1]{bayer2011bridgeland}}]\label{conj-1}
Let $(b,w)\in \cU$ and $E\in \cA^{b}(X)$ be a $\nu_{b,w}$-stable object with $\nu_{b,w}(E)=b$. Then
\begin{equation}\label{eq-conj-8.1}
    \bv_3^{b}(E)\leq \frac{2w-b^2}{6}\bv_1^{b}(E)+\Gamma.\bch^{\gamma}_1(E)-b\frac{\Gamma.H}{H^3}\bv_0(E).
\end{equation}
\end{conjecture}

It is useful to set $a\coloneqq \sqrt{2w-b^2}$. Then we have
\[\nu_{b,w}(E)=\frac{\bv_2^{b}(E)-\frac{a^2}{2}\bv_0(E)}{\bv_1^{b}(E)}+b=\mu_{a,b}(E)+b\]
and \eqref{eq-conj-8.1} becomes
\[\bv_3^{b}(E)\leq \frac{a^2}{6}\bv_1^{b}(E) +\Gamma.\bch^{\gamma}_1(E)-b\frac{\Gamma.H}{H^3}\bv_0(E),\]
which coincides with \cite{bayer2011bridgeland}.

Following \cite{bayer2016space}, we introduce two further formulations, which will be shown to be equivalent.

\begin{conjecture}[{\cite[Conjecture 4.1]{bayer2016space}}]\label{conj-2}
Let $(b,w)\in \cU$ and $E\in \cA^{b}(X)$ be a $\nu_{b,w}$-semistable object. Then
\begin{align*}
0 \leq Q^{\Gamma}_{b,w}(E)
&\coloneqq (2w-b^2)\!\left(\bDelta(E)+3\,\frac{\Gamma. H}{H^3}(\bv_0(E))^2\right) \\
&\quad +\,2\bv^b_2(E)\left(2\,\bv_2^b(E)-3\frac{\Gamma. H}{H^3}\bv_0(E)\right) \\
&\quad -\,6\bv_1^b(E)\left(\bv_3^b(E)-\Gamma.\bch^{\gamma}_1(E)+b\frac{\Gamma.H}{H^3}\bv_0(E)\right). 
\end{align*}
\end{conjecture}

If $\Gamma=0$, then we write $Q^{\Gamma}_{b,w}$ as $Q_{b,w}$ for simplicity.

\begin{remark}\label{rmk:compute-bmt-at-wall}
Note that by rearranging, we find that
\begin{align}\label{Q-simple}
    \frac{1}{2}Q^{\Gamma}_{b,w}(E) =  & w\left((\bv_1(E))^2-2\bv_0(E)\bv_2(E) +3 \frac{\Gamma.H}{H^3} (\bv_0(E))^2\right) \notag \\
    & +b \big(3\bv_0(E)\bv_3(E)-\bv_1(E)\bv_2(E) -3\bv_0(E)\Gamma.\bch^{\gamma}_1(E)\big)\\
    & +2(\bv_2(E))^2-3\bv_1(E)\bv_3(E) -3 \frac{\Gamma.H}{H^3}\bv_0(E)\bv_2(E) +3\bv_1(E)\Gamma.\bch^{\gamma}_1(E)  . \nonumber
\end{align}
In particular, if $\ell$ is a wall for $E$, then for any point $(b,w)\in \ell$, the value of $\frac{1}{\bv_1^b(E)}Q^{\Gamma}_{b,w}(E)$ remains unchanged.
\end{remark}

\begin{remark}\label{rmk:ch3}
If $\gamma=(1,0,-\mathsf{D}H^2,0)$ for a constant $\mathsf{D}$, the above inequality can be written as
\begin{align*}
0 \leq 
& (2w-b^2)\!\left((\bch_1(E).H^2)^2-2(\bch_0(E).H^3)(\bch_2(E).H)+\left(2\mathsf{D}+3\,\frac{\Gamma. H}{H^3}\right)(\bch_0(E).H^3)^2\right) \\
&+\,2\mathsf{L}\left(2\mathsf{L}-3\frac{\Gamma. H}{H^3}\bch_0(E).H^3\right)-\,6(\bch_1(E).H^2-b\bch_0(E).H^3)\left(\mathsf{M}-\Gamma.\bch_1(E)+b\frac{\Gamma.H}{H^3}\bch_0(E).H^3\right), 
\end{align*}
where
\[\mathsf{L}\coloneqq \bch_2(E).H-b\bch_1(E).H^2+\left(\frac{1}{2}b^2-\mathsf{D}\right)\bch_0(E).H^3\]
and
\[\mathsf{M}\coloneqq \bch_3(E)-b\bch_2(E).H+\left(\frac{1}{2}b^2-\mathsf{D}\right)\bch_1(E).H^2-\left(\frac{1}{6}b^3-b\mathsf{D}\right)\bch_0(E).H^3.\]
\end{remark}

For any $E\in \cA^{b}(X)$, we define 
\[
\overline{b}(E) \coloneqq
\begin{cases}
\frac{\bv_2(E)}{\bv_1(E)} & \text{if } \bv_0(E) = 0, \bv_1(E)\neq 0 \\ \\[10pt]
\frac{\bv_1(E) - \sqrt{\bDelta(E)}}{\bv_0(E)} & \text{if } \bv_0(E) \neq 0.
\end{cases}
\]
In particular, we have
\begin{equation}\label{eq-barbeta}
    \bv_2^{\overline{b}(E)}(E)=0
\end{equation}
and $\bv_1^{\overline{b}(E)}(E)=\sqrt{\bDelta(E)}$ or $\bv_1(E)$. Note that $\left(\overline{b}(E),\frac{(\overline{b}(E))^2}{2}\right)$ is the end point of the curve $\nu_{b,w}(E)=b$ in $\cU$.

We say an object $E\in \Db(X)$ is \emph{$\overline{b}$-stable} if $E$ or $E[1]\in \cA^{\overline{b}(E)}(X)$ is $\nu_{b,w}$-stable for a neighborhood $U$ of $\left(\overline{b}(E),\frac{(\overline{b}(E))^2}{2}\right)$ and any $(b,w)\in U\cap \cU$.

\begin{conjecture}[{\cite[Conjecture 5.3]{bayer2016space}}]\label{conj-3}
Let $E\in \Db(X)$ be a $\overline{b}$-stable object. Then
\[\bv_3^{\overline{b}(E)}(E)\leq \Gamma.\bch^{\gamma}_1(E)-\overline{b}(E)\frac{\Gamma.H}{H^3}\bv_0(E).\]
\end{conjecture}

As in \cite{bayer2016space}, the three conjectures above are equivalent.

\begin{theorem}\label{thm:equivalent-conj}
Conjecture \ref{conj-3} is equivalent to Conjecture \ref{conj-1}, which is also equivalent to Conjecture \ref{conj-2}.
\end{theorem}

\begin{proof}
First, we show that Conjecture \ref{conj-3} implies Conjecture \ref{conj-1}. Assume that Conjecture \ref{conj-3} holds. Given a $\nu_{b,w}$-stable object $E\in \cA^{b}(X)$ with $\nu_{b,w}(E)=b$. We claim that the function 
\begin{equation}\label{eq-1}
    f(a,b)\coloneqq\bv_3^{b}(E)-\frac{a^2}{6}\bv_1^{b}(E)-\Gamma.\bch^{\gamma}_1(E)+b\frac{\Gamma.H}{H^3}\bv_0(E)
\end{equation}
increases as $a$ decreases along $\nu_{b,w}(E)=b$, or in other words, along the curve $\mu_{a,b}(E)=0$. Indeed, $\nu_{b,w}(E)=b$ is equivalent to
\begin{equation}\label{eq-w}
    \bv_2^{b}(E)=\frac{a^2}{2}\bv_0(E).
\end{equation}
Applying $d/da$ to both sides of \eqref{eq-w}, we get
\begin{equation}\label{eq-2}
    -b'\bv_1^{b}(E)=a\bv_0(E)
\end{equation}
where $b'\coloneqq db/da$. In particular, $b'\bv_0(E)\leq 0$. Then by \eqref{eq-w} and \eqref{eq-2}, we get
\begin{align*}
\frac{df(a,b)}{da}|_{\mu_{a,b}(E)=0}=&-b'\bv_2^{b}(E)-\frac{a}{3}\bv_1^{b}(E)+\frac{a^2}{6}b'\bv_0(E)+b'\frac{\Gamma.H}{H^3}\bv_0(E)\\
=&-\frac{a}{3}\bv_1^{b}(E)-\frac{a^2}{3}b'\bv_0(E)+b'\frac{\Gamma.H}{H^3}\bv_0(E)\\
=&-\frac{a}{3}\bv_1^{b}(E)(1-(b')^2)+b'\frac{\Gamma.H}{H^3}\bv_0(E).
\end{align*}
Therefore, by $\bv^{b}_1(E)> 0$, $b'\bv_0(E)\leq 0$, and $\Gamma.H\geq 0$, to prove the claim, it suffices to prove $1-(b')^2\geq 0$. To this end, we first assume that $\bv_0(E)=0$. Then \eqref{eq-w} implies $\bv_2(E)=b\bv_1(E)$. If $\bv_1(E)\neq 0$, then it is clear that $1-(b')^2=1\geq 0$. When $\bv_1(E)= 0$, then we have $\bv_1^{b}(E)=0$, which is impossible since $\nu_{b,w}(E)=b<+\infty$. Now we assume that $\bv_0(E)\neq 0$. By solving \eqref{eq-w}, we have
\[b=\frac{\bv_1(E)\pm\sqrt{\bDelta(E)+a^2(\bv_0(E))^2}}{\bv_0(E)}.\]
Therefore, we know that
\[(b')^2=\frac{a^2(\bv_0(E))^2}{\bDelta(E)+a^2(\bv_0(E))^2}\]
and the claim follows from $\bDelta(E)\geq 0$.

Now we do induction on $\bDelta(E)$. If $\bDelta(E)$ is minimal among all tilt-semistable objects, then $E$ is $\overline{b}$-stable by Lemma \ref{lem-delta-JH}. After taking limit $(a,b)\to (0,\overline{b}(E))$ in \eqref{eq-1}, by the monotonicity of $f(a,b)$, we see
\begin{align*}
&\bv_3^{b}(E)-\frac{a^2}{6}\bv_1^{b}(E)-\Gamma.\bch^{\gamma}_1(E)+b\frac{\Gamma.H}{H^3}\bv_0(E)\\
&\leq \bv_3^{\overline{b}}(E)-\Gamma.\bch^{\gamma}_1(E)+b\frac{\Gamma.H}{H^3}\bv_0(E)\\
&\leq 0
\end{align*}
by Conjecture \ref{conj-3} as desired. If $\bDelta(E)$ is not minimal and $E$ is still $\overline{b}$-stable, we are also done. If $\bDelta(E)$ is not minimal and $E$ is not $\overline{b}$-stable, then $E$ is destabilized along a wall between $(b,w)$ and $\left(\overline{b}(E),\frac{\overline{b}(E)^2}{2}\right)$. Let $F_1,\dots,F_m$ be the stable factors of $E$ along this wall. By Lemma \ref{lem-delta-JH} and the induction hypothesis, the inequality holds for $F_1,\dots,F_m$ on the wall, and so it does for $E$ by linearity of $\bch^{\gamma}_i(-)$ and $\bv_i(-)$. Then the inequality for $E$ at $(b,w)$ follows from the monotonicity of $f(a,b)$.

Now, we prove that Conjecture \ref{conj-1} implies Conjecture \ref{conj-3}. Let $E\in \cA^{\overline{b}(E)}(X)$ be a $\overline{b}$-stable object. Then we can directly take the limit $(a,b)\to (0,\overline{b}(E))$ to \eqref{eq-conj-8.1}, which gives $$\bv_3^{\overline{b}}(E)\leq \Gamma.\bch^{\gamma}_1(E)-\overline{b}\frac{\Gamma.H}{H^3}\bv_0(E).$$

Next, we show that Conjecture \ref{conj-2} implies Conjecture \ref{conj-1}. Assume that Conjecture \ref{conj-2} holds. Let $E\in \cA^{b}(X)$ be a $\nu_{b,w}$-stable object with $\mu_{a,b}(E)=0$. Then we have $\bv_2^{b}(E)=\frac{a^2}{2}\bv_0(E)$. Therefore, we get
\[Q^{\Gamma}_{b,w}(E)=a^2(\bv_1^b(E))^2-6\bv_1^b(E)\left(\bv_3^b(E)-\Gamma.\bch^{\gamma}_1(E)+b\frac{\Gamma.H}{H^3}\bv_0(E)\right)\geq 0.\]
Since $E\in \cA^{b}(X)$, we have $\bv_1^{b}(E)\geq 0$. Moreover, by the assumption $\nu_{b,w}(E)=b$, we see $\bv_1^{b}(E)>0$. Then from $Q^{\Gamma}_{b,w}(E)\geq 0$, we get 
\[\bv_3^{b}(E)\leq \frac{a^2}{6}\bv_1^{b}(E) +\Gamma.\bch^{\gamma}_1(E)-b\frac{\Gamma.H}{H^3}\bv_0(E)\]
as desired.

Finally, the implication of Conjecture \ref{conj-1} to Conjecture \ref{conj-2} follows from the same calculation as in  \cite[Proposition 2.8]{macri:fano-threefold} or \cite[Theorem 4.2]{bayer2016space}.
\end{proof}



\subsection{Construction of stability conditions}

Using Conjecture \ref{conj-3} and results in Section \ref{sec:general-tilt}, we can construct a continuous family of stability conditions on $\Db(X)$ as in \cite[Section 8]{bayer2016space}.

Recall that $a\coloneqq \sqrt{2w-b^2}$. For any $\Gamma\in \Hom_{\QQ}(\CH^1_{\num}(X)_{\QQ}, \QQ)$, we define a new graded lattice $\Lambda^{\Gamma}$ as the image of
\[\Knum(X)\to \QQ^4,\quad E\mapsto (\bv_0(E), \bv_1(E), \bv_2(E),\bv_3(E)-\Gamma.\bch_1^{\gamma}(E)).\]
Then we also define a homomorphism $Z^{a,b}_{\alpha,\beta}\colon \Lambda^{\Gamma} \to \CC$ by
\[Z^{a,b}_{\alpha,\beta}(v)\coloneqq-v_3^{b}-b\frac{\Gamma.H}{H^3}v_0+\beta v_2^{b}+\alpha v_1^{b}+\mathfrak{i}\left(v_2^{b}-\frac{1}{2}a^2v_0\right).\]
In particular, for any $E\in \Db(X)$, we have
\[Z^{a,b}_{\alpha,\beta}(E)=-\bv_3^{b}(E)+\Gamma.\bch^{\gamma}_1(E)-b\frac{\Gamma.H}{H^3}\bv_0(E)+\beta\bv_2^{b}(E)+\alpha\bv_1^{b}(E)+\mathfrak{i}\left(\bv_2^{b}(E)-\frac{1}{2}a^2\bv_0(E)\right).\]
We also define a heart $\cA^{a,b}(X)$ by
\[\cA^{a,b}(X)\coloneqq \langle \cF^{a,b}[1], \cT^{a,b} \rangle,\]
where
\[\cF^{a,b}\coloneqq\langle E\in \cA^{b}(X) \text{ is }\mu_{a,b}\text{-semistable with } \mu_{a,b}(E)\leq 0 \rangle\]
and
\[\cT^{a,b}\coloneqq\langle E\in \cA^{b}(X) \text{ is }\mu_{a,b}\text{-semistable with } \mu_{a,b}(E)> 0 \rangle.\]
Alternatively, we have $$\cA^{a,b}(X)=\cA^{b,\frac{a^2+b^2}{2},b}(X).$$

As in \cite[Section 8]{bayer2016space}, Conjecture \ref{conj-2} gives the following.

\begin{theorem}\label{thm-stab}
Assume that Conjecture \ref{conj-2} holds for $(X,H,\gamma)$ and $\Gamma\in \Hom_{\QQ}(\CH^1_{\num}(X)_{\QQ}, \QQ)$ with $\Gamma.H\geq 0$ when $\left(b,\frac{a^2+b^2}{2}\right)$ lies in an open subset $R\subset \cU$. Then the map
\[V_R\to \Stab_{\Lambda^{\Gamma}}(\Db(X)), \quad (a,b,\alpha,\beta)\mapsto \sigma^{a,b}_{\alpha,\beta}\coloneqq (\cA^{a,b}(X), Z^{a,b}_{\alpha,\beta})\]
is a continuous embedding, where
\[V_R\coloneqq \left\{(a,b,\alpha,\beta)\in \RR^4 \colon a>0,\alpha>\frac{1}{6}a^2+\frac{1}{2}|\beta|a, ~\left(b,\frac{a^2+b^2}{2}\right)\in R\right\}.\]
\end{theorem}

\begin{proof}
We first show that $Z^{a,b}_{\alpha,\beta}$ is a stability function on $\cA^{a,b}(X)$ for $(a,b,\alpha,\beta)\in V_R$. Let $F[1]\in \cA^{a,b}(X)$ be a non-zero object with $\Im Z^{a,b}_{\alpha,\beta}(F[1])=0$. Then $F\in \cA^b(X)$ is $\mu_{a,b}$-semistable. Therefore, we get 
\[\bv_3^{b}(F)-\Gamma.\bch^{\gamma}_1(F)+b\frac{\Gamma.H}{H^3}\bv_0(F)\leq \frac{a^2}{6}\bv_1^b(F).\]
So together with $\mu_{a,b}(F)=0$ and $\bDelta(F)\geq 0$, we have $(\bv_2^b(F))^2\leq \frac{1}{4}a^2(\bv_1^b(F))^2$, hence
\[\Re Z^{a,b}_{\alpha,\beta}(F[1])\leq \frac{a^2}{6}\bv_1^b(F)+\frac{1}{2}|\beta|a\bv_1^b(F)-\alpha\bv_1^b(F)<0\]
as desired.

When $(a,b,\alpha,\beta)\in V_R$ such that $(a,b)\in \QQ^2$, the same argument as in Theorem \ref{thm:tilt-stability} shows that $Z^{a,b}_{\alpha,\beta}$ satisfies the HN property. Moreover, the support property follows from Lemma \ref{lem:bms-support-property}, as the argument in \cite[Theorem 8.7]{bayer2016space} applies verbatim. Therefore, $\sigma^{a,b}_{\alpha,\beta}$ is a stability condition for $(a,b,\alpha,\beta)\in V_R$ such that $(a,b)\in \QQ^2$. The extension to the whole $V_R$ then follows identically from \cite[Proposition 8.10]{bayer2016space}.
\end{proof}

\begin{lemma}\label{lem:bms-support-property}
Let $\RR^4$ be the $4$-dimensional real vector space with coordinates $x_0,x_1,x_2,x_3$. Define a quadratic form on $\RR^4$ by $$\bDelta(x)\coloneqq x_1^2-2x_2x_0.$$ Define another quadratic form by $$\nabla^{a,b,\xi}(x)\coloneqq 3\xi a^2 x_0^2+2x_2(2x_2-3\xi x_0)-6x_1(x_3+b\xi x_0),$$ where $a\in \RR_{>0}$ and $b,\xi\in \RR$. If $\alpha>\frac{1}{6}a^2+\frac{1}{2}|\beta|a$, then there exists an open interval $I^{\alpha,\beta}_a\subset (a^2, 6\alpha)\subset \RR_{>0}$ such that the kernel of $$-x_3-b\xi x_0+\beta x_2 +\alpha x_1+\mathfrak{i}\left(x_2-\frac{1}{2}a^2x_0\right)\colon \RR^4\to \CC$$ is negative definite with respect to the quadratic form $K\bDelta+\nabla^{a,b,\xi}$ for all $K\in I^{\alpha,\beta}_a$.
\end{lemma}

\begin{proof}
The kernel is spanned by vectors $v_1=(0,1,0,\alpha)$ and $$v_2=\left(1,0,\frac{1}{2}a^2, \frac{1}{2}\beta a^2-b \xi\right),$$ where $\alpha, \beta\in \RR$. Therefore, the symmetric matrix of $K\bDelta+\nabla^{a,b,\xi}$ with respect to the basis $v_1,v_2$ is
\[
M=
\begin{pmatrix}
K-6\alpha & -\frac{3}{2}\beta a^{2}\\[4pt]
-\frac{3}{2}\beta a^{2} & a^{4}-Ka^{2}
\end{pmatrix}
\]
which is the same matrix as in \cite[Lemma 8.5]{bayer2016space}. Now the result follows from the remaining argument of \cite[Lemma 8.5]{bayer2016space}.
\end{proof}

Similar to \cite[Proposition 26.3]{BLMNPS21}, we have:

\begin{theorem}
Let $(f\colon X\to S, \cL, \gamma)$ be as in Setup \ref{setup-gamma-relative} that $B_{\gamma}=0$ and has the standard BG function and $n=3$. Assume that $f$ is also Cohen--Macaulay, admissible, and each fiber of $f$ satisfies Conjecture \ref{conj-2} for some $\left(b,\frac{a^2+b^2}{2}\right)\in \cU\cap \QQ^2$ and $\Gamma\in \A^2_{\star}(X/S)_{\QQ}$, then the collection 
\[\underline{\sigma}^{a,b}_{\alpha,\beta}\coloneqq(\sigma^{a,b}_{\alpha,\beta,s}=(\cA^{a,b}(X_s), Z^{a,b}_{\alpha,\beta}))_{s\in S}\]
is a stability condition on $\Db(X)$ over $S$ with respect to $\Lambda^{\Gamma}$ for $\{(\alpha,\beta)\in \QQ^2\colon \alpha>\frac{1}{6}a^2+\frac{1}{2}|\beta|a\}$.
\end{theorem}

\subsection{BMT Conjecture via degeneration}

We end this section with one of our main theorems (cf.~Theorem \ref{thm-degeneration}). The main idea is to use a semistable reduction argument as in \cite{BLMNPS21}. However, since we work with weak HN structures, the restriction of a semistable object to the special fiber need not remain semistable (cf.~\cite[Lemma 15.7]{BLMNPS21}). To overcome this issue, we need the following lemma, which says that the restriction to the special fiber preserves semistability up to modifications.

\begin{theorem}[{Semistable reduction of tilt-stability}]\label{thm-lift-tilt-ss}
Let $(f\colon X\to S, \cL, \gamma)$ as in Setup \ref{setup-gamma-relative} with $f$ Cohen--Macaulay, $n\leq 3$, $d\geq 2$, and $B_{\gamma}=0$. Fix $(b,w)\in U^{\gamma}_{X/S, \cL}$. Assume furthermore $C\to S$ is a morphism essentially of finite type from the spectrum $C$ of a DVR with the fraction field $K$ and the closed point $p\in C$. If $E_{K}\in \cA^{b}(X_{K})$ is a $\nu_{b,w}$-semistable object with $\bv_1^{b}(E_{K})\neq 0$, then we can find a $\sigma^{b,w}_{C}$-semistable object $F$ such that $F_p\in \cA^{b}(X_{p})$ is $\nu_{b,w}$-semistable and we have an exact sequence
\[0\to E_{K}\to F_{K}\to T_{K}\to 0\]
in $\cA^{b}(X_{K})$, where $T_{K}\in \Coh(X_K)$ is a torsion sheaf supported in codimension $\geq 3$.
\end{theorem}

\begin{proof}

We define a pair $\sigma^{b,w}_C$ as in Theorem \ref{thm:tilt-HN-structure}(c). By Corollary \ref{cor:C-torsion-theory-real-b}, we know that $\cA^b(X_C)$ universally satisfies openness of flatness and has a $C$-torsion theory. Moreover, since $C$ is the spectrum of a DVR, the generic point is open in $C$. Thus, $\sigma^{b,w}_{C}$ meets all assumptions in \cite[Theorem 18.7]{BLMNPS21} and we can conclude that $\sigma^{b,w}_C$ is a weak HN structure with a $C$-torsion theory. 


By \cite[Lemma 3.18]{BLMNPS21}, we can lift $E_{K}$ to an object $E\in \Db(X_C)$. As $\cA^b(X_C)$ is $C$-local, the restriction functor $\Db(X_C)\to \Db(X_K)$ is t-exact, and we may assume that $E\in \cA^b(X_C)$. Therefore, applying \cite[Proposition 15.10, Lemma 15.11]{BLMNPS21}, we may assume that $E$ is $\sigma^{b,w}_{C}$-semistable. By \cite[Proposition 17.6]{BLMNPS21}, we know that $\sigma^{b,w}_{C}$ has a $C$-torsion theory, so we can also assume that $E$ is $C$-torsion-free. In particular, $E_p\in \cA^b(X_p)$ by Lemma \ref{lem:6.12}.

First, we show that $\cH^{-1}(E)$ is a torsion-free sheaf if it is non-zero. Since $\cH^{-1}(E)\in \cF^{b}_C$, any non-zero subsheaf $W$ of $\cH^{-1}(E)$ satisfies that $W[1]\in \cF^{b}_C[1]\subset \cA^b(X_C)$, so $W[1]$ is a subobject of $E$. Then by the $C$-torsion-freeness of $E$, we have $W_K\neq 0$. However, when $W$ is torsion, $W_K$ is also a torsion sheaf on $X_K$, so $\mu_C(W)=+\infty$, contradicts $\cH^{-1}(E)\in \cF^{b}_C$ as $\mu^+_C(\cH^{-1}(E))\leq b$.

Next, we verify that $\Hom_{X_C}(G, E)=0$ for any sheaf on $X_C$ supported in codimension at least $2$. Indeed, if $G$ is not supported on $X_p$, then $\mu_{\sigma^{b,w}_C}(G)=+\infty$ and we have $\Hom_{X_C}(G, E)=0$ by the $\sigma^{b,w}_{C}$-semistability of $E$ and $\bv_1^{b}(E_{K})\neq 0$. If $G$ is supported on $X_p$, then $G\in \cA^{b}(X_C)$ is $C$-torsion and we have $\Hom_{X_C}(G, E)=0$ by \cite[Lemma 6.6(2)]{BLMNPS21} since $E$ is $C$-torsion-free. 

Therefore, we can apply Lemma \ref{lem-blms-2.19} to $E$ to get an exact sequence $$0\to E\to F\coloneqq E^{\sharp \sharp}\to T\to 0$$ in $\cA^{b}(X_C)$, where $T$ is a torsion sheaf supported in codimension $\geq 3$. In the following, we prove that $F$ satisfies the desired properties.

The first step is to show that $F$ is $C$-torsion-free. Indeed, if $0\neq i_{p*}G\subset F$ for $G\in \cA^{b}(X_p)$, then applying the snake lemma to
\[\begin{tikzcd}
	& 0 & {i_{p*}G} & {i_{p*}G} \\
	0 & E & F & T & 0
	\arrow[from=1-2, to=1-3]
	\arrow[no head, from=1-3, to=1-4]
	\arrow[shift left, no head, from=1-3, to=1-4]
	\arrow[from=1-3, to=2-3]
	\arrow[from=1-4, to=2-4]
	\arrow[from=2-1, to=2-2]
	\arrow[from=2-2, to=2-3]
	\arrow[from=2-3, to=2-4]
	\arrow[from=2-4, to=2-5]
\end{tikzcd}\]
we get $i_{p*}G\subset T$ as $E$ is $C$-torsion-free. This means $i_{p*}G$ is supported in codimension $\geq 3$, which is impossible since $F$ can be written as a complex $C^{-1}\to C^0$ with $C^{-1}$ torsion-free and Cohen--Macaulay, and $C^0$ torsion-free and $S_2$ (cf.~Lemma \ref{lem-blms-2.19}).

Next, we show that $F$ is $\sigma^{b,w}_{C}$-semistable. This is easy, as if $0\to G_1\to F\to G_2\to 0$ is a destabilizing sequence of $F$ with respect to $\sigma^{b,w}_{C}$, then the pullback of this exact sequence along $E\hookrightarrow F$ is a destabilizing sequence of $E$ by $\codim_{X_C}(T)\geq 3$.

Since $F$ is $C$-torsion-free, we have $F_p\in \cA^b(X_p)$ by Lemma \ref{lem:6.12}. It remains to show that $F_p$ is $\nu_{b,w}$-semistable. We write $F=[C^{-1}\to C^0]$, then $F_p=[C^{-1}_p\to C^0_p]$. As $C^0$ is a torsion-free and $S_2$ sheaf, we see that $C^0_p$ is a torsion-free sheaf on $X_p$. Moreover, $C^{-1}_p$ is torsion-free and $S_2$. Therefore, we have $\Hom_{X_p}(G, F_p)=0$ for any sheaf $G\in \Coh(X_p)$ supported in codimension $\geq 3$. Then from \cite[Lemma 15.7]{BLMNPS21}, we get the $\nu_{b,w}$-semistability of $F_p$.
\end{proof}

As a corollary, we have the following result. We will not use it in this paper.

\begin{corollary}
Let $(X, H, \gamma)$ be a triple as in Setup \ref{setup-gamma} with $2\leq n\leq 3$ and $d\geq 2$, and $\kk_1/\kk$ be a field extension. Assume that $B_{\gamma}=0$ and $X$ is Cohen--Macaulay. If there exists a pair $(b,w)\in U^{\gamma}_{X, H}$ and a $\nu_{b,w}$-semistable object $F\in \cA^b(X_{\kk_1})$ with $\bv^b_1(F)\neq 0$, then there exists a $\nu_{b,w}$-semistable object $E\in \cA^b(X)$ so that $[E_{\kk_1}]-[F]=[T]\in \KK_0(X_{\kk_1})$ for a sheaf $T\in \Coh(X_{\kk_1})$ supported in codimension $\geq 3$.
\end{corollary}

\begin{proof}
By \cite[Proposition 5.9(3)]{BLMNPS21}, we may assume that $\kk_1/\kk$ is finitely generated. Then the result follows from Theorem \ref{thm-lift-tilt-ss} and the same argument as in \cite[Theorem 12.17(2)]{BLMNPS21}.
\end{proof}

Now, we can prove one of our main results: The BMT Conjecture can be checked after degeneration, which generalizes \cite[Proposition 27.1]{BLMNPS21}. Note that we have no assumptions about the characteristics of the base scheme.

\begin{theorem}\label{thm-degeneration}
Let $(f\colon X\to S, \cL, \gamma)$ as in Setup \ref{setup-gamma-relative} with $f$ admissible and Cohen--Macaulay that has the standard BG function, $B_{\gamma}=0$, $n=3$, $d\geq 2$. Assume furthermore that $S=C$ is a Dedekind scheme. Fix $(b,w)\in \cU$ and $\Gamma_c\in \Hom_{\QQ}(\CH^1_{\num}(X_c)_{\QQ}, \QQ)$ for each $c\in C$ so that $\Gamma_{p}.D_{p}=\Gamma_{K}.D_K$ for any $D\in \CH^1(X)_{\QQ}$ and any closed point $p\in C$.

If there exists a closed point $p\in C$ such that 
\[Q^{\Gamma_p}_{b,w}(E_p)\geq 0\]
for any $\nu_{b,w}$-semistable object $E_p\in \cA^{b}(X_p)$, then \[Q^{\Gamma_K}_{b,w}(E_{K})\geq 0\]
holds for any $\nu_{b,w}$-semistable object $E_{K}\in \cA^{b}(X_{K})$.
\end{theorem}

\begin{proof}
Since the question is local around $p\in C$, we can assume that $C$ is affine. By base change to the localization at $p\in C$, we can assume that $C$ is the spectrum of a DVR. Assume that there exists a $\nu_{b,w}$-semistable object $E_{K}\in \cA^{b}(X_{K})$ such that $Q^{\Gamma_K}_{b,w}(E_{K})<0$. We may assume that $\bv_1^{b}(E_{K})>0$, hence $Q^{\Gamma_K}_{b,w}(-)$ decreases as $v_3$ increases. Therefore, by applying Theorem \ref{thm-lift-tilt-ss} to $E_{K}$, we can assume that $E_{K}$ has a $\sigma^{b,w}_{C}$-semistable lift $E$ such that $E_p$ is $\nu_{b,w}$-semistable. However, this implies $Q^{\Gamma_p}_{b,w}(E_{p})<0$ by Lemma \ref{lem-constant-chM} and \ref{lem:constant-generalization}, which makes a contradiction.
\end{proof}

\section{Singular Calabi--Yau/Fano threefolds}\label{sec:bmt-cy-fano}

In this section, we first reduce Conjecture \ref{conj-2} to a simple inequality for $\bch_2$ (cf.~Theorem \ref{thm:ch2}). Then we apply it to prove Conjecture \ref{conj-2} for many Fano threefolds (cf.~Corollary \ref{cor:fano3}) and Calabi--Yau threefolds (cf.~Section \ref{subsec:bmt-cy}). In particular, by Theorem \ref{thm-stab}, there exists an explicit family of stability conditions on their derived categories.

Throughout this section, we follow the simplification of the notation as in Section \ref{sec:tilt-3}. We fix a triple $(X, H, \gamma)$ as in Setup \ref{setup-gamma} with $\gamma=1$ and has the standard BG function with $n=3$ such that either

\begin{enumerate}[(1)]
    \item $X$ is lci, or

    \item $X_{\overline{\kk}}$ is normal integral $\QQ$-factorial Gorenstein\footnote{The only parts that we need the Gorenstein assumption are when applying Lemma \ref{lem-bmt-dual} in Lemma \ref{lem:reduce-bmt-to-bn} and Theorem \ref{thm:ch2}. We expect that Cohen--Macaulay is already enough.} and is lci in codimension $2$.
\end{enumerate}

Recall that
\[
\mathsf{E}(X)_{\QQ} \coloneqq
\begin{cases}
\A^2(X)_{\QQ} & \text{if } X \text{ is lci}, \\[10pt]
\CH_1(X)_{\QQ} & \text{if } X \text{ is } \QQ\text{-factorial but not lci}.
\end{cases}
\]
Then we have natural a homomorphism
\[\mathsf{E}(X)_{\QQ}\to \Hom_{\QQ}(\CH^1_{\num}(X)_{\QQ}, \QQ).\]
We also denote by $D\in \Hom_{\QQ}(\CH^1_{\num}(X)_{\QQ}, \QQ)$ the image of a class $D\in \mathsf{E}(X)_{\QQ}$.


\subsection{Reduction to a $\mathbf{ch}_2$-inequality}\label{subsec:reduce-to-ch2}
A useful way to verify Conjecture \ref{conj-2} is to reduce the problem to stable objects near $(b,w)=(0,0)$.

\begin{definition}
For an object $E\in \cA^0(X)$, we say $E$ is \emph{Brill--Noether (BN) stable} if $E$ is $\nu_{b,w}$-stable for any $(b,w)$ in an open subset
\[\left\{(b,w)\in \RR^2 \colon b^2+w^2<\delta, \frac{1}{2}b^2<w\right\}\]
for some $\delta>0$. We say $E$ is \emph{BN-semistable} if $E$ is $\nu_{0,w}$-semistable for every $0<w \ll 1$. 
\end{definition}

By Theorem \ref{thm:wall-chamber-abstract}, if $\bv_0(E)\neq 0$ and $\bv_2(E)\neq 0$, then $E$ is BN-stable if and only if it is $\nu_{b,w}$-stable for some $(b,w) \in \cU$ proportional to $\Pi(E)$; if $E$ is BN-semistable, then $E$ is $\nu_{b,w}$-semistable for some $(b,w) \in \cU$ proportional to $\Pi(E)$. 
For any object $E\in \cA^0(X)$, we define its Brill--Noether slope as
\[\nu_{BN}(E)\coloneqq\frac{\bv_2(E)}{\bv_1(E)}\]
when $\bv_1(E)\neq 0$ and set $+\infty$ otherwise.

The reason we care about BN-stable objects is the following reduction result.

\begin{lemma}[{\cite[Theorem 3.2]{Li19}, \cite[Theorem 2.3]{koseki:stability-cy-solid}}]\label{lem:reduce-bmt-to-bn}
Let $\Gamma\in \Hom_{\QQ}(\CH^1_{\num}(X)_{\QQ}, \QQ)$ be a linear map such that $\Gamma.H \geq 0$. 
If $$Q^{\Gamma}_{0,0}(E)\geq 0$$ for any BN-stable object $E\in \cA^0(X)$ with $\nu_{BN}(E)\in [0,\frac{1}{2}]$, then Conjecture \ref{conj-2} holds for $(X,H)$ and $\Gamma$ with 
\begin{equation}\label{eq:restricted-range}
w>\frac{1}{2}b^2+\frac{1}{2}(b-\lfloor b \rfloor)(\lfloor b \rfloor+1-b).
\end{equation}
\end{lemma}

\begin{proof}
Using the derived dual, Lemma \ref{lem:property-ch3}(b), Lemma \ref{lem-derived-dual-ch}, and Lemma \ref{lem-bmt-dual}, we know that $Q^{\Gamma}_{0,0}(E)\geq 0$ holds for any BN-stable object $E\in \cA^0(X)$ with $\nu_{BN}(E)\in [-\frac{1}{2},\frac{1}{2}]$.

Now, suppose the statement is false. Note that the set of values of $\bDelta$ is discrete, then there exists a $\nu_{b,w}$-semistable object $E\in \cA^b(X)$ with \[w>\frac{1}{2}b^2+\frac{1}{2}(b-\lfloor b \rfloor)(\lfloor b \rfloor+1-b)\]
with $Q^{\Gamma}_{b,w}(E)<0$ and minimal $\bDelta(E)$ among these objects. We may assume that $\bv_1^b(E)>0$, otherwise 
\[Q^{\Gamma}_{b,w}(E)=4(\bv_2(E)-w\bv_0(E))\left(\bv_2(E)-\left(\frac{1}{2}b^2+\frac{3\Gamma.H}{2H^3}\right)\bv_0(E)\right),\]
and one can directly check that $Q^{\Gamma}_{b,w}(E)\geq 0$.

Using Lemma \ref{lem-tensor-H}, Lemma \ref{lem-chM-general}(c), and Lemma \ref{lem:property-ch3}(a), as in the proof of \cite[Theorem 3.2]{Li19}, we may assume that $E$ is $\nu_{0,w}$-semistable and $Q^{\Gamma}_{0,w}(E)<0$ for some $w>0$. If $E$ is strictly $\nu_{0,w_0}$-semistable for some $0<w_0\leq w$, then from Lemma \ref{lem-delta-JH}, each JH factor $E_i$ of $E$ satisfies $\bDelta(E_i)\leq \bDelta(E)$ with equality holding if and only if $\bDelta(E_i)=\bDelta(E)=0$ and $\bv_{\leq 2}(E_i)$ is proportional to $\bv_{\leq 2}(E)$. By \cite[Lemma 11.7]{bayer2016space}, there exists a JH factor $E'$ such that $Q^{\Gamma}_{0,w_0}(E')<0$. If $\bDelta(E)>0$, then we have $\bDelta(E')<\bDelta(E)$, contradicting the minimality assumption. If $\bDelta(E)=0$, then $E'$ is $\nu_{0,w}$-stable for any $w>0$ by Lemma \ref{lem-delta-JH} and we can replace $E$ by $E'$. In each case, we can assume that $E$ is $\nu_{0,w}$-stable for $0<w\ll 1$. Now, the remaining proof is the same as \cite[Theorem 3.2]{Li19}.
\end{proof}

To work with BN-stable objects, we need the following useful lemmas.

\begin{lemma}[{\cite[Lemma 6.5]{bayer:brill-noether}, \cite[Lemma 2.12]{Li19}}]\label{lem:cone-bn-stable}
Fix $E\in \cA^0(X)$ to be a BN-stable object. If $\nu_{BN}(E)>0$, then for any subspace $W\subset \mathrm{Hom}(\cO_X,E)$, the object
\[\mathrm{cone}(\cO_X\otimes W\xrightarrow{\mathrm{ev}} E)\]
is in $\cA^0(X)$ and BN-semistable, where $\mathrm{ev}$ is the evaluation map.

If $\nu_{BN}(E)<0$, then for any subspace $V\subset \mathrm{Hom}(E[-1],\cO_X)$, the object
\[\mathrm{cone}(E[-1]\xrightarrow{\mathrm{coev}} \cO_X\otimes V^{\vee})\]
is in $\cA^0(X)$ and BN-semistable, where $\mathrm{coev}$ is the coevaluation map.
\end{lemma}

\begin{proof}
We only treat the case $\nu_{BN}(E)>0$, since the other one is similar. Let $$\wt{E}\coloneqq \mathrm{cone}(\cO_X\otimes W\xrightarrow{\mathrm{ev}} E).$$ By definition, we have an exact sequence
\[0\to E \to \wt{E} \to \cO_X[1]\otimes W \to 0\]
in $\cA^0(X)$. If $\bv_1(E)=0$, then $\wt{E}$ satisfies $\bv_1(\wt{E})=0$. In particular, $\wt{E}\in \cA^0(X)$ is $\nu_{0,w}$-semistable for any $w>0$.

Now, we assume that $\bv_1(E)\neq 0$. In this case, we consider the numerical wall $\ell$ for $E$ passing through $(0,0)$. Since $E$ is BN-stable, for any $(b,w)\in \ell \cap \cU$ with $b>0$ sufficiently small, we know that $E\in \cA^b(X)$ is $\nu_{b,w}$-stable and $\wt{E}\in \cA^b(X)$ is $\nu_{b,w}$-semistable, both have finite $\nu_{b,w}$-slope. Note that by Lemma \ref{lem:wt-inj}(d) and Theorem \ref{thm:tilt-HN-structure}, since $(\cO_X[1])^{\sharp \sharp}=\cO_X[1]$, we have a commutative diagram
\[\begin{tikzcd}
	0 & E & {\wt{E}} & {\cO_X[1]\otimes W} & 0 \\
	0 & {E^{\sharp \sharp}} & {(\wt{E})^{\sharp \sharp}} & {\cO_X[1]\otimes W} & 0
	\arrow[from=1-1, to=1-2]
	\arrow[from=1-2, to=1-3]
	\arrow[hook, from=1-2, to=2-2]
	\arrow[from=1-3, to=1-4]
	\arrow[hook, from=1-3, to=2-3]
	\arrow[from=1-4, to=1-5]
	\arrow[shift right, no head, from=1-4, to=2-4]
	\arrow[no head, from=1-4, to=2-4]
	\arrow[from=2-1, to=2-2]
	\arrow[from=2-2, to=2-3]
	\arrow[from=2-3, to=2-4]
	\arrow[from=2-4, to=2-5]
\end{tikzcd}\]
with rows exact in $\cA^b(X)$. In particular, the exact sequence at the bottom is the canonical one that comes from $W\subset \Hom(\cO_X, E^{\sharp \sharp})$. So from Lemma \ref{lem:wtE-stable}, we may assume that $E=E^{\sharp \sharp}$. Hence $(\wt{E})^{\sharp \sharp}=\wt{E}$ as well.

By the local finiteness of walls for $\wt{E}$, we know that $\wt{E}$ is either $\nu_{0,w}$-semistable for any $0<w \leq \epsilon$ and $0<\epsilon \ll 1$, or $\ell$ is the only wall for $\wt{E}$ in the region $$\{(b,w)\in \cU\colon b^2+w^2\leq \epsilon^2\}.$$ In the latter case, let $K\hookrightarrow \wt{E}\twoheadrightarrow Q$ be the destabilizing sequence at a point $(b_0, w_0)\in \ell$ with $0<b_0\ll 1$. Then $K^{\sharp \sharp}=K$ by Lemma \ref{lem:wt-inj}(c). Let $G\subset K$ be the first factor in a JH filtration of $K$ with respect to $\nu_{b_0,w_0}$, so we also have $G^{\sharp \sharp}=G$ by Lemma \ref{lem:wt-inj}(c). Note that $\wt{E}$ has a JH filtration with factors $E$ and $\cO_X[1]$. Since $K$ destabilizes $\wt{E}$ at $(b_0,w_0+\delta)$ for $0<\delta \ll 1$ and $G$ is also a JH factor of $\wt{E}$, using Proposition \ref{prop:unique-JH}, we see that $G^{\sharp \sharp}=G$ is isomorphic to $\cO_X[1]$. In particular, we get a non-zero map $\cO_X[1]\to K$ in $\cA^b(X)$, and hence $\Hom(\cO_X[1], \wt{E})\neq 0$, which contradicts the construction of $\wt{E}$.
\end{proof}

\begin{lemma}[{\cite[Lemma 2.8]{FKLR}}]\label{lem:vanish-hom-bn}
Assume that $K_X.H^2\leq 0$. For any object $E\in \cA^0(X)$, we have
\[\chi(E)\leq \dim_{\kk}\mathrm{Hom}_X(\cO_X, E)+\dim_{\kk}\mathrm{Ext}_X^2(\cO_X, E).\]
Furthermore, if $E$ is BN-stable, then we have
\[\chi(E)\leq \dim_{\kk}\mathrm{Hom}_X(\cO_X, E)\]
in one of the following cases.

\begin{enumerate}
    \item $\omega_X[1]\in \cA^0(X)$ is BN-semistable with $K_X.H^2\neq 0$ and $\nu_{BN}(E)>\nu_{BN}(\omega_X[1])$.

    \item $K_X$ is numerically trivial and $\nu_{BN}(E)>0$.
\end{enumerate}
\end{lemma}

\begin{proof}
Since $\cO_X[1], E\in \cA^0(X)$, the first inequality follows from the same argument of \cite[Lemma 2.8]{FKLR}.

In case (a), it is clear from the definition that we can find $w>0$ sufficiently small so that
\[\nu_{0,w}(E)>\nu_{0,w}(\omega_X[1]).\]
Therefore, we have $\mathrm{Hom}_X(E,\omega_X[1])=\mathrm{Ext}^2_X(\cO_X, E)=0$ and the second inequality follows.

In case (b), by Lemma \ref{lem-large-volume} and Lemma \ref{lem-delta-JH}, we know that $\omega_X[1]\in \cA^b(X)$ is $\nu_{b,w}$-stable for any $b\geq 0$ and $w> \frac{1}{2}b^2$. From the BN-stability of $E$ and $\nu_{BN}(E)>0$, we can find $(b_0,w_0)\in U$ with $b_0>0$ sufficiently small such that $E\in \cA^{b_0}(X)$ is $\nu_{b_0,w_0}$-stable and
\begin{equation}\label{eq-lem-sec-2}
    0<\frac{w_0}{b_0}<\nu_{BN}(E).
\end{equation}
Since \eqref{eq-lem-sec-2} implies
\[\nu_{b_0,w_0}(\omega_X[1])=\frac{w_0}{b_0}<\nu_{b_0,w_0}(E),\]
we get $\mathrm{Hom}_X(E,\omega_X[1])=\mathrm{Ext}^2_X(\cO_X, E)=0$ and the second inequality holds in this case as well.
\end{proof}

Following \cite{FKLR}, for a real number $\epsilon>0$, we consider:


\begin{minipage}{15.5cm}\vspace{2mm}
\begin{enumerate}[label=$\mathbf{BG(\epsilon)}$]
\item\label{BGn} \emph{~For any $\mu_H$-stable sheaf $E$ on $X$ with $\mu_H(E)\in (0,\epsilon]$, we have
\begin{equation}\label{eq:stBG}
    \bv_2(E)<-\frac{1}{2}\bv_1(E).
\end{equation}
}
\end{enumerate}
\end{minipage}

\medskip

We have the following generalization of \cite[Theorem 3.4]{FKLR}.

\begin{theorem} \label{thm:ch2}
Assume that

\begin{enumerate}
    \item $K_X$ is numerically equivalent to $tH$ for some $t\in \mathbb{R}_{\leq 0}$, and



    \item \ref{BGn} holds for $(X,H)$ and some $\epsilon>0$. 
\end{enumerate}

Then there exists a class $\Gamma (\epsilon)\in \mathsf{E}(X)_{\QQ}$ such that 
$\Gamma(\epsilon).H \geq 0$ and 
$$Q^{\Gamma(\epsilon)}_{0,0}(E) \geq 0$$ for any BN-stable object $E \in \cA^0(X)$ with $\nu_{\text{BN}}(E) \in [0, \frac{1}{2}]$. 
Explicitly, we can take 
\[\Gamma(\epsilon)\coloneqq\theta H^2-\td_{2}(X), \quad 
\theta \geq \max\left\{\frac{2+|\chi(\cO_X)|}{H^3\epsilon}, \frac{\td_2(X).H}{H^3},  \frac{2+2|1-\chi(\cO_X)|}{H^3\epsilon}-\frac{\td_2(X).H}{H^3}  \right\}. \]
In particular, Conjecture \ref{conj-2} holds for $(X,H)$, the class $\Gamma(\epsilon)\in \mathsf{E}(X)_{\QQ}$, and $(b,w)$ in the range \eqref{eq:restricted-range}.
\end{theorem}

\begin{proof}
We consider the universal extension 
\[
0\to E \to \widetilde{E} \to \Hom_X(\cO_X, E) \otimes \cO_X[1] \to 0, \]
which is an exact sequence in $\cA^0(X)$. First, assume that $\nu_{BN}(E) > 0$. Then by Lemma \ref{lem:cone-bn-stable}, $\widetilde{E}$ is BN-semistable and $\nu_{BN}(\widetilde{E})=\nu_{BN}(E)$. Note that the same argument in \cite[Lemma 3.3]{FKLR} applies to this case, so from the assumption \ref{BGn}, we have $\mu_H(E), \mu_H(\widetilde{E}) \notin [-\epsilon, \epsilon]$, which gives
\begin{equation} \label{eq:randa}
\bv_0(E) \leq \frac{1}{\epsilon}\bv_1(E),
\end{equation}
\begin{equation}\label{eq:slope}
-\frac{\bv_1(E)}{\epsilon} \leq  \bv_0(E),
\end{equation}
and 
\begin{equation} \label{eq:bound-hom}
-\frac{\bv_1(E)}{\epsilon} \leq  \bv_0(\widetilde{E}) = \bv_0(E) - \dim_{\kk}\mathrm{Hom}_X(\cO_X, E)H^3.
\end{equation}

On the other hand, using Lemma \ref{lem-chM-general}(a) in the lci case and Definition \ref{def:ch3} in the normal case, together with Lemma \ref{lem:vanish-hom-bn}, we have 
\begin{equation} \label{eq:RR-BN}
\bv_3(E)-\frac{K_X}{2}.\bch_2(E)+\td_{2}(X).\bch_1(E)+\chi(\cO_X)\frac{1}{H^3}\bv_0(E)=\chi(E) \leq \dim_{\kk}\mathrm{Hom}_X(\cO_X, E). 
\end{equation}
Combining the inequalities \eqref{eq:bound-hom} and \eqref{eq:RR-BN}, we get 
\begin{equation} \label{eq:ch3ineq}
\bv_3(E) \leq \frac{1}{H^3}(1-\chi(\cO_X))\bv_0(E)+\frac{1}{H^3\epsilon}\bv_1(E)-\td_{2}(X).\bch_1(E)+\frac{1}{2}K_X.\bch_2(E). 
\end{equation}
Then by \eqref{Q-simple} via writing $v_i\coloneqq \bv_i(E)$, $\chi\coloneqq \chi(\cO_X)$, and $\Gamma(\epsilon) = \theta H^2 -\td_2(X)$, we get the required lower bound 
\begin{align*}
\frac{1}{2}Q^{\Gamma(\epsilon)}_{0,0}(E) = \ & 2v_2^2-3 \left(\theta -\frac{ \td_2(X).H}{H^3} \right)v_0v_2 -3v_1v_3 +3v_1(\theta H^2 -\td_2(X)).\bch_1(E) \\ 
\overset{\eqref{eq:ch3ineq}}{\geq} \ & 2v_2^2 -3 \left(\theta -\frac{ \td_2(X).H}{H^3} \right)v_0v_2+3\theta v_1^2 - 3v_1\left(\frac{v_0}{H^3}(1-\chi) +\frac{v_1}{H^3\epsilon}+\frac{1}{2}K_X.\bch_2(E)\right) \\
\overset{\mathrm{(BG)}}{\geq} \ & 2v_2^2-3 \left(\theta -\frac{ \td_2(X).H}{H^3} \right)\frac{v_1^2}{2}+3\theta v_1^2 - 3v_1\left(\frac{v_0}{H^3}(1-\chi) +\frac{v_1}{H^3\epsilon}+\frac{1}{2}K_X.\bch_2(E)\right)\\
= \ & 2v_2^2+\frac{3}{2}v_1^2 \left(\theta +\frac{ \td_2(X).H}{H^3} - \frac{2}{H^3\epsilon}\right) - 3v_1\left(\frac{v_0}{H^3}(1-\chi) +\frac{1}{2}K_X.\bch_2(E)\right)    \\ 
\geq \ & \frac{3}{2}v_1^2 \left(\theta +\frac{ \td_2(X).H}{H^3} - \frac{2}{H^3\epsilon}\right) - 3(1-\chi)\frac{v_1 v_0}{H^3}.
\end{align*}
Here, (BG) means the inequality given by the standard BG function, which is part of our assumption in this section. In addition, we used $v_1,v_2\geq 0$ and assumption (a) in the last inequality. If $1-\chi\geq 0$, then by \eqref{eq:randa}, we have
\[- 3(1-\chi)\frac{v_1 v_0}{H^3}\geq \frac{-3(1-\chi)}{H^3\epsilon}v_1^2.\]
Similarly, if $1-\chi< 0$, then by \eqref{eq:slope}, we have
\[- 3(1-\chi)\frac{v_1 v_0}{H^3}\geq \frac{3(1-\chi)}{H^3\epsilon}v_1^2.\]
Therefore, we obtain
\begin{align*}
\frac{1}{2}Q^{\Gamma(\epsilon)}_{0,0}(E) \geq \ & \frac{3}{2}v_1^2 \left(\theta +\frac{ \td_2(X).H}{H^3} - \frac{2}{H^3\epsilon}\right) - 3(1-\chi)\frac{v_1 v_0}{H^3} \\
\geq \ & \frac{3}{2}v_1^2 \left(\theta + \frac{\td_2(X).H}{H^3}   - \frac{2+2|1-\chi|}{H^3\epsilon} \right)\\
\geq \ & 0, 
\end{align*}
where the last inequality follows from the assumption on $\theta$.

It remains to consider the case that $\nu_{BN}(E)=0$, i.e. $\bv_2(E)=0$. 
In this case, the inequality $Q^{\Gamma(\epsilon)}_{0,0}(E) \geq 0$ is equivalent to 
$\bv_3(E) \leq \Gamma(\epsilon).\bch_1(E)$. 
Following the proof of \cite[Proposition 3.3]{Li19}, for any $0 < \delta \ll 1$, there exists a filtration of $\widetilde{E}$ such that each factor $E_i$ is $\nu_{0, \alpha_i}$-semistable for some $\alpha_i>0$, and satisfies $\nu_{BN}(E_i) < \delta$. 
Applying \cite[Lemma 3.3]{FKLR} to each $E_i$, we have 
$\mu_H(E_i) \notin [-\epsilon, \epsilon]$, which implies $\bv_1(E_i)>-\epsilon\bv_0(E_i)$.
Thus, we get $\bv_1(\widetilde{E})>-\epsilon\bv_0(\widetilde{E})$ and so inequality \eqref{eq:bound-hom} still holds in this case. When $K_X$ is numerically equivalent to $tH$ with $t<0$, Lemma \ref{lem:vanish-hom-bn} still implies \eqref{eq:RR-BN}, so the remaining argument is the same as above. When $K_X$ is numerically trivial, similar to the arguments in \cite[Proposition 3.3]{Li19}, using the derived dual, Lemma \ref{lem:property-ch3}(b), Lemma \ref{lem-derived-dual-ch}, and Lemma \ref{lem-bmt-dual}, we also have 
\[
\dim_{\kk}\mathrm{Ext}^2_X(\cO_X, E) \leq \frac{\bv_1(E)}{H^3\epsilon}-\frac{1}{H^3}\bv_0(E). 
\]
Combining these inequalities with Lemma \ref{lem-chM-general}(a) and \cite[Lemma 2.8]{FKLR}, we get 
\begin{align*}
\chi(E)&=\bv_3(E)+\td_{2}(X).\bch_1(E)+\chi(\cO_X)\frac{1}{H^3}\bv_0(E)\\
& \leq \dim_{\kk}\mathrm{Hom}_X(\cO_X, E)+\dim_{\kk}\mathrm{Ext}_X^2(\cO_X, E) \\
&\leq \frac{2\bv_1(E)}{H^3\epsilon}. 
\end{align*}
Therefore, we obtain
\[\bv_3(E)\leq \left( \frac{2}{H^3\epsilon}H^2-\td_2(X)\right).\bch_1(E)-\frac{\chi(\cO_X)}{H^3}\bv_0(E).\]
Combining with \eqref{eq:randa} and \eqref{eq:slope}, we obtain 
$\bv_3(E) \leq \Gamma(\epsilon).\bch_1(E)$ as required. The last statement now follows from Lemma \ref{lem:reduce-bmt-to-bn}.
\end{proof}

Using Proposition \ref{prop-restriction}, Theorem \ref{thm:wall-chamber-tilt}, and Lemma \ref{lem-delta-JH}, \cite[Lemma 3.8, Proposition 3.9]{FKLR} work in our case verbatim.

\begin{lemma} \label{lem:fklr-restriction}
Assume that \ref{BGn} fails for $(X, H)$ and some $0<\epsilon<1/3$. Then there exists a $\mu_H$-stable torsion-free $S_2$ sheaf $E$ on $X$ with
\begin{equation*}
    0 < \mu_H(E) \leq \frac{2\epsilon}{1 - \epsilon} \quad \text{and} \quad 
    \bv_2(E)\geq -\frac{1}{2} \bv_1(E),
\end{equation*}
such that for any divisor $D \in |H|$, each HN factor $F_i$ of $E|_D$ (with respect to $\mu_{H_D}$-semistability) satisfies
\[
0 \leq \mu_{H_D}(F_i) \leq \frac{2\epsilon}{1 - 3\epsilon}.
\]
\end{lemma}

\begin{lemma} \label{prop:surface-to-3fold-1}
Suppose that $D \in |H|$ is a divisor that is lci or geometrically normal, and \hyperref[BGn]{\ensuremath{\mathbf{BG(\delta)}}} holds for $(D, H_D)$ and some $\delta >0$. Assume furthermore that $\bv_2(F)\leq 0$ for any $\mu_{H_D}$-stable sheaf $F$ with $\mu_{H_D}(F)=0$. Then \ref{BGn} holds for $(X, H)$ and $\epsilon \coloneqq \frac{\delta}{2+3\delta}.$
\end{lemma}

\subsection{Singular Fano threefolds}
In the rest of this section, we assume that the base field $\kk$ is algebraically closed of characteristic zero.

A \emph{Fano threefold} $X$ is a normal projective $3$-dimensional variety over $\kk$ with rational Gorenstein singularities and $-K_X$ ample. Thus, the classical BG inequality holds for $X$ by Theorem \ref{thm-bg-normal}. In this case, by \cite[Proposition 2.1.2]{iskovskikh1999fano}, $\Pic(X)$ is always a finite rank lattice. We define its \emph{index $\iota(X)$}, \emph{degree $d(X)$}, and \emph{genus $g(X)$} by 
\[\iota(X)\coloneqq\max\left\{0\neq i\in \ZZ\colon \frac{1}{i}K_X\in \Pic(X)\right\},\]
\[d(X)\coloneqq\frac{(-K_X)^3}{\iota(X)^3},\]
and
\[g(X)\coloneqq\frac{1}{2}(-K_X)^3+1.\]
We set $$H\coloneqq -\frac{1}{\iota(X)}K_X.$$ If $\Pic(X)$ has rank one, we call it a \emph{prime Fano threefold}. In this case, $H$ is the unique ample generator of $\Pic(X)$. By \cite{iskovskikh1999fano}, for a Fano threefold $X$, we have $1\leq \iota(X)\leq 4$. Moreover, $\iota(X)=4$ if and only if $X\cong \PP^3$, and $\iota(X)=3$ if and only if $X\subset \PP^4$ is a quadric threefold. 

Now, we want to apply Theorem \ref{thm:ch2} to Fano threefolds. To this end, we need an easy lemma.

\begin{lemma}\label{lem:k3}
Let $S$ be a normal projective surface with $K_S$ numerically trivial and rational Gorenstein singularities, and $H$ be an ample line bundle on $S$. Assume that $\chi(\cO_S)>1$. Then \ref{BGn} holds for $(S, H)$ and some $\epsilon>0$.
\end{lemma}

\begin{proof}
The argument is similar to \cite[Proposition 3.11]{FKLR}. By replacing $H$ by a multiple of itself, we may assume that $H$ is very ample. We assume that \ref{BGn} fails for $(S, H)$ and some $0<\epsilon \ll 1$. Then by \cite[Lemma 3.8]{FKLR}, we can take a reflexive $\mu_H$-stable sheaf $E$ on $S$ that violates \ref{BGn} such that
\begin{equation*}
    0 < \mu_H(E) \leq \frac{2\epsilon}{1 - \epsilon}
\end{equation*}
and for a general integral divisor $C \in |H|$, each HN factor $F_i$ of $E|_C$ (with respect to $\mu_{H_C}$-stability) satisfies
\[
0 \leq \mu_{H_C}(F_i) \leq \frac{2\epsilon}{1 - 3\epsilon}.
\]
Therefore, if we fix an integral curve $C\in |H|$ and choose $\epsilon\leq \frac{1}{5}$, from \cite[Lemma 4.3]{FKLR}, we get
\[
\frac{\dim_{\kk}\mathrm{H}^0(F_i)}{\rk(F_i)} \leq 1+\frac{H^2\mu_{H_C}(F_i)}{2}.
\]
Summing over all factors, we have
\begin{equation} \label{eq:Cliff(F)-1}
\dim_{\kk}\mathrm{H}^0(E|_C) \leq \frac{1}{H^2}\bv_0(E)+\frac{\bv_1(E)}{2}. 
\end{equation}
Note that the exact sequence
\[0 \to E(-H) \to E \to E|_C \to 0\]
implies $\mathrm{H}^0(E)\subset \mathrm{H}^0(E|_C)$. Combining this with \eqref{eq:Cliff(F)-1}, Lemma \ref{lem-chM-general}(a), and 
\[\chi(E)\leq \dim_{\kk}\mathrm{H}^0(E)+\dim_{\kk}\mathrm{H}^0(E^{\vee}(K_S))=\dim_{\kk}\mathrm{H}^0(E),\]
we obtain
\[\bv_2(E)+\chi(\cO_S)\frac{1}{H^2}\bv_0(E)=\chi(E)\leq \dim_{\kk}\mathrm{H}^0(E|_C)\leq \frac{1}{H^2}\bv_0(E)+\frac{\bv_1(E)}{2},\]
which implies
\[\frac{\bv_2(E)}{\bv_0(E)}\leq \frac{1-\chi(\cO_S)}{H^2}+\frac{1}{2}\mu_H(E).\]
Therefore, once we take $$0<\epsilon<\min\left\{\frac{1}{5},\frac{\chi(\cO_S)-1}{2H^2+\chi(\cO_S)-1}\right\},$$ we get
\[\frac{\bv_2(E)}{\bv_0(E)}<-\frac{1}{2}\mu_H(E),\]
which is a contradiction.
\end{proof}

Combining this with Theorem \ref{thm:ch2}, we have:

\begin{corollary}\label{cor:fano3}
Let $X$ be a Fano threefold that is either lci or $\QQ$-factorial. Then \ref{BGn} holds for $(X,-K_X)$ and some $\epsilon>0$. In particular, Conjecture \ref{conj-2} holds for $(X,-K_X)$, a class  $\Gamma\in \mathsf{E}(X)_{\QQ}$ with $\Gamma.(-K_X)\geq 0$, and $(b,w)$ in the range \eqref{eq:restricted-range}.
\end{corollary}

\begin{proof}
By \cite[Theorem 2.3.3]{iskovskikh1999fano}, we can find a K3 surface $S\in |-K_X|$ with rational Gorenstein singularities. Therefore, Lemma \ref{lem:k3} and Lemma \ref{prop:surface-to-3fold-1} imply that \ref{BGn} holds for $(X,-K_X)$. Now, the last statement follows from Theorem \ref{thm:ch2}.
\end{proof}

Note that the above result is not optimal. For smooth Fano threefolds of Picard number one, Conjecture \ref{conj-2} for $\Gamma=0$ and the full range $w>\frac{1}{2}b^2$ is proved in \cite{li:fano-3fold}. For higher Picard rank cases, the result for $w>\frac{1}{2}b^2$ and an effective choice of $\Gamma$ is obtained in \cite{macri:fano-threefold}. We expect that a more explicit version of Corollary \ref{cor:fano3} can be proved using similar arguments in \cite{li:fano-3fold,macri:fano-threefold}. In the following, we treat some index $\geq 2$ cases.

We start with an easy lemma.

\begin{lemma}\label{lem-hrr-Fano}
Assume that $X$ is a prime Fano threefold of index $\iota(X)$ and degree $d(X)$, and is $\QQ$-factorial or is lci with $\td_2(X)$ numerically proportional to $K_X^2$. Then for any $E\in \Db(X)$, we have
\[\chi(E)=\bv_3(E)+\frac{\iota(X)}{2}\bv_2(E)+a(\iota(X))\bv_1(E)+\frac{1}{d(X)}\bv_0(E)\]
where $a(1)=\frac{1}{12}+\frac{2}{d(X)}$, $a(2)=\frac{1}{3}+\frac{1}{d(X)}$, $a(3)=\frac{13}{12}$, and $a(4)=\frac{11}{6}$.
\end{lemma}

\begin{proof}
By Lemma \ref{lem-chM-general}(a) in the lci case and Definition \ref{def:ch3} in the $\QQ$-factorial case, we know that
\[\chi(E)=\bch_3(E)+\frac{\iota(X)}{2}H.\bch_2(E)+\td_2(X).\bch_1(E)+\frac{1}{d(X)}\bv_0(E).\]
When $X$ is $\QQ$-factorial and Picard number one, there exists a unique $k(E)\in \QQ$ such that $\bch_1(E)$ is numerically equivalent to $k(E)H$. Thus, we have $$\td_2(X).\bch_1(E)=k(E).\td_2(X).H.$$ Since $\bv_1(E)=H^2.\bch_1(E)=k(E)d$, if we set $a_2\coloneqq\frac{\td_2(X).H}{d}$, then $\td_2(X).\bch_1(E)=a_2\bv_1(E)$. When $X$ is lci and $\td_2(X)$ is proportional to $K_X^2$, we also have $\td_2(X).\bch_1(E)=a_2\bv_1(E)$. Therefore, in both cases, we can write 
\[\chi(E)=\bv_3(E)+\frac{\iota(X)}{2}\bv_2(E)+a_2\bv_1(E)+\frac{1}{d(X)}\bv_0(E).\]
Now using $\chi(\oh_X(-H))=0$ for $\iota(X)\geq 2$ and $\chi(\oh_X(-H))=-1$ for $\iota(X)=1$, we can solve $a_2$ and the result follows.
\end{proof}

Now, using Lemma \ref{lem-hrr-Fano}, a similar argument as in \cite[Section 2]{li:fano-3fold} gives the following result. The only difference is that we do not have \cite[Proposition 3.12]{bayer2016space}, so we need to deal with the case $\bDelta(E)=0$ separately.

\begin{theorem}\label{thm-fano3}
Let $X$ be a prime Fano threefold of index $\iota(X)\geq 2$ and degree $d(X)$. Assume that it also satisfies one of the following conditions.

\begin{itemize}
    \item $X\subset \PP^4$.
    
    \item $\iota(X)=2$ and $X$ is $\QQ$-factorial with $d(X)\leq 6$.

\end{itemize}
Then Conjecture \ref{conj-2} holds for $(X,H)$ with $\Gamma=0$ and any $w>\frac{1}{2}b^2$.
\end{theorem}

\begin{proof}
By Theorem \ref{thm:equivalent-conj}, we only need to verify Conjecture \ref{conj-3}. The argument of \cite{li:fano-3fold} applies with minor modifications. We assume $\iota(X)= 2$ for simplicity. The case $\iota(X)= 3,4$ follows from a similar proof as below.

Let $E\in \cA^{\overline{b}}(X)$ be a $\overline{b}$-stable object. By Lemma \ref{lem-tensor-H} and \ref{lem-large-volume-converse}, we may assume that $\overline{b}(E)\in [0,1)$. By Lemma \ref{lem-delta-JH}, we know that $\oh_X(2H)$,$\oh_X(H)$,$\oh_X[1]$, $\oh_X(-H)[1]\in \cA^{\overline{b}}(X)$ are $\nu_{\overline{b},w}$-stable for any $w>\overline{b}^2/2$. If $\bDelta(E)>0$, then the slope comparison as in \cite[Section 2]{li:fano-3fold} gives
\begin{equation}\label{eq-index2-1}
    \Hom_X(\oh_X(H),E)=\Hom_X(E,\oh_X(-H)[1])=0,
\end{equation}
and when $\overline{b}\in (0,1)$, we have
\begin{equation}\label{eq-index2-2}
    \Hom_X(\oh_X(2H),E)=\Hom_X(E,\oh_X[1])=0.
\end{equation}
If $\bDelta(E)=0$, then $\overline{b}(E)=\bv_1(E)/\bv_0(E)$. Since $E\in \cA^{\overline{b}}(X)$, we see that $\bv_0(E)<0$. Then from $\bDelta(E)=0$ and $\overline{b}(E)=\bv_1(E)/\bv_0(E)\in [0,1)$, we see $\bv_0(E)<\bv_1(E)\leq 0$ and $$\bv_2(E)=\frac{\bv_1(E)^2}{2\bv_0(E)}=\frac{1}{2}\overline{b}\bv_1(E)\leq 0.$$ Therefore, for any $0<\overline{b}<b<1$ we have $E\in \cA^{b}(X)$ and
\begin{align*}
\nu_{b,w}(\oh_X(H))-\nu_{b,w}(E)=&\frac{\frac{1}{2}-w}{1-b}-\frac{\bv_2(E)-w\bv_0(E)}{\bv_1^{b}(E)}\\
\geq & \frac{\frac{1}{2}-w}{1-b}+\frac{w\bv_0(E)}{\bv_1^{b}(E)} \\
= & \frac{\bv_0(E)((\frac{1}{2}-w)\overline{b}+w-\frac{1}{2}b)}{(1-b)\bv_1^{b}(E)}.
\end{align*}
Note that if we take $w=\frac{b^2}{2}$, then
\[(1-2w)\overline{b}+2w-b=(1-\overline{b})b^2-b+\overline{b}<0\]
for any $b=1-\epsilon$ and $0<\epsilon\ll 1$. Therefore, we can find $b<1$ sufficiently close to $1$ and $w>\frac{b^2}{2}$ sufficiently close to $\frac{b^2}{2}$ with 
\[\nu_{b,w}(\oh_X(H))>\nu_{b,w}(E).\]
This implies \eqref{eq-index2-1} when $\bDelta(E)=0$. The argument for \eqref{eq-index2-2} is similar, but we choose $b<2$ sufficiently close to $2$ instead.

Therefore, \eqref{eq-index2-1} and Lemma \ref{lem-serre-duality} give
\[\chi(\oh_X(H),E)\leq \dim_{\kk} \Hom_X(\oh_X(H),E)+\dim_{\kk} \Hom_X(\oh_X(H),E[2])=0.\]
Then applying Lemma \ref{lem-hrr-Fano} and \eqref{eq-barbeta}, we have
\begin{align*}
0\geq & \chi(\oh_X(H),E)\\
=&\bv^{\overline{b}}_3(E)+\left(\frac{\overline{b}^2}{2}+\frac{1}{d(X)}-\frac{1}{6}\right)\bv_1^{\overline{b}}(E)+\left(\frac{\overline{b}^3}{6}+\overline{b}\left(\frac{1}{d(X)}-\frac{1}{6}\right)\right)\bv_0(E).
\end{align*}
Note that the second term on the right-hand side is always non-negative by $d(X)\leq 6$ and $\bv_1^{\overline{b}}(E)\geq 0$, and the third term is non-negative if $\bv_0(E)\geq 0$ or $\overline{b}=0$. So it remains to deal with the case when $\bv_0(E)< 0$ and $\overline{b}\in (0,1)$.

In this case, a similar argument as above by using \eqref{eq-index2-2} and Lemma \ref{lem-serre-duality} gives
\[\chi(\oh_X(2H),E)\leq \dim_{\kk} \Hom_X(\oh_X(2H),E)+\dim_{\kk} \Hom_X(\oh_X(2H),E[2])=0.\]
Then applying Lemma \ref{lem-hrr-Fano} and \eqref{eq-barbeta}, we have
\begin{align*}
0\geq & \chi(\oh_X(2H),E)\\
=&\bv^{\overline{b}}_3(E)+\left(\frac{\overline{b}^2}{2}-\overline{b}+\frac{1}{d(X)}+\frac{1}{3}\right)\bv_1^{\overline{b}}(E)+\left(\frac{\overline{b}^3}{6}-\frac{\overline{b}^2}{2}+\overline{b}\left(\frac{1}{d(X)}+\frac{1}{3}\right)-\frac{1}{d(X)}\right)\bv_0(E).  
\end{align*}
As $d(X)\leq 6$ and $\bv_1^{\overline{b}}(E)\geq 0$, the second term is non-negative. And from $\bv_0(E)<0$ and $\overline{b}\in (0,1)$, the last term is non-negative as well. So we obtain $\bv^{\overline{b}}_3(E)\leq 0$ as desired.
\end{proof}

\subsection{Singular Calabi--Yau threefolds}\label{subsec:bmt-cy}
In \cite{FKLR}, the above reductions of Conjecture \ref{conj-2} are applied to smooth Calabi--Yau threefolds. Using results in our paper, we can extend them to mild singular cases. 

We recall the definition of Brill--Noether number.

\begin{definition}[{\cite[Definition 1.2]{FKLR}}]\label{def:bnc}
Let $C$ be an integral projective Gorenstein curve of (arithmetic) genus $g$. We define the \emph{Brill--Noether number} $\bn_C$ of $C$ as
\begin{equation*}
    \bn_C \coloneqq \lim_{t\to 0^+} \sup \left\{ \frac{\dim_{\kk}\mathrm{H}^0(E)}{\rk(E)} \colon \ \text{$E$ is a stable sheaf on $C$ with } \frac{\deg(E)}{\rk(E)}\in (g-1-t,g-1]    \right\}. 
\end{equation*}
\end{definition}

The following is proved by the same argument as \cite[Theorem 3.1]{FKLR}.

\begin{theorem}\label{thm:main-criterion}
Let $(X, H)$ be a polarised normal projective threefold with rational Gorenstein singularities such that $K_X$ is numerically trivial, $\mathrm{H}^1(\cO_X)=0$, and either $X$ is lci or $\QQ$-factorial. Given a surface $S\in |H|$ and a curve $C\in |H_S|$. Assume either

\begin{enumerate}
    \item $S$ and $C$ are both smooth with
    \begin{equation*}
        \bn_C< \chi(\cO_X(H)),
    \end{equation*}
    or

    \item $S$ has rational singularities and $C$ is integral with
    \begin{equation*}\label{eq:thm-bn-sing}
        \bn_C< \chi(\cO_X(H))-1.
    \end{equation*}
\end{enumerate}

Then \ref{BGn} holds for $(X,H)$ and some $\epsilon>0$. In particular, Conjecture \ref{conj-2} holds for $(X,H)$, a class $\Gamma\in \mathsf{E}(X)_{\QQ}$ with $\Gamma.H\geq 0$, and $(b,w)$ in the range \eqref{eq:restricted-range}.
\end{theorem}

\begin{proof}
Note that $(X, H)$ has the standard BG function by Theorem \ref{thm-bg-normal}. By \cite[Proposition 3.10, 3.11]{FKLR}, we know that \ref{BGn} holds for $(S, H_S)$. Applying Lemma \ref{prop:surface-to-3fold-1}, \ref{BGn} holds for $(X, H)$ as well. Now, using Theorem \ref{thm:ch2}, Conjecture \ref{conj-2} holds for $(X,H)$ as desired.
\end{proof}

Analogously to \cite[Corollary 5.1]{FKLR}, we obtain:

\begin{corollary}\label{cor:trivial-cor-basepoint-free}
Let $(X, H)$ be a polarised normal projective threefold with rational Gorenstein singularities, such that $K_X$ is numerically trivial, $\mathrm{H}^1(\cO_X)=0$, and either $X$ is lci or $\QQ$-factorial. Suppose that there exists a surface $S\in |H|$ and an integral curve $C\in |H_S|$ such that one of the following conditions is satisfied:

\begin{enumerate}
    \item Both $S$ and $C$ are smooth (e.g. $\dim \mathrm{Sing}(X)=0$ and $H$ is basepoint-free) with $$\td_2(X).H> \frac{1}{3}H^3+1.$$

    \item Both $S$ and $C$ are smooth with $2H$ very ample and  $$\td_2(X).H> \frac{1}{3}H^3+\frac{1}{2}.$$

    \item $S$ has rational singularities and $$\td_2(X).H> \frac{1}{3}H^3+2.$$
\end{enumerate}
Then \ref{BGn} holds for $(X,H)$ and some $\epsilon>0$. In particular, Conjecture \ref{conj-2} holds for $(X,H)$, a class $\Gamma\in \mathsf{E}(X)_{\QQ}$ with $\Gamma.H\geq 0$, and $(b,w)$ in the range \eqref{eq:restricted-range}.
\end{corollary}

Using this criterion, many examples in \cite[Section 5]{FKLR} can be generalized to singular cases. 

\begin{example}
Let $M$ be a Fano fourfold of index $3$ that is either lci or $\QQ$-factorial, $X\in |-K_M|$ be a general divisor, and $H\coloneqq -\frac{1}{3}(K_M)|_X$. Then $(X, H)$ is a polarised normal projective threefold with rational singularities and is either lci or $\QQ$-factorial, such that $K_X=0$ and $\mathrm{H}^1(\cO_X)=0$. When $M$ is smooth, the statement \ref{BGn} for $(X,H)$ is established in \cite[Theorem 5.2]{FKLR}. Using \cite[Proposition 4.10]{FKLR} and Theorem \ref{thm:main-criterion}, a similar argument as in \cite[Theorem 5.2]{FKLR} verifies \ref{BGn} for $(X,H)$ in this singular setting.
\end{example}

Moreover, we can apply Theorem \ref{thm:main-criterion} to singular Calabi--Yau threefolds which do not admit a smoothing. A typical example is the following.

\begin{corollary}\label{cor:x8}
Let $X$ be a general degree $8$ hypersurface in the weighted projective space $\PP(1,1,1,2,3)$ and $H\coloneqq (\cO_{\PP(1,1,1,2,3)}(3))|_X$. Then $K_X=0$ and $\mathrm{H}^1(\cO_X)=0$. Moreover,

\begin{itemize}
    \item $X$ is a normal projective Gorenstein variety with isolated quotient non-lci singularities and does not admit a smoothing, and

    \item \ref{BGn} holds for $(X,H)$ and some $\epsilon>0$. In particular, Conjecture \ref{conj-2} holds for $(X,H)$, a class $\Gamma\in \CH_1(X)_{\QQ}$ with $\Gamma.H\geq 0$, and $(b,w)$ in the range \eqref{eq:restricted-range}.
\end{itemize}
\end{corollary}

\begin{proof}
By \cite[Proposition 2.4]{wang:exponential}, we know that $X\subset \PP(1,1,1,2,3)$ is quasi-smooth. Moreover, by \cite[Proposition 2.6]{wang:exponential}, the unique singular point of $X$ is an isolated quotient singularity of type $\frac{1}{3}(1,1,1)$. In particular, $X$ has isolated rational Gorenstein $\QQ$-factorial singularities, and has no smoothing by \cite{schlessinger:rigid}. We also know that $H$ is an ample (Cartier) divisor on $X$ and $|H|$ is basepoint-free.

Now, we take a general surface $S\in |H|$ and a general curve $C\in |H_S|$, which are both smooth. By definition, we can write $S=X\cap T_1$ and $C=X\cap T_1\cap T_2$, where $T_1,T_2\subset \PP(1,1,1,2,3)$ are general degree $3$ hypersurfaces. Consider the surface $S'\coloneqq T_1\cap T_2$. By \cite[Proposition 2.6]{wang:exponential}, $S'$ only has $A_1$-singularities. Moreover, by adjunction, we have $K_{S'}=(\cO_{\PP(1,1,1,2,3)}(-2))|_{S'}$ and $(-K_{S'})^2=6$. In particular, $S'$ is a del Pezzo surface of degree $6$ with rational Gorenstein singularities. Since $C\in |-4K_{S'}|$,  we can apply Proposition \ref{prop:fklr} to get $$\bn_C\leq \frac{49}{4}-\frac{1}{76}<14.$$ Now, the result follows from $\chi(\cO_X(H))=14$ and Theorem \ref{thm:main-criterion}(a).
\end{proof}

\begin{proposition}[{\cite[Proposition 4.10]{FKLR}}]\label{prop:fklr}
Let $S$ be a del Pezzo surface with rational Gorenstein singularities and set $H=-K_S$. Then for any even integer $s>0$ and any integral curve $C\in |sH|$, we have
\[\bn_C\leq \max\left\{ s, 1+\frac{H^2(s^2-1)}{8}+\frac{H^2-8}{s(s+2)H^2+8}\right\}.\]
\end{proposition}

\begin{proof}
Compared to the odd integer case in \cite[Proposition 4.10]{FKLR}, the only cases that need different calculations are when $P_1=Q$ and when $P_1$ lies on $OQ$. In the former case, a direct computation gives 
\[\bn_C\leq 1+\frac{H^2(s^2-1)}{8}+\frac{H^2-8}{s(s+2)H^2+8},\]
while the latter case gives 
\[\bn_C\leq \frac{(s+1)(s^2H^2-2sH^2+8)}{4(2s-1)}\leq \max\left\{ s, 1+\frac{H^2(s^2-1)}{8}+\frac{H^2-8}{s(s+2)H^2+8}\right\}.\]
\end{proof}

We end this section with the quintic example. Conjecture \ref{conj-2} for smooth quintic threefolds is proved in \cite{Li19} for $\Gamma=0$ and $(b,w)$ in the range \eqref{eq:restricted-range}. Using the above results, the argument in \cite{Li19} applies to singular quintic as well.

\begin{theorem}[{After \cite{Li19}}]\label{thm:sing-quintic}
Let $X\subset \PP^4_{\kk}$ be a quintic normal threefold with rational singularities and $H=c_1(\cO_{\PP^4_{\kk}}(1)|_X)$. Then Conjecture \ref{conj-2} holds for $(X,H)$, $\Gamma=0$, and $(b,w)$ in the range \eqref{eq:restricted-range}.
\end{theorem}

\begin{proof}
By Lemma \ref{lem:reduce-bmt-to-bn}, \cite[Theorem 3.2]{Li19} holds without any change. Similar to Theorem \ref{thm:ch2}, \cite[Proposition 3.3]{Li19} is also valid in this case. So it remains to verify the statement of \cite[Theorem 5.5]{Li19} for $X$. Using Theorem \ref{thm-bg-normal} and Lemma \ref{lem-naoki}, it suffices to prove the statement of \cite[Proposition 5.2]{Li19} for a general surface $S\in |2H|$. Note that $S$ is a normal lci surface with rational singularities and a general curve $C\in  |2H_S|$ is smooth. Therefore, using Lemma \ref{lem-naoki} and the same argument as in \cite[Proposition 5.2]{Li19}, the result follows from \cite[Proposition 4.1]{Li19}.
\end{proof}

\begin{remark}
The inequality in \cite[Theorem 1.1]{xu:gepner} also works in this case. Moreover, the main results of \cite{liu:bg-ineqaulity-quadratic,koseki:stability-cy-solid} hold in the case of normal rational singularities as well.
\end{remark}


\section{Kuznetsov components of singular Fano threefolds}\label{sec:ku}

In this section, we focus on Kuznetsov components of singular Fano threefolds. As another application of our framework, we construct stability conditions on Kuznetsov components of certain singular Fano threefolds. Moreover, its relative version proves a singular analog of \cite[Conjecture 1.8]{kuznetsov:fano-threefold-degneration} (cf.~Corollary \ref{cor-ks-conj}).

\subsection{Relative exceptional collections}

We first review the notion of relative exceptional collections.

\begin{definition}
Let $f\colon X\to S$ be a proper flat morphism between Noetherian schemes. We say an object $E\in \Dperf(X)$ is a \emph{relative exceptional object in $\Db(X)$ over $S$} if $$Rf_*R\cH om_X(E,E)\cong \oh_S.$$ A \emph{relative exceptional collection in $\Db(X)$ over $S$} is a sequence $E_1,\cdots, E_m$ of relative exceptional objects in $\Db(X)$ over $S$ such that $Rf_*R\cH om_X(E_i,E_j)\cong 0$ for all $i>j$.
\end{definition}

By \cite[Lemma 3.22]{BLMNPS21}, a sequence $E_1,\cdots, E_m\in \Dperf(X)$ is a relative exceptional collection in $\Db(X)$ over $S$ if and only if $(E_1)_s,\cdots, (E_m)_s\in \Dperf(X_s)$ is an exceptional collection in $\Db(X_s)$ for every point $s\in S$, or equivalently for every closed point $s\in S$.

We first recall the following result from \cite{BLMNPS21}.

\begin{lemma}[{\cite[Lemma 3.23]{BLMNPS21}}]\label{lem-10.2}
Let $f\colon X\to S$ be a proper morphism of finite Tor-dimension between Noetherian schemes. If $E\in \Dperf(X)$ is a relative exceptional object in $\Db(X)$ over $S$, then there is a fully faithful functor
\[\alpha_E\colon \Db(S)\to \Db(X),\quad F\mapsto Lf^*F\otimes^{\LL} E,\]
which has a right adjoint. Moreover, if $\omega^{\bullet}_{f}$ is invertible, then $\alpha_E$ has a left adjoint.
\end{lemma}

Using the properties of dualizing complexes (cf.~Lemma \ref{lem-serre-duality}), an argument similar to that of \cite[Lemma 3.25]{BLMNPS21} gives:

\begin{lemma}\label{lem-10.3}
Let $f\colon X\to S$ be a proper morphism of finite Tor-dimension between Noetherian schemes and $E_1,\cdots, E_m\in \Dperf(X)$ be a relative exceptional collection. Then there is an $S$-linear semi-orthogonal decomposition  of finite cohomological amplitude $$\Db(X)=\langle \cD_1, \alpha_{E_1}(\Db(S)),\cdots, \alpha_{E_m}(\Db(S)) \rangle$$ such that the inclusion $\cD_1\hookrightarrow \Db(X)$ has a left adjoint.

Furthermore, if $f$ is Gorenstein and $S$ is regular, then all components above are admissible; in particular, the decomposition above is strong.
\end{lemma}

\begin{proof}
The first part is proved in \cite[Lemma 3.25]{BLMNPS21}. Note that in the situation of the second part, $\omega^{\bullet}_{f}$ is a shift of a line bundle. By Lemma \ref{lem-10.2} and \cite[Lemma 3.10]{perry:noncommutative-hpd}, $$\cD\coloneqq \langle \alpha_{E_1}(\Db(S)),\cdots, \alpha_{E_m}(\Db(S))\rangle$$ is admissible. Then we have a semi-orthogonal decomposition $\Db(X)=\langle \cD, \cD'_1 \rangle$ so that $\cD'_1$ is admissible. As $S$ is regular, we have $\cD\subset \Dperf(X)$. Therefore, by applying Lemma \ref{lem-serre-duality}, we obtain $\cD_1'\otimes^{\mathrm{L}} \omega^{\bullet}_f=\cD_1$. Since $-\otimes^{\mathrm{L}} \omega^{\bullet}_f$ is an auto-equivalence of $\Db(X)$, we can conclude that $\cD_1$ is also admissible.
\end{proof}

\subsection{Inducing stability conditions}

Now, we state some general criteria on inducing hearts or stability conditions on semi-orthogonal components.

The following result is a variant of \cite[Corollary 7.6]{BLMNPS21}.

\begin{proposition}\label{prop-induce-heart}
Let $f\colon X\to S$ be a flat projective morphism to a regular Nagata scheme $S$ of finite Krull dimension which is quasi-projective over a Noetherian affine scheme. Let $\cA_S\subset \Db(X)$ be the heart of a bounded $S$-local t-structure. Let $$\Db(X)=\langle \cD_1, \cD_2 \rangle$$ be an $S$-linear semi-orthogonal decomposition. Let $\cG$ be a relative spanning class of $\cD_2$ with $\cG\subset \cA_S\cap \Dperf(X)$ and every $G\in \cG$ satisfies $$G_s\otimes^{\LL} \omega_{X_s}^{\bullet}\in \cA_s[1]$$ for every closed point $s\in S$, where $\cA_s\subset \Db(X_s)$ is the heart induced by \cite[Theorem 5.6]{BLMNPS21}. Then
\[(\cA_S)_1\coloneqq\cA_S\cap \cD_1\subset \cD_1\]
is the heart of a bounded $S$-local t-structure on $\cD_1$ such that the inclusion functor $\cD_1\to \Db(X)$ is t-exact.
\end{proposition}

\begin{proof}
Let $F\in \cA_S$ and $G\in \cG$. By Lemma \ref{lem-S-perf-lem-1}(c) and Lemma \ref{lem-S-perf-lem-2}(b), we know that $F_s, G_s\in \Db(X_s)$ for every $s\in S$. Moreover, $F$ is $S$-perfect by our assumption and Lemma \ref{lem-S-perf-lem-1}(c). Therefore, $R\cH om_X(G,F)\in \Db(X)$ is $S$-perfect by Lemma \ref{lem-S-perf-lem-2}(c), and we get \[Li^*_sRf_*R\cH om_X(G,F)\cong \RHom_{X_s}(G_s, F_s)\]
by Lemma \ref{lem-S-perf-lem-2}(a), where $i_s\colon s\hookrightarrow S$ is the inclusion of a closed point. From Lemma \ref{lem-serre-duality}, we see
\[\RHom_{X_s}(G_s, F_s)^{*}=\RHom_{X_s}(F_s, G_s\otimes^{\LL}\omega_{X_s}^{\bullet}).\]
Therefore, our assumption implies that $\RHom_{X_s}(G_s, F_s)\in \Db(\kappa(s))^{\leq 1}$. Since this holds for any closed point $s\in S$, we see that $Rf_*R\cH om_X(G,F)\in \Db(S)^{\leq 1}$ and the result follows from \cite[Lemma 7.4]{BLMNPS21}.
\end{proof}

The same argument as in \cite[Proposition 5.1]{bayer2017stability} yields the following criterion.

\begin{proposition}\label{prop-blms}
Let $X$ be a projective scheme over a field $\kk$ and $E_1,\dots, E_m$ be an exceptional collection. We set $\cD_1\coloneqq\langle E_1,\dots, E_m \rangle^{\perp}$. Let $\sigma=(\cA, Z)$ be a weak stability condition on $\Db(X)$ such that

\begin{enumerate}
    \item $E_i\in \cA\cap \Dperf(X)$, 

    \item $E_i\otimes^{\LL}\omega_{X}^{\bullet}\in \cA[1]$

    \item $Z(E_i)\neq 0$, and

    \item there are no objects $0\neq F\in \cA_1\coloneqq\cA\cap \cD_1$ with $Z(F)=0$.
\end{enumerate}
Then $\sigma_1=(\cA_1, Z_1\coloneqq Z|_{\KK(\cD_1)})$ is a stability condition on $\cD_1$.
\end{proposition}

\begin{proof}
By (a), (b), and Proposition \ref{prop-induce-heart}, we know that $\cA_1$ is the heart of a bounded t-structure on $\cD_1$ such that $\cD_1\to \Db(X)$ is t-exact. Moreover, from (d), it is clear that $Z_1$ is a stability function on $\cA_1$. Therefore, $Z_1$ has the HN property by \cite[Lemma 5.2]{bayer2017stability}. Now the same proof as in \cite[Proposition 5.1]{bayer2017stability} implies the support property and the result follows.
\end{proof}

Now we can prove a variant of \cite[Theorem 23.1]{BLMNPS21}. Given a weak stability condition $\underline{\sigma}=(\sigma_s=(\cA_s, Z_s))_{s\in S}$ on $\Db(X)$ over $S$ with respect to a relative Mukai homomorphism $$\bv\colon \Knum(\Db(X)/S)\to \Lambda.$$ If $E_1,\dots, E_m$ is a relative exceptional collection in $\Db(X)$ over $S$, then we set
\[\Lambda'\coloneqq\langle \bv(\Knum(\cD_1/S)) \rangle\subset \Lambda,\] where $\cD_1\coloneqq\langle \alpha_{E_1}(\Db(S)),\dots, \alpha_{E_m}(\Db(S)) \rangle^{\perp}$.

\begin{theorem}\label{thm-induce-stab}
Let $f\colon X\to S$ be a Gorenstein projective morphism to a regular Nagata scheme $S$ of finite Krull dimension which is quasi-projective over a Noetherian affine scheme. Let $\underline{\sigma}=(\sigma_s=(\cA_s, Z_s))_{s\in S}$ be a weak stability condition on $\Db(X)$ over $S$ and $E_1,\dots, E_m$ be a relative exceptional collection. We set $$\cD_1\coloneqq\langle \alpha_{E_1}(\Db(S)),\dots, \alpha_{E_m}(\Db(S)) \rangle^{\perp}.$$ Assume that

\begin{enumerate}
    \item $(E_i)_s\in \cA_s$ for all $i$ and $s\in S$,

    \item $(E_i)_s\otimes^{\LL} \omega_{X_s}^{\bullet}\in \cA_s[1]$ for all $i$ and $s\in S$,

    \item $\bv(E_i)\notin \Lambda^{Z}$ for all $i$,

    \item $\Lambda^{Z}\cap \Lambda'=0$, and

    \item for all $\bv\in \Lambda$, the set
    \[\{[F]\in \cM_{\underline{\sigma}}(\bv') \colon \bv'\in \bv+\Lambda^{Z},~\chi(E_i, F)\geq 0 \text{ for all }i\}\] is bounded.
\end{enumerate}
For each $s\in S$, let $\cA_{s,1}$ be the heart in $(\cD_1)_s$ given in Proposition \ref{prop-induce-heart}, and let $Z_{s,1}$ be the central charge given by the restriction of $Z_s$ along $\KK((\cD_1)_s)\to \KK_0(X_s)$. Then the collection
\[\underline{\sigma_1}=((\sigma_s)_1=(\cA_{s,1}, Z_{s,1}))_{s\in S}\]
is a stability condition on $\cD_1$ over $S$ with respect to $\Lambda'$.
\end{theorem}

\begin{proof}
By Lemma \ref{lem-10.3}, $\cD_1$ is a strong $S$-linear semi-orthogonal component of finite cohomological amplitude; in particular, it satisfies Assumption \ref{assum-stab}. By Proposition \ref{prop-blms}, we know that $\underline{\sigma_1}$ is a collection of numerical stability conditions on $\cD_1$ over $S$. Then the verification of conditions in  Definition \ref{def-flat-collection} and Definition \ref{def-stab-family} follows from the same proof as \cite[Theorem 23.1]{BLMNPS21} after replacing \cite[Corollary 7.6]{BLMNPS21} and \cite[Proposition 5.1]{bayer2017stability} by Proposition \ref{prop-induce-heart} and Proposition \ref{prop-blms}, respectively.
\end{proof}

\subsection{Kuznetsov components of singular Fano threefolds}



By Kodaira's vanishing theorem, for a Fano threefold $X$, we always have $\rH^{i}(X,\oh_X)=0$ for all $i\neq 0$. In particular, any line bundle on $X$ is an exceptional object in $\Db(X)$. As a generalization of \cite[Theorem 6.7]{bayer2017stability}, we have:

\begin{theorem}\label{thm:stab-ox}
Let $f\colon X\to S$ be an admissible morphism to a regular Nagata scheme $S$ of finite Krull dimension which is quasi-projective over a Noetherian affine scheme of characteristic zero. Assume that each geometric fiber is a Fano threefold. Then we have an $S$-linear strong semi-orthogonal decomposition of finite cohomological amplitude
\[\Db(X)=\langle \cR_{X/S},\alpha_{\oh_X}(\Db(S)) \rangle\]
such that there exists a stability condition on $\cR_{X/S}$ over $S$.
\end{theorem}

\begin{proof}
Note that $f$ is automatically Cohen--Macaulay by our assumption. Let $\underline{\sigma}=\underline{\sigma}^{b,\frac{a^2+b^2}{2},b}$ constructed using $\cL\coloneqq\omega_{X/S}^{\vee}$ and $\gamma=1$ as in Proposition \ref{prop-rotate-tilt}. Note that for any $s\in S$, we have $(\omega_{X/S}^{\bullet})_s=\omega_{X_s}^{\bullet}$ by \cite[\href{https://stacks.math.columbia.edu/tag/0E2Y}{Tag 0E2Y}]{stacks-project}. We will verify assumptions in Theorem \ref{thm-induce-stab}.

By Lemma \ref{lem-delta-JH} and Lemma \ref{lem-large-volume-converse}, we know that $\oh_{X_s}, \omega_{X_s}[1]\in \cA^{b}(X_s)$ is $\mu_{a,b}$-stable for any $a>0$ and $-1<b<0$. Next, as
\[\mu_{a,b}(\oh_{X_s})=\frac{\frac{b^2-a^2}{2}}{-b}>0>\frac{\frac{1}{2}+b+\frac{b^2-a^2}{2}}{-1-b}=\mu_{a,b}(\omega_{X_s}[1])\]
when $-\frac{1}{2}\leq b<0$ and $0<a<-b$ or $-1<b<-\frac{1}{2}$ and $0<a< 1+b$, we see that $$\oh_{X_s},\omega_{X_s}[2]\in \cA^{a,b}(X_s)=\cA^{b,\frac{a^2+b^2}{2},b}(X_s)$$ for $a,b$ in this range, which verifies (a) and (b) in Theorem \ref{thm-induce-stab}. Moreover, assumptions (c) and (d) are clear. Finally, (e) follows from the strong boundedness result in Theorem \ref{thm:tilt-HN-structure}(a).
\end{proof}

For an index $2$ Fano threefold $X$, since $K_X=-2H$, we always have $$\rH^{i}(X,\oh_X(-H))=0$$ for all $i$ by Kodaira's vanishing theorem. Therefore, $\oh_X,\oh_X(H)$ is an exceptional collection of $X$. In this situation, we have the following generalization of \cite[Theorem 6.8]{bayer2017stability} and \cite[Corollary 26.2]{BLMNPS21}.

\begin{theorem}\label{thm:stab-index-2}
Let $f\colon X\to S$ be an admissible morphism to a regular Nagata scheme $S$ of finite Krull dimension which is quasi-projective over a Noetherian affine scheme of characteristic zero. Assume that there exists a relative ample line bundle $\cO_{\cX}(1)$ on $X$ such that for each geometric point $t\to S$, $X_t$ is an index $2$ Fano threefold and $\cO_{X_t}(-2)=\omega_{X_t}$. Then we have an $S$-linear strong semi-orthogonal decomposition of finite cohomological amplitude
\[\Db(X)=\langle \cB_{X/S}, \alpha_{\oh_X}(\Db(S)), \alpha_{\cO_{\cX}(1)}(\Db(S)) \rangle\]
such that there exists a stability condition on $\cB_{X/S}$ over $S$.
\end{theorem}

\begin{proof}
Let $\underline{\sigma}=\underline{\sigma}^{{b,\frac{a^2+b^2}{2},b}}$ constructed using $\cL\coloneqq \cO_{\cX}(1)$ and $\gamma=1$ as in Proposition \ref{prop-rotate-tilt}. We will verify assumptions in Theorem \ref{thm-induce-stab}.

By Lemma \ref{lem-delta-JH} and Lemma \ref{lem-large-volume-converse}, we know that $G, G(-2)[1]\in \cA^{b}(X_s)$ is $\mu_{a,b}$-stable for any $a>0$ and $-1<b<0$, where $G\in \{\oh_{X_s}, \cL_s\}$. Next, by a computation of slopes, we see that $G, G(-2)[2]\in \cA^{a,b}(X_s)$ when $-\frac{1}{2}\leq b<0$ and $0<a<-b$ or $-1<b<-\frac{1}{2}$ and $0<a<1+b$, where $G\in \{\oh_{X_s}, \cL_s\}$. This verifies (a) and (b) in Theorem \ref{thm-induce-stab}. Moreover, assumptions (c) and (d) are clear. Finally, (e) follows from the strong boundedness result in Theorem \ref{thm:tilt-HN-structure}.
\end{proof}

Finally, we discuss the index $1$ case. For an index $1$ prime Fano threefold $X$, when $g(X)$ is even, we say a rank $2$ vector bundle $\cU_X$ is a \emph{Mukai bundle} if it is $\mu_{-K_X}$-stable and satisfies

\begin{itemize}
    \item $\det(\cU_X)\cong \oh_X(K_X)$,

    \item $\rH^i(X,\cU_X)=0$ for all $i$, and 
    
    \item $\RHom_X(\cU_X, \cU_X)=\kk$.
\end{itemize}
In particular, $\oh_X, \cU^{\vee}_X$ is an exceptional collection in $\Db(X)$. When $X$ is factorial, terminal, and $g(X)\geq 6$, the existence and the uniqueness of a Mukai bundle is known (cf.~\cite{bayer:mukai-bundle,bayer:mukai-model-fano-var}).

In the case of index $1$ Fano threefolds, we have the following.

\begin{theorem}\label{thm-ku-index-1}
Let $f\colon X\to S$ be an admissible morphism to a regular Nagata scheme $S$ of finite Krull dimension which is quasi-projective over a Noetherian affine scheme of characteristic zero. Assume that for each geometric point $t\to S$, $X_t$ is an index $1$ prime Fano threefold of even genus $g\geq 6$ and there exists a vector bundle $\cU_X$ on $X$ such that $\cU_{X_t}$ is a Mukai bundle. Assume furthermore that $f$ has a smooth fiber. Then we have an $S$-linear strong semi-orthogonal decomposition of finite cohomological amplitude
\[\Db(X)=\langle \cA_{X/S},\alpha_{\oh_X}(\Db(S)),\alpha_{\cU^{\vee}_X}(\Db(S)) \rangle\]
such that there exists a stability condition on $\cA_{X/S}$ over $S$.
\end{theorem}

\begin{proof}
The same argument in \cite[Theorem 6.9]{bayer2017stability} works once we know that $\cU_{X_s}^{\vee},\cU_{X_s}[1]\in \cA^{-\epsilon}(X_s)$ are both $\nu_{-\epsilon,w}$-stable for any $w>\frac{\epsilon^2}{2}$ and $0<\epsilon \ll 1$ as in \cite[Lemma 6.11]{bayer2017stability}. To prove this, using Lemma \ref{lem-delta-0-rk1}(c) instead of \cite[Proposition 2.14]{bayer2017stability} in the proof of \cite[Lemma 6.11]{bayer2017stability}, it suffices to prove that $\cU_{X_s}^{\vee}$ and $\cU_{X_s}[1]\in \cA^{0}(X_s)$ are $\nu_{0,w}$-stable for any $w>0$.

In the following, we only prove this for $\cU_{X_s}^{\vee}$, as the argument for $\cU_{X_s}[1]$ is similar. Since $f$ has a smooth fiber, from Lemma \ref{lem-constant-chM}, \ref{lem:constant-generalization}, and the computation in the smooth case, we get $$\bv_{\leq 2}(\cU^{\vee}_{X_{s}})=\left(2d(X), d(X), \frac{g(X)}{2} -2\right).$$ By Theorem \ref{thm:wall-chamber-tilt}, if $\cU_{X_s}^{\vee}$ is strictly $\nu_{0,w_0}$-semistable for some $w_0>0$, then we have an exact sequence
\[0\to A\to \cU^{\vee}_{X_{s}}\to B\to 0\]
of $\nu_{0,w_0}$-semistable objects in $\cA^0(X_s)$ such that $\nu_{0,w_0}(A)=\nu_{0,w_0}(\cU_{X_s}^{\vee})=\nu_{0,w_0}(B)$. Then $A$ is a torsion-free sheaf. If we set $$\bv_{\leq 2}(A)=(d(X)r,d(X)x,d(X)y),$$ then $r\in \ZZ_{\geq 1}$. By $\bv_1(A), \bv_1(B)>0$, we also have $0<x<1$. From $\nu_{0,w_0}(A)=\nu_{0,w_0}(\cU_{X_s}^{\vee})$, we get
\[y-x\left(\frac{g}{2d}-\frac{2}{d}\right)=(r-2x)w_0.\]
Since $w_0>0$, we obtain
\[(r-2x)\left(y-x\left(\frac{g}{2d}-\frac{2}{d}\right)\right)\geq 0.\]
Together with inequalities
$\bDelta(A)+\bDelta(B)\leq \bDelta(\cU^{\vee}_{X_{s}})$ and $\bDelta(A),\bDelta(B)\geq 0$ from Lemma \ref{lem-delta-JH}, the only possibility is $r=2x=1$ and $\bv_{\leq 2}(A)=\bv_{\leq 2}(B)$, which cannot be a wall. This implies the $\nu_{0,w}$-stability of $\cU^{\vee}_{X_{s}}$ for any $w>0$ and the result follows.
\end{proof}


As a corollary, we have the following singular analog of \cite[Conjecture 1.8]{kuznetsov:fano-threefold-degneration}. Recall that for each $1\leq d\leq 5$, there exists a morphism $\cX\to B$ between varieties over $\kk$ constructed in \cite[Theorem 3.6]{kuznetsov:fano-threefold-degneration} such that $B$ is a smooth curve, the total space $\cX$ is smooth, $\cX_{o}$ is a $1$-nodal index $1$ prime Fano threefold of genus $2d+2$ for a closed point $o\in B$, and $\cX_b$ is a smooth index $1$ prime Fano threefold of genus $2d+2$ for each $b\in B\setminus o$. Moreover, there exists a vector bundle $\cU_{\cX}$ on $\cX$ such that $\cU_{\cX_b}$ is a Mukai bundle on $\cX_b$ for each $b\in B$ and we have a $B$-linear semi-orthogonal decomposition
\[\Db(\cX)=\langle i_{o*}P_{\cX_o}, \bar{\cA}_{\cX}, \alpha_{\oh_{\cX}}(\Db(B)), \alpha_{\cU^{\vee}_{\cX}}(\Db(B)) \rangle \]
such that $P_{\cX_o}$ is a $\PP^{\infty,2}$ object on the central fiber $\cX_{o}$ and $i_{o*}\colon \cX_o\hookrightarrow\cX$ is the inclusion. Moreover, $\bar{\cA}_{\cX}$ is smooth and proper over $B$.

In \cite[Conjecture 1.8]{kuznetsov:fano-threefold-degneration}, it is expected that there exists a stability condition on $\bar{\cA}_{\cX}$ over $B$, which then may lead to applications of \cite[Theorem 3.6]{kuznetsov:fano-threefold-degneration} in studying moduli spaces. When $B$ is the spectrum of a complete DVR, this is proved by \cite{LMPSZ:deformation}. Applying Theorem \ref{thm-ku-index-1} to this setting, we obtain an analog of \cite[Conjecture 1.8]{kuznetsov:fano-threefold-degneration} by replacing $\bar{\cA}_{\cX}$ with
\[\cA_{\cX/B}=\langle i_{o*}P_{\cX_o}, \bar{\cA}_{\cX}\rangle.\]

\begin{corollary}\label{cor-ks-conj}
In the setting above, if $d\geq 2$, then there exists a stability condition on $\cA_{\cX/B}$ over $B$.
\end{corollary}

\begin{proof}
Since each geometric fiber of $\cX\to B$ is either smooth or has at most one node, it has terminal Gorenstein lci singularities. In particular, $\cX\to B$ is lci and Theorem \ref{thm-ku-index-1} applies.
\end{proof}


\begin{appendix}

\section{An alternative approach to tilt-stability}\label{appendix:approach}
In this appendix, we discuss another construction of tilt-stability on an arbitrary projective scheme.

Let \( X \) be a projective scheme of dimension \( n \geq 2 \) over a field \( \kk \), and let \( H \) be an ample divisor. One can define a version of tilt-stability on \( \Db(X) \) as follows. For any object \( E \in \Db(X) \), let \( \alpha_i(E) \) be the rational number such that \( i! \, \alpha_i(E) \) is the coefficient of the degree \( i \) term in the Hilbert polynomial of \( E \) with respect to \( H \). We set $$\rk(E)\coloneqq \frac{\alpha_n(E)}{\alpha_n(\cO_X)},$$ $$\deg_H(E)\coloneqq \alpha_{n-1}(E)-\rk(E)\alpha_{n-1}(\cO_X),$$ and $$\mathsf{c}^H_2(E)\coloneqq \alpha_{n-2}(E)-\rk(E)\alpha_{n-2}(\cO_X)-\frac{\alpha_{n-1}(\cO_X)}{\alpha_n(\cO_X)}\deg_H(E).$$
When $X$ is lci and integral, by Lemma \ref{lem-chM-general}(a), we have 
\[\rk(E)H^n=\bch_0(E).H^{n},\]
\[\deg_H(E)=\bch_1(E).H^{n-1},\]
and
\[\mathsf{c}^H_2(E)=\bch_2(E).H^{n-2} +\left(\td_1(X).H^{n-2}-\frac{H^{n-1}.\td_1(X)}{H^n}H^{n-1}\right).\bch_1(E).\]
Thus, if the numerical class of $K_X$ is proportional to $H$, we have
\[\bv_{H,\leq 2}(E)=(\rk(E)H^n, \deg_H(E), \mathsf{c}^H_2(E)).\]

We define a pair 
\[\tau_H\coloneqq (\Coh(X),\ -\deg_H(-) + \mathfrak{i} \rk(-)),\]
then it is easy to check that $\tau_H$ is a weak stability condition on $\Db(X)$ with respect to the lattice $\mathsf{H}$, which is the image of the homomorphism
\[\Knum(X)\to \QQ^{n+1},\quad \xi\mapsto (\alpha_0(\xi), \alpha_1(\xi),\dots,\alpha_n(\xi)).\]
Analogously to the Le Potier function defined in Section \ref{sec:bg}, we set
\[\Psi_{X, H}(x)\coloneqq \limsup_{\mu\to x}\left\{\frac{\mathsf{c}_2^H(E)}{\rk(E)}\colon E \text{ is a }\tau_H\text{-semistable torsion-free sheaf with }\mu_{\tau_H}(E)=\mu\right\}\in \RR\cup\{\pm \infty\}.\]

As in Theorem \ref{thm:exist-bg-function}, we have the following Bogomolov--Gieseker inequality.

\begin{theorem}\label{thm-bg-general}
Let $X$ be a projective scheme of dimension $n\geq 2$ over a field $\kk$ and $H$ be an ample divisor. Then there exists a constant $\mathsf{N}_{X, H}\geq 0$ such that 
\[(\deg_H(E))^2-2\rk(E)H^n \mathsf{c}^H_2(E)\geq -\mathsf{N}_{X, H} (\rk(E)H^n)^2\]
for any $\tau_{H}$-semistable torsion-free sheaf $E\in \Coh(X)$.
\end{theorem}

\begin{proof}
As in the proof of Theorem \ref{thm:exist-bg-function}, we may assume that $\overline{\kk}=\kk$ and $H$ is very ample. Fix a finite surjective morphism $\pi\colon X\to \PP^n$ such that $\pi^*h=H$, where $h$ is the hyperplane class on $\PP^n$. Then for any $E\in \Coh(X)$, we have
\[\rk(\pi_*E)=\alpha_n(\cO_X)\rk(E),\]
\[\deg_h(\pi_*E)=\deg_H(E)+\rk(\pi_*E)\left(\frac{\alpha_{n-1}(\cO_X)}{\alpha_n(\cO_{X})}-\frac{n+1}{2}\right),\]
and
\begin{align*}
\mathsf{c}^h_2(\pi_*E)=&\mathsf{c}^H_2(E)+\left(\frac{\alpha_{n-1}(\cO_X)}{\alpha_n(\cO_X)}-\frac{n+1}{2}\right)\deg_H(E)\\
&+\left(\alpha_{n-2}(\cO_X)-\frac{(n+1)(3n+2)}{24}\alpha_n(\cO_X)-\frac{n+1}{2}\alpha_{n-1}(\cO_X)+\frac{(n+1)^2}{4}\alpha_n(\cO_X)\right)\rk(E).
\end{align*}
Note that for any torsion-free $\mu_h$-semistable sheaf $F\in \Coh(\PP^n)$, we have
\[(\deg_h(F))^2-2\rk(F) \mathsf{c}^h_2(F)\geq 0.\]
Therefore, after fixing a surjection $\cO_{\PP^n}(-N)^{\oplus m}\twoheadrightarrow \pi_*\cO_X$, the argument in Step 1 of Theorem \ref{thm:exist-bg-function} applies to this case. In particular, we can take
\[\mathsf{N}_{X, H}\coloneqq N^2+\left( \frac{\alpha_{n-1}(\cO_X)}{\alpha_n(\cO_X)}\right)^2-\frac{2\alpha_{n-2}(\cO_X)}{\alpha_n(\cO_X)}-\frac{n+1}{12}.\]
\end{proof}

We define an open subset $V_{X, H}\subset \RR^2$ as
\[V_{X, H}\coloneqq \{(s,t)\in \RR^2\colon t>\Psi_{X, H}(s)\}.\]
Then we have the following version of Theorem \ref{thm-bms}.

\begin{theorem}\label{thm:tilt-hilbert}
Let $X$ be a projective scheme of dimension $n\geq 2$ over a field $\kk$ and $H$ be an ample divisor. Then we have a continuous injective map
\[V_{X,H}\to \Stab_{\mathsf{H}}^{\mathsf{w}}(\Db(X)), \quad (s,t)\mapsto \tau^{s,t}_H=(\Coh^{sH}(X), -\mathsf{c}^H_2(-) + t \rk(-) + \mathfrak{i}(\deg_H(-) - s \rk(-))),\]
where the heart $\Coh^{sH}(X)$ is defined by tilting of $\tau_H$ at the slope $s$ as in Section \ref{subsec-tilting-property}.

Moreover, there is a wall-chamber structure on $V_{X, H}$ as described in Theorem \ref{thm:wall-chamber-abstract}.
\end{theorem}

\begin{proof}
By \cite[Corollary 33]{kollar:duality} and Lemma \ref{lem-general-S2}(d), $\tau_H$ satisfies tilting property. Therefore, it is easy to check that $\tau_H$ meets all assumptions in Theorem \ref{thm:tilt-stability}, and the first statement follows. Moreover, by Theorem \ref{thm-bg-general}, we have $\Psi_{X, H}<+\infty$, so Theorem \ref{thm:wall-chamber-abstract} gives a wall-chamber structure on $V_{X, H}$.
\end{proof}

In particular, when $\dim X=2$, this approach yields a family of stability conditions on $\Db(X)$. Furthermore, the relative results established in Section \ref{sec:tilt-3} extend naturally to this setting. Notably, the central charges defined via this method are automatically locally constant in families, thereby obviating the need to verify the admissibility of morphisms.

However, the Bogomolov-Gieseker-type inequalities in this setting are not sharp for line bundles, even in the smooth case, unless $K_X$ is proportional to $H$. Because this sharpness is essential for investigating the BMT conjecture, we do not pursue this approach in the main text of the paper.


\end{appendix}

\bibliography{singular}

\bibliographystyle{alpha}

\end{document}